\documentclass[12pt, reqno]{amsart}
\usepackage{amssymb}
\usepackage{graphicx}
\usepackage{xcolor} 
\usepackage{tensor}
\usepackage[color=yellow]{todonotes}

\usepackage{enumitem}
\setlist[enumerate]{leftmargin=*, label=(\arabic*)}
\setlist[itemize]{leftmargin=2em}

\usepackage{fullpage} 

\usepackage{hyperref}

\usepackage{xifthen} 

\definecolor{green}{rgb}{0,0.8,0} 

\newtheorem{theorem}{Theorem}[section]
\newtheorem{corollary}[theorem]{Corollary}
\newtheorem{lemma}[theorem]{Lemma}
\newtheorem{proposition}[theorem]{Proposition}
\theoremstyle{definition}
\newtheorem{definition}[theorem]{Definition}

\theoremstyle{remark}
\newtheorem{remark}[theorem]{Remark}
\numberwithin{equation}{section}

\newcommand{\tA}{{\tilde A}}

\newcommand{\ta}{{\tilde a}}
\newcommand{\tb}{{\tilde b}}
\newcommand{\tF}{{\tilde F}}
\newcommand{\tf}{{\tilde f}}

\newcommand{\A}{{\mathbf A}}

\newcommand{\nrm}[1]{\Vert#1\Vert}
\newcommand{\abs}[1]{\vert#1\vert}
\newcommand{\brk}[1]{\langle#1\rangle}
\newcommand{\set}[1]{\{#1\}}

\newcommand{\tr}{\textrm{tr}}

\newcommand{\aleq}{\lesssim}
\newcommand{\ageq}{\gtrsim}

\newcommand{\lap}{\Delta}

\newcommand{\ud}{d}
\newcommand{\rd}{\partial}
\newcommand{\nb}{\nabla}

\newcommand{\alp}{\alpha}
\newcommand{\bt}{\beta}
\newcommand{\gmm}{\gamma}

\newcommand{\dlt}{\delta}
\newcommand{\Dlt}{\Delta}
\newcommand{\eps}{\epsilon}

\newcommand{\sgm}{\sigma}
\newcommand{\Sgm}{\Sigma}

\newcommand{\Tht}{\Theta}

\newcommand{\omg}{\omega}

\newcommand{\bfa}{{\bf a}}
\newcommand{\bfb}{{\bf b}}

\newcommand{\bfw}{{\bf w}}

\newcommand{\bfA}{{\bf A}}
\newcommand{\bfB}{{\bf B}}

\newcommand{\bfD}{{\bf D}}

\newcommand{\bfF}{{\bf F}}

\newcommand{\bfH}{{\bf H}}

\newcommand{\bfQ}{{\bf Q}}

\newcommand{\bfW}{{\bf W}}

\newcommand{\bbN}{\mathbb N}

\newcommand{\bbR}{\mathbb R}
\newcommand{\R}{\mathbb R}
\newcommand{\bbS}{\mathbb S}

\newcommand{\bbZ}{\mathbb Z}

\newcommand{\calB}{\mathcal B}
\newcommand{\calC}{\mathcal C}

\newcommand{\calE}{\mathcal E}
\newcommand{\calF}{\mathcal F}
\newcommand{\calG}{\mathcal G}

\newcommand{\calQ}{\mathcal Q}

\newcommand{\weakto}{\rightharpoonup}

\newcommand{\pfstep}[1]{\vskip.5em \noindent {\bf #1.}}
\newcommand{\pfsubstep}[1]{\vskip.5em \noindent {\it #1.}}

\newcommand{\covD}[1][]{%
\ifthenelse{\isempty{#1}}{\bfD}{{\bfD^{(#1)}}}%
}

\newcommand{\Cal}{\mathrm{Cal}}
\newcommand{\Str}{\mathrm{Str}}
\newcommand\DA{{\mathbf{DA}}}
\newcommand\DB{{\mathbf{DB}}}

\newcommand{\tw}{\tilde w}

\renewcommand{\div}{\mathrm{div}\,}
\newcommand{\curl}{\mathrm{curl}\,}

\newcommand{\g}{\mathfrak  g}
\newcommand{\G}{\mathbf{G}}
\newcommand{\la}{\langle}
\newcommand{\ra}{\rangle}

\renewcommand{\P}{\mathbf P}
\newcommand{\rst}{\!\upharpoonright}		

\newcommand{\spE}{\calE_{e}} 
\newcommand{\Egs}{\calE_{GS}} 

\newcommand{\hM}{\calQ} 

\newcommand{\nE}{{\calE}} 
\newcommand{\cstr}{c^{Str}}
\newcommand{\cstrd}{c^{Str,[\delta_{0}]}}

\newcommand{\blue}[1]{{
#1}}

\newcommand{\LieBr}[2]{[#1, #2]}

\newcommand{\smexp}{\sgm_{0}}
\newcommand{\dltStr}{\dlt_{s}}

\newcommand{\smin}{\sgm^{(0)}}
\newcommand{\pmin}{p^{(0)}}
\newcommand{\smid}{{\sgm^{(1)}}}
\newcommand{\pmid}{p^{(1)}}
\newcommand{\smidp}{{\sgm^{(2)}}}
\newcommand{\pmidp}{p^{(2)}}
\newcommand{\smax}{\sgm^{(3)}}
\newcommand{\pmax}{p^{(3)}}
\newcommand{\smaxp}{\sgm^{(4)}}
\newcommand{\pmaxp}{p^{(4)}}

\begin{document}

\title[]{The Yang--Mills heat flow and the caloric gauge }
\author{Sung-Jin Oh}%
\address{Department of Mathematics, UC Berkeley, Berkeley, CA 94720 and KIAS, Seoul, Korea 02455}%
\email{sjoh@math.berkeley.edu}%

\author{Daniel Tataru}%
\address{Department of Mathematics, UC Berkeley, Berkeley, CA, 94720}%
\email{tataru@math.berkeley.edu}%


\date{\today}%
\begin{abstract}
This is the first part of the four-paper sequence, which establishes the Threshold Conjecture and the Soliton Bubbling vs.~Scattering Dichotomy for the energy critical hyperbolic Yang--Mills equation in the (4 + 1)-dimensional Minkowski space-time.

The primary subject of this paper, however, is another PDE, namely the energy critical Yang--Mills heat flow on the 4-dimensional Euclidean space. Our first goal is to establish sharp criteria for global existence and asymptotic convergence to a flat connection for this system in $\dot{H}^{1}$, including the Dichotomy Theorem (i.e., either the above properties hold or a harmonic Yang--Mills connection bubbles off) and the Threshold Theorem (i.e., if the initial energy is less than twice that of the ground state, then the above properties hold). Our second goal is to use the Yang--Mills heat flow 
in order to define the caloric gauge, which will play a major role in the analysis of the hyperbolic Yang--Mills equation in the subsequent papers.

\end{abstract}
\maketitle
\setcounter{tocdepth}{2}
\tableofcontents

\section{Introduction}
The goal of this paper is two-fold:
\begin{itemize}
\item {\it To develop a large data global theory of the Yang--Mills heat flow on $\bbR^{4}$.} Consider the Yang--Mills heat flow on $\bbR^{4}$ with a compact structure group. For initial data $a \in \dot{H}^{1}(\bbR^{4})$, we establish sharp criteria for global existence and asymptotic convergence to a flat connection, including the Dichotomy Theorem (Theorem~\ref{thm:dich-simple}) and the Threshold Theorem (Theorem~\ref{thm:thr-simple}).

\item {\it To define the caloric gauge for the hyperbolic Yang--Mills equation.} Using the large data global theory of the Yang--Mills heat flow, we define the so-called \emph{caloric gauge} (Definition~\ref{d:C}) and identify the structure of the hyperbolic Yang--Mills equation in this gauge (Theorem~\ref{t:main-wave}). 
\end{itemize}

While this paper is primarily devoted to analysis of the Yang--Mills heat flow, in the larger scheme of things it constitutes the first part of a four-paper sequence, whose overall aim is to prove the Threshold Conjecture and the Dichotomy Theorem for the energy critical \emph{hyperbolic} Yang--Mills equation in $\bbR^{1+4}$. The four installments of the series are concerned with the following topics:
\begin{enumerate}
\item the \emph{caloric gauge} for the hyperbolic Yang--Mills equation, present paper.
\item large data \emph{energy dispersed} caloric gauge solutions, \cite{OTYM2}.
\item \emph{topological classes} of connections and large data local well-posedness, \cite{OTYM2.5}
\item \emph{soliton bubbling} vs. scattering dichotomy for large data solutions, \cite{OTYM3}.
\end{enumerate}
A short overview of the whole sequence is provided in the survey paper \cite{OTYM0}.

In the remainder of the introduction, we formulate the three Yang--Mills equations that play a role in this paper, namely the harmonic Yang--Mills equation (elliptic), the Yang--Mills heat flow (parabolic) and the hyperbolic Yang--Mills equation. Then in Section~\ref{sec:results}, the main results are stated in a more precise form, along with discussion of some major ideas.

\subsection{Lie groups and Lie algebras}
Let $\G$ be a compact noncommutative Lie group and $\g$ its associated Lie
algebra. We denote by $Ad(O) X = O X O^{-1}$ the action of $\G$ on
$\g$ by conjugation (i.e., the adjoint action), and by $ad(X) Y = [X,
Y]$ the associated action of $\g$, which is given by the Lie
bracket. We introduce the notation $\brk{X, Y}$ for a bi-invariant inner product on $\g$,
\begin{equation*}
\la [X,Y],Z \ra = \la X, [Y,Z] \ra, \qquad X,Y,Z \in \g, 
\end{equation*}
or equivalently 
\begin{equation*}
\la X,Y \ra = \la Ad(O) X,  Ad(O) Y  \ra, \qquad X,Y \in \g, \quad O \in \G. 
\end{equation*}
If $\G$ is semisimple then one can take 
$\brk{X, Y} = -\tr(ad(X) ad(Y))$ i.e. negative of the Killing form on $\g$, which is then positive definite,
However, a bi-invariant inner product on $\g$ exists for any compact Lie group $\G$.

\subsection{Connections and curvature}
The objects of study here are connection $1$-forms taking values in the Lie algebra $\g$,
\[
A_{j}: \R^{d}\rightarrow \g.
\]
The associated covariant differentiation  operators $\covD_{j} = (\covD[A])_j$,  acting on $\g$-valued functions $B$,
are defined by
\[
\covD_{j} B:= \partial_{j} B + ad(A_{j}) B.
\]
 Their commutators yield the curvature tensor
\[
F_{jk}: = \partial_{j}A_{k} - \partial_{k}A_{j} +[A_j,A_k],
\]
in the sense that $\covD_{j} \covD_{k} - \covD_{k} \covD_{j} = ad(F_{jk})$. A basic property of $F$ is that it satisfies the \emph{Bianchi identity}:
\begin{equation*}
	\covD_{\alp} F_{\bt \gmm} + \covD_{\bt} F_{\gmm \alp} + \covD_{\gmm} F_{\alp \bt} = 0.
\end{equation*}
Given a $\G$-valued function $O$, its action $B \to Ad(O) B$ induces a gauge transformation for $A$,
namely 
\[
A_k \rightarrow O A_k O^{-1} -\partial_k O  O^{-1} =: \calG(O) A_{k}.
\]
Correspondingly, we have for $F$
\[
F_{jk} \to O F_{jk} O^{-1}.
\]

\subsection{Yang--Mills equations}
While this article is primarily devoted to the Yang--Mills heat flow, there are in effect three Yang--Mills equations
which play a role in our work. These are as follows:

\subsubsection{The harmonic Yang--Mills equations in the 
Euclidean space $\R^d$.}
This is obtained as the Euler--Lagrange equation for the Lagrangian (or the energy)
\begin{equation}\label{Le}
\spE[A]: = \frac{1}{2}\int_{\R^{d}} \la F_{jk}, F^{jk}\ra \,dx,
\end{equation}
and has the form
\begin{equation}\label{YM-e}
\covD^j F_{jk} = 0 
\end{equation}
Both the Lagrangian and the harmonic Yang--Mills equation are invariant with respect to gauge transformations,
therefore in order to have a good theory for these equations one needs to fix the gauge. A common choice 
here is the \emph{Coulomb gauge},
\begin{equation}\label{Coulomb-e}
\partial^j A_j = 0 
\end{equation}
which formally turns the equations \eqref{YM-e} into a strongly elliptic system,
\begin{equation}\label{YM-e-Coulomb}
\Delta_A A_k = - [A^j, \covD_k A_j] 
\end{equation}
where $\Delta_A$ is the covariant Laplacian, given by
\begin{equation}\label{DeltaA}
\Delta_A = \covD^j \covD_j.
\end{equation}

\subsubsection{The Yang--Mills heat flow in $\R^+ \times \R^d$}
This can be viewed as the gradient flow for the above Lagrangian, but
written in a gauge invariant fashion. Using the letter $s$ for the
heat-time, we add a heat-time connection component $A_s$ and
the corresponding curvatures $F_{sj}$.  Then the covariant Yang--Mills
heat flow has the form
\begin{equation}\label{caloric}
F_{sj} = \covD^{\ell} F_{\ell j}, \qquad A_j(s=0) = a_j
\end{equation}
The solutions of the harmonic Yang--Mills equation play the role of
steady states for the Yang--Mills heat flow.

This flow is also gauge invariant. A natural way to fix the gauge is
via the \emph{de Turck gauge condition}
\begin{equation}\label{deTurck}
A_s = \partial^k A_k, 
\end{equation}
which formally turns the system \eqref{caloric} into a semilinear strongly parabolic system,
\begin{equation}\label{YM-h-deTurck}
(\partial_s - \Delta_A)A_j = [A_j, \partial^k A_k] - [A^k, \partial_j A_k].
\end{equation}

However, there is a second gauge choice which plays the leading role
in this article, namely the \emph{local caloric gauge}\footnote{The word \emph{local} is here to differentiate this gauge with the global caloric gauge defined in Definition~\ref{d:C} below.},
\begin{equation}\label{local-caloric}
A_s = 0.
\end{equation}
In this gauge  the system \eqref{caloric} becomes a semilinear degenerate parabolic system, 
\begin{equation}\label{YM-h-lc}
(\partial_s - \Delta_A)A_j = - \covD^k \partial_j A_k
\end{equation}
where the degenerate part occurs at the level of the divergence of
$A$, namely 
\begin{equation}
\partial_s (\partial^k A_k) = [ A^j,\covD^kF_{jk}].
\end{equation}
This would seem to be less favorable from an analytic point of
view. However, as it turns out, under this gauge the long-time
behavior is better. Incidentally, this is exactly the gauge which
corresponds to directly taking the gradient flow for the Lagrangian in
\eqref{Le}.

\begin{remark} \label{rem:ymhf}
In the literature, the gradient flow $\rd_{s} A_{j} = \covD^{\ell} F_{\ell j}$ is usually called the \emph{Yang--Mills heat flow}. In our work, however, we find it conceptually and technically useful to adopt the fully gauge-covariant formulation \eqref{caloric}, and view this flow as the equation in a gauge defined by \eqref{local-caloric}. Of course, these viewpoints are equivalent.
\end{remark}

The first part of the paper will be devoted to the study of global 
solutions for the Yang--Mills heat flow.

\subsubsection{ The hyperbolic Yang--Mills equation in the 
Minkowski space 
$\R^{1+d}$.}
Let $\bbR^{1+d}$ be the $(d+1)$-dimensional Minkowski space, equipped with the Minkowski metric $\mathrm{diag}(-1, +1, \cdots, +1)$ in the rectangular coordinates $(x^{0}, x^{1}, \ldots, x^{d})$. 

The \emph{hyperbolic Yang--Mills equations} are the Euler-Lagrange
equations associated with the formal Lagrangian action functional
\begin{equation}\label{Lh}
\mathcal{L}[A]: = \frac{1}{2}\int_{\R^{1+d}} \la F_{\alpha\beta}, F^{\alpha\beta}\ra \,dx dt.
\end{equation}
Here we are using the standard convention for raising indices using the Minkowski metric, and
greek letters for the Minkowski setting. The (hyperbolic) time is
denoted by $t$, and corresponds to the index $0$ (i.e., $t = x^{0}$).  Thus, the
hyperbolic Yang--Mills equations take the form
\begin{equation}\label{ym}
\covD^\alpha F_{\alpha \beta} = 0.
\end{equation}
In order to consider the Yang--Mills problem as an evolution equation
we need to consider initial data sets. An \emph{initial data set} for
\eqref{ym} consists of two pairs of 1-forms $(a_{j}, e_{j})$  on $\bbR^{d}$. 
We say that $(a_{j},e_{j})$ is the initial data for a Yang--Mills wave $A$ if
\begin{equation*}
	(A_{j}, F_{0 j}) \rst_{\set{t=0}} = (a_{j}, e_{j}).
\end{equation*}
Note that \eqref{ym} for $\beta = 0$ imposes the condition that the following
equation be true for any initial data for \eqref{ym}:
\begin{equation} \label{eq:YMconstraint}
	\covD^{j} e_{j} = 0.
\end{equation}
where $\covD^j$ denotes the covariant derivative with respect to the $a_j$
connection.  This equation is the \emph{Gauss} (or the
\emph{constraint}) \emph{equation} for \eqref{ym}.

We observe again that harmonic Yang--Mills connections play the role of
steady states for the hyperbolic Yang--Mills evolution. However, here
we have an additional class of symmetries, namely the Lorentz group.
Taking a Lorentz transform of a steady state yields a soliton, which
evolves with constant speed less than $1$. It is a simple computation 
to verify that the energy of this soliton is larger than the energy of the original 
harmonic Yang--Mills connection.

Yet again, \eqref{ym} is gauge invariant. There are several interesting gauge choices one can make for the 
hyperbolic Yang--Mills equations:

\begin{itemize}
\item The Lorenz gauge, 
\begin{equation}\label{lorenz}
\partial^\alpha A_\alpha = 0
\end{equation}
In this gauge the  hyperbolic Yang--Mills equations  become a semilinear wave system
\begin{equation}\label{YM-h-Lorenz}
\Box_A A_\alpha = - [A^\beta, \covD_\alpha A_\beta] 
\end{equation}
In particular it has finite speed of propagation.  This gauge is very
convenient for local well-posedness for large but regular data.
Unfortunately there are multiple technical difficulties if one tries
to implement such a gauge in the low regularity setting, see
e.g. \cite{MR3519539}.

\item the temporal gauge, 
\begin{equation}\label{temporal}
A_0 = 0
\end{equation}
This is akin to the local caloric gauge for the heat flow. In particular 
at the level of the divergence of $A_x$ we again get a pure transport equation. 
In this gauge the  Yang--Mills system is still strictly hyperbolic, and in
particular it has finite speed of propagation.  Because of this, it is also
convenient for local well-posedness for large but regular data.
Unfortunately working with either low regularity solutions or long time solutions
runs into difficulties largely caused by the lack of decay/dispersion in the transport part.

\item The Coulomb gauge
\begin{equation}\label{coulomb}
\partial^j A_j = 0.
\end{equation}
where only the spatial divergence is considered.
Here the causality is lost; however, the Coulomb gauge is an
``elliptic'' gauge which captures well the null structure of the
problem, and thus works well in low regularity settings. Indeed, the
Coulomb gauge was used in \cite{KT} to prove the small data result for
this problem in the critical Sobolev space. Unfortunately, for large data 
there are issues with the Coulomb gauge.
\end{itemize}

The aim of the second part of this paper will be to use the Yang--Mills
heat flow in the local caloric gauge in order to introduce a new gauge
choice for the hyperbolic Yang--Mills flow, which we call the \emph{caloric gauge}. The final objective here is to arrive at a good
formulation of the hyperbolic Yang--Mills equation in the caloric
gauge.

\subsection{Energy, scaling and criticality}
Here we review the standard energy, scaling and criticality considerations 
which apply to the three Yang--Mills problems described above.

\subsubsection{The harmonic Yang--Mills equation}
In the context of the harmonic Yang--Mills equation, the Lagrangian $\mathcal{L}_e[A]$ plays the 
role of the energy of the connection $A$; we will suggestively use the alternate notation $\calE_{\bbR^{d}}[A]$
for it.

On the other hand, the equations \eqref{YM-e} also have a scale invariance property,
\[
A(x) \to \lambda A(\lambda x).
\]
The Sobolev space with the same scaling is $\dot H^{\frac{d-2}2}$,
which we will view as the natural space for solutions to
\eqref{YM-e}. We will refer to this space as the critical Sobolev
space.  The energy is scale invariant in dimension $d = 4$, which we
will refer to as the \emph{energy critical dimension}.

\subsubsection{The Yang--Mills heat flow}

Here the energy plays the role of a Lyapunov functional,
\begin{equation}\label{Lyapunov}
\frac{d}{ds} \spE[A] = - \int_{\R^d} |\covD^j F_{jk}|^2 dx
\end{equation}
The scale invariance property now reads 
\[
A(x, s) \to \lambda A(\lambda x, \lambda^2 s).
\]
The critical Sobolev space  for the initial data $a$ is again $\dot H^{\frac{d-2}2}$,
and the energy critical dimension is $d = 4$ as well.

\subsubsection{The hyperbolic Yang--Mills equation}

Here the gauge invariant energy is given by
\[
\calE_{\set{t} \times \bbR^{d}}[A] = \frac12  \int_{\set{t} \times \R^d} \sum_{\alp < \bt} |F_{\alpha \beta}|^2 dx.
\]
The scale invariance property has the form 
\[
A(t,x) \to \lambda A(\lambda t, \lambda x).
\]
The critical Sobolev space  for the initial data $(a,e)$ is 
$\dot H^{\frac{d-2}2} \times \dot H^{\frac{d-4}2}$,
and the energy critical dimension is again $d = 4$.

\addtocontents{toc}{\protect\setcounter{tocdepth}{-1}}
\subsection*{Acknowledgments} 
Part of the work was carried out during the semester program ``New
Challenges in PDE'' held at MSRI in Fall 2015. S.-J. Oh was supported
by the Miller Research Fellowship from the Miller Institute, UC
Berkeley and the TJ Park Science Fellowship from the POSCO TJ Park
Foundation. D. Tataru was partially supported by the NSF grant
DMS-1266182 as well as by the Simons Investigator grant from the
Simons Foundation.

\addtocontents{toc}{\protect\setcounter{tocdepth}{2}}

\section{ The main results} \label{sec:results} In this section, we
present the main results proved in this paper. We focus primarily on
the energy critical dimension $d = 4$, although some of our results
and techniques easily extend to higher dimensions; see
Remark~\ref{rem:higher-d} below.

\subsection{Review of results for harmonic Yang--Mills 
connections on  \texorpdfstring{$\bbR^{4}$}{R4}} \label{subsec:har-ym} 
Here we do not consider any new
results, but instead recall the known results concerning the harmonic
Yang--Mills equation.

We start with the classical elliptic regularity result:
\begin{theorem}[\cite{MR648356}]
  All harmonic Yang--Mills connections in $H^1_{loc}$ are smooth up to
  a gauge transformation $O \in H^2_{loc}$ (or equivalently $O_{;x}
  \in H^1_{loc}$).
\end{theorem}
Key to this theorem is existence of a good gauge (namely, the local
Coulomb gauge) in which the harmonic Yang--Mills equation becomes a
nice elliptic system. The importance of a judicious choice of gauge to
reveal essential analytic features of the equation is a theme that we
will see repeatedly below for the other Yang--Mills equations.

Another fundamental issue, which for example arises in the study of
moduli spaces \cite{Don}, is the behavior of $H^{1}_{loc}$ connections
with uniformly bounded energy $\spE$. We recall the following basic
results:
\begin{theorem} \label{thm:weak-conv} Consider a sequence $a^{n}$ of
  $H^{1}_{loc}$ harmonic Yang--Mills connections with uniformly bounded energies $\spE[a^{n}]
  \leq \spE < \infty$. Then:
  \begin{enumerate}
  \item (Uhlenbeck compactness) After passing to a subsequence,
    $a^{n}$ is weakly convergent in $H^{1}$ up to $H^{2}_{loc}$ gauge
    transformations.
  \item (Small energy) If $\spE$ is sufficiently small, then after
    passing to a subsequence, $a^{n}$ converges strongly in
    $H^{1}_{loc}$ to the flat connection up to $H^{2}_{loc}$ gauge
    transformations.
  \item (Dichotomy) Generally, one of the following two scenarios must
    hold:
    \begin{enumerate}[label=(\alph*)]
    \item After passing to a subsequence, $a^{n}$ converges strongly
      in $H^{1}_{loc}$ up to $H^{2}_{loc}$ gauge transformations.
    \item The sequence ``bubbles off'' a nontrivial harmonic
      Yang--Mills connection. More precisely, there exists a finite
      set of points $\Sgm$ such that, after passing to a subsequence,
      $a^{n}$ converges strongly in $H^{1}_{loc}$ on $\bbR^{4}
      \setminus \Sgm$ up to $H^{2}_{loc}$ gauge
      transformations. Moreover, for each $x_{0} \in \Sgm$, there
      exist sequences $x_{n} \to x_{0}$ and $r_{n} \to 0$ such that
      the rescaled sequence
      \begin{equation*}
        b^{n} (x) = r_{n} a^{n}(x_{n} + r_{n} x)
      \end{equation*}
      converges strongly in $H^{1}_{loc}$ to a nontrivial harmonic
      Yang--Mills connection.
    \end{enumerate}
  \end{enumerate}
\end{theorem}
\noindent
For more sophisticated results on the structure of the possible
singular set $\Sgm$  we refer to \cite{Riv}.

By part (2), we see that
\begin{equation} \label{eq:GS-en} \Egs = \inf \set{\spE[Q] : \hbox{$Q$
      is a nontrivial harmonic Yang--Mills connection on $\bbR^{4}$}}
\end{equation}
is strictly positive. Part (3) then implies that, if $\Egs < \infty$,
then there exists a harmonic Yang--Mills connection $Q$, which we call
a \emph{ground state}, such that
\begin{equation*}
  \spE[Q] = \Egs.
\end{equation*}

For noncommutative compact structure groups, we indeed have $\Egs <
\infty$. This nontrivial fact comes from the beautiful interplay among
the harmonic Yang--Mills equation on $\bbR^{4}$, topology and the
theory of Lie groups. To property describe it, we need to introduce
the concept of \emph{topological classes}.

For a compact base manifold, such as $\bbS^{4}$, this term refers to
the isomorphism classes of principal $\G$-bundles which supports the
connection. On the other hand, for $\bbR^{4}$, which is contractible
and thus supports only the trivial fiber bundles, a topological class
must be interpreted rather as a property of a connection. If $a$ is a
harmonic Yang--Mills connection on $\bbR^{4}$, then consider its
pullback on $\bbS^{4} \setminus \set{point}$ by the stereographic
projection. By conformal invariance, it is a harmonic Yang--Mills
connection on $\bbS^{4} \setminus \set{point}$ with the same energy,
and by the singularity removal theorem of Uhlenbeck \cite{U2}, it can
be uniquely extended to a principal $\G$-bundle on the whole sphere
$\bbS^{4}$. We identify its isomorphism class with the topological
class of $a$.

\begin{remark} 
  More generally, if $A$ is smooth with a rapidly decaying curvature
  $F$, then it uniquely determines a principal $\G$-bundle on the
  one-point compactification $\bbS^{4}$ of $\bbR^{4}$, whose
  isomorphism class may be identified with the topological class of
  $A$. In fact, this procedure essentially works under the mere
  condition that $A \in H^{1}_{loc}$ with finite energy. We refer the
  reader to \cite[Section~3]{OTYM2.5} for a precise definition of
  topological classes of such connections along these ideas.
\end{remark}

By the above discussion, it suffices to consider the harmonic
Yang--Mills connections on $\bbS^{4}$. For concreteness, we also
restrict our attention to $\G = SU(2)$, in which case we normalize $\brk{A, B} = -
\tr(AB)$. Then the Chern number $c_{2}$ of a principal
$\G$-bundle, which is always an integer, characterizes its topological
class. It may be computed from $a$ by the Chern--Weyl formula
\begin{equation} \label{eq:chern-no} c_{2}(a) = \frac{1}{8 \pi^{2}}
  \int_{\bbS^{4}} \tr(f \wedge f).
\end{equation}

One way to construct harmonic Yang--Mills connections is to look for
the absolute minimizers of the energy functional $\spE$ in a fixed
topological class; such connections are called \emph{instantons}. A
procedure due to Atiyah--Drinfeld--Hitchin--Manin \cite{ADHM} gives
explicit construction of all instantons in the case of $\G =
SU(2)$. In particular, we have:
\begin{theorem}[\cite{ADHM}] \label{thm:ADHM} Consider $\G = SU(2)$
  with $\brk{AB} = - \tr(AB)$. In every topological class of principal
  $\G$-bundles on $\bbS^{4}$, there exists an instanton (ground state)  $Q$ with
  energy
  \begin{equation*}
    \spE(Q) = 8 \pi^{2} \abs{c_{2}}.
  \end{equation*}
\end{theorem}

One may wonder if the instantons exhaust all harmonic Yang--Mills
connections. Remarkably, the answer is \emph{no}, as demonstrated by
\cite{SSU, Bor, SaSe, Parker}. Nevertheless, Gursky--Kelleher--Streets
\cite{GKS} recently proved the following lower bound:
\begin{theorem}[\cite{GKS}] \label{thm:gks-simple} Consider $\G =
  SU(2)$ with $\brk{AB} = - \tr(AB)$. If $a$ is a harmonic Yang--Mills
  connection on a principal $\G$-bundle on $\bbS^{4}$, then either $a$
  is an instanton, or
  \begin{equation*}
    \spE(a) \geq 8 \pi^{2} \abs{c_{2}} + 16 \pi^{2}.
  \end{equation*}
\end{theorem}
In particular, observe that we have a refined threshold for a
topologically trivial harmonic Yang--Mills connection: It is either
flat, or must have energy at least $2 \Egs = 16 \pi^{2}$.

For a general noncommutative Lie group $\G$, instantons may be
constructed from the $SU(2)$ version by an appropriate embedding
$su(2) \to \g$, which always exists and forms the basis of the
classical Cartan--Weyl structure theory of compact Lie
groups. Moreover, Theorem~\ref{thm:gks-simple} holds more generally
for any simple compact Lie group, and also has important implications
for general noncommutative compact Lie groups; see
Theorem~\ref{thm:gks}. For details, we refer the reader to
\cite[Section~6]{OTYM2.5}.

\subsection{Large data global theory for the Yang--Mills heat flow on
 \texorpdfstring{ $\bbR^{4}$}{R4}}
Here we consider the question of local and global well-posedness, as
well as asymptotic convergence property, for the Yang--Mills heat flow
\eqref{caloric} at energy regularity.

Here and in the rest of this paper, we restrict our attention to
connection $1$-forms in $\dot{H}^{1}$, in anticipation of the results
we will prove for the hyperbolic Yang--Mills equation (see
Section~\ref{subsec:hyp-ym}). More precisely, we will need a global
theory of the Yang--Mills heat flow only for those connections whose
hyperbolic Yang--Mills evolutions scatter. These must be in the
trivial topological class (``topologically trivial''), or
equivalently, gauge-equivalent to a $\dot{H}^{1}$ connection
\cite[Section~4]{OTYM2.5}.

In the question of well-posedness of the Yang--Mills heat flow, the
gauge choice is critical, as we would like to avoid pure gauge
singularities. Local well-posedness is most readily understood in the
\emph{de Turck gauge}, where the Yang--Mills heat flow becomes a
strongly parabolic semilinear flow. The following result is relatively
easy to establish, and is provided without proof:

\begin{theorem}\label{thm:deturck}
  The Yang--Mills heat flow \eqref{caloric} in the de Turck gauge is
  locally well-posed in $\dot H^1$, as well as globally well-posed for
  small data.
\end{theorem}
In particular the small data solutions will satisfy
\begin{equation}\label{decay-heat}
  \| A \|_{L^\infty \dot H^1 \cap L^2 \dot H^2} \lesssim \|a\|_{\dot H^1}
\end{equation}
and will decay to zero at infinity. We remark that a similar small
data result can also be proved in the Coulomb gauge; however this
gauge no longer extends to all large data.

Unfortunately, the small data global well-posedness result in the de
Turck gauge does not readily extend to large data. To understand why,
consider initial data $a$ whose curvature $f$ vanishes, $f = 0$. By
the energy dissipation relation \eqref{Lyapunov} we expect the
solution to satisfy $F_{jk} = 0 $ at all heat-times $s > 0$. Such
solutions are gauge equivalent to the zero solution, so they can be
represented as
\[
A = O_{;x}.
\]
The equation \eqref{caloric} is clearly globally solvable in $\dot
H^1$, for any extension of $O = Id$ to $s > 0$ would yield a solution.
However, if we now impose the de Turck gauge condition we arrive at
the following equation for $O$:
\[
O_{;s} = \covD^j O_{;j}
\]
which is nothing but the harmonic heat flow equation for $\G$-valued
maps. It is well-known (see \cite{CoGh, ChDi}) that this flow can
develop singularities in finite time.  Hence the same will happen for
the Yang--Mills heat flow \eqref{caloric} in the de Turck gauge.

The above discussion motivates the introduction of the \emph{local
  caloric gauge} \eqref{local-caloric} as a substitute for the de
Turck gauge. The pure gauge blow-up described above no longer arises
as pure gauge solutions are stationary in the local caloric gauge.

We start with the basic local well-posedness result in the local
caloric gauge:
\begin{theorem} \label{thm:hf-loc-simple} The Yang--Mills heat flow
  \eqref{caloric} in the local caloric gauge is locally well-posed in
  $\dot H^1$, as well as globally well-posed for small data.
\end{theorem}

A precise statement of this theorem is provided in
Section~\ref{s:caloric-loc}. However, for reader's convenience we
briefly describe here the main features:

\begin{enumerate}[label=(\alph*)]
\item Existence, uniqueness and $C^1$ local dependence on the initial
  data. However, in contrast to the case of the de Turck gauge, $C^2$
  dependence does not seem to hold.

\item Higher regularity also holds; in particular the data to solution
  map is Lipschitz in $\dot H^1 \cap \dot H^\sgm$ for all $\sgm >
  1$. Also a frequency envelope version of the result is valid.

\item The curvature $F$ satisfies global parabolic bounds,
  \[
  \|F\|_{L^\infty L^2 \cap L^2 \dot H^1} \lesssim_{\|a\|_{\dot H^1}} 1
  \]
  and decays to zero at infinity.

\item The connection $A(s)$, however, does not decay to $0$ at
  infinity. Instead, the limit
  \[
  a_\infty = \lim_{s \to \infty} A(s)
  \]
  exists in $\dot H^1$, and has zero curvature $f_{\infty} = 0$.
\end{enumerate}

Next, we introduce the gauge- and scaling-invariant space-time norm
\begin{equation*}
  \nrm{F}_{L^{3}(J; L^{3})}
\end{equation*}
which plays an important role in our study. In fact, we have the
following Structure Theorem:
\begin{theorem}[Structure Theorem] \label{thm:str-simple} Let $A$ be a
  Yang--Mills heat flow given by Theorem~\ref{thm:hf-loc-simple} on a
  heat-time interval $J$, such that
  \begin{equation} \label{eq:l3-simple} \nrm{F}_{L^{3}(J; L^{3})} \leq
    \hM < \infty.
  \end{equation}
  When $J$ is finite, $A$ can be extended as a Yang--Mills heat flow
  past its endpoint. When $J$ is infinite, properties (b)--(d) holds
  for $A$, where the implicit constants depend also on $\hM$.
\end{theorem}
Motivated by this result, for $a \in \dot{H}^{1}$ with a \emph{global}
solution $A$, we introduce the notation
\begin{equation} \label{eq:hM-def} \hM(a) = \nrm{F}_{L^{3}([0,
    \infty); L^{3})}.
\end{equation}
We will refer to $\hM(a)$ as the \emph{caloric size} of the connection $a$.
In what follows, by writing $\hM(a) < \infty$, it is implicit that $a$
has a global associated Yang--Mills heat flow $A$.

We note that many more strong conclusions about $A$ and nearby
Yang--Mills heat flows can be made from \eqref{eq:l3-simple}, which is
why we call Theorem~\ref{thm:str-simple} the Structure Theorem. We
refer the reader to Section~\ref{s:caloric-loc} for a more precise
statement and the proof.

\begin{remark} [Remark on the techniques]
  The subtlety in the proofs of Theorems~\ref{thm:hf-loc-simple} and
  \ref{thm:str-simple} lies in the fact that in the local caloric
  gauge, \eqref{YM-h-lc} is only degenerate parabolic. One way to
  handle this issue is via the \emph{de Turck trick}, which in our
  formalism amounts to working in the de Turck gauge; as discussed
  above, however, this approach is effective only locally in time. Our
  approach instead is to rely on a version of the de Turck trick for
  the linearization of the Yang--Mills heat flow; in this scheme, an
  auxiliary flow called the \emph{dynamic Yang--Mills heat flow} plays
  a major role. We refer the reader to Section~\ref{subsec:dymhf} for
  a further discussion.
\end{remark}

Next, we describe our main large data results for the Yang--Mills heat
flow.  By a blow-up analysis based on the monotonicity formula (or the
energy identity)
\begin{equation*}
  \int \frac{1}{2} \brk{F_{ij}, F^{ij}} (s_{1}) \, \ud x + \int_{s_{0}}^{s_{1}} \int \brk{\covD^{\ell} F_{\ell i}, \covD^{\ell} \tensor{F}{_{\ell}^{i}}} \, \ud x \ud s= \int \frac{1}{2} \brk{F_{ij}, F^{ij}} (s_{0}) \, \ud x,
\end{equation*}
Theorem~\ref{thm:str-simple} can be considerably strengthened as
follows:
\begin{theorem} [Dichotomy Theorem] \label{thm:dich-simple} Let $a$ be
  a connection 1-form in $\dot{H}^{1}$, and let $A$ be the solution to
  \eqref{eq:cYMHF} with initial data $A_{i}(s=0) = a_{i}$ given by
  Theorem~\ref{thm:hf-loc-simple}.  Then one of the following two
  properties must hold for the maximal solution:
  \begin{enumerate}[label=(\alph*)]
  \item The solution is global, $\hM(a) = \nrm{F}_{L^{3}([0, \infty);
      L^{3})}< \infty$, and $A(s)$ converges to a flat connection
    $a_{\infty}$ in $\dot{H}^{1}$ as $s \to \infty$.
  \item The solution ``bubbles off'' a nontrivial harmonic Yang--Mills
    connection, either
    \begin{enumerate}[label=(\roman*)]
    \item at a finite blow-up time $s < \infty$, or
    \item at infinity $s = \infty$.
    \end{enumerate}
  \end{enumerate}
\end{theorem}
A more precise form of Theorem~\ref{thm:dich-simple}, as well as its
proof, may be found in Section~\ref{s:thr}.

Theorem~\ref{thm:dich-simple} identifies the possible obstruction for
global existence and asymptotic convergence to a flat connection as
``bubbling off'' a nontrivial harmonic Yang--Mills connections. Taking
into account their theory reviewed in Section~\ref{subsec:har-ym}, we
obtain:

\begin{theorem} [Threshold Theorem]\label{thm:thr-simple}
  Let $a$ be a connection 1-form in $\dot{H}^{1}$, and let $A$ be the
  solution to \eqref{eq:cYMHF} with initial data $A_{i}(s=0) = a_{i}$
  given by Theorem~\ref{thm:hf-loc-simple}. If
  \begin{equation*}
    \spE[a] < 2 \Egs,
  \end{equation*}
  then the solution $A$ is global, $\hM(a) = \nrm{F}_{L^{3}([0,
    \infty); L^{3})}< \infty$, and $A(s)$ converges to a flat
  connection $a_{\infty}$ in $\dot{H}^{1}$ as $s \to \infty$.
\end{theorem}
For a more precise form of Theorem~\ref{thm:thr-simple}, as well as
its proof, we again refer to Section~\ref{s:thr}.

Observe that the threshold energy is $2\Egs$ instead of the obvious
value $\Egs$! This refinement is a result of taking into account the
``topological triviality'' of $\dot{H}^{1}$ connections, as well as
``topological nontriviality'' of harmonic Yang--Mills connections with
energy below $2 \Egs$, as suggested by Theorem~\ref{thm:gks-simple}
(at least for $\G = SU(2)$).

\begin{remark} [Brief historical remarks]
  The bubbling analysis, which forms the basis of the proofs of
  Theorems~\ref{thm:dich-simple} and \ref{thm:thr-simple}, has its
  origin in the classical work of Struwe \cite{Str} (see also
  Schlatter \cite{Sch1}).  In the context of a general compact
  (Riemannian) base manifold, Schlatter \cite{Sch2} proved global
  existence and (weak) asymptotic convergence under non-sharp energy
  restrictions, and a sharp threshold theorem was proved recently in
  by Gursky--Kelleher--Streets \cite{GKS}, as a corollary of their
  lower bound on the energy of non-instanton harmonic Yang--Mills
  connections.  In comparison to these works, the significance of our
  results lies in the precise asymptotics of the Yang--Mills heat flow
  on the noncompact space $\bbR^{4}$ (encapsulated by the bound
  $\hM(a) < \infty$ via Theorem~\ref{thm:str-simple}), which allows us
  to define the \emph{global caloric gauge}, to be described below. We
  also refer to the interesting recent work of Waldron on (im)possible
  finite time singularities \cite{Wal1, Wal2}, and of
  Kelleher--Streets \cite{KeSt1, KeSt2} on the structure of general
  singular set.
\end{remark}

\subsection{Preview of results for the hyperbolic Yang--Mills
 equation  in \texorpdfstring{$\bbR^{1+4}$}{R1+4}} \label{subsec:hyp-ym} 
In order to motivate the
subsequent results concerning the caloric gauge, we state here the
local and global well-posedness results for the energy critical
hyperbolic Yang--Mills equation that will be proved in the subsequent
papers \cite{OTYM2, OTYM2.5, OTYM3} of the series.

For the local and global well-posedness properties of the hyperbolic
Yang--Mills equation, the question of the gauge choice is again
paramount. We begin with some classical, higher regularity local
well-posedness results:

\begin{theorem}
  The hyperbolic Yang--Mills equation \eqref{ym} is locally well-posed
  in $H^3 \times H^2$ in both the Lorenz and the temporal gauge.
\end{theorem}

As mentioned before, both of these gauges are consistent with
causality, i.e, the corresponding evolution has finite speed of
propagation. Because of this, the large data problem is easily
localized and reduced to a small data problem. The small data problem
is perturbative in the Lorenz gauge, but slightly nonperturbative in
the temporal gauge.

While in both of these gauges one can lower somewhat the regularity of
the data, descending to (critical) energy regularity while working
directly in these gauges\footnote{We remark that, a-posteriori, we
  obtain local well-posedness at the critical regularity in the
  temporal gauge (see Theorem~\ref{t:dich-hyp}). However, its proof is
  highly indirect, and the key analysis is still performed in the
  caloric gauge \cite{OTYM2, OTYM2.5}. } appears to be fraught with
difficulties.  The same applies to the global problem; the two are in
effect equivalent at least to a certain extent.

The Coulomb gauge \eqref{coulomb}, on the other hand, provides a much
better structure for the equations. In recent work of the second
author with Krieger~\cite{KT}, the small data problem was considered
at energy regularity:

\begin{theorem}
  The hyperbolic Yang--Mills equation \eqref{ym} in the Coulomb
  gauge~\eqref{coulomb} is globally well-posed for small data in $\dot
  H^1 \times L^2$.
\end{theorem}

At the same time, blow-up solutions are known to exist in certain
cases (see \cite{KST2, MR2929728}) just above the ground state
energy. This lead one to a ``Threshold Conjecture'' similar to the one
for the Yang--Mills heat flow.

Unfortunately, it appears\footnote{At this point, one should regard
  this as a conjecture for which we have some evidence but not a
  proof.} that in general the Coulomb gauge cannot be extended to all
data below the threshold energy. Hence we need an alternative gauge
choice which should retain as much as possible of the algebraic
structure associated to the Coulomb gauge, but which is well-defined
at least for subthreshold connections.

This gauge, which we will call the \emph{(global) caloric gauge}, is
defined using the Yang--Mills heat flow.  It is precisely the
Threshold Theorem for the Yang--Mills heat flow which guarantees that
the caloric gauge is defined for all subthreshold connections.

The primary aim of the four-paper series, of which the present one is
the first, is to prove the following two results, which are analogous
to Theorems~\ref{thm:thr-simple} and \ref{thm:dich-simple}. The first
result gives an affirmative answer to the Threshold Conjecture:
\begin{theorem} [Threshold Theorem] \label{t:caloric} The hyperbolic
  Yang--Mills equation \eqref{ym} in the caloric gauge is globally
  well-posed, and the solution scatters, for all initial data $(a,e)
  \in \dot H^1 \times L^2$ with energy below $2 \Egs$.
\end{theorem}
In Theorem~\ref{t:caloric}, the restriction to ``topologically
trivial'' data $(a, e) \in \dot{H}^{1} \times L^{2}$ is natural, since
one of the conclusions is scattering of the solution. This explains
the refined threshold $2 \Egs$. On the other hand, for more general
data $(a, e) \in H^{1}_{loc} \times L^{2}_{loc}$ with finite energy,
we establish the following sharp dichotomy:
\begin{theorem} [Dichotomy Theorem] \label{t:dich-hyp} The hyperbolic
  Yang--Mills equation \eqref{ym} is locally well-posed in the
  temporal gauge for all initial data $(a,e) \in H^1_{loc} \times
  L^2_{loc}$ with finite energy. Moreover, one of the following two
  properties must hold:
  \begin{enumerate}[label=(\alph*)]
  \item The solution is global, and $A$ scatters as $t \to \infty$
    after a suitable gauge transformation.
  \item The solution ``bubbles off'' a soliton, either
    \begin{enumerate}[label=(\roman*)]
    \item at a finite blow-up time $t < \infty$, or
    \item at infinity $t = \infty$.
    \end{enumerate}
  \end{enumerate}
\end{theorem}
In (a), for sufficiently large $t$ the solution $A(t)$ can in fact be
gauge-transformed into the caloric gauge, which then scatters as $t
\to \infty$ in the same sense as Theorem~\ref{t:caloric}. In (b), a
soliton for the hyperbolic Yang--Mills equation is simply a Lorentz
transformation of a nontrivial harmonic Yang--Mills connection $Q$. By
time reversibility, this theorem applies also (separately) to the past
time direction.

Further discussion of gauges and of the above results, as well as a
more complete formulation of these, are postponed for the other three
papers.  The main goal of the second part of the present paper, which
consists of Sections~\ref{s:caloric}--\ref{s:wave}, is to properly
define the caloric gauge, and to provide a formulation of the
hyperbolic Yang--Mills equations in this gauge which suffices for the
proof of the above theorems. The caloric gauge and the corresponding
results are described in the next subsection, and the proofs are
provided later on.

\subsection{The global caloric gauge and the manifold of caloric
  connections}
Here, we state our main results concerning the caloric gauge
connections in a simplified form for the reader's convenience.  More
accurate statements are given and proved in Section~\ref{s:caloric}.

Consider a connection $a$ on $\R^4$, whose associated Yang--Mills heat
flow is global and satisfies $\calQ(a) < \infty$.  Since the limiting
connection $a_\infty$ is flat, it must be gauge equivalent to the zero
connection.  Precisely, there must be a gauge transformation $O$ with
the property that
\[
(a_{\infty})_{j} = O^{-1} \partial_j O .
\]
Here $O= O(a) \in \dot H^2$ (interpreted in the sense that $O_{;j} \in
\dot H^1$) is unique up to constant conjugations.  Conjugating the
full heat flow with respect to such an $a$ yields a gauge-equivalent
connection
\[
\tA_j = O A_j O^{-1} - O_{;j}
\]
which solves the Yang--Mills heat flow, and satisfies $\tA(\infty) =
0$.  This leads us to the following definition of caloric connections:

\begin{definition} \label{d:C} We will say that a connection $a \in
  \dot H^1$ with $\hM(a) < \infty$ is (globally) \emph{caloric} if the
  corresponding limiting connection vanishes, $a_{\infty}= 0$; we
  denote the set of all such connections by $\calC$. More
  quantitatively, we denote by $\calC_{\hM}$ the set of all caloric
  connections whose caloric size $\hM(a)$ satisfies
  \begin{equation*}
    \hM(a) = \nrm{F}_{L^{3}([0, \infty); L^{3})} \leq \hM.
  \end{equation*}
\end{definition}

Then the Threshold Theorem for the Yang--Mills heat flow can be
restated as an existence result for gauge-equivalent caloric
connections:

\begin{theorem}\label{t:caloric-g}
  For every connection $\ta_j \in \dot H^1$ with energy below $2\Egs$,
  there exists a gauge-equivalent caloric connection $a \in \dot H^1$,
  which is unique up to constant gauge transformations.
\end{theorem}

The connection $a$ is defined as
\[
a_j = \Cal(\ta) := O \ta_j O^{-1} - O_{;j}, \qquad O = O(\ta) .
\]

To properly solve the hyperbolic Yang--Mills equation in the caloric
gauge, we need to view the family $\calC$ of the caloric gauge
connections as an infinite dimensional manifold. Here the $\dot H^1$
topology is no longer sufficient, so we introduce the slightly
stronger topology\footnote{Here $\ell^1$ stands for dyadic summation
  in frequency. We prefer this notation to the more classical Besov
  style notation, as we can apply it to a larger class of spaces, see
  also Section~\ref{sec:notations}.}
\[
\bfH = \{ a \in \dot H^1: \partial^{\ell} a_{\ell} \in \ell^1 L^2\},
\]
which reflects the fact, discussed later in greater detail, that
caloric connections satisfy a nonlinear form of the Coulomb gauge
condition. Then we have
\begin{theorem} \label{t:smooth} For any caloric connections $a \in
  \calC_{\hM}$ with energy $\calE$, we have the $\bfH$ bound
  \begin{equation}
    \| a \|_{\bfH} \lesssim _{\hM, \calE} 1.
  \end{equation}
  The set $\calC$ of all $\dot H^1$ caloric connections is a $C^1$
  infinite dimensional submanifold of $\bfH$.
\end{theorem}

For an arbitrary $a \in \calC$, $O(a)$ is only defined as an
equivalence class, modulo constant gauge transformations. However, if
in addition we know that $a \in \bfH$, then $O(a)$ is continuous, and
we can fix its choice by imposing the additional condition
\begin{equation}\label{O-inf}
  \lim_{x \to \infty} O(x) = Id.
\end{equation}
With this choice we have the following regularity property:

\begin{theorem}
  The map $a \to O(a)$ is locally
  $C^{1}$
  from $\bfH$ to $\dot{H}^{2} \cap C^{0}$, and from $H^{\sgm}$ to
  $\dot{H}^{1} \cap \dot{H}^{\sgm+1}$ for $\sgm \geq 2$.  It is also
  also continuous from $\dot{H}^{1}$ to\footnote{Here $\dot H^2$ needs
    to be interpreted as a quotient space, modulo constant gauge
    transformations.}  $\dot H^2$.
\end{theorem}

Finite energy solutions to the hyperbolic Yang--Mills equation will be
continuous functions of time which take values into $\calC$.  They are
however not smooth in time, instead their time derivative will merely
belong to $L^2$.  Because of this, we need to take the closure of its
tangent space $T_{a} \calC$ (which a-priori is a closed subspace of
$\bfH$) in $L^2$.  This is denoted by $T_a^{L^2}\calC$. It is also
convenient to have a direct way of characterizing this space; that is
naturally done via the linearization of the Yang--Mills heat flow:

  \begin{definition}\label{d:tangent}
    For a caloric gauge connection $a \in \calC$, we say that $ L^2
    \ni b \in T_a^{L^2} \calC$ if the solution to the linearized local
    caloric gauge Yang--Mills heat flow equation
    \begin{equation} \label{eq:YM-h-lc-lin}
      \partial_s B_k = [B^j,F_{kj}] + \covD^j ( \covD_k B_j - \covD_j B_k), \qquad B_k(0) = b_k
    \end{equation}
    satisfies
    \[
    \lim_{s \to \infty} B(s) = 0 \qquad \hbox{ in } L^{2}.
    \]
  \end{definition}

  A key property of the tangent space $T_a^{L^2} \calC$ is the
  following nonlinear div-curl type decomposition:

  \begin{theorem} \label{t:cal-proj} Let $a \in \calC$. Then for each
    $e \in L^2$ there exists a unique decomposition
    \begin{equation} \label{eq:cal-proj-simple} e = b - \covD a_0,
      \qquad b \in T_a^{L^2} \calC, \qquad a_0 \in \dot H^1.
    \end{equation}
    with the corresponding bound
    \begin{equation}
      \| b\|_{L^2} +\|a_0\|_{\dot H^1} \lesssim \|e \|_{L^2}.
    \end{equation}
  \end{theorem}
  For any $e \in L^{2}$, we introduce the notation
  \begin{equation*}
    b = \Pi_{a}(e)
  \end{equation*}
  for its projection $b$ as in \eqref{eq:cal-proj-simple} to
  $T^{L^{2}}_{A} \calC$.

  Finally, as already hinted by Theorem~\ref{t:smooth}, a key property
  of a caloric connection is that its divergence $\rd^{\ell} A_{\ell}$
  satisfies a generalized Coulomb condition. We separate out the
  quadratic par, which can be explicitly determined, and the remaining
  higher order terms, which only play a perturbative role in the
  subsequent analysis:
  \begin{theorem}
    For $a \in \calC$, we have the representation
    \begin{equation}\label{coulomb+}
      \partial^k A_k = \DA(A) = \bfQ(A,A) + \DA^3(A).
    \end{equation}
    where $ \bfQ(A,A)$ is a symmetric\footnote{Note that the symbol of
      $\bfQ$ is odd, but this is combined with the antisymmetry of the
      Lie bracket appearing in the bilinear form; see
      Definition~\ref{def:mult-form}.} bilinear form (see
    Definition~\ref{def:mult-form} below) with symbol
    \begin{equation} \label{eq:Q-sym-simple} \bfQ(\xi, \eta) =
      \frac{\abs{\xi}^{2} - \abs{\eta}^{2}}{2 (\abs{\xi}^{2} +
        \abs{\eta}^{2})}.
    \end{equation}
    and $ \DA^3(A)$ is a $C^1$ map on $\calC$ containing cubic and
    higher order terms, and satisfying better bounds.
  \end{theorem}

\begin{remark} [Brief historical remarks]
  The caloric gauge was introduced by Tao \cite{Tao-caloric} in the
  context of mappings from $\bbR^{2}$ into hyperbolic space, using the
  harmonic map heat flow on $\bbR^{2}$. Its construction was extended
  to general targets up to the ground state energy
  (cf. Theorem~\ref{t:caloric-g}) by Smith \cite{Sm1}. Various authors
  successfully applied the caloric gauge in analysis of dispersive
  equations for mappings, including Tao
  \cite{Tao:2008wn,Tao:2008tz,Tao:2008wo,Tao:2009ta,Tao:2009ua} for
  the wave maps on $\bbR^{1+2}$; Bejenaru--Ionescu--Kenig--Tataru
  \cite{BIKT}, Smith \cite{Sm2} and Dodson--Smith \cite{DoSm} for
  Schr\"odinger maps on $\bbR \times \bbR^{2}$. We also note some
  recent applications of the caloric gauge in the context of wave maps
  on the hyperbolic space \cite{LOS5, Li1, LMZ, Li2}.

  The idea of caloric gauge was extended to the Yang--Mills setting
  (i.e., for vector bundle connections) by the first author \cite{Oh1,
    Oh2} at subcritical regularity. In that case, since the scaling
  symmetry is broken, it is more natural to only impose the
  \emph{local} caloric gauge condition \eqref{local-caloric} for
  heat-times below certain threshold $s_{0} > 0$ dictated by the
  initial data; this is in contrast to the global caloric gauge used
  in this work.
\end{remark}

\subsection{The hyperbolic Yang--Mills equation in the 
caloric gauge} \label{subsec:ym-cal-evol}

We now turn our attention to the hyperbolic Yang--Mills equation. We will consider
solutions which at any fixed time $t$ are in the caloric gauge,
$A_x(t) \in \calC$. We will refer to such solutions as \emph{caloric Yang--Mills waves}.

We first clarify the notion of an initial data set for the Yang--Mills equation in the caloric gauge. On the one hand, we have the gauge covariant notion $(a, e)$, which satisfies the constraint equation \eqref{eq:YMconstraint}. On the other hand, in the caloric gauge, we will consider the spatial component of the 
connection as the dynamic variable, and view the temporal part of the connection
as an auxiliary variable (which is analogous to the Coulomb case). From this point if view, we have:
 \begin{definition}  \label{d:cal-data}
An initial data for the Yang--Mills equation in the caloric gauge is a pair 
$(a,b)$ where $a \in \calC$ and $b_k \in T_a^{L^2} \calC$. 
\end{definition}
These two notions are related to each other by the following result:
\begin{theorem}\label{t:data-simple}
  \begin{enumerate}
  \item Given any Yang--Mills initial data pair $(a, e) \in \dot H^1
    \times L^2$ such that $\hM(a) < \infty$, there exists a caloric gauge
    Yang--Mills data $(\tilde a, b) \in T^{L^{2}} \calC$ and $a_0 \in
    \dot H^1$, so that the initial data pair $(\tilde a, \tilde e)$ is
    gauge equivalent to $(a,e)$, where
    \[
    \tilde e_k = b_k - \covD^{(\tilde{a})}_k a_0 .
    \]
  \item Given any caloric gauge data $(\tilde a, b) \in T^{L^{2}}
    \calC$, there exists a unique $a_0 \in \dot H^1$, with Lipschitz
    dependence on $(a, b) \in \dot H^1 \times L^2$, so that
    \[
    e_k = b_k - \covD^{(a)}_k a_0
    \]
    satisfies the constraint equation
    \eqref{eq:YMconstraint}. 
  \end{enumerate}
\end{theorem}
For a proof, see Theorem~\ref{t:data} below. In view of this result, we can fully describe caloric Yang--Mills waves as continuous functions
\[
I \ni t \to (A_x(t), \partial_{t} A_x(t)) \in T^{L^2} \calC.
\]

Next, we consider the task of providing a full description of the gauge-dependent system of nonlinear wave equations satisfied  
by a caloric Yang--Mills wave $A$. Recall the Yang--Mills equation \eqref{ym}:
\begin{equation*}
	\covD^{\alp} F_{\alp \bt} = 0.
\end{equation*}
Separating these equations into the spatial ($\beta \neq 0$) case
\begin{equation}\label{ym-x}
\covD^\alpha \covD_\alpha A_k = \covD^k \covD^\alpha A_\alpha - [A^\alpha, \covD_k A_\alpha],
\end{equation}
and respectively the temporal ($\beta = 0 $) case
\begin{equation}\label{ym-0}
\covD^k \covD_k A_0 = \covD_0 \covD^k A_k  - [A^k, \covD_0 A_k],
\end{equation}
we will seek to interpret the first equation as a hyperbolic evolution for $A_x$,
and the second as an elliptic compatibility condition for $A_0$. This is achieved 
as follows:

\pfstep{Step~1}
Use the equation \eqref{ym} to show 
that $A_0$ is uniquely determined by $A = A_{x}$ and $B = B_x = \partial_{t} A_x$,
\begin{equation}\label{a0-map}
A_0 = \bfA_0(A, B)
\end{equation}
where $\bfA_0$ is a $C^1$  map on $T^{L^2} \calC$ which contains
an explicitly computed quadratic part $\bfA_{0}^{2}$, as well as purely perturbative
higher order terms $\bfA_{0}^{3}$:
\begin{equation} \label{a0-quad-}
\bfA_0(A_x,B_x) = \bfA_0^2(A, B) +  \bfA_0^3(A, B) .
\end{equation}

\pfstep{Step~2} Use the equations \eqref{ym-x} to show 
that $\covD^0 A_0$ is uniquely determined by $A$ and $B$, i.e.,
\begin{equation}\label{da0-quad-}
\covD^0 A_0 = \DA_0(A_x,B_x)
\end{equation}
where $\DA_0$ is a $C^1$ map on $T^{L^2} \calC$ which plays a purely perturbative
role in the analysis.

\bigskip

The above steps allow us to recast the equations \eqref{ym-x} in the form
\begin{equation}\label{ymcg-main}
\Box_A A_k = P [A^j,\partial_k A_j] 
+  2\Delta^{-1} \partial_k \bfQ(\partial^\alpha A^j,\partial_\alpha A_j) 
+ R(A,\partial_t A)
\end{equation}
Here on the right we have two quadratic terms depending only on $A$
and $B$, both of which have a favorable null structure, and a
higher order remainder term $R$, which admits favorable $L^1L^2$
bounds and thus only plays a perturbative role. However, in the
covariant d'Alembertian $\Box_{A}$ on the left, we still have the coefficients
$A_0$ and $\covD^0 A_0$, which are determined as above in terms of $A_x$
and $\partial_t A_x$. Of these only the quadratic part $\bfA_0^2$ of $A_0$ plays
a nonperturbative role. We arrive at:

\begin{theorem}\label{t:main-wave}
The hyperbolic Yang--Mills equation in the caloric gauge takes the form \eqref{ymcg-main},
where  
\begin{itemize}
\item $\bfQ$ is a symmetric quadratic form of order zero with symbol \eqref{eq:Q-sym-simple}.
\item $R$ is a $C^1$ map on $T^{L^2} \calC$ satisfying perturbative bounds.
\item $A_0$ and $\covD^0 A_0$ are uniquely determined by $A$ and $A_x$ via 
\eqref{a0-map}, \eqref{a0-quad-} and \eqref{da0-quad-}.
\item The generalized Coulomb condition \eqref{coulomb+} holds.
\end{itemize}
\end{theorem} 

A more precise form of this result is contained in Section~\ref{s:wave}.
All perturbative terms are shown to satisfy favorable bounds purely in terms 
of Strichartz type norms for the connection $A$. The 
exact structure of all explicit quadratic forms $\bfQ$, $\bfA_0^2$ and $\DA^2$
will play a key role in the next paper of our three paper sequence.

To study the small data problem it would be sufficient to work with
the equation \eqref{ymcg-main}. However, for the large data problem we
also need to flow the wave equation in the parabolic direction. To smooth out the space-time connection $(A_{0}, A)$, we use the \emph{dynamic Yang--Mills heat flow}
\begin{equation} \label{eq:dymhf-simple}
	F_{s \alp} = \covD^{\ell} F_{\ell \alp}, \qquad \alp = 0, 1, \ldots 4
\end{equation}
 which is the Yang--Mills heat flow for $A$ adjoined with an $s$-evolution equation for $A_{0}$.
Then at nonzero heat-times $s$ the equation \eqref{ym} becomes
\begin{equation}\label{ym-cov-w}
\covD^\alpha F_{\alpha \beta}(s) = w_\alpha(s).
\end{equation}
The space-time 1-form $w_{\alp}$ is called the \emph{(hyperbolic) Yang--Mills tension field}. In general $w_\alpha(s) \neq 0$ for $s > 0$, as the two flows (wave and heat) do not commute. In order to proceed, additional steps are needed:

\pfstep{Step~3} Compute parabolic evolutions for $w_\alpha$, showing that at time $t$ they
depend only on the data $A(t)$, $B(t)$ and of course on $s$,
\[
w_\alpha = \bfw_\alpha(A(t),B(t), s).
\]
A key point is that the initial data $w_{\alp}(s=0)$ is zero, thanks to the hyperbolic Yang--Mills equation \eqref{ym}. As a consequence, $w_{\alp}$ turns out to be quadratic and higher order.

\pfstep{Step~4} Separate $\bfw_\alpha$ into the quadratic part and 
a higher order term,
\[
\bfw_\alpha(s) = \bfw_\alpha^2(s) + \bfw_\alpha^3(s).
\]
 where the former can be explicitly computed, and the latter is purely perturbative.

\pfstep{Step~5} Recalculate $A_0$ and $\covD^0 A_0$ to include the dependence on 
$w(s)$, and write the analogue of the equation \eqref{ymcg-main}
for $A(s)$:
\begin{equation}\label{ymcg-main-s}
\begin{split}
\Box_{A(s)} A_k(s) = & \ P [A^j(s),\partial_k A_j(s)]  +  
2\Delta^{-1} \partial_k \bfQ(\partial^\alpha A^j(s),\partial_\alpha A_j(s)) + R(A(s),\partial_t A(s)).
\\ 
& \ + P \bfw_k^2(s) + R_s(A,\partial_t A)
\end{split}
\end{equation}
The extra term on the right is matched by a like contribution to the quadratic part of $A_0$,
i.e. \eqref{a0-quad-} is replaced by
\begin{equation} \label{a0-quad}
A_0(s) = \bfA_0(A(s),B(s)) + \Delta^{-1} \bfw_0^2+ \bfA_{0,s}^3(A,B) 
\end{equation}

Now we can state

\begin{theorem}\label{t:main-wave-s}
The caloric flow $A(s)$ of a hyperbolic Yang--Mills wave in the caloric gauge takes 
the form \eqref{ymcg-main-s}, where the additional terms   $\bfw_k^2(s)$, $R_s$ 
and $\bfA_{0,s}^3(A,B)$ satisfy the following properties:
\begin{itemize}
\item  $\bfw_k^2(s)(A,B)$ are explicit quadratic forms localized\footnote{with decaying tails}
at frequency $s^{-\frac12}$. 
\item $R_s$ and $\bfA_{0,s}^3(A,B)$ are $C^1$ maps on $T^{L^2} \calC$
  satisfying perturbative bounds, also localized at frequency
  $s^{-\frac12}$.
\end{itemize}
\end{theorem} 
The analysis of the equation \eqref{ymcg-main-s}, carried out in \cite{OTYM2}, will be very similar
to that of \eqref{ymcg-main}, with the minor proviso that the
quadratic terms in $\bfw$ in the two equations above have a very mild
nonperturbative role, and exhibit a null form type cancellation. 

\begin{remark} \label{rem:higher-d}
Finally, we briefly discuss what happens in dimensions higher than $4$. After replacing $\dot{H}^{1}$ by $\dot{H}^{\frac{d-2}{2}}$ and defining 
\begin{equation*}
\calQ(a) = \nrm{F}_{L_{s}^{d-1}([0, s_{+}); L^{d-1}_{x})},
\end{equation*}
the analogue of Theorems~\ref{thm:hf-loc-simple} and
\ref{thm:str-simple} may be proved for any dimension $d \geq 4$, by
essentially a technical extension of the arguments in this paper. For
connections $a \in \dot{H}^{\frac{d-2}{2}}$ with $\hM(a) < \infty$,
the caloric gauge condition can be defined exactly in the same way
(Definition~\ref{d:C}), and the subsequent results concerning the
caloric gauge also extend easily to higher dimensions.

On the other hand, our proof of the Dichotomy and Threshold Theorems
(Theorems~\ref{thm:dich-simple} and \ref{thm:thr-simple}) rely on
specific features of dimension $4$ (energy criticality, knowledge of
harmonic Yang--Mills connections etc.), and does not admit direct
generalization to higher dimensions.
\end{remark}

\subsection{Remarks on the dynamic Yang--Mills heat flow and the de
  Turck trick} \label{subsec:dymhf} In this paper, the dynamic
Yang--Mills heat flow
\begin{equation}
  F_{s \alp} = \covD^{\ell} F_{\ell \alp}
\end{equation}
plays a central role in multiple major ways. This subsection is
devoted to a brief discussion of these aspects.
 
The same flow appears in our work in three distinct capacities:
\begin{enumerate}
\item {\it As a gauge covariant smoothing flow for space-time
    connections.}  This is the most direct interpretation of the
  dynamic Yang--Mills heat flow (as opposed to the original
  Yang--Mills heat flow, which is for spatial connections). It is used
  in the energy induction argument in \cite{OTYM2}; for this reason,
  we derive equations obeyed by the dynamic Yang--Mills heat flow
  $A(s)$ of a caloric Yang--Mills wave (see
  Section~\ref{subsec:ym-cal-evol}). Noncommutativity of the
  hyperbolic and parabolic Yang--Mills equations gives rise to a
  nontrivial Yang--Mills tension field $w_{\mu}(s)$, whose analysis is
  key for deriving the equations for $A(s)$.

  Curiously, $w_{\mu}(s)$ $(s > 0)$ also makes appearance in estimates
  for $\rd_{0} A_{0}$, even at $s = 0$. This is due to other uses of
  related to the ``infinitesimal de Turck trick'', which we explain
  below.

\item {\it As a means to perform the ``infinitesimal de Turck trick''
    for the linearized Yang--Mills heat flow in the local caloric
    gauge.}  The usual de Turck trick is a way of compensating for the
  degeneracy of \eqref{YM-h-lc} by an $s$-dependent gauge
  transformation; in our gauge-covariant formalism, it amounts to
  working in the de Turck gauge \eqref{deTurck}. As we have seen,
  however, this approach is problematic for large data global theory.

  Instead, we perform the de Turck trick not for $A$, but rather at
  the level of the linearized flow \eqref{eq:YM-h-lc-lin} (thus the
  name ``infinitesimal de Turck trick''). The algorithm is as
  follows. Given a one-parameter family of Yang--Mills heat flows
  $A_{j}(t, x, s)$ with data $a_{j}(t, x)$ $(t \in I, x \in \bbR^{4},
  s \in J)$, we add a $t$-component $A_{0}(t, x, s)$ and view it as a
  connection 1-form on $I \times \bbR^{4} \times J$. In the
  $s$-direction, we then \emph{impose} the equation
  \begin{equation} \label{eq:Fs0} F_{s0} = \covD^{\ell} F_{\ell 0}
  \end{equation}
  which, combined with \eqref{caloric}, forms the dynamic Yang--Mills
  heat flow system.

  The key idea is to work with
  \begin{equation} \label{eq:F0j-deT} F_{0j} = \rd_{t} A_{j} -
    (\covD[A])_{j} A_{0}.
  \end{equation}
  As opposed to $\rd_{t} A_{j}$, which solves \eqref{eq:YM-h-lc-lin},
  $F_{0j}$ has the advantage of obeying a \emph{nondegenerate}
  covariant parabolic equation:
  \begin{equation*}
    \covD_{s} F_{0j} - \lap_{A} F_{0j} - 2 ad(\tensor{F}{_{j}^{\ell}})F_{0\ell} = 0.
  \end{equation*}
  Solving this equation would determine $F_{0j}$ from any data
  $F_{0j}(s=0) = e_{j}$.  We \emph{choose} $e_{j} = \rd_{t} a_{j}$,
  which amounts to prescribing $a_{0} = 0$.  Then $A_{0}$ may be
  determined by integrating
  \begin{equation} \label{eq:ds-a0} \rd_{s} A_{0} = F_{s0} =
    \covD^{\ell} F_{\ell 0}
  \end{equation}
  where the first equality is the local caloric condition, and the
  second one is \eqref{eq:Fs0}. Finally, using \eqref{eq:F0j-deT}, we come
  back to the solution $\rd_{t} A$ of the linearized Yang--Mills heat
  flow.

  The success of this approach is based on solvability of $\lap_{A}$
  and $\rd_{s} - \lap_{A}$, which is developed in
  Section~\ref{sec:cov}. It forms the basis of our analysis in
  Section~\ref{s:caloric-loc}.

\item {\it As a means to obtain useful representation of projection to
    the caloric manifold.}  This is a variant of the ``infinitesimal
  de Turck trick''.  Previously, we chose to initialize $a_{0} =
  0$. When $a(t=0)$ is a caloric connection, another natural choice is
  to set $A_{0}(s = \infty) = 0$, which amounts to making a gauge
  transformation in $t$ so that the nearby $a(t)$'s are also
  caloric. Integrating \eqref{eq:ds-a0} from $s = \infty$ to $0$, we
  obtain the following representation of $a_{0}$:
  \begin{equation*}
    a_{0} = - \int_{0}^{\infty} \covD^{\ell} F_{\ell 0}(s) \, ds.
  \end{equation*}
  By \eqref{eq:F0j-deT}, we have
  \begin{equation*}
    e_{j} = \rd_{t} a_{j} - (\covD[a])_{j} a_{0}.
  \end{equation*}
  Since $a(t)$'s are caloric, $\rd_{t} a_{j}$ clearly belongs to
  $T_{a} \calC$ since each $a(t)$ is caloric, whereas $\covD[a] a_{0}$
  is a pure covariant gradient. By Theorem~\ref{t:cal-proj}, $\rd_{t}
  a_{j}$ is precisely the projection $\Pi_{a} e_{j}$.

  The procedure just described gives an explicit algorithm for
  computing $\Pi_{a}$, which we will use extensively in
  Section~\ref{s:caloric} and onward.  The same idea also allows us to
  relate the second order variation $\rd_{0} a_{0}(s=0)$ with integral
  of $\covD_{0} \covD^{\ell} F_{\ell 0}(s) = \covD^{\ell} w_{\ell}(s)$
  from $s = \infty$ to $0$ (up to minor error terms), which explains
  the usefulness of $w(s)$ in (1).
\end{enumerate}

\section{Notations, conventions and other preliminaries}
\label{sec:notations}
\subsection{Notations and conventions}
 Here we collect some notation and conventions used in this paper.

\begin{itemize}
\item We employ the usual asymptotic notation $A \aleq B$ to denote $A\leq C B$ for some implicit 
constant $C > 0$. The dependence of $C$   on various parameters is specified by subscripts.

\item We use the notation $\rd$ (without sub- or superscripts) for the
  spatial gradient $\rd = (\rd_{1}, \rd_{2}, \ldots, \rd_{d})$, and
  $\nb$ for the space-time gradient $\nb = (\rd_{0}, \rd_{1}, \ldots,
  \rd_{d})$. We write $\rd^{(n)}$ (resp. $\nb^{(n)}$) for the
  collection of $n$-th order spatial (resp. space-time) derivatives,
  and $\rd^{(\leq n)}$ (resp. $\nb^{(\leq n)}$) for those up to order
  $n$.
\end{itemize}

Linear, translation invariant operators acting on functions in $\R^4$ are viewed as multipliers,
and described in a standard fashion via their symbol. Bilinear operators also play an important role
in this paper. In the Lie algebra context, the connection between bilinear operators and symbols 
is described as follows:

\begin{definition} \label{def:mult-form} By a \emph{bilinear operator
    with symbol} $m(\xi, \eta) = m^{jk}(\xi, \eta)$ (which is a
  complex-valued $4 \times 4$-matrix), we mean an expression of the
  form
\begin{equation*}
	\mathfrak{L}(a, b) = 
	\iint 
	\left( m^{\bfa \bfb}(\xi, \eta) [\hat{a}_{\bfa}(\xi), \hat{b}_{\bfb}(\eta)] \right)
	e^{i (\xi + \eta) \cdot x}  \, \frac{\ud \xi \, \ud \eta}{(2 \pi)^{8}}.
\end{equation*}
\end{definition}
If $\mathfrak{L}$ were symmetric, then the symbol $m(\xi, \eta)$ is 
\emph{anti-symmetric in $\xi, \eta$}, in the sense that 
$m^{\bfa \bfb}(\xi, \eta) = - m^{\bfb \bfa}(\eta, \xi)$; this is due to the antisymmetry of the Lie bracket.

\subsection{Function spaces}
We begin with the standard Sobolev spaces:
\begin{itemize}
\item  The $n$-th homogeneous $L^{p}$-Sobolev space for functions from
  $\bbR^{d}$ into a normed vector space $V$ is denoted by $\dot{W}^{n,
    p}(\bbR^{d}; V)$. In the special case $p = 2$, we write
  \[
\dot{H}^{n}(\bbR^{d}; V) = \dot{W}^{n, 2}(\bbR^{d}; V).
\]
The Lebesgue spaces (i.e.,
  when $n = 0$) are denoted by $L^{p}(\bbR^{d}; V)$.

\item The mixed space-time norm $L^{q}_{s} \dot{W}^{\sgm, r}_{x}$ [resp. $L^{q}_{t} \dot{W}^{\sgm, r}_{x}$] of
  functions on $\bbR^{d}_{x} \times J_{s}$ [resp. $I_{t} \times \bbR^{d}_{x}$] is often abbreviated as $L^{q} \dot{W}^{\sgm, r}$. It will be clear from the context which variable (either $s$ or $t$) is involved.

\item Given a function space $X$ (on either $\bbR^{d}$ or
  $\bbR^{1+d}$), we define the space $\ell^{p} X$ by
  \begin{equation*}
    \nrm{u}_{\ell^{p}X}^{p} = \sum_{k} \nrm{P_{k} u}^{p}_{X}
  \end{equation*}
  (with the usual modification for $p = \infty$), where $P_{k}$ $(k
  \in \bbZ)$ are the usual Littlewood--Paley projections to dyadic
  frequency annuli.
\end{itemize}

In the last section of the paper, where we make the connection with the hyperbolic Yang-Mills
equation, we need Strichartz type norms to describe bounds for various remainder terms.
Generally the Strichartz norms are used to describe the dispersive decay of solutions for the linear 
wave equation. In  particular for solutions to the homogeneous wave equation 
$\Box u = 0$ in $\R^{1+4}$  we have
\[
\| \nabla u \|_{L^p \dot W^{\sigma,q}} \lesssim \| \nabla u(0)\|_{L^2}
\]
for exponents $(p,q,\sigma)$ in the admissible Strichartz range
\begin{equation}\label{Str}
\frac{1}{p} + \frac{4}{q} = 2 + \sigma, \qquad 2 \leq p,q \leq \infty, \qquad \frac{2}{p}+ \frac{3}{q} \leq \frac32.
\end{equation}
The exponents $(p,q)$ for which equality holds in the last relation above are referred to as sharp Strichartz
exponents. For the remainder term bounds in the last section we seek to avoid using sharp Strichartz
norms, and instead use only a restricted range of exponents. For this reason we 
choose a sufficiently small universal threshold $\dltStr > 0$ and define
\begin{equation}
\| u\|_{\Str} = \sup\{ \| u \|_{L^p \dot W^{\sigma,q}}; \ (p,q,\sigma)\  \text{admissible}, \ 
\dltStr \leq \frac1p  \leq \frac{1}{2} - \dltStr, \ \   \frac{2}{p}+ \frac{3}{q} \leq \frac32 - \dltStr \}.
\end{equation}
as well as 
\begin{equation}
\| u\|_{\Str^1} = \| \nabla u\|_{\Str}
\end{equation}
These norms have two key properties, which will play an important role in the next paper of the sequence
\cite{OTYM2}:
\begin{itemize}
\item They are divisible in time, i.e. can be made small by subdividing the time interval. 
\item Saturating the associated Strichartz inequalities requires strong pointwise concentration, rather
than the usual range of Knapp examples (wave packets). 
\end{itemize}

\subsection{Frequency envelopes}
To provide more accurate versions of many of our estimates and results we use the 
language of frequency envelopes. 

Given a sequence $c_{k}$ $(k \in \bbZ)$ of positive numbers and a translation invariant norm $\nrm{\cdot}_{X}$, we introduce the shorthand
\begin{equation*}
	\nrm{u}_{X_{c}} := \sup_{k} \frac{\nrm{P_{k} u}_{X}}{c_{k}}.
\end{equation*}

\begin{definition}
Given a translation invariant space of functions $X$, we  say that a sequence 
$c_k$ of positive numbers is a frequency envelope for a function $u \in X$ if 
\begin{enumerate}[label=(\roman*)]
\item The dyadic pieces of $u$ satisfy
\[
\nrm{u}_{X_{c}} \leq 1, \hbox{ or equivalently, } \| P_k u\|_{X} \leq c_k
\]
\item The sequence $c_k$ is slowly varying,
\[
	2^{-\delta(j-k)} \aleq \frac{c_{k}}{c_{j}} \aleq 2^{\delta(j-k)}, \qquad j > k.
\]
\end{enumerate}
\end{definition}
Here $\delta$ is a small positive universal constant. For some of the results we need to relax 
the slowly varying property in a quantitative way. Fixing a universal small constant $0 < \eps \ll 1$,
we set
\begin{definition} Let $\sgm_1,\sgm_2 > 0$. A frequency envelope $c_k$ is called 
$(-\sgm_{1}, \sgm_{2})$-admissible if
\begin{equation*}
	2^{-\sgm_{1} (1-\eps)(j-k)} \aleq \frac{c_{k}}{c_{j}} \aleq 2^{\sgm_{2} (1-\eps) (j-k)}, \qquad j > k.
\end{equation*}
\end{definition}

Another situation that will occur frequently is that where we have a reference frequency envelope 
$c_k$, and then a secondary envelope $d_k$ describing properties which apply on a background 
controlled by $c_k$. In this context the envelope $d_k$ often cannot be chosen arbitrarily
but instead must be in a constrained range depending on $c_k$. To address such matters we 
 set:

\begin{definition} \label{def:fe-compat}
We say that the envelope $d_k$ is \emph{$\sigma$-compatible} with $c_k$ if 
we have
\[
c_k \sum_{j < k} 2^{\sigma(1-\eps) (j-k)} d_j \lesssim d_k.
\]
\end{definition}

We will often replace envelopes $d_k$ which do not satisfy the above compatibility condition
by slightly larger envelopes that do:

\begin{lemma}
Assume that $c_k$ and $d_k$ are $(-\sgm_{1}, S)$ envelopes, and also that $c_{k}$ is bounded.  Then for $\tilde{\sigma} < \sgm(1-\eps)$ the envelope
\[
e_k = d_k + c_k \sum_{j<k} 2^{\tilde{\sigma} (j-k)} d_{j}
\]
is $\sgm$-compatible with $c_k$. The implicit constant in Definition~\ref{def:fe-compat} is bounded above by $1+C_{\sgm(1-\eps) - \tilde{\sgm}} \nrm{c}_{\ell^{\infty}}$. 
\end{lemma}
\begin{proof}
We need to show that 
\[
c_k  \sum_{j<k} 2^{\sgm(1-\eps) j-k} e_{j} \lesssim e_k.
\]
This is trivial for the first term in $e_j$, so we consider the contribution of the second,
\[
c_k  \sum_{\ell < j<k} 2^{\tilde{\sigma}(\ell-k)}  d_{\ell}  2^{(\sgm(1-\eps) -\tilde{\sigma})(j-k)} c_j  \lesssim c_k  \sum_{\ell < j<k} 2^{\sigma(\ell-k)}  d_{\ell}. 
\]
The claim regarding the bound on the implicit constant in Definition~\ref{def:fe-compat} follows by inspection.
\end{proof}

Finally we need the following additional frequency envelope notations:
\begin{align*}
(c \cdot d)_{k} =  \ c_{k} d_{k}, \qquad & \qquad  
a_{\leq k} =  \ \sum_{j \leq k} a_{j}, \\
c_{k}^{[\sgm]} =  \ \sup_{j < k} 2^{(1-\eps) \sgm (j-k)} c_{j} &  \qquad (\sigma > 0).
\end{align*}

\section{ Linear covariant elliptic and parabolic flows} \label{sec:cov}

\subsection{Solvability for \texorpdfstring{$\Delta_A$}{DeltaA}}
Our goal here is  to study the elliptic equation 
\begin{equation}\label{deltaA}
\Delta_A B = F,
\end{equation}
where $A$ is a connection $1$-form on $\bbR^{4}$, $\covD$ is the covariant derivative associated to $A$ and $\Delta_{A} = \covD^{\ell} \covD_{\ell}$ is the covariant Laplacian. Moreover, $B, F$ are $\g$-valued functions on $\bbR^{4}$. In this subsection, we assume that $A \in \dot{H}^{1}$, and omit the dependence of all implicit constants on $\nrm{A}_{\dot{H}^{1}}$.

The main result is:
\begin{theorem}\label{t:deltaA}
Assume that $A \in \dot H^1$. 
Then the equation \eqref{deltaA} is solvable with bounds as follows:
\begin{equation}
\| B\|_{\dot{H}^\sgm} \lesssim \| F\|_{\dot{H}^{\sgm-2}}, \qquad (0 < \sgm < 2).
\end{equation}
If in addition  $\partial^j A_j \in \ell^1 L^2$, then we  also have
\begin{equation}
\| B\|_{\ell^1 \dot{H}^2} \lesssim \| F\|_{\ell^1 L^2}.
\end{equation}
\end{theorem}

\begin{proof}
All these bounds are perturbative if we assume in addition that $A$ is small in $\dot H^1$.
Else we proceed with the following steps:

\pfstep{The case $\sgm = 1$} 
Here the solutions are variationally interpreted as minimum points for the 
functional
\[
L(B) = \int \frac12 \langle \covD^j B, \covD_j B \rangle -   \langle B, F\rangle dx.
\]
The desired solvability result may be proved with a standard calculus of variations argument combined with the diamagnetic inequality for $\bfD$ as follows. Note that, for every $\eps > 0$ and smooth $B$, $\rd_{j} (\eps + \brk{B, B})^{\frac{1}{2}} \leq (\eps + \brk{B, B})^{-\frac{1}{2}} \brk{B, \bfD_{j} B} \leq \abs{\bfD_{j} B}$. Multiplying by a nonnegative test function $\varphi$ and taking $\eps \to 0$, we obtain the \emph{diamagnetic inequality}: 
\begin{equation*}
\abs{\rd \abs{B}} \leq \abs{\covD B} \hbox{ in the sense of distributions}.
\end{equation*}
By Sobolev embeddings, we immediately see that $\nrm{B}_{L^{4}} \aleq \nrm{\covD B}_{L^{2}}$; then expanding $\covD_{j} B = \rd_{j} B + ad(A_{j}) B$ and estimating $\nrm{ad(A_{j}) B}_{L^{2}} \aleq \nrm{A}_{\dot{H}^{1}} \nrm{B}_{L^{4}}$, we obtain
\begin{equation} \label{eq:deltaA-sgm=1}
	\nrm{B}_{L^{4}} + \nrm{B}_{\dot{H}^{1}} \leq C (1 + \nrm{A}_{\dot{H}^{1}}) \nrm{\covD B}_{L^{2}}.
\end{equation}
As a result, we obtain the following lower bound on $L(B)$:
\begin{align*}
	L(B) \geq \frac{1}{2} \nrm{\bfD B}_{L^{2}}^{2} - \nrm{B}_{\dot{H}^{1}} \nrm{F}_{\dot{H}^{-1}} 
	\geq \frac{1}{4} \nrm{\bfD B}_{L^{2}}^{2} - C^{2} (1+\nrm{A}_{\dot{H}^{1}})^{2} \nrm{F}_{\dot{H}^{-1}}^{2}.
\end{align*}
By \eqref{eq:deltaA-sgm=1} and the convexity of $L$, used in the form of the identity
\begin{equation*}
	L(\tfrac{B+B'}{2}) + \frac{1}{2} \int \brk{\covD^{j} (\tfrac{B-B'}{2}), \covD_{j}(\tfrac{B-B'}{2})} \, \ud x  = \frac{1}{2} \left( L(B) + L(B') \right),
\end{equation*}
we see that any minimizing sequence $B^{(n)} \in \dot{H}^{1}$ for $L$ converges strongly to a minimizer $B \in \dot{H}^{1}$, which is unique.

Finally, the desired bound follows from $L(B) \leq L(0) = 0$ and \eqref{eq:deltaA-sgm=1}.

\pfstep{The case $\sgm > 1$} By duality the case $\sgm < 1$ reduces to this.

It suffices to start with $F$ localized at frequency $1$ with $\nrm{F}_{\dot{H}^{\sgm-2}} = 1$, and prove
that the bounds above hold.  The smaller frequencies of $B$ are
obtained from the $\dot H^1$ bound, so we need to get the higher
frequencies. A perturbative argument at high frequencies shows that we
must have $B \in \dot{H}^{\sgm}$, but this proof depends on the frequency
envelope of $A$. It remains to remove this dependence. 

Let $c_{k} = \nrm{P_{k} B}_{\dot{H}^{2}}$, and let $d_{k}$, $e_{k}$ be $(-\dlt, \dlt)$ frequency envelopes for $A$ in $\dot{H}^{1}$ and $\rd^{\ell} A_{\ell}$ in $L^{2}$, respectively. 
A direct application of the $\dot{H}^{1}$ bound to $\lap_{A} B = F$ yields
\begin{equation} \label{eq:deltaA-B-low}
	c_{k} \aleq 2^{k} \qquad \hbox{ for any } k.
\end{equation}
which is effective only for $k \leq 0$. For $k > 0$ we view our equation as an equation for $B_{\geq k} = (1-P_{< k}) B$, i.e.,
\[
\Delta_A B_{\geq k} = [P_{<k}, \Delta_A] B.
\]
We furthermore decompose
\begin{align*}
	[P_{<k}, \Delta_A] B
	= & P_{<k} \left( (2 ad(A^{\ell}) \rd_{\ell}  + ad(\rd^{\ell} A_{\ell} ) + ad(A_{\ell}) ad(A^{\ell}) ) P_{\geq k} B \right) \\
	& - P_{\geq k} \left( (2 ad(A^{\ell}) \rd_{\ell}  + ad(\rd^{\ell} A_{\ell} ) + ad(A_{\ell}) ad(A^{\ell}) ) P_{<k} B \right).
\end{align*}
In what follows, we omit the tensor index $\ell$. Applying the $\dot H^1$ result to the above equation, and using Littlewood--Paley trichotomy, we obtain the bound
\begin{align*}
\nrm{P_{\geq k} B}_{\dot{H}^{1}}
\aleq & \nrm{[P_{<k}, \lap_{A}] B}_{\dot{H}^{-1}} \\
\aleq & \nrm{P_{\geq k} B}_{\dot{H}^{1}} \sum_{j \geq k - 5} 2^{j-k} d_{k} + 2^{-2k} \nrm{\rd P_{<k} B}_{L^{\infty}} \sum_{j \geq k - 5} 2^{2(k-j)} d_{j}\\
& + 2^{-k} \nrm{P_{<k} B}_{L^{\infty}} \sum_{j \geq k - 5}  2^{k-j}(e_{j} + d_{j}^{2}) \\
\aleq & d_{k} \nrm{P_{\geq k} B}_{\dot{H}^{1}} + 2^{-k} d_{k} ( 2^{-k} \nrm{\rd P_{<k} B}_{L^{\infty}} )
 + 2^{-k} (e_{k} + d_{k}^{2}) \nrm{P_{<k} B}_{L^{\infty}} .
\end{align*}
We apply this only for those good $k$'s where $d_k \ll 1$, which are all but finitely many. Then the first term on the far RHS can be absorbed into the LHS, and we obtain
\begin{equation} \label{eq:deltaA-B-hi}
	\nrm{B_{\geq k}}_{\dot{H}^{1}} \aleq 2^{-k} d_{k} ( 2^{-k} \nrm{\rd P_{<k} B}_{L^{\infty}}) + 2^{-k} (e_{k} + d_{k}^{2} ) \nrm{B_{<k}}_{L^{\infty}} .
\end{equation}
The LHS controls any $\nrm{B_{j}}_{\dot{H}^{1}}$ with $j \geq k$. Since for any $j \in \bbZ$ we can find a good $k < j$ such that $d_{k} \ll 1$ and $j -k = O(1)$, we have (after relabeling $j \to k$)
\begin{equation*} 
	c_{k} \aleq d_{k} \sum_{j < k} 2^{j - k} c_{j} + (e_{k} + d_{k}^{2}) \sum_{j < k} c_{j} \qquad \hbox{ for any } k > 0.
\end{equation*}
The first term on the RHS may be essentially absorbed into the second term after reiteration:
\begin{align*}
d_{k} \sum_{j < k} 2^{j - k} c_{j} 
\aleq & 2^{-k} d_{k} + d_{k} \sum_{0 < j < k} 2^{j - k} \left(d_{j} \sum_{i < j} 2^{i - j} c_{i} + (e_{j} + d_{j}^{2}) \sum_{i < j} c_{i} \right) \\
\aleq & 2^{-k} d_{k} + d_{k} \sum_{i < k} c_{i} \sum_{\max \set{i, 0} < j < k} \left(d_{j} 2^{i - k} + (e_{j} + d_{j}^{2}) 2^{j - k} \right) \\
\aleq & 2^{-k} d_{k} + d_{k}^{2} \sum_{i < k} 2^{(1-\dlt) (i - k)} c_{i} + d_{k} (e_{k} + d_{k}^{2})  \sum_{i < k} c_{i} \\
\aleq & 2^{-k} d_{k} + (e_{k} + d_{k}^{2}) \sum_{i < k} c_{i}.
\end{align*}
Plugging this bound back to the preceding bound, summing up in $k$ and using the relation $\sum_{k > 0} 2^{-k} d_{k} \aleq 1$, we arrive at
\begin{equation*} 
	1 + \sum_{0 < j \leq k} c_{k} \aleq \sum_{0 < j \leq k}(e_{j} + d_{j}^{2}) \left( 1+ \sum_{0 < i < j} c_{i} \right) \qquad \hbox{ for any } k > 0.
\end{equation*}
By induction on $k$, it follows that
\begin{equation*}
	1 + \sum_{j=1}^{k} c_{j} \aleq \prod_{j=1}^k \left( 1+C (d_j^2+e_j) \right).
\end{equation*}
In the first case we simply have $e_k \lesssim d_k$. 
Since $d_k \in \ell^2$, this yields 
\[
\sum_{j=1}^k c_j \lesssim e^{C \sqrt{k}},
\]
which suffices for the $\dot{H}^\sgm$ bound when $\sgm < 2$.

In the second case we have $d_k^2 + e_k \in \ell^1$ so we get instead instead
\[
\sum_{j=1}^k c_j \lesssim 1,
\]
which leads to
\[
c_k \lesssim d_k^2 + e_k . \qedhere
\]
\end{proof}

We continue with the frequency envelope version of the above result:

\begin{theorem}\label{t:deltaA-fe}
  Assume that $A \in \dot H^1$, with a $(-1,S)$ frequency envelope $c_k$.
  Also assume that $F \in \dot H^{-1}$ has a $1$-compatible $(-1,S)$
  frequency envelope $d_k$. Then the equation \eqref{deltaA} is solvable with bounds as follows:
\begin{equation}
\| B_k \|_{\dot H^1} \lesssim d_k .
\end{equation}
\end{theorem}

\begin{proof}
If $S = 1$ then  this follows from the previous result, and no compatibility condition is needed; so we may assume that $S > 1$. 
Note also that it suffices to consider a $(- \dlt, S)$ frequency envelope $d_{k}$, as $-\dlt$ can be improved to $-1$ using Theorem~\ref{t:deltaA}.
Let $C$ be minimal with the property that 
\begin{equation}\label{boot-ell}
\| B_k \|_{\dot H^1} \leq C d_k.
\end{equation}
To guarantee that such a $C$ exist, we can always replace $d_k$ by 
\[
d_k^\epsilon = \max \{d_k, \epsilon \} 
\]
These envelopes are  still $1$-compatible with $c_{k}$, and the desired result is obtained by letting $\epsilon \to 0$.

Now we write the equation for $B_k$ in the paradifferential form,
\begin{align*}
\Delta_{A_{<k}} B_k = G_k =:& P_k F 
 - P_{k} \left(2 ad(P_{\geq k} A^{\ell}) \rd_{\ell} B + ad(\rd^{\ell} P_{\geq k} A_{\ell}) B \right) \\
& - P_{k} \left( ad(P_{\geq k} A^{\ell}) ad(A_{\ell}) B + ad(P_{< k} A^{\ell}) ad(P_{\geq k} A_{\ell}) B\right)
 - [P_k, \Delta_{A_{<k}}] B ,
\end{align*}
where $A_{<k} = P_{<k} A$.
We use Littlewood--Paley to estimate
\[
\|  G_k\|_{\dot H^{-1}} \lesssim d_k (1+ C \sum_{j < k} 2^{j-k} c_j ) 
+ C c_k \sum_{j < k}  2^{j-k} d_j,
\]
where all contributions in the $high \times high \to low$ case are rapidly 
decreasing and subsumed in the $j= k$ term, provided that we choose $\dlt$ sufficiently small.

If $C \lesssim 1$ then we are done. Else, let $k$ be so that \eqref{boot-ell}
is near optimal. Then we must have 
\begin{equation*}
	C d_{k} \aleq d_{k} (1 + C \sum_{j < k} 2^{j-k} c_{j}) + C c_{k} \sum_{j < k} 2^{j-k} d_{j},
\end{equation*}
so either
\[
1 \aleq  \sum_{j < k} 2^{j-k} c_j ,
\]
or 
\[
d_k \aleq c_k \sum_{j < k}  2^{j-k} d_j .
\]
In the first case we must clearly have $c_k \approx 1$. In the second case, by the compatibility condition,
\begin{equation*}
	d_{k} \aleq c_{k} \sum_{j < k} 2^{j-k} d_{j} \aleq c_{k} 2^{-m \eps} \sum_{j < k - m} 2^{(1-\eps) (j-k)} d_{j} + c_{k} 2^{(S-1) m} d_{k} \aleq d_{k} \left(2^{-m \eps} + 2^{(S-1)m} c_{k} \right),
\end{equation*}
so that after choosing $m$ appropriately large, the same conclusion $c_{k} \approx 1$ holds.

Therefore, in both cases, by the compatibility condition we must also have
\[
d_k \ageq c_{k} \sum_{j < k} 2^{(1-\eps) (j-k)} d_{j} \ageq d_k^{[1]},
\]
and then the conclusion follows from the $S = 1$ case.
\end{proof}

\subsection{Solvability for \texorpdfstring{$\partial_s - \Delta_A$: $L^{2}$}{ds-DeltaA} theory}
Our goal here is  to study the parabolic equation 
\begin{equation}\label{heatA}
\partial_s B_{j} - \Delta_A B_{j} - 2 ad(\tensor{F}{_{j}}^{k}) B_{k} = G_{j} + \covD^{k} H_{k j}. 
\end{equation}
Here $B$, $G$ are $\g$-valued 1-forms and $H$ is a $\g$-valued covariant 2-tensor on $\bbR^{4} \times J$.
We assume that $A$ is a $L^\infty (J; \dot H^1)$ connection, with curvature
$F \in L^2 (J; \dot H^1)$ and $\partial_s A \in L^2 (J; L^{2})$. For simplicity, we will often omit the indices and abbreviate \eqref{heatA} as
\begin{equation*}
\left( \partial_s - \Delta_A - 2 ad(F) \right) B = G + \covD H.
\end{equation*}
Furthermore, we will skip writing out the heat-time interval $J$, and drop the dependence of implicit constants on $\nrm{A}_{L^{\infty} \dot{H}^{1}}$, $\nrm{F}_{L^{2} \dot{H}^{1}}$ and $\nrm{\rd_{s} A}_{L^{2} L^{2}}$.

We start with a basic solvability result:
\begin{theorem} \label{t:heatA}
Let $-2 < \sgm < 2$, and $A$ as above. Then the above equation is well-posed in $\dot H^{\sgm}$, with bounds\footnote{Here it is important to have covariant derivatives on the left if $\sgm \geq 1$, and on the right 
if $\sgm \leq -1$.}
\begin{equation}\label{parab-lin}
\| B \|_{L^\infty \dot H^\sgm} +  \|\covD B\|_{L^2 \dot H^{\sgm}} \lesssim \|B(0)\|_{\dot H^\sgm} + \|G\|_{L^1 \dot H^\sgm}
 +\|H\|_{ L^2 \dot H^{\sgm}}.
\end{equation}
\end{theorem}

\begin{proof}
We begin with the case $\sgm = 0$. When the term $2 ad(F) B$ is absent, the desired estimate follows by multiplying \eqref{heatA} by $B_{j}$ and integrating by parts over $\bbR^{4} \times J$. The contribution of the term $2 ad(F) B$ is then treated perturbatively, by splitting $J$ in to a finite number of intervals on each of which the $L^{2} \dot{H}^{1}$ norm of $F$ is small.

For $\sgm = 1$ we differentiate the equation to obtain the following (schematic) linear equation 
for $\covD B$:
\begin{equation} \label{eq:heatA-dB}
(\partial_s - \Delta_A - 2 Ad(F)) \covD B = \covD G + \covD^{2} H + (\partial_s A) B + (\covD F) B + F \covD B,
\end{equation}
and apply the $\sgm = 0$ result. The last three terms on the right are perturbative. 
By interpolation this yields the result for $0 \leq \sgm \leq 1$.

For $1 < \sgm < 2$ we use again the differentiated equation \eqref{eq:heatA-dB}, and perturb off the $\sgm-1$ result.
To insure that no additional derivative falls on $\rd_{s} A$ and $\covD F$, we write
\begin{equation*}
	(\rd_{s} A) B + (\covD F) B = \covD_{\ell} \covD^{\ell} \lap_{A}^{-1} ((\rd_{s} A) B + (\covD F) B)
\end{equation*}
and note that, by Theorem~\ref{t:deltaA},
\begin{align*}
	\nrm{\covD \lap_{A}^{-1} ((\rd_{s} A) B + (\covD F) B)(s)}_{\dot{H}^{\sgm-1}} \aleq & \nrm{((\rd_{s} A) B + (\covD F) B)(s)}_{\dot{H}^{\sgm-2}} \\
	\aleq & ( \nrm{\rd_{s} A(s)}_{L^{2}} + \nrm{F(s)}_{\dot{H}^{1}})\nrm{B(s)}_{\dot{H}^{\sgm}} .
\end{align*}
Hence the last three terms on the right in \eqref{eq:heatA-dB} can be treated perturbatively, by putting $(\rd_{s} A) B + (\covD F) B$ in $L^{2} \dot{H}^{\sgm-2}$ and $F \covD B$ in $L^{1} \dot{H}^{\sgm-1}$.

Finally, for negative $\sgm$ we use duality, as our assumptions are invariant with respect to  heat-time reversal. 
\end{proof}

We will also need a frequency envelope version of the above result. Simply the fact that 
this result applies for a range of indices $\sgm$ already allows us to obtain
the following

\begin{corollary} \label{cor:heatA-fe-low}
Assume that $d_k$ is a $(-2,2)$ frequency envelope for $B(0)$ in $L^2$, $G$ in $L^1 L^2$ 
and $H$ in $L^{2} L^2$. Then 
\begin{equation}
\| P_k B\|_{L^\infty L^2} + \|P_k \covD B\|_{L^2} \lesssim d_k.
\end{equation}
\end{corollary}

The lower range limit for this bound (more precisely, the lower admissibility of $d_{k}$) is entirely satisfactory, but we would like to  increase the upper limit in order to also have a higher regularity result. The obvious price to pay is that  we need a stronger 
assumption on the connection $A$. To quantify that we will use an $\ell^2$ frequency envelope $c_k$
so that
 \begin{equation}\label{hf-A-fe}
\| P_k A\|_{L^\infty \dot H^1} + \|P_k F\|_{L^2 \dot H^1}+ \|P_k \partial_s A\|_{L^2 L^{2}} \lesssim c_k.
\end{equation}
Then we have the following:

\begin{theorem} \label{t:heatA-fe}
Assume that \eqref{hf-A-fe} holds for some $(-1,S)$ frequency envelope $c_k$.
Let $d_k$ be a $1$-compatible $(-2,S)$ frequency envelope for $B(0)$ in $L^2$, $G$ in $L^1 L^2$ 
and $H$ in $L^2 L^{2}$. Then we have
\begin{equation}
\| P_k B\|_{L^\infty L^2} + \|P_k \covD B\|_{L^2 L^{2}} \lesssim d_k.
\end{equation}
\end{theorem}

\begin{proof}
The proof is analogous to that of Theorem~\ref{t:deltaA-fe}. We may assume that $S > 2$ and that $d_{k}$ is a $(-\dlt, S)$ frequency envelope.  Let $C$ be minimal so that 
  \begin{equation} \label{eq:boot-heat}
\| P_k B\|_{L^\infty L^2} + \|P_k \covD B\|_{L^2 L^{2}} \leq C d_k.
\end{equation}
To insure that such a $C$ exists, we can always relax $d_{k}$ in the high frequencies while keeping the compatibility condition, as in the proof of Theorem~\ref{t:deltaA-fe}. 
  
 We rewrite \eqref{heatA} in the paradifferential form
 \begin{align*}
	(\rd_{s} - \lap_{A_{<k}}) B_{k} 
	= & P_{k} G + \covD[A_{<k}] P_{k} H + P_{k} \left(2 ad(P_{\geq k} A) \rd B + ad(\rd P_{\geq k} A) B \right) \\
	& + P_{k} \left(ad(P_{\geq k} A) ad(A) B + ad(P_{< k} A) ad(P_{\geq k} A)  B \right) + P_{k} (2 ad(F) B) \\
	& + [P_{k}, \lap_{A_{<k}}] B + [P_{k}, \covD[A_{<k}] ] H \\
	= & P_{k} G + G'_{k} + \covD[A_{<k}] (P_{k} H + \covD[A_{<k}] \lap_{A_{<k}}^{-1} G''_{k}),
\end{align*}
where $A_{<k} = P_{< k} A$, $B_{k} = P_{k} B$ and
\begin{align*}
	G'_{k} = & P_{k} (2 ad(F) B), \\
	G''_{k} = &  P_{k} \left(2 ad(P_{\geq k} A) \rd B + ad(\rd P_{\geq k} A) B \right)\\
	& + P_{k} \left(ad(P_{\geq k} A) ad(A) B + ad(P_{< k} A) ad(P_{\geq k} A)  B \right) \\
	& +  [P_{k}, \lap_{A_{<k}}] B + [P_{k}, \covD[A_{<k}]] H.
\end{align*}
Recall that, by Theorem~\ref{t:deltaA},
\begin{equation*}
	\covD[A_{<k}] \lap_{A_{<k}}^{-1} : \dot{H}^{-1} \to L^{2}.
\end{equation*}
Applying the $L^{2}$ bound for $B_{k}$ in Theorem~\ref{t:heatA} (which does not require any curvature information), it follows that
\begin{equation*}
\| B_{k}\|_{L^\infty L^2} + \|\covD[A_{<k}] B_{k}\|_{L^2 L^{2}} \aleq d_{k} + \nrm{G'_{k}}_{L^{1} L^{2}} + \nrm{G''_{k}}_{L^{2} \dot{H}^{-1}}.
\end{equation*}
By Littlewood--Paley theory, we may estimate
\begin{align*}
	\nrm{G'_{k}}_{L^{1} L^{2}} + \nrm{G''_{k}}_{L^{2} \dot{H}^{-1}}
	\aleq & C c_{k}\sum_{j < k} 2^{j-k} d_{j} + C d_{k} \sum_{j < k} 2^{j-k} c_{j},
\end{align*}
and
\begin{align*}
	\nrm{P_{k} \covD[A] B}_{L^{2} L^{2}} \aleq & \nrm{\covD[A_{<k}] B_{k}}_{L^{2} L^{2}} + \nrm{P_{k} (ad(A_{\geq k}) B}_{L^{2} L^{2}} + \nrm{[P_{k}, \covD[A_{<k}] ] B}_{L^{2} L^{2}} \\
	\aleq & \nrm{\covD[A_{<k}] B_{k}}_{L^{2} L^{2}} + C d_{k} \sum_{j < k} 2^{j-k} c_{j}.
\end{align*}
where the $high \times high \to low$ interaction terms are again rapidly decreasing, and thus is subsumed to the $j = k$ terms after fixing $\dlt$ to be sufficiently small. Thus, we arrive at the estimate
\begin{equation}
\| P_k B\|_{L^\infty L^2} + \|P_{k} \covD[A] B\|_{L^2 L^{2}} \lesssim d_k + C(c_k \sum_{j < k} 2^{j-k} d_j 
+  d_k \sum_{j < k} 2^{j-k} c_j).
\end{equation}
Let $k$ be so that \eqref{eq:boot-heat} is near optimal.
Then we must have 
\[
C d_k  \lesssim d_k + C(c_k \sum_{j < k} 2^{j-k} d_j
+  d_k \sum_{j < k} 2^{j-k} c_j),
\]
and it follows that
\begin{equation*}
\hbox{either} \qquad	1 \aleq \sum_{j < k} 2^{j-k} c_{j} \qquad \hbox{ or } \qquad d_{k} \ageq c_{k} \sum_{j < k} 2^{j-k} d_{j}.
\end{equation*}
As in the proof of Theorem~\ref{t:deltaA-fe}, in either case we must have $c_{k} \approx 1$. Then, by the compatibility condition, $d_{k} \ageq d_{k}^{[1]}$, so the desired bound follows from the case $S = 1$. \qedhere
\end{proof}

Next, we prove an $L^{1}_{s}$-type bound.
\begin{theorem} \label{t:heatA-L1}
Consider the equation \eqref{heatA} with $H = 0$, $G \in L^{1} \dot{H}^{\sgm-1}$ and $B(0) \in \dot{H}^{\sgm-1}$, where $-1 < \sgm < 1$.
Then the solution $B$ obeys the bound
\begin{equation*}
	\nrm{\covD B}_{L^{1} \dot{H}^{\sgm}} \aleq \nrm{B(0)}_{\dot{H}^{\sgm-1}} + \nrm{G}_{L^{1} \dot{H}^{\sgm-1}}.
\end{equation*}
Assume that \eqref{hf-A-fe} holds for some $(-1, S)$ frequency envelope $c_{k}$. Let $d_{k}$ be a $1$-compatible $(-1, S)$ frequency envelope for $B(0)$ in $\dot{H}^{-1}$ and $G$ in $L^{1} \dot{H}^{-1}$. Then
\begin{equation*}
	\nrm{P_{k} (\covD B)}_{L^{1} L^{2}} \aleq d_{k}.
\end{equation*}
\end{theorem}

Note that the scaling of Theorem~\ref{t:heatA-L1} differs from Corollary~\ref{cor:heatA-fe-low} and Theorem~\ref{t:heatA-fe}. This reflects the fact that in what follows, Theorem~\ref{t:heatA-L1} will typically be applied to a covariant derivative $\covD B$ of a solution $B$ to \eqref{heatA}.

\begin{proof}
We directly prove the frequency envelope version. By Duhamel's formula, the general case is easily reduced to the homogeneous case $G = 0$. Note that $2^{k} d_{k}$ is a $(-2, S)$ frequency envelope for $B$ in $L^{2}$, which is $1$-compatible with $c_{k}$. Thus, by Theorem~\ref{t:heatA-fe}, we have
\begin{equation} \label{eq:heatA-L1-pf-1}
	\nrm{P_{k} B}_{L^{\infty} L^{2}} + \nrm{P_{k} (\covD B)}_{L^{2} L^{2}} \aleq 2^{k} d_{k}.
\end{equation}

On the other hand, commuting the equation with $s^{\frac{1}{2} + \eps}$, we obtain
\begin{equation*}
\left( \rd_{s} - \lap_{A} - 2ad(F) \right) s^{\frac{1}{2} + \eps}B = \frac{1+2\eps}{2} s^{-\frac{1}{2}+\eps}B = \frac{1+2\eps}{2} \covD^{\ell} \left( s^{-\frac{1}{2} + \eps} H_{\ell} \right).
\end{equation*}
where
\begin{equation*}
	H_{\ell} = \covD_{\ell} (\Dlt_{A})^{-1} B.
\end{equation*}
We claim that
 \begin{equation} \label{eq:heatA-L1-pf-H}
	\nrm{P_{k} H}_{L^{\infty} L^{2}}
	+ \nrm{P_{k} H}_{L^{2} \dot{H}^{1}} \aleq d_{k}.
\end{equation}
Assuming \eqref{eq:heatA-L1-pf-H}, we first conclude the proof. Using the $L^{\infty} L^{2}$ bound for $s < 2^{-2k}$, and the $L^{2} \dot{H}^{1}$ bound for $s > 2^{-2k}$, it follows that
\begin{equation*}
	\nrm{P_{k} (s^{-\frac{1}{2} + \eps} H)}_{L^{2} L^{2}} \aleq 2^{- 2 \eps k} d_{k}.
\end{equation*}
For $\eps > 0$ sufficiently small, $2^{-2 \eps k} d_{k}$ still satisfies the admissibility and the compatibility conditions. Thus, by Theorem~\ref{t:heatA-fe}, we have
\begin{equation} \label{eq:heatA-L1-pf-2}
	\nrm{s^{\frac{1}{2} + \eps} P_{k} (\covD B)}_{L^{2} L^{2}} \aleq 2^{- 2 \eps k} d_{k}.
\end{equation}
Interpolating \eqref{eq:heatA-L1-pf-1} and \eqref{eq:heatA-L1-pf-2}, we obtain the desired frequency envelope bound for $\covD B$ in $L^{1} L^{2}$.

It remains to prove \eqref{eq:heatA-L1-pf-H}. For the $L^{\infty} L^{2}$ bound, we first note that Theorem~\ref{t:deltaA-fe} and the $L^{\infty} L^{2}$ bound in \eqref{eq:heatA-L1-pf-1} imply 
\begin{equation*}
	\nrm{P_{k} (\lap_{A}^{-1} B)}_{L^{\infty} \dot{H}^{1}} \aleq d_{k}.
\end{equation*}
Then splitting $\covD = \rd + ad(A)$ and using Littlewood--Paley trichotomy, we estimate
\begin{align*}
	\nrm{P_{k} H}_{L^{\infty} L^{2}} 
	\aleq & \nrm{P_{k} (\covD \lap_{A}^{-1} B)}_{L^{\infty} L^{2}} \\
	\aleq & \nrm{P_{k} (\lap_{A}^{-1} B)}_{L^{\infty} \dot{H}^{1}} + \nrm{P_{k} (ad(A) \lap_{A}^{-1} B)}_{L^{\infty} L^{2}} \\
	\aleq & d_{k} + c_{k} \sum_{j < k} 2^{j - k} d_{j} \aleq d_{k},
\end{align*}
where we used the compatibility condition in the last inequality.

On the other hand, the $L^{2} \dot{H}^{1}$ bound needs a bit of additional work, in order to make use of the $L^{2} L^{2}$ bound for $\covD B$ in \eqref{eq:heatA-L1-pf-1}. We begin by writing
\begin{equation*}
	H = \lap_{A}^{-1} \lap_{A} H =  \lap_{A}^{-1} \covD_{\ell} B + \lap_{A}^{-1} [\lap_{A}, \covD_{\ell}] \lap_{A}^{-1} B.
\end{equation*}
The contribution of the first term on the RHS is directly dealt with Theorem~\ref{t:deltaA-fe}. For the second term, note that schematically, $[\lap_{A}, \covD_{\ell}] = ad (\covD F) + ad(F) \covD$. Using Littlewood--Paley trichotomy, as well as the frequency envelope bounds for $F \in L^{2} \dot{H}^{1}$, $\covD F \in L^{2} L^{2}$, $\lap_{A}^{-1} B \in L^{\infty} \dot{H}^{-1}$ and $\covD \lap_{A}^{-1} B \in L^{\infty} L^{2}$, we have
\begin{equation*}
	\nrm{P_{k} (ad(\covD F) \lap_{A}^{-1} B)}_{L^{2} \dot{H}^{-1}}
	+ \nrm{P_{k} (ad(F) \covD \lap_{A}^{-1} B)}_{L^{2} \dot{H}^{-1}}
	\aleq d_{k} + c_{k} \sum_{j < k} 2^{j - k} d_{j} \aleq d_{k}.
\end{equation*}
The desired $L^{2} \dot{H}^{1}$ bound for $H$ now follows from Theorem~\ref{t:deltaA-fe} and an argument to peel off $ad(A) = \covD - \rd$ as in the previous case.
\end{proof}

\subsection{Solvability for \texorpdfstring{$\rd_{s} - \lap_{A}$: $L^{p}$}{ds-DeltaA-p} 
parabolic regularity} \label{subsec:heat-Lp}
Next, we consider $L^{p}$ solvability of \eqref{heatA}, and parabolic regularity properties of its solutions. 
For this, we assume that $A$ satisfies a stronger parabolic regularity property:
\begin{equation}\label{A=parab}
\| P_k A (s)\|_{\dot H^1} \lesssim c_k (1+2^{2k} s)^{-N} \qquad \hbox{ for all } s > 0,
\end{equation}
where $\{c_k\} \in \ell^2$ and $N > 2$. Then for the homogeneous covariant equation
\begin{equation}\label{heatA+}
(\partial_s - \Delta_A - 2Ad(F)) B = 0, \qquad B(0) = b,
\end{equation}
the following result holds:
\begin{theorem} \label{t:heatB}
Let $A$ be as in \eqref{A=parab} with a $(-\dlt, S)$ compatible frequency envelope $c_{k}$, and 
\begin{equation} \label{eq:heatB-sp}
-2 < \sgm < \frac{4}p, \qquad 2 \leq p \leq \infty, \qquad 0 < \dlt \leq \min\set{\frac{4}{p} - \sgm, 2 + \sgm}.
\end{equation}
Let $d^{\sgm, p}_{k}$ be a $(-\dlt, S)$ frequency envelope for $b$ in $\dot{W}^{\sgm, p}$, which is $\dlt$-compatible with $c_{k}$.  
Then there exists a unique solution $B \in C([0, \infty); \dot{W}^{\sgm, p})$ to \eqref{heatA+}, and we have
\begin{equation}
\| P_k B(s)\|_{\dot W^{\sgm, p}} \lesssim_{\nrm{c}_{\ell^{2}}, N}  d_k^{\sgm, p} (1+2^{2k} s)^{-N}.
\end{equation}
\end{theorem}

\begin{proof}
We proceed in several steps.

\pfstep{Step~1: $A = 0$, inhomogeneous case}
We start with the constant coefficient case $A = 0$. We will treat the general case as a perturbation of this case. For this purpose, we need a slightly refined estimate for the inhomogeneous equation. 

Consider the solution $B$ to the inhomogeneous equation
\begin{equation*}
	(\rd_{s} - \lap) B = G, \qquad B(s_{0}) = 0,
\end{equation*}
with
\begin{equation*}
	\nrm{P_{k} G}_{\dot{W}^{\sgm, p}} \leq 2^{2k} d^{\sgm, p}_{k} (2^{2k} s)^{-\bt} (1+2^{2k} s)^{-M}.
\end{equation*}
where $d^{\sgm, p}_{k}$ an arbitrary positive sequence here, and $\bt < 1$. Then for any $0 < \alp \leq 1$, we have
\begin{equation} \label{eq:inhom-free}
	\nrm{P_{k} B(s)}_{\dot{W}^{\sgm, p}} \aleq_{\alp} d^{\sgm, p}_{k} (2\min \set{2^{2k} \abs{s - s_{0}}, 2^{2k} s}^{\alp} (2^{2k} s)^{-\bt} (1+2^{2k} s)^{-M}.
\end{equation}
Indeed, note by Duhamel's formula, $B$ takes the form
\begin{equation*}
	B(s) = \int_{s_{0}}^{s} e^{(s - \tilde{s}) \lap} G(\tilde{s}) \, d \tilde{s}.
\end{equation*}
We consider two cases:
\begin{enumerate}
\item {\it Short interval.} If $s < 2 s_{0}$, then $\tilde{s} \in (s_{0}, s)$ obeys $\tilde{s} \approx s$, so that
\begin{align*}
	\nrm{P_{k} B(s)}_{\dot{W}^{\sgm, p}} 
	\aleq & \int_{s_{0}}^{s} \nrm{e^{(s - \tilde{s}) \lap} G(\tilde{s})}_{\dot{W}^{\sgm, p}} \, d \tilde{s} \\
	\aleq & \int_{s_{0}}^{s} (s - \tilde{s})^{-1+\alp} \nrm{G(\tilde{s})}_{\dot{W}^{\sgm - 2 (1-\alp), p}} \, d \tilde{s} \\
	\aleq & d^{\sgm, p}_{k} (2^{2k} \abs{s - s_{0}})^{\alp} (2^{2k} s)^{-\bt} (1+2^{2k} s)^{-M}.
\end{align*}
\item {\it Long interval.} If $s \geq 2 s_{0}$, then we split $\int_{s_{0}}^{s} = \int_{s_{0}}^{\frac{s}{2}} + \int_{\frac{s}{2}}^{s}$ and proceed as follows:
\begin{align*}
	\nrm{P_{k} B(s)}_{\dot{W}^{\sgm, p}}
	\leq &   \nrm{e^{\frac{s}{2} \lap} \int_{s_{0}}^{\frac{s}{2}} e^{(\frac{s}{2} - \tilde{s}) \lap} G(\tilde{s}) \, d \tilde{s}}_{\dot{W}^{\sgm, p}} 
	+ \int_{\frac{s}{2}}^{s} \nrm{e^{(s - \tilde{s}) \lap} G(\tilde{s})}_{\dot{W}^{\sgm, p}} \, d \tilde{s} \\
	\aleq & (1+2^{2k} s)^{-M} \int_{s_{0}}^{\frac{s}{2}} (\frac{s}{2} - \tilde{s})^{-1+\alp} \nrm{G(\tilde{s}) }_{\dot{W}^{\sgm-2(1-\alp), p}} \, d \tilde{s} \\
	& + \int_{s_{0}}^{\frac{s}{2}} (s - \tilde{s})^{-1+\alp} \nrm{G(\tilde{s}) }_{\dot{W}^{\sgm - 2(1-\alp), p}} \, d \tilde{s} \\
	\aleq & d^{\sgm, p}_{k} (2^{2k} s)^{\alp - \bt}  (1+2^{2k} s)^{-M}.
\end{align*}
\end{enumerate}

\pfstep{Step~2: $A \neq 0$, homogeneous case}
Next, we consider the covariant homogeneous equation \eqref{heatA+} with an arbitrary $A$ satisfying \eqref{A=parab}.
Let $\dlt_{0} = \min \set{2+\sgm - \dlt(1-\eps), \frac{1}{2}}$, and consider the slowly varying envelope 
\begin{equation*}
\tilde{c}_{k} = \sup_{j} 2^{-\dlt_{0}(1-\eps) \abs{j-k}} c_{j}.
\end{equation*}
Given a small constant $\eps_{0}$ to be fixed below, we split the time interval $[0,\infty)$ into finitely many subintervals $J_j=[s_{j},s_{j+1}]$ so that each $J_j$ has one of the following properties:
\begin{itemize}
\item Either $\tilde{c}_{k(s)} \leq \eps_{0}$ for all $s \in J_j$;

\item or, $s_{j+1} - s_j \leq \eps_{0} s_j$.
\end{itemize}
Once $\eps_{0}$ is fixed, the number of such subintervals can be bounded by a constant depending on $\nrm{c}_{\ell^{2}}$ and $\eps_{0}$.

In both cases we solve the problem perturbatively off the constant coefficient case (Case~1). We expand $\Dlt_{A}$ and $F$ in terms of $A$, and then write the equation schematically as 
\[
(\partial_s - \Delta) B = ad(A) \rd B + ad(\rd A) B + ad(A) ad(A) B, \qquad B(s_j) = b_j,
\]
or equivalently 
\[
B(s) = e^{(s-s_j) \Delta} b_j + (L B)(s) 
\]
where 
\[
(LB)(s) = \int_{s_j}^s e^{(\tilde s-s_j) \Delta}   (ad(A) \rd B + ad(\rd A) B + ad(A) ad(A) B)(\tilde s) d\tilde s.
\]
We make the induction hypothesis that 
\begin{equation} \label{eq:heatB-ind}
\| P_k b_j \|_{\dot W^{\sgm,p}} \lesssim  d_k^{\sgm,p} (1+2^{2k} s_j)^{-N},
\end{equation}
and use a fixed point argument in the space $X$ with the norm
\[
\| B \|_{X} = \sup_{s \in J_j} \sup_{k}  (d_k^{\sgm,p})^{-1}(1+2^{2k} s)^{-N} \|B_k(s)\|_{\dot W^{\sgm,p}}.
\]
Observe that \eqref{eq:heatB-ind} follows from the hypothesis in the initial step $j = 1$. To continue the induction, it suffices to show that in both cases $L$ is a contraction in $X$. 

Indeed, suppose $B$ satisfies
\[
  \|P_{k} B(s)\|_{\dot W^{\sgm,p}} \leq  (d_k^{\sgm,p})(1+2^{2k} s)^{-N}.
\]
Then we seek to estimate the expression
\[
P_{k} G(s) =  P_{k} (ad(A) \rd B + ad(\rd A) B + ad(A) ad(A) B )(s)
\]
We separate out essentially the $high \times high \to low$ interaction:
\begin{align*}
	G_{k}^{hh} = & ad(P_{\geq k-5} A) \rd P_{\geq k} B + ad(\rd P_{\geq k-5}A) P_{\geq k} B \\
	& + ad(P_{\geq k-5} A) ad(A) P_{\geq k} B + ad(P_{< k-5} A) ad(A_{\geq k-5}) P_{\geq k} B 
\end{align*}
and decompose $P_{k} G$ into $P_{k} G_{k}^{hh}$ and $P_{k} G_{k}^{lh} = P_{k} G - P_{k} G_{k}^{hh}$.
Using the standard Littlewood--Paley trichotomy, we obtain
\begin{align*}
\nrm{P_{k} G_{k}^{lh}(s)}_{\dot{W}^{\sgm, p}}
\aleq & 2^{2k} c_{k} \sum_{\ell < k} 2^{(\frac{4}{p} - \sgm) (\ell-k)} d_{\ell}^{\sgm, p}(1 + 2^{2k} s)^{-N} (1+2^{2\ell} s)^{-N} \\
& + 2^{2k} d^{\sgm, p}_{k} \sum_{\ell < k} 2^{\ell-k} c_{\ell} (1 + 2^{2k} s)^{-N} (1+2^{2\ell} s)^{-N},  \\
\nrm{P_{k} G_{k}^{hh}(s)}_{\dot{W}^{\sgm, p}}
\aleq & 2^{2k} \sum_{\ell > k} 2^{\sgm (k -\ell)} c_{\ell} d^{\sgm, p}_{\ell} (1+ 2^{2\ell} s)^{-2N}.
\end{align*}
For $0 < \alp < 2 N$, we have the simple inequalities
\begin{align}
	\sum_{\ell < k} 2^{\alp \ell} (1+2^{2\ell} s)^{-N} \aleq & 2^{\alp k} (1+2^{2k} s)^{-\frac{\alp}{2}}, \\
	\sum_{\ell > k} 2^{\alp \ell} (1+2^{2\ell} s)^{-N} \aleq & 2^{\alp k} (2^{2k} s)^{-\frac{\alp}{2}} (1+2^{2k} s)^{-N+\frac{\alp}{2}},
\end{align}
so it follows that
\begin{align*} 
	\nrm{P_{k} G^{lh}_{k}(s)}_{\dot{W}^{\sgm, p}} 
	\aleq & 2^{2k} c_{k} (\sum_{\ell < k} 2^{\dlt (\ell-k)} d_{\ell}^{\sgm, p} ) (1+2^{2k} s)^{-N - \dlt_{1}} \\
	& + 2^{2k} d_{k} ( \sum_{\ell < k} 2^{\dlt_{0} (\ell-k)} c_{\ell} ) (1+2^{2k} s)^{-N-\frac{1}{2} (1-\dlt_{0})} \\
	\nrm{P_{k} G^{hh}_{k}(s)}_{\dot{W}^{\sgm, p}} 
	\aleq & 2^{2k} ( \sup_{\ell > k} 2^{(2 + \sgm)(1 - \eps)(k-\ell)} c_{\ell} d^{\sgm, p}_{\ell} ) (2^{2k} s)^{-1+ \frac{2+\sgm}{2} \eps}(1+2^{2k} s)^{-2N+1 - \frac{2+\sgm}{2} \eps}, 
\end{align*}
where $\dlt_{1} = \frac{1}{2} (\frac{4}{p} - \sgm - \dlt) > 0$. Using the slowly varying properties
\begin{equation*}
d_{\ell}^{\sgm, p} \aleq 2^{\dlt(1-\eps)(\ell-k)} d_{k}^{\sgm, p} \hbox{ for } \ell > k, \qquad
c_{\ell} \aleq 2^{\dlt_{0}(1-\eps) \abs{\ell-k}} \tilde{c}_{k} \hbox{ for any }\ell,
\end{equation*}
as well as the following consequence of $\dlt$-compatibility
\begin{equation*}
	c_{k} \sum_{\ell < k} 2^{\dlt(\ell - k)} d^{\sgm, p}_{\ell}
	\aleq 2^{- \dlt \eps m} c_{k}  \sum_{\ell < k-m} 2^{\dlt(1-\eps)(\ell - k)} d^{\sgm, p}_{\ell} + 2^{S m} c_{k} d^{\sgm, p}_{\ell}
	\aleq (2^{- \dlt \eps m} + 2^{S m} \tilde{c}_{k}) d^{\sgm, p}_{k},
\end{equation*}
where $m > 0$ is to be chosen below, we obtain
\begin{align} 
	\nrm{P_{k} G^{lh}_{k}(s)}_{\dot{W}^{\sgm, p}} 
	\aleq & 2^{2k} (2^{- \dlt \eps m} + 2^{S m} \tilde{c}_{k}) d^{\sgm, p}_{k} (1+2^{2k} s)^{-N - 2 \dlt_{2}} , \label{eq:heatB-key-lh} \\
	\nrm{P_{k} G^{hh}_{k}(s)}_{\dot{W}^{\sgm, p}} 
	\aleq & 2^{2k} \left(\frac{2^{2k} s}{1 + 2^{2k } s} \right)^{-1 + 2 \dlt_{2}} \tilde{c}_{k} d_{k} (1+2^{2k} s)^{-N-2\dlt_{2}} , \label{eq:heatB-key-hh}
\end{align}
with $\dlt_{2}= \frac{1}{2}\min\set{\dlt_{1}, \frac{1}{2}(1-\dlt_{0}), N-1+\frac{2+\sgm}{2}\eps, \frac{2+\sgm}{2} \eps}$.

Applying Step~1 to each piece, with $\alp = \dlt_{2}$ for $G^{lh}_{k}$ and $\alp = 1 - \dlt_{2}$ for $G^{hh}_{k}$,  this leads to
\begin{equation*}
	\nrm{P_{k} LB(s)}_{\dot{W}^{\sgm, p}} \aleq  (1+2^{2k} s)^{-N-\dlt_{2}} (2^{- \dlt \eps m} + 2^{S m} \tilde{c}_{k})  d^{\sgm, p}_{k} \min\set{1, 2^{2k} \abs{J_{j}}}^{\dlt_{2}} .
\end{equation*}

Now we consider the two scenarios (this is where we fix $m$ and $\eps_{1}$):
\begin{enumerate}
\item {\it Short intervals.} 
Here $J_{j} = [s_j,s_{j+1}]$ where $s_j,s_{j+1} \approx 2^{-2k_j}$ and $|s_{j+1}-s_j| \aleq \eps_{1} 2^{-2k_j}$.
Then we gain from the first factor if $2^{2k} \gg 2^{2k_j}$, and from the last otherwise.
\item {\it Long intervals.} 
Here $J_{j} = [s_j,s_{j+1}]$ where $s_j \approx 2^{-2k_j}$ and $s_1 \approx 2^{-2k_{j+1}}$, and 
$|J_{j}| \approx 2^{-2k_{j+1}}$.  Then we gain from the first factor if $2^{2k} \gg 2^{2k_j}$, from $2^{-\dlt \eps m} + 2^{S m} \tilde{c}_{k} \leq 2^{-\dlt \eps m} + 2^{Sm} \eps_{1}$ if $2^{2k_{j}} \aleq 2^{2k} \aleq 2^{2k_{j+1}}$, and from the last factor otherwise. \qedhere
\end{enumerate}
\end{proof}

We also need a version of the previous result
for the inhomogeneous equation
\begin{equation}\label{heatB+}
(\partial_s - \Delta_A - 2Ad(F)) B = G, \qquad B(0) = 0,
\end{equation} 

\begin{theorem} \label{t:heatB-inhom} Let $0 \leq \beta < 1$.
 Let $d_{k}$ be $(-1, S)$ admissible and $1$-compatible with $c_{k}$. If
\begin{equation}
\| P_k G \|_{L^2} \lesssim  2^{2k} (2^{2k} s)^{-\beta} d_{k} (1+2^{2k} s)^{-N}.
\end{equation}
Then
\begin{equation}
\| P_k B \|_{L^{2}} \lesssim  (2^{2k} s)^{1-\beta}  d_{k}     (1+2^{2k} s)^{-N}.
\end{equation}
Similarly, for general $(\sgm, p)$ as in \eqref{eq:heatB-sp}, 
 if $d_{k}$ is $(-\dlt, S)$ admissible and $\dlt$-compatible with $c_{k}$, and if
\begin{equation}
\| P_k G \|_{\dot W^{\sgm,p}} \lesssim  2^{2k} (2^{2k} s)^{-\beta} d^{\sgm,p}_{k} (1+2^{2k} s)^{-N}.
\end{equation}
then we have
\begin{equation}
\| P_k B \|_{\dot W^{\sgm,p}} \lesssim  (2^{2k} s)^{1-\frac{\alp}{2}} d^{\sgm,p}_{k}     (1+2^{2k} s)^{-N}.
\end{equation}
\end{theorem}

\begin{proof}
The proof repeats the proof of the previous theorem. Step 1 is reused in its entirety, as well as the 
interval partition in Step 2.  The only difference is in the choice of the induction hypothesis 
\eqref{eq:heatB-ind}, which is now replaced by 
\begin{equation} \label{eq:heatB-ind-inhom}
\| P_k b_j \|_{\dot W^{\sgm,p}} \lesssim (2^{2k} s)^{1-\frac{\alp}{2}}  d_k^{\sgm,p} (1+2^{2k} s_j)^{-N},
\end{equation}
But this is still consistent with \eqref{eq:inhom-free}, so the rest of the argument is again identical 
to the previous proof.
\end{proof}

\section{The Yang--Mills heat flow and the local caloric gauge} \label{s:caloric-loc}
 
\subsection{The covariant and dynamic Yang--Mills heat flows} \label{subsec:dymhf-re}

Our first goal here is to introduce the equations for the Yang--Mills heat flow
and its linearization in a gauge independent fashion. The Yang--Mills heat flow
models the parabolic evolution of a connection 1-form $A = A_{j} \, \ud x^{j}$.
This can be thought of as the gradient flow associated to the functional
\[
\spE[A] = \frac12 \int_{\bbR^{4}} \langle F_{jk}, F^{jk} \rangle \, d x.
\]
We will denote by  $s \in [0,\infty)$ the heat-time variable. 
To describe this flow covariantly we add an $s$ component $A_s$ 
to the connection, so that our connection 1-form 
is now
\[
 A_{j} \, \ud x^{j} + A_{s} \, \ud x
\]
One can think of $A_s$ as the generator of a semigroup of gauge 
transformations.  

\begin{definition}
Given an interval $J \subseteq [0, \infty)$, we say that a connection
1-form $A = A_{j} \, \ud x^{j} + A_{s} \, \ud x$ on $\bbR^{4} \times
J$ is a \emph{covariant Yang--Mills heat flow} if it solves
\begin{equation} \label{eq:cYMHF}
	F_{si} = \covD^{\ell} F_{\ell i}
\end{equation}
\end{definition}
\noindent
This equation is invariant under pointwise gauge transformations on $\bbR^{4} \times J$.

Assuming that \eqref{eq:cYMHF} holds, it is not difficult to see that the curvature
tensors also must solve their own system of covariant parabolic equations:
\begin{lemma} \label{lem:cymhf-eqns} 
Let $ A_{j} \, \ud x^{j} + A_{s} \, \ud s$ be a sufficiently regular solution to
  \eqref{eq:cYMHF}. Then the curvature components $F_{ij}$ and
  $F_{si}$ obey the following covariant parabolic equations.
\begin{align} 
\label{eq:YMHF-Fij}
	\covD_{s} F_{ij} - \covD^{\ell} \covD_{\ell} F_{ij}
	= & - 2 \LieBr{\tensor{F}{_{i}^{\ell}}}{F_{j \ell}}, \\
\label{eq:YMHF-Fsi}
	\covD_{s} F_{si} - \covD^{\ell} \covD_{\ell} F_{si}
	= & - 2 \LieBr{\tensor{F}{_{s}^{\ell}}}{F_{i \ell}}. 
\end{align}
\end{lemma}

Next we consider the linearization of the covariant Yang--Mills heat flow. 
To describe this we consider a smooth one parameter family $A(t)$ 
of covariant Yang--Mills heat flows. Here, anticipating the use of $t$ 
as the hyperbolic Yang--Mills time variable, we will use the index $0$ 
for $t$ derivatives. Then we seek to write the equations for 
\[
B = \partial_0 A
\]
A direct linearization in the equations \eqref{eq:cYMHF} yields the evolution
\begin{equation}\label{eq:cYMHF-lin}
\covD_s B_i - \covD_i B_s =  \covD^\ell (\covD_\ell B_i - \covD_i B_\ell)+ [B^\ell,F_{\ell i}] 
\end{equation}
This is invariant with respect to gauge transformations which depend 
on $s$ and $x$. Further, the effect of a one-parameter family of gauge transformations $O(t)$ in the 
equations \eqref{eq:cYMHF} that satisfies $O(t=0) = Id$ is the invariance of \eqref{eq:cYMHF-lin} under the following transformation:
\begin{equation} \label{eq:cYMHF-lin-gt}
B_i \to B_i - \covD_i A_0, \qquad B_s \to B_s - \covD_s A_0
\end{equation}
where $A_0(s,x) = O_{;t}(0, s, x)$ is the linearized action of $O$ at $t = 0$. In this sense, \eqref{eq:cYMHF-lin-gt} may be regarded as a \emph{linearized gauge transformation} for \eqref{eq:cYMHF-lin}. Another useful statement that is equivalent to the invariance of \eqref{eq:cYMHF-lin} under \eqref{eq:cYMHF-lin-gt} is that $B_{i} \ud x^{i} + B_{s} \ud s = \covD_{i} A_{0} \ud x^{i} + \covD_{s} A_{0} \ud s$ solves the linearized flow \eqref{eq:cYMHF-lin} for any $\g$-valued function $A_{0}$. 

 To achieve full covariance, we 
proceed as above and add to our connection the $s$ dependent $A_0$ component:
\[
 A_{\alpha} \, \ud x^{\alpha} + A_{s} \, \ud s.
\]
Instead of tracking $B$, we are now considering the parabolic 
evolution of the corresponding curvature tensor $F_{0j}, F_{s0}$, 
which  can be interpreted as a gauge-covariant deformation
of $A_j,A_0$.  These are related to $B_j, B_0$ via the relations
\begin{equation}\label{BvsF}
F_{0j} = B_j - \covD_j A_0, \qquad F_{0s} = B_s - \covD_s A_0.
\end{equation}
Due to the above gauge invariance, these still solve the equations
\eqref{eq:cYMHF-lin}.

So far, the choice of $A_0$ was arbitrary. 
We \emph{specify} covariantly the parabolic evolution of $A_{0}$:

\begin{definition}
Given an interval $J \subseteq [0, \infty)$, we say that a connection
1-form $ A_{\alpha} \, \ud x^{\alpha} + A_{s} \, \ud x$ on $\bbR^{4} \times
J$ is a \emph{dynamic (covariant) Yang--Mills heat flow} if it solves
\begin{equation} \label{eq:cdYMHF}
	F_{s\alpha} = \covD^{\ell} F_{\ell \alpha}
\end{equation}
\end{definition}
\noindent
Compared with \eqref{eq:cYMHF}, here we have added the $\alpha = 0$
equation, which specifies $A_0$ given its value at the initial heat-time $s=0$. Indeed, $F_{s0} = \covD^{\ell} F_{\ell 0}$ may be equivalently written as
\begin{equation} \label{eq:cdYMHF-0-exp}
	\covD_{s} A_{0} - \covD^{\ell} \covD_{\ell} A_{0} = B_{s} - \covD^{\ell} B_{\ell}.
\end{equation}
This choice is independent of the equation \eqref{eq:cYMHF-lin} satisfied by $B$. On the other hand, the advantage of writing \eqref{eq:cdYMHF} instead of \eqref{eq:cdYMHF-0-exp} is that the former is manifestly covariant under any $t$- and/or $s$-dependent gauge transformation $O = O(t, x, s)$.

Assuming that $A$ is a dynamic Yang--Mills heat flow,
one can differentiate to obtain the desired covariant parabolic equations 
for the curvature, which expand the equations in Lemma~\ref{lem:cymhf-eqns}:
\begin{lemma} \label{lem:cymhf-eqns+} 
Let $A = A_{\alpha} \, \ud x^{\alpha} + A_{s} \, \ud s$ be a sufficiently regular solution to
  \eqref{eq:cdYMHF}. Then the curvature components $F_{\alpha\beta}$ and
  $F_{s\alpha}$ obey the following covariant parabolic equations.
\begin{align} 
\label{eq:YMHF-Fij+}
	\covD_{s} F_{\alpha \beta} - \covD^{\ell} \covD_{\ell} F_{\alpha \beta}
	= & - 2 \LieBr{\tensor{F}{_{\alpha}^{\ell}}}{F_{\beta \ell}}, \\
\label{eq:YMHF-Fsi+}
	\covD_{s} F_{s\alpha} - \covD^{\ell} \covD_{\ell} F_{s\alpha}
	= & - 2 \LieBr{\tensor{F}{_{s}^{\ell}}}{F_{\alpha \ell}}. 
\end{align}
\end{lemma}

At this point, the way one should think of the $\alpha = 0$ component of these  equations is as a
covariant  equivalent formulation of  the linearized equations  \eqref{eq:cYMHF-lin},
given that $A_0$ satisfies the gauge fixing condition \eqref{eq:cdYMHF} for $\alpha = 0$; 
see Section~\ref{subsec:dymhf}.

\begin{remark}[\eqref{eq:cdYMHF} with $\alpha = 0$ as an infinitesimal de Turck trick]
Note that, by imposing the equation $F_{s0} = \covD^{\ell} F_{\ell 0}$, we have arrived at a \emph{nondegenerate} parabolic equation for $F_{0 j}$ in \eqref{eq:YMHF-Fij+}, as opposed to the degenerate parabolic equation \eqref{eq:cYMHF-lin}. This statement can be alternatively seen as follows, in a way that resembles the classical de Turck trick for \eqref{eq:cYMHF} (cf.~\eqref{deTurck}--\eqref{YM-h-deTurck}). Using $\covD^{\ell} \covD_{i} B_{\ell} = \covD_{i} \covD^{\ell} B_{\ell} + [\tensor{F}{^{\ell}_{i}}, B_{\ell}]$, \eqref{eq:cYMHF-lin} can be rewritten as 
\begin{equation*}
	\covD_{s} B_{i} - \covD^{\ell} \covD_{\ell} B_{i} + 2 [\tensor{F}{^{\ell}_{i}}, B_{\ell}] = \covD_{i} (B_{s} - \covD^{\ell} B_{\ell}).
\end{equation*}
If one attempts to cancel the undesirable RHS by a linearized gauge transformation of the form \eqref{eq:cYMHF-lin-gt}, then one is naturally led to the condition \eqref{eq:cdYMHF-0-exp}, which is equivalent to \eqref{eq:cdYMHF} with $\alp = 0$ under the definitions $F_{0 j} = B_{j} - \covD_{j} A_{0}$ and $F_{0 s} = B_{s} - \covD_{s} A_{0}$.
\end{remark}

\subsection{Covariant bounds for solutions}

Postponing for the moment the gauge dependent well-posedness question,
we now explore the possible covariant bounds for sufficiently regular  solutions.
These are necessarily curvature based. The first is the monotonicity formula:

\begin{proposition}
Let $A$ be a sufficiently regular covariant Yang--Mills heat flow. 
Then for $0 \leq s_0 \leq s_1$ we have the relation
\begin{equation}
  \int \frac{1}{2} \brk{F_{ij}, F^{ij}} (s_{1}) \, \ud x + \int_{s_{0}}^{s_{1}} \int 
\brk{\covD^{\ell} F_{\ell i}, \covD^{\ell} \tensor{F}{_{\ell}^{i}}} \, \ud x \ud s
= \int \frac{1}{2} \brk{F_{ij}, F^{ij}} (s_{0}) \, \ud x.
\end{equation}
\end{proposition}
This is verified by a direct computation. Based on this relation, one
expects that well-behaved solutions to the curvature equations satisfy
\begin{equation*}
	F \in L^{\infty} L^{2}, \quad \covD F \in L^{2} L^{2}.
\end{equation*}
By the diamagnetic inequality, the latter bound implies $F \in L^{2} L^{4}$. Interpolating between this norm and $L^{\infty} L^{2}$, one obtains in particular
\begin{equation*}
	F \in L^{3} L^{3}.
\end{equation*}
This norm plays a key role in our 
analysis, as it serves as a continuation criteria for solutions. 
More precisely, we have the following:
\begin{proposition} \label{p:cov-smth}
Let $A$ be a sufficiently regular covariant Yang--Mills heat flow on the heat-time interval $J = [s_{0}, s_{1})$. Suppose that
\begin{equation}\label{l3}
	\nrm{F}_{L^{3}_{s}(J, L^{3}_{x})} \leq \hM < \infty.
\end{equation}
for some $\hM < \infty$. 
\begin{enumerate}
\item Then for any $n \geq 0$, there exists $\hM_{n}  = \hM_{n}(\nrm{F(s=s_{0})}_{L^{2}}, \hM) < \infty$ such that 
\begin{align} \label{est-cov}
\nrm{s^{n/2} \covD_{x}^{(n)} F}_{L^{\infty}_{\frac{\ud s}{s}} (J; L^{2} _{x})}
	+ \nrm{s^{(n+1)/2} \covD_{x}^{(n+1)} F}_{L^{2}_{\frac{\ud s}{s}} (J; L^{2} _{x})}
\aleq \hM_{n}.
\end{align}
\item Suppose, in addition, that $(A_{0}, A)$ is a sufficiently regular dynamic Yang--Mills heat flow on $J$. Then for any $n \geq 0$, there exists $\hM_{0, n} = \hM_{0, n}(\nrm{F_{0x}(s = s_{0})}_{L^{2}}, \nrm{F(s=s_{0}}_{L^{2}}, \hM) < \infty$ such that
\begin{align} \label{est-cov-0}
\nrm{s^{n/2} \covD_{x}^{(n)} F_{0j}}_{L^{\infty}_{\frac{\ud s}{s}} (J; L^{2} _{x})}
	+ \nrm{s^{(n+1)/2} \covD_{x}^{(n+1)} F_{0j}}_{L^{2}_{\frac{\ud s}{s}} (J; L^{2} _{x})}
\aleq \hM_{0, n}.
\end{align}
\end{enumerate}
\end{proposition}
Here the $L^3 L^{3}$ bound in the hypothesis is applied only to the spatial components of the 
curvature. However, the conclusion applies as well to the $F_{0j}$ components.

\begin{proof}
We prove each part in order. 
\pfstep{Proof of (1)}
We proceed by an induction argument. Assume that, for $0 \leq n' \leq n-1$, we have
\begin{equation} \label{eq:cov-est-ind}
	\sum_{0 \leq m \leq n'}
	\left( \nrm{s^{m/2} \covD_{x}^{(m)} F}_{L^{\infty}_{\frac{\ud s}{s}} (J; L^{2} _{x})}^{2}
	+ \nrm{s^{(m+1)/2} \covD_{x}^{(m+1)} F}_{L^{2}_{\frac{\ud s}{s}} (J; L^{2} _{x})}^{2} \right)
	\leq 2 \hM_{n'}^{2}
\end{equation}
for some $\hM_{n'} < \infty$, then we claim that \eqref{eq:cov-est-ind} holds for $n' = n$. In the base case $n =0$, we make no induction hypothesis.

For each fixed $i, j$, the curvature component $F_{i j}$ obeys
\begin{equation*}
	\covD_{s} F_{i j} - \lap_{A} F_{i j} = - 2 [\tensor{F}{_{i}^{k}}, F_{j k}].
\end{equation*}
Commuting with $\covD_{x}^{(n)}$, we obtain a schematic equation of the form
\begin{equation*}
	\covD_{s} \covD_{x}^{(n)} F_{i j} + \lap_{A} \covD_{x}^{(n)} F_{i j} = \sum_{n' = 0}^{n} [\covD_{x}^{(n')} F, \covD_{x}^{(n - n')}F].
\end{equation*}
Multiplying by $s^{n} \covD_{x}^{(n)} F_{i j}$ (using the bi-invariant inner product) and integrating over $\bbR^{4} \times (s_{1}, s_{2})$, we get
\begin{align*}
&\hskip-2em 
	\frac{1}{2} \int s_{2}^{n} \abs{\covD_{x}^{(n)} F_{i j}}^{2} (s_{2})\, \ud x 
	+ \int_{s_{1}}^{s_{2}} \int s^{n+1} \, \abs{\covD_{x}^{(n+1)} F_{i j}}^{2} \, \ud x \, \frac{\ud s}{s} \\
= & \ 	\frac{1}{2} \int s_{1}^{n} \abs{\covD_{x}^{(n)} F_{i j}}^{2} (s_{1})\, \ud x 
	+ n \int_{s_{1}}^{s_{2}} \int s^{n} \, \abs{\covD_{x}^{(n)} F_{i j}}^{2} \, \ud x \, \frac{\ud s}{s} \\
&	+ \sum_{n' = 0}^{n} \int_{s_{1}}^{s_{2}} \int s^{n+1}  \brk{[\covD_{x}^{(n')} F, \covD_{x}^{(n - n')} F], \covD_{x}^{(n)} F} \, \frac{\ud s}{s}
\end{align*}
Therefore,
\begin{align*}
&\hskip-2em  	
	\nrm{s^{n/2} \covD_{x}^{(n)} F}_{L^{\infty}_{\frac{\ud s}{s}} ((s_{1}, s_{2}); L^{2} _{x})}^{2}
	+ \nrm{s^{(n+1)/2} \covD_{x}^{(n+1)} F}_{L^{2}_{\frac{\ud s}{s}} ((s_{1}, s_{2}); L^{2} _{x})}^{2} \\
\leq & \ C \left(
	\nrm{s_{1}^{n/2} \covD_{x}^{(n)} F(s)}_{L^{2} _{x}}^{2}
	+ n \nrm{s^{n/2} \covD_{x}^{(n)} F}_{L^{2}_{\frac{\ud s}{s}} ((s_{1}, s_{2}); L^{2} _{x})}^{2}
	+ I_{n}^{2}
	\right)
\end{align*}
where
\begin{equation*}
	I_{n}(s_{1}, s_{2})
	= \sum_{n' = 0}^{n} \int_{s_{1}}^{s_{2}} \int s^{n+1} \brk{[\covD_{x}^{(n')} F, \covD_{x}^{(n - n')} F], \covD_{x}^{(n)} F} \, \frac{\ud s}{s}.
\end{equation*}

We decompose
\begin{align*}
	I_{n}(s_{1}, s_{2})
	= I_{n, high}(s_{1}, s_{2}) + I_{n, low}(s_{1}, s_{2})
\end{align*}
where
\begin{align*}
	I_{n, high}(s_{1}, s_{2})
= &	\int_{s_{1}}^{s_{2}} \int s^{n+1} \brk{[F, \covD_{x}^{(n)} F], \covD_{x}^{(n)} F} \, \frac{\ud s}{s}, \\
	I_{n, low}(s_{1}, s_{2})
= &	\sum_{0\leq n', n-n' < n}\int_{s_{1}}^{s_{2}} \int s^{n+1} \brk{[\covD_{x}^{(n')} F , \covD_{x}^{(n-n')} F ], \covD_{x}^{(n)} F} \, \frac{\ud s}{s}.
\end{align*}

Observe that $I_{n, low}$ is nontrivial only when $n \geq 2$. In that case, we estimate $I_{n, low}$ using H\"older and covariant Sobolev inequalities as follows:
\begin{align*}
	\abs{I_{n, low}(s_{1}, s_{2})}
\leq & C \left(\sum_{0 \leq n' < n} \nrm{s^{(n'+1)/2} \covD_{x}^{(n')} F}_{L^{2}_{\frac{\ud s}{s}} ((s_{1}, s_{2}); L^{4} _{x})} \right)^{2} \nrm{s^{n/2} \covD_{x}^{(n)} F}_{L^{\infty}_{\frac{\ud s}{s}} ((s_{1}, s_{2}); L^{2} _{x})} \\
\leq & C \left(\sum_{0 \leq n' < n} \nrm{s^{(n'+1)/2} \covD_{x}^{(n'+1)} F}_{L^{2}_{\frac{\ud s}{s}} ((s_{1}, s_{2}); L^{2} _{x})} \right)^{2} \nrm{s^{n/2} \covD_{x}^{(n)} F}_{L^{\infty}_{\frac{\ud s}{s}} ((s_{1}, s_{2}); L^{2} _{x})}
\end{align*}
By the induction hypothesis,
\begin{align*}
	\abs{I_{n, low}(s_{1}, s_{2})}
\leq & C \hM_{n}^{2} \nrm{s^{n/2} \covD_{x}^{(n)} F}_{L^{\infty}_{\frac{\ud s}{s}} ((s_{1}, s_{2}); L^{2} _{x})}
\end{align*}
which is acceptable since it is linear in $F$.

On the other hand, for $I_{n, high}$ we proceed differently in the base case $n = 0$ and the inductive case $n > 0$. In the base case, we simply have
\begin{align*}
\abs{I_{n, high}(J)} 
\leq C \nrm{F}_{L^{3}_{s} (J; L^{3}_{x})}^{3},
\end{align*}
and $I_{n, low} = 0$, so the desired conclusion \eqref{eq:cov-est-ind} follows with $n = 0$.

In the inductive case $n > 0$, we estimate
\begin{align*}
  & \hskip-2em
  \abs{I_{n, high}(s_{1}, s_{2})} \\
  \leq & C \nrm{s^{1/2} F}_{L^{2}_{\frac{\ud s}{s}} ((s_{1}, s_{2}); L^{4} _{x})} \nrm{s^{(n+1)/2} \covD_{x}^{(n)} F}_{L^{2}_{\frac{\ud s}{s}} ((s_{1}, s_{2}); L^{4} _{x})} \nrm{s^{n/2} \covD_{x}^{(n)} F}_{L^{\infty}_{\frac{\ud s}{s}} ((s_{1}, s_{2}); L^{2} _{x})} \\
  \leq & C \nrm{s^{1/2} \covD_{x} F}_{L^{2}_{\frac{\ud s}{s}} ((s_{1},
    s_{2}); L^{2} _{x})} \nrm{s^{(n+1)/2} \covD_{x}^{(n+1)}
    F}_{L^{2}_{\frac{\ud s}{s}} ((s_{1}, s_{2}); L^{2} _{x})}
  \nrm{s^{n/2} \covD_{x}^{(n)} F}_{L^{\infty}_{\frac{\ud s}{s}}
    ((s_{1}, s_{2}); L^{2} _{x})}.
\end{align*}
Therefore, this term can be absorbed into the LHS if
\begin{equation} \label{eq:cov-est-small}
C \nrm{s^{1/2} \covD_{x} F}_{L^{2}_{\frac{\ud s}{s}} ((s_{1}, s_{2}); L^{2} _{x})} \ll 1.
\end{equation}
By the induction hypothesis, $\nrm{s^{1/2} \covD_{x} F}_{L^{2}_{\frac{\ud
      s}{s}} (J; L^{2} _{x})} \leq \hM_{0} < \infty$, so the interval
$J$ can be split into $O(\hM_{0})$-many of intervals on each of which
\eqref{eq:cov-est-small} holds. Reinitializing data at every (left)
endpoint of these intervals, we obtain the conclusion
\eqref{eq:cov-est-ind} for $n' = n$.

\pfstep{Proof of (2)}
As in (1), we again proceed by an induction. Here, the key point is that the equation obeyed by $F_{0x}$, namely
\begin{equation*}
	\covD_{s} F_{0j} - \lap_{A} F_{0j} = - 2[\tensor{F}{_{0}^{k}}, F_{jk}]
\end{equation*}
is linear in $F_{0x}$. The contribution of the RHS can be treated perturbatively, using the bound proved in (1), and splitting $J$ into small intervals to gain smallness of $F$. We leave the details to the reader.
\end{proof}

While in these bounds we cannot substitute covariant derivatives by
regular derivatives, we do have the corresponding $L^p$ bounds which would normally follow 
from Sobolev embeddings:

\begin{corollary}
Under the same assumptions as the previous proposition, we also have
\begin{equation}\label{GNS}
\| s^{ 2(\frac12 -\frac1q)+(\frac14-\frac1p)}  F_{\alp \bt} \|_{L^p_{\frac{ds}s}(J;L^q)} \lesssim_{\hM} 1, \qquad \frac1p + \frac2q \leq \frac12
\end{equation}
respectively
\begin{equation}\label{GNS1}
\| s^{ \frac{n}2 + 2(\frac12 -\frac1q)+(\frac14-\frac1p)}  \covD_{x}^{(n)} F_{\alp \bt} \|_{L^p_{\frac{ds}s}(J;L^q)} \lesssim_{\hM} 1, \qquad 
p,q \geq 2, \quad n \geq 1.
\end{equation}
\end{corollary}

\begin{proof}
This follows from the diamagnetic inequality $\rd_{j} \abs{F_{\alp \bt}} \leq \abs{\covD_{j} F_{\alp \bt}}$ and the standard Sobolev embeddings; we leave the details to the reader.
\end{proof}

\subsection{Main results in the local caloric gauge}
To prove local solvability for the Yang--Mills heat flow we need to fix the gauge.
One natural choice in this context is the de Turck gauge,
\begin{equation} \label{eq:deturck}
	A_{s} = \rd^{\ell} A_{\ell},
\end{equation}
With this gauge choice, the covariant Yang--Mills heat flow equations
\eqref{eq:cYMHF} are reduced to a genuinely \emph{parabolic}
semilinear system, and their local theory is relatively
straightforward.  Unfortunately, it is not clear to us whether this
gauge leads to global solutions in the large data case, rather than
gauge related singularities. Part of the difficulty is that this flow 
is nontrivial even for flat connections, where it corresponds to a
critical harmonic map heat flow into the Lie group $\G$.

In order to avoid such difficulties we will work from the start in the the local caloric gauge,
\begin{equation} \label{eq:i-caloric}
	A_{s} = 0.
\end{equation}
The Yang--Mills heat flow written in the local caloric gauge takes the form
\begin{equation}\label{caloric-re}
\partial_s A_i = \covD^{\ell} \covD_{\ell} A_i - \covD^{\ell}\partial_i A_\ell , \qquad A_i(0) = a_i.
\end{equation}
Here one can see the downside of working in this gauge, namely that our evolution is only degenerate
parabolic. This will cause some small difficulties with the local theory, but has the chief advantage 
that it is very well suited for the global theory. 

The non-parabolic component of the above system is captured in the
evolution of $\partial^k A_k$, which we capture here for later use:
\begin{equation}\label{DA-rep}
\partial_s \partial^\ell A_\ell = - [A^k, \covD^\ell F_{\ell k}].
\end{equation}

This evolution retains some gauge freedom, namely that which 
corresponds to purely spatial (i.e. $s$ independent) gauge transformations.
Later in this section we will  take advantage of this gauge freedom to construct our caloric gauge
for the hyperbolic Yang--Mills equation.

In the same vein, the linearized Yang--Mills heat flow written in the local caloric gauge
$B_s = 0$ has the form 
\begin{equation}\label{lin-heat}
\partial_s B_i = \covD^j ( \covD_j B_i - \covD_i B_j)  +  [B^j,F_{ji}] , \qquad B_i(0) = b_i.
\end{equation}
We introduce the notation
\begin{equation}
	(\curl_{A} B)_{ij} = \covD_{i} B_{j} - \covD_{j} B_{i}.
\end{equation}

In addition to the above gauge transformations $B \to OBO^{-1}$, here we have the additional
gauge freedom arising from pure gauge variations of $A$, namely\footnote{This is the linearized gauge freedom \eqref{eq:cYMHF-lin-gt} of \eqref{eq:cYMHF-lin} under the additional restriction $O_{;s}= 0$ in order to keep the local caloric gauge condition $A_{s} = 0$.}
\begin{equation}
B_i \to B_i - \covD_i C
\end{equation}
where $C$ is again independent of $s$ (though its covariant derivatives are not).

We repeat this discussion with the $A_0$ equation of the dynamic Yang--Mills heat flow, which takes the form
\begin{equation}\label{A0-cal}
\partial_s A_0 = \covD^{\ell} F_{\ell 0},
\end{equation}
where $F_{0\ell}$ are uniquely determined in terms of their initial data
from the equations
\begin{equation}\label{F0l}
\partial_s F_{0\ell} - \Delta_A F_{0\ell} = - 2 [\tensor{F}{_{0}^{k}},F_{\ell k}] .
\end{equation}
The gauge freedom here is even simpler,
\begin{equation}
A_0 \to A_0 +  C,
\end{equation}
where $C = C(x)$ is independent of the heat-time $s$.

The Yang--Mills heat flow written in the form \eqref{caloric-re} has the
disadvantage that the principal part $\lap A_{i} - \rd_{i} \rd^{\ell}
A_{\ell}$ of the RHS is not strictly negative-definite; hence the
principal part does not exhibit (forward in heat-time) smoothing for
the whole $A_{i}$.  Instead, at the leading order this equation
decouples into a nondegenerate parabolic equation for the curl of $A$,
coupled with an ODE evolution for the divergence of $A$. The same
considerations apply to the linearized equation. 

Unfortunately we cannot take advantage of this decoupling directly.
However, a carefully covariant version  of it turns out to be effective, and 
leads us to the critical $\dot H^1$ well-posedness theory
for the system \eqref{caloric-re}.

From the above discussion we retain the special role played
by the divergence of $A$, which in general gains no regularity in time.
For later use, we will also consider solutions for which the divergence 
of $A$ is more regular, and extend the well-posedness theory to the 
space
\[
\bfH = \{ a \in \dot H^1: \partial^{\ell} a_{\ell} \in \ell^1 L^2\}.
\]

The same considerations as above apply to the linearized equation \eqref{lin-heat}.
On the other hand, the evolution \eqref{F0l}  of $F_{0\ell}$ is nondegenerate parabolic.
For this reason we will use a roundabout way to obtain solutions
to the linearized flow. Precisely, we follow the following algorithm:

\begin{itemize}
\item We initialize $a_0 = 0$ and $f_{0\ell} = b_\ell$ at $s = 0$.
\item We solve the parabolic equation for $F_{0j}$.
\item We obtain $A_0$ by integrating \eqref{A0-cal} in heat-time.
\end{itemize}

As discussed in Section~\ref{subsec:dymhf}, this may be though of as an infinitesimal de Turck trick for \eqref{lin-heat}.

Our main local well-posedness result for the Yang--Mills heat flow in
the local caloric gauge is as follows:

\begin{theorem} \label{t:hf-loc}
  The Yang--Mills heat flow in the local caloric gauge is locally well-posed
  in both $\dot H^1$ and $\bfH$, with locally Lipschitz dependence on the
  initial data. The same result holds in $\dot H^1 \cap H^\sgm$ for all $\sgm \geq 1$.
\end{theorem}

To study the long time behavior of solutions it is useful to have 
bounds which depend only on the $L^3 L^{3}$ norm of $F$:

\begin{theorem}\label{t:ymhf-l3}
Let $a \in \dot H^1$,   and $A$ be the corresponding $\dot H^1$
  solution for the Yang--Mills heat flow in the local caloric gauge on
  a time interval $J= [0,s_0)$. Assume that, for some $\hM, M_{1} > 0$, the $L^3 L^{3}$ curvature 
bound  \eqref{l3} holds in $J$ and that
\begin{equation} \label{eq:h1-hyp}
	\nrm{a}_{\dot{H}^{1}} \leq M_{1}.
\end{equation}
\begin{enumerate}
\item  Then
  we have the uniform bound
\begin{equation}
\| A\|_{L^\infty \dot H^1} + \| F\|_{L^\infty L^2 \cap L^2 \dot H^1} + \| \partial_s A\|_{L^2 L^{2}}  \lesssim_\hM \| a\|_{\dot H^1}.
\end{equation}
as well as the similar $\bfH$ bound.

\item Let $b \in \dot H^\sigma$ be a corresponding linearized data.
 Then  we have the uniform bound
\begin{equation}
\| B\|_{L^\infty \dot H^\sigma} + \| \mathrm{curl}_{A} B\|_{L^2 \dot H^\sigma}  \lesssim_{\hM,M_1} \| b\|_{\dot H^\sigma}, 
\qquad -2 < \sgm < 2 ,
\end{equation}
as well as the similar $\bfH$ bound.

\item Assume in addition that $a \in \dot H^{\sigma}$ for some $\sigma > 1$. Then 
 we have the uniform bound
\begin{equation}
\| A\|_{L^\infty \dot H^{\sigma}} + \| F\|_{L^\infty \dot H^{\sigma-1} \cap L^2 \dot H^{\sigma}} + 
\| \partial_s A\|_{L^2 \dot H^{\sigma-1} }  \lesssim_{M,M_1} \| a\|_{\dot H^{\sigma}}.
\end{equation}
Also for the linearized equation we have 
\begin{equation}
\| B\|_{L^\infty \dot H^\sigma} + \| \mathrm{curl}_{A} B\|_{L^2 \dot H^\sigma}  \lesssim_{\hM,M_1} \| b\|_{\dot H^\sigma}
+ \| b\|_{\dot H^1} \|a\|_{\dot H^\sigma}.
\end{equation}
\end{enumerate}
\end{theorem}

These bounds assert that the solution to data map for the Yang--Mills heat flow in the local caloric gauge 
is  uniformly Lipschitz in $\dot H^1$, $\bfH$ and $\dot H^1 \cap \dot H^\sgm$ ( $\sgm > 1$) 
on bounded sets in $\dot H^1$ for as long as the $L^{3} L^3$ norm of $F$ remains controlled.

We will prove parts (1) and (2) directly. However, for part (3) we will instead establish a stronger 
and more accurate frequency envelope version of the above result.  For this, we will use the following 
notation:
\begin{itemize}
\item  $c_k$ is a $(-1,S)$  frequency envelope for the connection $a$ in $\dot H^1$. 
\item $d_k$ is a  $(-1,S)$  frequency envelope for the linearized data $b$ in $L^2$.
\end{itemize}
If $S =  2$ then these envelopes can be taken independently of each other. For larger $S$, we require that $d_{k}$ be $2$-compatible\footnote{In fact, any $\sgm$-compatibility with $\sgm > 1$ would do. The reason for this range is that it allows us to treat the linearized equation simultaneously at the two regularities that we are interested in, namely $\dot{H}^{1}$ and $L^{2}$.} with $c_{k}$.

With the above notation, we have:
\begin{theorem}\label{t:ymhf-l3-fe}
Let $a \in \dot H^1$ with $(-1, S)$ frequency envelope $c_k$, and $b \in L^2$   
with $2$-compatible $(-1, S)$ frequency envelope $d_k$. Let  $A$, $B$ be the corresponding solutions for the 
Yang--Mills heat flow, respectively its linearization around $A$, in the local caloric gauge on
a heat-time interval $J= [0,s_0)$. Assume that  the $L^{3} L^3$ curvature 
bound  \eqref{l3} holds in $J$, and also that \eqref{eq:h1-hyp} holds.  
Then we have the  frequency envelope bounds
\begin{equation}\label{ymhflc-feA}
\| P_k A\|_{L^\infty \dot H^1} +  \| P_k F\|_{L^\infty L^2 \cap L^2 \dot H^1}  + \| P_k \covD F\|_{L^1 \dot H^1}  \lesssim_{\hM, M_1} c_k,
\end{equation}
\begin{equation} \label{ymhflc-feDA}
\sup_{s_1,s_2} \|P_k( \partial^\ell A_\ell(s_1) - \partial^\ell A_\ell(s_2)) \|_{L^2}\lesssim_{\hM, M_{1}}  c_k c_k^{[1]}.
\end{equation}

respectively,
\begin{equation}\label{ymhflc-feB}
\| P_k B\|_{L^\infty L^2} + \| P_k \mathrm{curl}_{A} B\|_{L^2}  \lesssim_{\hM, M_1}   d_k,
\end{equation}
\begin{equation} \label{ymhflc-feDB}
\sup_{s_1,s_2}  \| \covD^k B_k(s_1) - \covD^k B_k(s_2) \|_{\dot{H}^{-1}} \aleq_{\hM, M_{1}} c_{k} d_{k}^{[2]} + d_{k} c_{k}^{[1]} + \sum_{j > k} 2^{k-j} c_{j} d_{j}.
\end{equation}
\end{theorem}

Theorem~\ref{t:ymhf-l3} allows us to consider the question of continuation of solutions 
for the Yang--Mills heat flow:

\begin{corollary}\label{c:ymhf-cont}
\begin{enumerate}
\item Let $a \in \dot H^1$ and $A$ be the corresponding $\dot H^1$
  solution for the Yang--Mills heat flow in the local caloric gauge on
  a heat-time interval $J= [0,s_0)$. Assume that  the $L^{3} L^3$ curvature 
bound  \eqref{l3} holds in $J$.  Then  the following limit  exists in $\dot H^1$:
\begin{equation}
A(s_0) = \lim_{s \to s_0} A(s) .
\end{equation} 
Further, there exists a strictly larger interval $\tilde J= [0,\tilde
s_0)$ as well as $\epsilon= \epsilon(\hM, \|a\|_{\dot H^1})$ so that for
all data $\tilde a$ with $\|\tilde a - a\|_{\dot H^1} \leq \epsilon$, the
corresponding solution $\tilde A$ exists in $\tilde J$, satisfies \eqref{l3}
with $\hM$ replaced by $2\hM$, as well as
\begin{equation}
\| A-\tilde A\|_{L^\infty(\tilde J; \dot H^1)} \lesssim_\hM \| a-\ta\|_{\dot H^1}.
\end{equation}

\item If in addition $a \in \bfH$, respectively $a \in \dot H^1 \cap \dot H^\sigma$ ($\sigma > 1$)
then the above limit exists in $\bfH$, respectively $\dot H^1 \cap \dot H^\sigma$.
\end{enumerate}
\end{corollary}

As a consequence of the last result we have the following continuation criteria:

\begin{corollary}\label{c:max}
Let $A$ be a maximal $\dot H^1$ solution for the Yang--Mills heat flow in the local caloric gauge
 in an interval $J = [0,s_0)$. Then we have either $s_0 = \infty$ or $\| F\|_{L^3(J; L^{3})} = \infty$.
\end{corollary}

Here we are especially interested in  the global behavior of solutions. Given a global covariant Yang--Mills heat flow $A$ with initial data $a = A(s=0)$, define
\begin{equation}
	\hM(a) = \nrm{F}_{L^{3}([0, \infty); L^{3})}.
\end{equation}
First, we show that small initial energy leads to a global solution, with an explicit bound on $\hM(a)$:
\begin{corollary} \label{c:global-small}
Let $a$ be a $\dot{H}^{1}$ connection with a sufficiently small $\spE[a]$. 
Then the corresponding solution $A$ to the Yang--Mills heat flow in the local caloric gauge exists globally, and obeys
\begin{equation*}
	\hM(a)^{2} \aleq \spE[a].
\end{equation*}
\end{corollary}

For solutions with $\hM(a) < \infty$, we obtain uniform global-in-time bounds for the Yang--Mills heat flow and its linearization in the local caloric gauge by Theorem~\ref{t:ymhf-l3}. Moreover, the following asymptotic convergence properties also hold.
\begin{corollary}\label{c:global}
   Let $A$ be a global $\dot H^1$ solution for the Yang--Mills heat
  flow in the local caloric gauge with $\calQ(a) = \|F\|_{L^3([0, \infty); L^{3})} \leq \hM < \infty$. Then
  the limiting connection $A_\infty$ exists in $\dot H^1$ and has zero
  curvature. The same applies to the linearized equation in $\dot{H}^{\sgm}$ with
  $-1 < \sgm <  2$ as well as in $\bfH$.

Furthermore, the map $a \to A_\infty$ is Lipschitz in $\dot H^1$, $\bf H$,
$\dot H^1 \cap \dot H^\sigma$ ($\sigma > 1$) uniformly on bounded convex 
subsets of $\dot H^1$ where \eqref{l3} holds uniformly.
\end{corollary}

\subsection{Proof of the main results in the local caloric gauge}
This subsection is devoted to the proof of the results stated in the preceding subsection.  All the solutions 
for the Yang--Mills heat flow are assumed from here on to be in the local caloric gauge \eqref{eq:i-caloric}.

Due to the degeneracy of the 
parabolic system for $A$, we cannot address directly the local well-posedness question in $\dot H^1$,
and we begin with a more regular setting:

\begin{lemma}\label{l:h3-loc}
  The Yang--Mills heat flow in the local caloric gauge
  \eqref{caloric-re} is locally well-posed for initial data $a \in
  \dot H^1 \cap \dot H^3$, with Lipschitz dependence on the initial
  data and lifespan depending only on the initial data size.
\end{lemma}
The same argument applies in $\dot H^1 \cap  \dot H^\sgm$ for all $\sgm \geq 3$.

\begin{proof}
We write the system as a mixed parabolic/transport for the curl/div of $A$:
\[
\left\{
\begin{array}{ll}
(\partial_s -\Delta) \curl A = [\covD A,\covD A] + [\covD^2A,A] + [A,A,\covD A],
\cr
\partial_s\  \div A =  [A,\partial \curl A]  + [A,A,\covD A].
\end{array}
\right.
\]
Here the only important structural information is that $\div A$ does
not appear differentiated in the second equation. For initial data 
$A(0) \in \dot H^1 \cap \dot H^3$ we solve this system in the space
\[
X =  (L^\infty H^2 \cap L^2  H^3)  \times L^\infty  H^2
\]
by estimating the right hand side in
\[
Y =  (L^1 H^2 + L^2 H^1)  \times L^1  H^2
\]
A standard fixed point argument in this setting yields local well-posedness,
with a lifespan depending only on the initial data size. The control of the $H^3$ 
norm is useful as it guarantees that $A \in L^\infty$. \qedhere

\end{proof}

We now turn our attention to obtaining scale invariant bounds for such
solutions. We will use the  $L^{3} L^3$ bound for $F$ as a key
a-priori assumption, while the initial data is assumed to have finite energy, i.e.,
\begin{equation} \label{eq:fin-en-sp}
	\frac{1}{2} \nrm{F[a]}_{L^{2}}^{2} \leq \calE < \infty.
\end{equation}

\begin{lemma}\label{l:h1}
  Let $A$ be a sufficiently regular Yang--Mills heat flow on a heat-time interval $J$, so that
  \eqref{l3}  and \eqref{eq:fin-en-sp} hold. Then we have the bound
\begin{equation} \label{eq:h1}
\| A \|_{L^\infty \dot H^1} \lesssim_{\hM, \calE}  1 + \inf_{s \in J} \| A(s)\|_{\dot H^1}.
\end{equation}
\end{lemma}
\begin{remark}
The lemma is also valid for the dynamic Yang--Mills heat flow with the $A_{0}$ component added in \eqref{eq:h1}, after replacing \eqref{eq:fin-en-sp} by
\begin{equation} \label{eq:fin-en-spt}
	\frac{1}{2} \sum_{\alp < \bt} \nrm{F_{\alp \bt}(s=0)}_{L^{2}}^{2} \leq \calE < \infty.
\end{equation}
Later we will seek to imbalance the $A_0$ bounds.
 \end{remark}
 
 \begin{remark} \label{rem:h1-loc}
From the proof, it will also be evident that a spatially localized version of Lemma~\ref{l:h1} also holds, i.e.,
\begin{equation} \label{eq:h1-loc}
\| A \|_{L^\infty (J; \dot H^1(B_{R}))} \lesssim_{\hM, \calE}  1 + \inf_{s \in J} \| A(s)\|_{\dot H^1(B_{R})},
\end{equation}
for any fixed ball $B_{R} \subseteq \bbR^{4}$.
\end{remark}
\begin{proof}
We will proceed in two steps, where we first establish an $L^4$ bound for $A$, and then 
a $\dot H^1$ bound. For the $L^4$ bound we need to estimate 
\[
I = \left \|\int_{0}^{s_0} \covD^\ell F_{\ell i} ds \right\|_{L^4}^4 = \int_0^{s_0}  \langle \covD^\ell F_{\ell i}(s_1) ,  \covD^\ell F_{\ell i}(s_2)  \rangle
 \langle \covD^\ell F_{\ell i}(s_3) ,  \covD^\ell F_{\ell i}(s_4)  \rangle  ds_1 \cdots ds_{4}.
\]
We restrict to $s_1 \leq s_2 \leq s_3 \leq s_4$, as the other cases are similar. Then we integrate 
by parts the derivative on $F(s_1)$, and it will fall on one of the other factors. In the worst case
it applies to $F(s_2)$. Dividing time into dyadic regions and using the bounds \eqref{GNS} and \eqref{GNS1},
 we estimate the corresponding integral $I_1$ by
\[
\begin{split}
I_1 \lesssim & \  \sum_{dyadic} \left\| \int F(s_1) ds_1 \right\|_{L^4} \left\| \int \covD^2 F(s_2)
\right \|_{L^2} \left\|\int \covD F(s_3) ds_3\right\|_{L^8} \left\|\int \covD F(s_4)ds_4 \right\|_{L^8} 
\\ \lesssim_M & \ \sum_{dyadic} s_1^{\frac12} s_3^{-\frac14} s_{4}^{-\frac14}  \| \covD F\|_{L^2_{s_1}L^2}
\| s^{-\frac12} \covD^2 F\|_{L^2_{s_2}  L^2} \| s^{\frac18} \covD F\|_{L^2_{s_3} L^8}\| s^{\frac18} \covD F\|_{L^2_{s_4} L^8}.
\end{split}
\]
Given the ordering of the $s_i$'s, this sum has  has off-diagonal decay, and thus converges.

In a similar manner, for the $\dot H^1$ bound we need to estimate 
\[
\left \|  \int_{0}^{s_0} \partial \covD^\ell F_{\ell i} ds \right\|_{L^2}^2 \lesssim 
\left\|  \int_{0}^{s_0}  \covD^2 F ds \right\|_{L^2}^2 + \left\|  \int_{0}^{s_0}  [A, \covD F] ds \right\|_{L^2}^2 .
 \]
We proceed as above. The first term is written as
\[
\int_{s_1 \leq s_2} \langle \covD F(s_1), \covD^3F(s_2)\rangle ds_1 d s_{2},
\]
and we can combine the two $L^2L^2$ bounds for the two factors.

The second term is written as 
\[
\int_{s_1 \leq s_2} \langle[A, \covD F](s_1), [A,\covD F](s_2)\rangle ds_1 ds_{2}.
\]
Here we directly use the $L^4$ bound for $A$, the $L^2L^2$ bound for the first $\covD F(s_{1)}$ and 
the $L^2 L^\infty$ bound for $\covD F(s_{2})$. \qedhere

\end{proof}

Next, we establish higher regularity bounds.
\begin{lemma}\label{l:h3}
  Let $A$ be a sufficiently regular Yang--Mills heat flow so that
  \eqref{l3} holds. Then we have the bounds
\begin{equation}\label{A-high}
\| A \|_{L^\infty \dot H^{k+1}} \lesssim_{\hM,M_1}  \| A(0)\|_{\dot H^{k+1}}, \qquad k \geq 1
\end{equation}
\end{lemma}

\begin{proof}
We first remark that bounds for $A(0)$ directly translate into bounds for $F(0)$,
namely 
\[
\|\covD^{(k)} F(0)\|_{L^2} \lesssim_{M_1}  \|A(0)\|_{\dot H^{k+1}}.
\]
This directly leads to improved bounds in Proposition~\ref{p:cov-smth}, with 
$k$ covariant derivatives added:
\begin{align} \label{est-cov-k}
\nrm{s^{m/2} \covD_{x}^{(m+k)} F}_{L^{\infty}_{\frac{\ud s}{s}} (J; L^{2} _{x})}
	+ \nrm{s^{(m+1)/2} \covD_{x}^{(m+k+1)} F}_{L^{2}_{\frac{\ud s}{s}} (J; L^{2} _{x})}
\leq_{\hM, k}  \|\covD^{(k)} F(0)\|_{L^2}
\end{align}

Next we prove the covariant version of \eqref{A-high}, arguing by induction on $k$.
Our starting point is the $k=0$ bound in the previous proposition. For the induction step 
we differentiate \eqref{A0-cal} $k+1$ times to obtain the schematic equation
\[
\partial_s \covD^{(k+1)} A  = \covD^{(k+2)} F + \sum_{j=0}^k [\covD^{(j+1)} F, \covD^{(k-j)} A].
\]
Then we integrate in $s$ from $s=0$ and estimate separately each term.
The analysis for the first term is identical to the proof of the $L^2$ bound in the previous Lemma,
but using \eqref{est-cov-k} instead of \eqref{est-cov}. For the summand in the second term
we need to consider the integral 
\[
I = \int_{s_1 \leq s_2}  \langle [\covD^{(j+1)} F,\covD^{(k-j)} A](s_1),  [\covD^{(j+1)} F,\covD^{(k-j)} A](s_2) \rangle dx ds_1 ds_2  
\]
Now we use \eqref{est-cov-k} for $F$, respectively our induction hypothesis for $A$ 
to bound the four factors in $L^2$, $L^\infty L^4$, $L^2 L^\infty$ respectively $L^\infty L^4$
to obtain
\[
\begin{split}
|I| \lesssim & \ \sum_{s_1< s_2, dyadic} s_1^\frac12 s_2^{-\frac12} \| \covD^{(j+1)} F\|_{L^2_{s_1} L^2}
\| \covD^{(k-j+1)} A\|_{L^\infty L^2} \| s  \covD^{(j+1)} F\|_{L^2_{s_2} L^\infty} \|  \covD^{(k-j+1)} A\|_{L^\infty L^2}
\\
\lesssim & \ \| \covD^{(j)} F(0)\|_{L^2}^2  \|  \covD^{(k-j+1)} A\|_{L^\infty L^2}.
\end{split}
\] 
Then the RHS is bounded using the induction hypothesis,
provided that $j \neq 0$. If $j = 0$ we can argue in a similar fashion 
if we mildly unbalance the estimate, using instead the norms 
 $L^2 L^4$, $L^\infty L^2$, $L^2 L^\infty$ respectively $L^\infty L^4$.
Alternatively, we can also divide the heat-time interval into subintervals 
where the above $L^2L^p$ norms of $F$ are small, and then reiterate.

It remains to make the transition from covariant to regular derivatives.
For $k=1$ we estimate as follows:
\[
\begin{split}
\| \partial^{(2)} A\|_{L^2} \lesssim_{M_1} & \ \| \covD \partial A\|_{L^2} \lesssim 
\| \covD^{(2)} A\|_{L^2} + \| [ \covD A, A]\|_{L^2}
\\  \lesssim_{M_1} & \  \| \covD^{(2)} A\|_{L^2} + \|  \covD A\|_{L^4} \|A\|_{L^4}
\lesssim_{M_1} \| \covD^{(2)} A\|_{L^2} .
\end{split}
\]
A similar argument inductively applies for higher $k$.
\end{proof}

We continue with some finer scale invariant estimates for $A$ and $F$:
\begin{lemma}\label{l:nocov-smth}
 Let $A$ be a sufficiently regular Yang--Mills heat flow so that
  \eqref{l3} and \eqref{eq:h1-hyp} hold. Then we have the bounds
  \begin{align} 
\| F\|_{L^2 \dot H^1} + \| s F\|_{L^\infty \dot H^2} + \|s^\frac12 F\|_{L^2 \dot H^2} 
\lesssim_{\hM,M_1} & 1 , \label{eq:nocov-smth-1} 
\end{align}
as well as
\begin{align}
\| \covD F\|_{\ell^2 L^1 \dot H^1} \lesssim_{\hM,M_1} & 1 ,\label{eq:nocov-smth-2} \\
\| A\|_{\ell^2 L^{\infty} \dot H^1} \lesssim_{\hM,M_1} & 1 , \label{eq:nocov-smth-3} \\
\| F- e^{s\Delta} f\|_{\ell^1 L^2 \dot H^1} \lesssim_{\hM,M_1} & 1 . \label{eq:nocov-smth-4}
\end{align}
\end{lemma}

\begin{proof}
We start with \eqref{eq:nocov-smth-1}. The first bound is obvious. For the second we need to estimate at fixed $s$
\[
\| s \partial^2 F\|_{L^2} \lesssim \| s \covD \partial F(s)\|_{L^2} \lesssim \| s  \partial \covD F(s)\|_{L^2}
+ \| s \partial A F\|_{L^2} \lesssim \| s  \covD^2 F(s)\|_{L^2}
+ \|  \partial A \|_{L^2} \|sF\|_{L^\infty}  
\]
which suffices. For the third we  argue similarly.

More covariant derivatives are also allowed here.  In particular the
bound \eqref{eq:nocov-smth-2} for $\| \covD F\|_{\ell^2 L^1 \dot H^1}$ is obtained by combining
bounds for $\| \covD F\|_{L^2}$ and $\| s \covD F\|_{L^2 \dot H^2}$. (Alternatively, we could use the fact that $\covD F_{i x}$ solves the \eqref{heatA} with $G \in L^{1} \dot{H}^{-1}$, and appeal to Theorem~\ref{t:heatA-L1}.)

The improved $A$ bound \eqref{eq:nocov-smth-3} follows by integrating 
\begin{equation*}
A_{i}(s) = \int_{0}^{s} \covD^{\ell} F_{\ell i} (\tilde{s}) \, \ud \tilde{s}
\end{equation*}
using \eqref{eq:nocov-smth-2}.

Finally we consider the difference
\[
\tilde F = F - e^{s \Delta} f
\]
which solves the schematic equation
\[
(\partial_s - \Delta)\tilde F = [A , \rd F] + [\rd A, F] + [A, [A, F]] =: G.
\]
Then we use off-diagonal decay, as well as the preceding bounds, to obtain the following estimate for $G$:
\begin{equation*}
	\nrm{G}_{\ell^{1} L^{2} \dot{H}^{-1}} \aleq_{\hM, M_{1}} 1.
\end{equation*}
Then \eqref{eq:nocov-smth-4} follows from the usual heat flow estimate (or use Theorem~\ref{t:heatA-fe} with $A = 0$). \qedhere
\end{proof}

We now have sufficient estimates in order to establish the continuation of 
regular solutions  for  the Yang--Mills heat flow with regular data:

\begin{lemma}\label{l:ext-h3}
Let $J$ be the maximal time of existence for the Yang--Mills heat flow problem
in the local caloric gauge, with initial data $a \in \dot H^1 \cap \dot H^3$.
Then either $J = [0,\infty)$ or $\|F\|_{L^3(J; L^{3})} = \infty$.
\end{lemma}

\begin{proof}
We assume that $A$ is a solution in $\dot{H}^{1} \cap \dot{H}^3$ in a finite time interval $J$, so that 
$\|F\|_{L^3(J; L^{3})} < \infty$. By Lemma~\ref{l:h3}, the solution is uniformly bounded in 
$\dot{H}^{1} \cap \dot{H}^3$ for $s \in J$. Since the lifespan of the solution 
to the initial value problem depends 
only on the size of the data, it immediately follows that we can extend the solution past $J$. 
\end{proof}

Next, we consider the $L^2$ well-posedness for the linearized equation \eqref{lin-heat}.
It is convenient to also consider the corresponding inhomogeneous problem, which we write in the form
\begin{equation}\label{lin-inhom}
\partial_s B_k -  \covD^j ( \covD_j B_k - \covD_k B_j) - [B^j,F_{jk}] = H_k + \covD^{j} G_{jk}, \qquad B_k(0) = b_k
\end{equation}
where $G$ is antisymmetric.

\begin{lemma}\label{l:lin}
Let $A$ be a sufficiently regular  Yang--Mills heat flow so that \eqref{l3} holds. 
Then the linearized equation \eqref{lin-inhom} is well-posed in $\dot H^\sigma$ for $-2 < \sigma  < 2$,
and we have the bounds
\begin{equation}\label{B-hs+}
\| B\|_{L^\infty \dot H^\sigma} + \| \covD_i B_j - \covD_j B_i\|_{L^2 \dot H^\sigma} \lesssim \|B(0)\|_{\dot H^\sigma}
+ \|H \|_{L^1 \dot{H}^\sigma} + \| G\|_{L^2 \dot{H}^\sigma}
\end{equation}
\end{lemma}

\begin{proof} 

Depending on $\sigma$ we divide the problem in three cases:
\pfstep{Case 1: $-1 < \sigma  < 1$} 
For this part of the proof we do not use the fact that $A$ is an Yang--Mills heat flow. Instead, we use only the following
properties:

\begin{itemize}
\item $A$ is bounded in $\dot H^1$, $\|A\|_{L^\infty \dot H^1} \aleq_{\hM, M_{1}} 1$.

\item $F$ is bounded in $L^2 \dot H^{1}$, $\|F\|_{L^2 \dot H^{1}} \lesssim_{\hM, M_1} 1$.

\item $\partial_s A \in L^2$, $\|\partial_s A\|_{L^2}  \lesssim_{\hM, M_1} 1$.
\end{itemize}
These are precisely what we need to apply Theorem~\ref{t:heatA}.

Our approach  here is to solve the linearized 
equation via the $A_0$ flow, as follows:

\begin{itemize}
\item We assign for simplicity the initial value $a_0 = 0$. Then we have matching initial data 
$f_{0j} = b_j$. 

\item We find $F_{0j}$ by solving  the inhomogeneous version of the covariant curvature flow \eqref{F0l},
namely 
\begin{equation}\label{F0l-inhom}
\partial_s F_{0\ell} - \Delta_A F_{0\ell} = - 2 [{F_{\ell}}^k,F_{k0}] +  H_\ell + \covD^j G_{j\ell}.
\end{equation}
\item We recover $A_0$ by integrating \eqref{A0-cal} in heat-time
\[
A_0(s) = \int_0^s \covD^k F_{k 0 }(s_1)\ ds_1.
\]

\item We find the  solution $B$ to the linearized inhomogeneous problem \eqref{lin-inhom}
by 
\[
B_j = F_{0j} + \covD_j A_0.
\]
\end{itemize}

This approach may at first appear more roundabout, but it has the chief advantage 
that we only need to solve dynamically a strongly parabolic evolution. It may be thought of as an ``infinitesimal de Turck's trick'' for the linearized Yang--Mills heat flow (cf. Section~\ref{subsec:dymhf}).

The bounds for $F_{0j}$ are already provided by Theorem~\ref{t:heatA}, which yields
\[
\| F_{0j} \|_{L^\infty \dot H^\sigma } + \| \covD F_{0j}\|_{ L^2 \dot H^\sigma}  \lesssim \|f_{0j} \|_{\dot H^\sigma } + \| H\|_{L^1 \dot H^\sigma }
+ \|G\|_{L^2 \dot H^\sigma } .
\]
By the formula $\mathrm{curl}_{A} B = \curl_{A} F_{0 x} + [F, A_{0}]$, it remains to estimate 
\begin{equation}
\| \covD A_0\|_{L^\infty \dot H^\sigma} + \|[F,A_0] \|_{L^2 \dot H^\sigma }
\lesssim \| f_{0j}\|_{\dot H^\sigma }+ \| H\|_{L^1 \dot H^\sigma }
+ \|G\|_{L^2 \dot H^\sigma } .
\end{equation}
Since $-1 < \sgm < 1$, this reduces to 
\[
\|A_0 \|_{L^\infty \dot H^{\sigma+1}} \lesssim \| f_{0j}\|_{\dot H^\sigma }+ \| H\|_{L^1 \dot H^\sigma }
+ \|G\|_{L^2 \dot H^\sigma } .
\]

To bound $A_0$, we need to better understand the expression $\covD^j F_{0j}$. This
solves the (schematic) parabolic equation
\[
(\partial_s -\Delta_A) \covD^j F_{0j}= [\covD F, F_{0j}] + [F, \covD F_{0j}] + \covD^{j} H_{j} + [F,  G],
\]
where we used the fact that $ \covD^{k} \covD^{j}G_{jk} = -\frac{1}{2} [F^{jk}, G_{jk}]$ by antisymmetry. As $\sigma \in (-1,1)$, we can estimate the right hand side in $L^1 \dot H^{\sigma-1}$.
By Theorem~\ref{t:heatA-L1}, it follows that
\[
\|\covD^j F_{0j}\|_{\ell^2 L^1 \dot H^{\sigma+1}} \lesssim  \| f_{0j}\|_{\dot H^\sigma }+ \| H\|_{L^1 \dot H^\sigma }
+ \|G\|_{L^2 \dot H^\sigma } 
\]
which in turn leads by integration to the desired $L^\infty \dot H^{\sgm+1}$ bound for $A_0$.

As a final remark, we observe that by interpolating the bounds \eqref{B-hs+} with different $\sigma$ we 
obtain the slightly stronger form
\begin{equation}\label{B-hs++}
\| B\|_{\ell^2 L^\infty \dot H^\sigma} + \| \covD_i B_j - \covD_j B_i\|_{L^2 \dot H^\sigma} \lesssim \|B(0)\|_{\dot H^\sigma}
+ \|H \|_{L^1 \dot{H}^\sigma} + \| G\|_{L^2 \dot{H}^\sigma}.
\end{equation}

\pfstep{Case 2: $1 \leq \sigma < 2$}
In addition to the previous case, here we use the bounds 
\begin{itemize}
\item $\covD F$ is bounded in $\ell^2 L^1 \dot H^{1}$, $\|\covD F\|_{\ell^2 L^1 \dot H^{1}} \lesssim_{M,M_1} 1$.

\item $\partial_s A \in \ell^2 L^1 \dot H^{1}$, $\|\rd_{s} A \|_{\ell^2 L^1 \dot H^{1}} \lesssim_{M,M_1} 1$.
\end{itemize}
which were established in Lemma~\ref{l:nocov-smth}.

Here we apply the previous estimates to $\covD^\ell B$. The  equations for $\covD^\ell B$ have the form
\begin{equation}\label{dlin-inhom}
\partial_s \covD^\ell B_k -  \covD^j ( \covD_j \covD^\ell B_k - \covD_k \covD^\ell B_j) - [\covD^\ell B^j,F_{jk}] =  H_k^\ell +  \covD^j G_{jk}^\ell 
\end{equation}
where
\[
\begin{split}
H_k^\ell =&  \covD^\ell H_k  +  \frac{1}{2} [F^{\ell j}, G_{jk}]  + [F^{\ell j},  \covD_j  B_k - \covD_k  B_j] + [\rd_{s} A, B_k] + [B^j,\covD^{\ell} F_{jk}],
\\
G_{jk}^\ell = &  \covD^\ell G_{jk} + ([ \tensor{F}{^{\ell}_j}, B_k] - [\tensor{F}{^{\ell}_k}, B_j]).
\end{split}
\]
Obtaining $\dot H^{\sigma -1}$ bounds for $\covD B$ suffices, in view of the elliptic bound
\[
\| \covD B\|_{\ell^2 L^\infty \dot H^{\sigma -1}} + \| \mathrm{curl}_{A} \, \covD B \|_{L^2 \dot H^{\sigma -1}} \approx_{\hM, M_1} 
\| B\|_{\ell^2 L^\infty \dot H^{\sigma}} + \| \mathrm{curl}_{A} \, B \|_{L^2 \dot H^{\sigma}} .
\]
The LHS of the last relation comes from the bounds \eqref{B-hs++} for $\covD B$. To obtain those, we 
treat the $B$ dependent terms in $H_k^\ell $ and $G_{jk}^\ell $ perturbatively. To guarantee smallness,
we partition the heat-time interval $J$ into finitely many subintervals where $\|F\|_{L^2 \dot H^1}$ and $\| \covD F\|_{\ell^2 L^1 \dot H^1}$
are small. The last norm is used in order to estimate the last two terms in $H_k^\ell$:
\[
\| [\covD F,B]\|_{L^1 \dot H^{\sigma -1}} \lesssim \| \covD F\|_{\ell^2 L^1 \dot H^1} \| B\|_{\ell^2 L^\infty \dot H^{\sigma-1}}.
\]
This is where the improved $\ell^2$ summation is essential in the last term. We omit further details.

\pfstep{Case 3: $-2 <  \sigma \leq -1$}
Here we argue by duality. The well-posedness for the linearized flow in $\dot{H}^{\sigma}$ is equivalent to the 
well-posedness for the adjoint linearized flow in $\dot{H}^{-\sigma}$. The adjoint linearized flow is a backward 
degenerate parabolic flow, which has exactly the same form as  \eqref{lin-inhom} but with the sign of $\partial_s$ 
reversed. Since our assumptions on $A$ and $F$ in the previous two steps are stable with respect to time reversal,
it follows that their conclusion applies to the adjoint linearized flow as well. Thus the desired conclusion follows.
\end{proof}

We now complement the previous result with a frequency envelope bound.
We assume that $A$ satisfies
 \begin{equation}\label{hf-A-fe-re}
\| P_k A\|_{L^\infty \dot H^1} + \|P_k F\|_{L^2 \dot H^1} + \|P_k \partial_s A\|_{L^2 L^{2}} + \| P_k \covD F\|_{L^1 \dot H^1}\lesssim c_k.
\end{equation}
Then we have the following:

\begin{lemma} \label{l:lin-fe}
Assume that \eqref{hf-A-fe} holds for some $(-1,S)$ frequency envelope $c_k$.
Let $d_k$ be a 2-compatible $(-1,S)$ frequency envelope for $B(0)$ in $L^2$, $G$ in $L^1 L^2$ 
and $H$ in $L^2$. Then we have
\begin{equation}
\| P_k B\|_{L^\infty L^2} + \|P_k \mathrm{curl}_{A} B\|_{L^2 L^{2}} \lesssim d_k .
\end{equation}
\end{lemma}

\begin{proof}
We prove the result in two steps repeating the analysis in the previous proof.

\pfstep{Step~1: The result for $1$-compatible frequency envelopes}  This is similar to the argument in the previous lemma  
 in the case $|\sigma| < 1$. The desired bound for
  $F_{0j}$ follows from Theorem~\ref{t:heatA-fe}. 

Next we consider the parabolic flow for $\covD^j F_{0j}$, 
\[
(\partial_s -\Delta_A) \covD^j F_{0j} =R           
\]
where (schematically)
\[
R =  [\covD F, F_{0j}] + [F, \covD F_{0j}] + \covD^j H_j + [F,G].
\]
Repeatedly using the Littlewood-Paley trichotomy we  estimate the right hand side
\[
\| P_k R\|_{L^1 L^2} \lesssim  2^k d_k 
\]
where the worst term is the first one where $\covD F$ is the high frequency factor. 
There we combine the $L^1 \dot H^1$ bound for $\covD F$ with the $L^\infty L^2$ bound for $F_{0j}$.

The frequency envelope $2^{k} d_k$ is still admissible. 
Hence by Theorem~\ref{t:heatA-fe} this bound for $R$ yields
\[
\| P_k \covD^j F_{0j}\|_{L^\infty L^2 \cap L^2 \dot H^1} \lesssim 2^k d_k .
\]
Moreover, by Theorem~\ref{t:heatA-L1} we get
\[
\| P_k \covD^j F_{0j}\|_{L^1 \dot H^1} \lesssim d_k.
\]
This in turn after integration yields
\[
\| P_k A_0\|_{L^\infty \dot H^1} \lesssim d_k,
\]
and thus the similar bound
\[
\| P_k \covD A_0\|_{L^\infty L^2} \lesssim d_k.
\]

Finally it remains to estimate $[F,A_0]$ in $L^2 L^{2}$. Here we combine the $A_0$ bound in $L^{\infty} \dot H^1$
and the $L^2 \dot H^1$ bound for $F$.

\pfstep{Step~2: The result for $2$-compatible frequency envelopes}  This is similar to the argument in the previous lemma  
 in the case $1 \leq \sigma < 2$. 

Here we work with the equations for $\covD B$.  If $d_k$ is a $2$-compatible frequency envelope for the initial data $b$,
\[
\| P_k b\|_{L^2} \lesssim d_k,
\]
then the data $\covD b$ for $\covD B$ satisfies (by Littlewood--Paley trichotomy)
\[
\| P_k \covD b \|_{L^2} \lesssim 2^k (d_k + c_k \sum_{j< k} 2^{2(j-k)} d_j) \lesssim  2^k d_k,
\]
and is $1$-compatible with $c_k$. 

We make the bootstrap assumption 
\[
\| P_k B\|_{L^{\infty} L^2} + \| P_k \mathrm{curl}_{A} B\|_{L^2 L^{2}} \leq C d_k.
\]
on some subinterval. To show that this extends to the whole interval, we need to verify that the $B$ dependent terms in the right hand side of \eqref{dlin-inhom}
can be treated perturbatively. The terms to bound are
\[
\| P_k ( [F, \mathrm{curl}_{A} B])\|_{L^1 L^2} + \| P_k [\covD F,B]\|_{L^1 L^2} + \| P_k [F,B]\|_{L^2}  \ll 2^k d_k
\]
The desired bound is obtained by standard Littlewood-Paley bilinear theory, which yields
\[
2^k (d_k + c_k \sum_{j < k} 2^{2(j-k)}d_j) \lesssim  
2^k (d_k + c_k \sup_{j < k} 2^{(2-\epsilon)(j-k)}d_j) \lesssim
2^k d_k
\]
However  we also need to gain smallness. In the last chain of inequalities it is clear that smallness holds
unless 
\begin{equation}\label{bad-case}
c_k \approx 1, \qquad 2^{(2-\epsilon)(j-k)} d_j \lesssim d_k, \qquad j < k .
\end{equation}
The first property selects only finitely many values of $k$; in those cases, we can gain smallness from the divisible norms of 
$F$ by subdividing the time interval.

Now we are in position to apply the bound in Step~1 for $\covD B$ to obtain
\[
\| P_k \covD B\|_{L^{\infty} L^2} + \| P_k\ \mathrm{curl}_{A} \covD B\|_{L^2} \lesssim 2^k d_k.
\]
It remains to return to $B$ and show that this implies
\[
\| P_k B\|_{L^{\infty} L^2} + \| P_k\ \mathrm{curl}_{A} B\|_{L^2} \lesssim d_k.
\] 
But this is done again perturbatively, where the errors are small unless \eqref{bad-case} holds.
But then we are in a position to apply Lemma~\ref{l:lin} directly. \qedhere
\end{proof}

Finally we use the frequency envelope bounds for the linearized equations 
in order to prove frequency envelope bounds for solutions to the Yang--Mills heat flow:
 \begin{lemma}\label{l:nolin}
Let $A$ be a sufficiently regular  Yang--Mills heat flow so that \eqref{l3} holds. 
Let $c_k^0$ be a $(-1,S)$ frequency envelope for the initial data $a$  in $\dot H^1$.
Then we have 
\begin{equation}
\label{hf-A-fe-re+}
\| P_k A\|_{L^\infty \dot H^1} + \|P_k F\|_{L^2 \dot H^1}+ \|P_k \partial_s A\|_{L^2 L^{2}}+ \| P_k \covD F\|_{L^1 \dot H^1} \lesssim c_k^0
\end{equation} 
\end{lemma}

\begin{proof}
This is based on the observation that if $A$ is a solution for the Yang--Mills heat flow 
then $\partial_x A$ are solutions for the corresponding linearized equations. 

\pfstep{Case 1} Here we consider the easier case when $c_k^0$ is a $(-1,1)$ frequency envelope.
Then the conclusion immediately follows from Lemma~\ref{l:lin}, with the exception of the $L^1 \dot H^1$
bound, which is instead obtained from Theorem~\ref{t:heatA-L1} applied to $\covD F$.

\pfstep{Case 2} 
In order to work with more general envelopes, we denote by $c_k^{[1]}$ a minimal $(-1,1)$ frequency 
envelope for $(A,F)$ in the sense of \eqref{hf-A-fe}, and by $c_k$ a minimal $(-1,S)$ frequency  envelope for $(A,F)$ in the same sense. 
By the result in Case 1 above we have $c_k^{[1]} \aleq c_{k}^{0 [1]}$. 
Define the envelope
\[
d_k =  c_k^0 + c_k \sup_{j < k} 2^{2(1-\epsilon)(j-k)} c_j^0.
\]
Then $2^k d_k$ is $2$-compatible with respect to $c_k$, therefore
applying Lemma~\ref{l:lin-fe} to $\covD A$ it follows that
\[
\| P_k \covD A\|_{L^\infty \dot H^1} + \|P_k \covD F\|_{L^2 \dot H^1} \lesssim 2^k d_k. 
\]
Then from Theorem~\ref{t:heatA-L1}, we obtain
\[
\| P_k \covD F\|_{L^1 \dot H^1} \lesssim d_k.
\]
Moreover, removing the covariant derivative from the first bound, we have
\[
\| P_k A\|_{L^\infty \dot H^1} + \|P_k F\|_{L^2 \dot H^1} \lesssim  d_k + c_k c_k^{[1]}.
\]
Thus, it follows that
\[
c_k \lesssim d_k   + c_k c_k^{[1]}  \lesssim  c_k^0 + c_k c_k^{[1]} .
\]
This implies $c_{k} \lesssim c_k^0$ unless $c_k^{[1]} \approx 1$. But in that case $c_k \approx c_k^{[1]}$,
and again we win by the result in Case 1, which implies $c_k^{[1]} \approx c_{k}^{0[1]} \approx c_{k}^{0}$. \qedhere

\end{proof}

The bounds for the linearized equation allow us to consider the case of
rough data $a \in \dot H^1$.  Our strategy will be to  define rough solutions not
directly, but rather as limits of smooth solutions. 

 \begin{lemma}\label{p:ext-h1}
   For every initial data $a \in \dot H^1$ for the Yang--Mills heat
   flow in the local caloric gauge there exists a nontrivial time
   interval $J = [0,s_0)$ and a local solution $A \in C(J;\dot H^1)$,
   which is the unique limit of regular solutions.
\end{lemma}

\begin{proof}

  For the existence part, we denote by $c_k$ a frequency envelope for
  $a$ in $\dot H^1$.  Then we consider a continuum of regularized data
  $a_{< k}= P_{<k} a$ for $k \geq k_0$, where
  $k_0$ will be chosen later.  We denote the corresponding
  solutions by $A_{<k}$, and the curvature $2$-forms by $F_{<k}$.  
  Given $\epsilon > 0$, we choose $J =
  [0,s_0]$ so that
\[
\| F_{<k_0} \|_{L^3(J; L^{3})} \leq 1.
\]
Then we consider the maximal interval $K = [k_0,k_1]$ with the property that for $k \in K$,
the solution $A_k$ exists in $J$ and satisfies
\[
 \| F_{<k} \|_{L^3(J; L^{3})} \leq 2, \ \text{for } k \in K.
\]

Within this range, the solutions $A_k$ are uniformly bounded in $L^\infty \dot H^1$.
Further,  we can combine the results in Lemmas~\ref{l:h3},\ref{l:lin} to conclude
that 
\[
\| (A_{<k+1} - A_{<k})\|_{L^\infty \dot H^1} + \|F_{<k+1} - F_{<k}\|_{L^{3} L^3} \lesssim c_k .
\]
with further decay away from frequency $2^k$,
\[
\| (A_{<k+1} - A_{<k})\|_{L^\infty \dot H^N} + \|F_{<k+1} - F_{<k}\|_{L^3 \dot W^{N-1,3}} \lesssim_N 2^{(N-1)k} c_k , \qquad N \geq 0.
\]
Here the implicit constants depend only on $\|a\|_{\dot H^1}$. Summing up, it follows that
\begin{equation} \label{diff-reg}
\| A_{<k_1} - A_{<k_0}\|_{L^\infty \dot H^1}^{2} +
\|F_{<k_1} - F_{<k_0}\|_{L^{3} L^3}^2 \lesssim \sum_{k_0 \leq  k  \leq k_1} c_k^2, \qquad k \in K.
\end{equation}
Now we choose $k_0$ so that 
\[
\sum_{k \geq k_0} c_k^2 \ll 1.
\]
Then for $k \in K$ we obtain 
\[
\| F_{<k} \|_{L^3 L^{3}} \leq \frac32.
\]
By the maximality of $K$ it follows that $K = [k_0,\infty)$. Further, by \eqref{diff-reg}
it follows that the limit
\[
A = \lim_{k \to \infty} A_{<k}
\]
exists in $L^\infty \dot H^1$. This is the desired solution. Further, we remark that 
the solution $A$ satisfies the frequency envelope bounds
\begin{equation}\label{fe-As}
\| P_k A\|_{L^\infty \dot H^1} + \| P_k F\|_{L^{3} L^{3}} \lesssim c_k, \qquad k > k_0,
\end{equation}
and the regularization bounds
\begin{equation}\label{fe-reg-diff}
\| A - A_{<k_0}\|_{L^\infty \dot H^1}^2 + \| F- F_{<k_0}\|_{L^{3} L^{3}}^2 \lesssim \sum_{k > k_0}
c_k^2,
\end{equation}
which will be useful later.
\end{proof}
\begin{remark} 
Our argument only insures that the rough solution we construct is unique among the limits of smooth solutions.
We leave open the question of establishing unconditional uniqueness in the larger class of $\dot H^{1}$ solutions.
\end{remark}

Next we consider the lifespan of $\dot H^1$ solutions.

\begin{lemma} \label{l:l3-cont}
Let $J$ be the maximal time of existence for a $\dot H^1$ solution to
the Yang--Mills heat flow in the local caloric gauge.  Then either $J
= [0,\infty)$ or $\|F\|_{L^3(J; L^{3})} = \infty$.
\end{lemma}

\begin{proof}
Let $A$ be an $\dot H^1$ solution on a finite time interval $J$ so that 
\[
\| F\|_{L^3(J; L^{3})} = \hM < \infty.
\]
Then we seek to show that the solution $A$ can be continued past
$J$. Let $J_0 \subset J$ and $k_0 < \infty$ be maximal so that the
approximate solutions $A_{<k}$ exist on $J_0$ and satisfy
\[
\| F_{<k} \|_{L^3(J_0; L^{3})} \leq 2 \hM , \qquad k \geq k_0.
\]
Such $J_0$ and $k_0$ exist by the existence part. Also by the argument
in the existence part, we obtain the bounds \eqref{fe-As} and
\eqref{fe-reg-diff}, but where the implicit constant now depends on
$F$. Choosing $k_0$ large enough so that 
\[
\sum_{k \geq k_0} c_k^2 \ll_\hM 1,
\]
from \eqref{fe-reg-diff} it follows that in effect 
\[
\| F_{<k} \|_{L^3(J_0; L^{3})} \leq \frac32 \hM , \qquad k \geq k_0.
\]
Hence we must have $J_0 = J$, else we contradict the maximality of $J$.
Thus the bounds \eqref{fe-As} and \eqref{fe-reg-diff} hold in $J$. Further, 
by Lemma~\ref{l:h3}, the solution $A_{<k_0}$ extends beyond $J$ to an interval $J_1$
where 
\[
\| F_{<k_0} \|_{L^3(J_1; L^{3})} \leq \frac54 \hM.
\]
Then the prior existence argument shows that the solutions $A$ and $A_{<k}$ all extend to 
$J_1$, satisfying similar bounds. \qedhere

\end{proof}

As a corollary of the above result (or rather its proof) we also obtain the following stability
result:

\begin{corollary}
Let $a$ be a $\dot H^1$ initial data with $\nrm{a}_{\dot{H}^{1}} \leq M_{1}$, and $J$  a time interval where the solution exists 
and satisfies 
\[
\|F\|_{L^3(J; L^{3})} \leq \hM.
\]
Then there exists $\epsilon = \epsilon(\hM, M_{1}) > 0$ so
that for all data $\tilde a$ satisfying
\[
\| a - \tilde a\|_{\dot H^1} \leq  \epsilon,
\]
the corresponding solution exists in $J$, satisfies 
\[
\|\tilde F\|_{L^3(J; L^{3})} \leq 2 \hM,
\]
and has a Lipschitz dependence on the initial data.
\end{corollary}

This concludes the proof of Theorems~\ref{t:hf-loc}, \ref{t:ymhf-l3} and \ref{t:ymhf-l3-fe} in $\dot H^1$, as well as
Corollary~\ref{c:max}.  

We next consider the similar problems for the $\bfH$ space, where all
that is needed is the $\ell^1 L^2$ norm both for $\partial^\ell A_\ell$ and
for $ \partial^\ell B_\ell$ in the context of the linearized equation.
This works because  the equation for $\partial^\ell A_\ell$ is
not strongly parabolic, instead it is merely a transport equation.
Our first goal is to show that the bounds for $\partial^\ell A_\ell$ propagate in time:

\begin{lemma} \label{l:da}
  Let $A$ be an $\dot H^1$  Yang--Mills heat flow so that
  \eqref{l3} and $\nrm{A(0)}_{\dot{H}^{1}} \leq M_{1}$ hold. Then we have the bounds
\begin{equation}
\sup_{s_1,s_2} \| \partial^{\ell} A_{\ell}(s_1) - \partial^{\ell} A_{\ell}(s_2) \|_{\ell^1 L^2  \cap \dot W^{1,\frac43}} 
\lesssim_{\hM, M_{1}}  1.
\end{equation}
In addition, if $c_k$ is a $(-1,S)$ frequency envelope for $a$ then we have
\begin{equation} \label{eq:da-fe}
\begin{split}
\sup_{s_1,s_2} \|P_k( \partial^\ell A_\ell(s_1) - \partial^\ell A_\ell(s_2)) \|_{L^2}\lesssim_{\hM, M_{1}}  & c_k c_k^{[1]}, \\
\sup_{s_1,s_2} \|P_k( \partial^\ell A_\ell(s_1) - \partial^\ell A_\ell(s_2)) \|_{\dot{W}^{1,\frac43}}  \lesssim_{\hM, M_{1}} &  c_k c_{\leq k}.
\end{split}
\end{equation}
\end{lemma}
\begin{proof}
We use the relation 
\[
\partial_s \partial^\ell A_\ell = \partial^\ell \covD^j F_{j \ell} =  - [A^\ell,\covD^j F_{j \ell}].
\]
Let $c_k$ be a frequency envelope for $a$ in $\dot H^1$. Then $c_k$ is
also a frequency envelope for $A$ in $L^\infty \dot H^1$, and also for
$\covD F$ in $L^1 \dot H^1$. The conclusion easily follows by an integration in heat-time.  \qedhere
\end{proof}

We remark that for the $L^2$ bound we have off-diagonal decay,
therefore we obtain a $c_k^2$ envelope (at least if $c_k$ is a $(-1,1)$ envelope), whereas the case of $L^{\frac{4}{3}}$ is borderline.

Now we switch to  the linearized Yang--Mills heat flow \eqref{lin-heat}:

\begin{lemma}\label{l:db}
  Let $A$ be an $\dot H^1$  Yang--Mills heat flow so that
  \eqref{l3} holds, and $B$ a corresponding $\dot H^1$ solution for the linearized 
equation. Then we have the bounds
\begin{equation}
\sup_{s_1,s_2}  \| \covD^\ell B_\ell(s_1) - \covD^\ell B_\ell(s_2) \|_{\ell^1 L^2  \cap \dot{W}^{1,\frac43}} 
\lesssim_{\hM, M_{1}}  \|b\|_{\dot H^1},
\end{equation}
and its analogue for $L^2$ data,
\begin{equation}
\sup_{s_1,s_2}  \| \covD^\ell B_\ell(s_1) - \covD^\ell B_\ell(s_2) \|_{\ell^1 \dot H^{-1}  \cap L^{\frac43}} 
\lesssim_{\hM, M_{1}}  \|b\|_{L^2}.
\end{equation}
In addition, if $c_{k}$ is a $(-1, S)$ frequency envelope for $a$ in $\dot{H}^{1}$, and $d_{k}$ is a $2$-compatible $(-1, S)$ frequency envelope for $b$ in $L^{2}$, then
\begin{equation} \label{eq:db-fe}
\begin{aligned}
\sup_{s_1,s_2}  \| P_{k}(\covD^{\ell} B_{\ell}(s_1) - \covD^{\ell} B_{\ell}(s_2)) \|_{\dot{H}^{-1}} \aleq_{\hM, M_{1}} & c_{k} d_{k}^{[2]} + d_{k} c_{k}^{[1]} + \sum_{j > k} 2^{k-j} c_{j} d_{j}, \\
\sup_{s_1,s_2}  \| P_{k}(\covD^{\ell} B_{\ell}(s_1) - \covD^{\ell} B_{\ell}(s_2)) \|_{L^{\frac43}} \aleq_{\hM, M_{1}} & c_{k} d_{k}^{[1]} + d_{k} c_{\leq k} + \sum_{j > k} 2^{k-j} c_{j} d_{j}.
\end{aligned}
\end{equation}
\end{lemma}
As a consequence of \eqref{eq:db-fe}, if $d'_{k}$ is a $(-1, S)$ frequency envelope for $b$ in $\dot{H}^{1}$ which is $1$-compatible with $c_{k}$, then 
\begin{equation} \label{eq:db-fe-h1}
\begin{aligned}
\sup_{s_1,s_2}  \| P_{k}(\covD^{\ell} B_{\ell}(s_1) - \covD^{\ell} B_{\ell}(s_2)) \|_{L^{2}} \aleq_{\hM, M_{1}} & c_{k} (d_{k}')^{[1]} + d_{k}' c_{k}^{[1]}, \\
\sup_{s_1,s_2}  \| P_{k}(\covD^{\ell} B_{\ell}(s_1) - \covD^{\ell} B_{\ell}(s_2)) \|_{\dot{W}^{1, \frac43}} \aleq_{\hM, M_{1}} & c_{k} d_{\leq k}' + d_{k} c_{\leq k}.
\end{aligned}	
\end{equation}
Indeed, note that $d_{k} = 2^{-k} d'_{k}$ is a $2$-compatible $(-1, S)$ frequency envelope for $b$ in $L^{2}$.
\begin{proof}
The equation for $\covD^{\ell} B_{\ell}$ is
\[
\partial_s \covD^{\ell} B_{\ell} = [\covD_j F^{j \ell}, B_{\ell}]  +  \covD^{\ell} [B_j,F_{j \ell}] + [F_{\ell j}, \covD^j B_{\ell}] 
= -2 [B^j,\covD^{\ell} F_{\ell j}].
\]
If $c_{k}$ and $d_{k}$ are as above, then they are frequency envelopes for $\covD F$ in $L^{1} \dot{H}^{1}$ and $B$ in $L^{\infty} L^{2}$, respectively. The desired lemma follows by integration in heat-time. \qedhere
\end{proof}

Next, we prove the explicit bound for $F \in L^{3} ([0, \infty); L^{3})$ when the energy is small:
\begin{lemma} \label{l:global-small}
Let $A$ be a $\dot{H}^{1}$ Yang--Mills heat flow with energy $\calE \ll 1$. Then $A$ exists globally, and
\begin{equation*}
	\nrm{F}_{L^{3}([0, \infty); L^{3})}^{2} \aleq \calE.
\end{equation*}
\end{lemma}
This proves Corollary~\ref{c:global-small}.
\begin{proof}
By local well-posedness and Lemma~\ref{l:l3-cont}, it suffices to prove the following: Assuming that $A$ exists on $J$ and satisfies the bootstrap assumption
\begin{equation*}
	\nrm{F}_{L^{3}(J; L^{3})}^{2} \leq 2 C_{0} \calE,
\end{equation*}
we claim that
\begin{equation} \label{eq:global-small-goal}
	\nrm{F}_{L^{3}(J; L^{3})}^{2} \leq C_{0} \calE,
\end{equation}
provided that $C_{0}$ is large enough, and $\calE$ is sufficiently small.

From the proof of Proposition~\ref{p:cov-smth}, recall that
\begin{equation*}
	\nrm{F}_{L^{\infty}_{\frac{\ud s}{s}}(J; L^{2})}^{2} + \nrm{s^{\frac{1}{2}} \covD F}_{L^{2}_{\frac{\ud s}{s}} (J; L^{2})}^{2} \aleq \calE + \nrm{F}_{L^{3}(J; L^{3})}^{3} \aleq \calE + C_{0}^{3/2} \calE^{3/2}.
\end{equation*}
Then by covariant Sobolev and interpolation, we have
\begin{equation*}
\nrm{F}_{L^{3}(J; L^{3})}^{2} \aleq \nrm{F}_{L^{\infty}_{\frac{\ud s}{s}}(J; L^{2})}^{2} + \nrm{s^{\frac{1}{2}} \covD F}_{L^{2}_{\frac{\ud s}{s}} (J; L^{2})}^{2} \aleq \calE + C_{0}^{3/2} \calE^{3/2},
\end{equation*}
from which \eqref{eq:global-small-goal} is clear.
\end{proof}

Finally, we are in a position to prove that the limits of $A$ and $B$ at infinity exist
if $F \in L^3([0, \infty); L^{3})$:
\begin{lemma} \label{l:global}
\begin{enumerate}
\item Let $A$ be an $\dot H^1$ (resp. $\bf H$)  Yang--Mills heat flow so that
  \eqref{l3} holds. Then the limit 
\[
A_{\infty} = \lim_{s \to \infty} A(s)
\]
exists in $\dot H^1$ (resp. $\bf H$), and has zero curvature 
\[
F_\infty = 0.
\]
Further, the map $A(0) \to A_\infty$ is $C^1$ in $\dot H^1$, $\bfH$ and $\dot H^1 \cap \dot H^\sigma$, where $\sigma > 1$.

\item Let $B$ be a  solution for the corresponding linearized 
equation. If $b \in \dot H^\sgm $ (resp. $\bfH$), $0 \leq \sgm \leq 1$, then the limit 
\[
B_{\infty} = \lim_{s \to \infty} B(s)
\]
exists in $\dot H^\sigma$ (resp. $\bfH$) and satisfies
\[
\covD[A_\infty]_{k} B_{\infty,j} - \covD[A_\infty]_{j} B_{\infty,k} = 0.
\]
If in addition $A(0) \in \dot H^{\sigma'}$ $(\sgm' > 1)$ then the same property holds in $\dot{H}^{1} \cap \dot H^{\sgm'}$.
\end{enumerate}
\end{lemma}

This proves Corollary~\ref{c:global}.
\begin{proof}
We proceed in several steps.
\pfstep{Step~1: Proof of (1), existence of $A_{\infty}$}
The result is obtained   by revisiting the proof of
Lemma~\ref{l:h1}. Precisely, the same computation but between two
times $s_0$ and $s_1$ shows that
\[
\lim_{s_0,s_1 \to \infty} \| A(s_0) - A(s_1)\|_{\dot H^1} = 0,
 \]
 as a consequence of the similar decay estimates for the parabolic
 space-time norms of $F$. The $\bfH$ bound follows in a similar manner 
as in Lemma~\ref{l:da}.

\pfstep{Proof of (2)}
We first consider $b \in \dot{H}^{1}$. We write the (schematic) equation for $G = \mathrm{curl}_{A} B$, 
\[
(\partial_s  - \Delta_A  - 2 ad(F)) G  = [F,\covD B]+ [\covD F,B].
\]
The RHS is bounded in $L^2 \dot H^{-1}$, so using the $L^2$ solvability for this 
problem we conclude that $\mathrm{curl}_{A} B \to 0$ in $L^2$ as $s \to \infty$. On the other hand 
as in Lemma~\ref{l:db}, $\mathrm{div}_{A} B = \covD^{\ell} B_{\ell}$ has a limit. Then the limit for $B$ is obtained by solving the associated covariant div-curl system, using the fact that the curvature decays to zero in $L^2$.

First, we consider the case $0 \leq \sgm < 1$. To avoid solving the covariant div-curl system, we employ a more 
roundabout route using the ``infinitesimal de Turck trick'' (cf. Section~\ref{subsec:dymhf}). As in Case~1 of the proof of Lemma~\ref{l:lin}, 
we introduce a dynamic component $A_{0}$ and the equation $F_{s0} = \covD^{\ell} F_{\ell 0}$ with $A_{0}(s=0) = 0$. Then $F_{0j}$ solves
\begin{equation*}
(\rd_{s} - \lap_{A} - 2 ad(F)) F_{0j} = 0, \qquad F_{0j}(s=0) = b_{j}.
\end{equation*}
Thus $F_{0j}$ vanishes at $s = \infty$ in $\dot{H}^{\sgm}$.

To transfer this behavior to $B_{j} = F_{0j} + \covD_{j} A_{0}$, it suffices to show that $\covD_{j} A_{0}$ has a limit as $s \to \infty$ in $\dot{H}^{\sgm}$. Proceeding as in Case~1 of the proof of Lemma~\ref{l:lin}, we obtain
\begin{equation*}
	\nrm{\covD^{\ell} F_{\ell 0}}_{\ell^{2} L^{1} \dot{H}^{\sgm+1}} \aleq \nrm{b}_{\dot{H}^{\sgm}}.
\end{equation*}
Since $\rd_{s} A_{0} = F_{s0} = \covD^{\ell} F_{\ell 0}$, we see that $A_{0} \to a_{0, \infty}$ in $\dot{H}^{\sgm+1}$. Then
using the $\dot{H}^{1}$ convergence of $A \to a_{\infty}$ in part~(1), and also the fact that $\sgm +1 < 2$, it follows that $\covD A_{0} \to \covD[a_{\infty}] a_{0, \infty}$ in $\dot{H}^{\sgm}$, as desired.

Finally, the $\bfH$ bound follows in a similar manner 
as in Lemma~\ref{l:db}, and the $\dot{H}^{1} \cap \dot{H}^{\sgm'}$ $(\sgm' > 1)$ bound follows from $\dot{H}^{1}$ case and the frequency envelope bound in Lemma~\ref{l:lin-fe}.

\pfstep{Step~3: Proof of (1), $C^{1}$ dependence of $A_{\infty}$ on $A(0)$}
Let $A^{h}(0)$ be a $C^{1}$ family of initial data in $\dot{H}^{1}$, $\bfH$ or $\dot{H}^{1} \cap \dot{H}^{\sgm}$ with $\sgm > 1$), and let $A^{h}$ be the corresponding Yang--Mills heat flows in the local caloric gauge. Then $\rd_{h} A^{h}$ solves the linearized equation, with data in the respective topology.  Thus, the desired statement follows from part (2). \qedhere
\end{proof}

\section{The Dichotomy and the Threshold Theorems} \label{s:thr}
In this section, we present the Dichotomy and the Threshold Theorems for the Yang--Mills heat flow, which are sharp criteria for global well-posedness and convergence to the flat connection of the Yang--Mills heat flow in $\dot{H}^{1}(\bbR^{4})$. 

\subsection{The Dichotomy Theorem} \label{subsec:dich}
Here, we precisely state and prove the Dichotomy Theorem for the Yang--Mills heat flow.
\begin{theorem}[Dichotomy Theorem for the Yang--Mills heat flow] \label{thm:dich}
Let $a$ be a connection 1-form in $\dot{H}^{1}$, and let $A$ be the solution to \eqref{eq:cYMHF} with initial data $A(s=0) = a$.
Then either the solution is global and $\hM(a) < \infty$, or the solution ``bubbles off'' a soliton in the following sense:
\begin{enumerate}[label=(\roman*)]
\item (Finite blow-up time). If the blow-up time (maximal existence time) $s_{0}$ is finite, then there exist a point $x_{0} \in \bbR^{4}$, a sequence of points $(x_{n}, s_{n}) \to (x_{0}, s_{0})$ and a sequence of scales $r_{n} \to 0$ such that
\begin{align*}
	\lim_{n \to \infty} \calG(O^{n}) (r_{n} A)(x_{n} + r_{n} x, s_{n} + r_{n}^{2} s) = Q(x) \quad \hbox{ in } L^{2}_{loc}(\bbR^{4} \times [-1, 0]), \\
	\lim_{n \to \infty} Ad(O^{n}) (r_{n}^{2} F)(x_{n} + r_{n} x, s_{n} + r_{n}^{2} s) = F[Q](x) \quad \hbox{ in } L^{2}_{loc}(\bbR^{4} \times [-1, 0]),
\end{align*}
for some sequence of $s$-independent gauge transformations $O^{n} \in H^{2}_{loc}$ (in the sense that $O^{n}_{;x} \in H^{1}_{loc}$) and a nontrivial finite energy harmonic Yang--Mills connection $Q$.

\item (Infinite blow-up time). If the maximal existence time is $s = \infty$, then there exist a point $x_{0} \in \bbR^{4}$, a sequence of points $(x_{n}, s_{n}) \to (x_{0}, \infty)$ and a sequence of scales $r_{n}$ such that
\begin{align*}
	\lim_{n \to \infty} \calG(O^{n}) (r_{n} A)(x_{n} + r_{n} x, s_{n} + r_{n}^{2} s) = Q(x) \quad \hbox{ in } L^{2}_{loc}(\bbR^{4} \times [-1, 0]), \\
	\lim_{n \to \infty} Ad(O^{n}) (r_{n}^{2} F)(x_{n} + r_{n} x, s_{n} + r_{n}^{2} s) = F[Q](x) \quad \hbox{ in } L^{2}_{loc}(\bbR^{4} \times [-1, 0]),
\end{align*}
for some sequence of $s$-independent gauge transformations $O^{n} \in H^{2}_{loc}$ (in the sense that $O^{n}_{;x} \in H^{1}_{loc}$) and a nontrivial finite energy harmonic Yang--Mills connection $Q$.
\end{enumerate}
\end{theorem}

The remainder of this subsection is devoted to the proof of Theorem~\ref{thm:dich}. Unless otherwise stated, we always consider Yang--Mills heat flows in the local caloric gauge given by Theorem~\ref{t:hf-loc}.

The key starting point of the proof is the monotonicity formula (or the energy identity)
\begin{equation} \label{eq:mono}
	\int \frac{1}{2} \brk{F_{ij}, F^{ij}} (s_{1}) \, \ud x + \int_{s_{0}}^{s_{1}} \int \brk{\covD^{\ell} F_{\ell i}, \covD^{\ell} \tensor{F}{_{\ell}^{i}}} \, \ud x \ud s= \int \frac{1}{2} \brk{F_{ij}, F^{ij}} (s_{0}) \, \ud x,
\end{equation}
which, in the case $\hM(a) = \infty$, allows us to locate intervals $I_{n}$ on which the Yang--Mills heat flow is concentrated in the sense that $\nrm{F}_{L^{3}(I_{n}; L^{3})} = 1$, but the \emph{harmonic (Yang--Mills) tension field} $\covD^{\ell} F_{\ell j}$ vanishes in $L^{2}(I_{n}; L^{2})$ in the limit as $n \to \infty$.

\begin{lemma} \label{lem:select-int}
Let $a^{n}$ be a sequence of initial data, and let $A^{n}$ be the corresponding solution to \eqref{eq:cYMHF} with $A^{n}(s=0) = a^{n}$, with the maximal heat-time interval of existence $J_{n}$, such that
\begin{equation*}
	\nrm{f^{n}}_{L^{2}} \leq \calE, \qquad \lim_{n \to \infty} \nrm{F^{n}}_{L^{3}(J_{n}; L^{3})} = \infty,
\end{equation*}
where $f^{n}$ and $F^{n}$ are the curvature $2$-forms for $a^{n}$ and $A^{n}$, respectively.
Then there exist subintervals $I_{n} \subseteq J_{n}$ with the following properties:
\begin{equation} \label{eq:select-int} 
	\nrm{F^{n}}_{L^{3}(I_{n}; L^{3})} = 1, \qquad
	\nrm{(\covD[A^{n}])^{\ell} F^{n}_{\ell i}}_{L^{2}(I_{n}; L^{2})} \to 0, \qquad
	\nrm{\covD[A^{n}] F^{n}}_{L^{2}(I_{n}; L^{2})} \aleq_{\calE} 1. 
\end{equation}
\end{lemma}
\begin{proof}
We partition the time intervals $J_n$ into subintervals $I_{n, m}$ where
$\|F^n\|_{L^3(I_{n, m}; L^{3})} = 1$ for each $n, m$. We select the subinterval with minimal energy
dissipation, and denote it by $I_{n}$. Thus, we now have a sequence of solutions 
$A_n$ in time intervals $I_n$ with the following properties:
\begin{equation*}
\| F^n\|_{L^3 (I_n \times \R^4)} = 1, \qquad \nrm{(\covD[A^{n}])^{\ell} F^{n}_{\ell i}}_{L^{2}(I_{n}; L^{2})} \to 0.
\end{equation*} 
Further, a straightforward integration by parts argument shows that the 
two bounds above also allow us to control the full gradient for the curvature,
\begin{equation*}
\nrm{\covD[A^{n}] F^{n}}_{L^{2}(I_{n}; L^{2})} \aleq 1. \qedhere
\end{equation*} 
\end{proof}

Naively, we wish to rescale $A^{n}$ so that each $I_{n}$ becomes the unit interval, and pass to the limit; if successful, the limiting curvature $F$ would satisfy $\|F\|_{L^3} = 1$, as well as
$\covD^{\ell} F_{\ell j} = 0$. The first property would guarantee that $F \neq 0$, and 
the second, that $F$ is a harmonic Yang--Mills connection, as claimed in Theorem~\ref{thm:dich}. 
However, to make such an argument precise, we need to handle several issues:
\begin{itemize}
 \item The major concern when taking a weak limit is to insure that this is nonzero. To achieve that,
we need to break the scaling and translation symmetries. This is done using localized $L^2$ norms. (Lemma~\ref{lem:conc})
 \item To pass to the weak limit at the level of $A^{n}$, we need a uniform bound for the size of $A^{n}$. While we do have uniformity in all $F^{n}$ bounds, there  is no such guarantee for $A_n$. We will address this by renormalizing $a_n$ via a suitable gauge transformation, and then using the local caloric gauge bounds in order to propagate these uniform bounds in heat-time. (Lemma~\ref{lem:rad-gauge})
\item Compactness fails, the obvious culprits being the scaling and translation (in both space and heat-time) symmetries. Even after factoring out the symmetry group, we still cannot hope for full compactness in the $\dot{H}^{1}$ topology, and we will have to settle for a weaker sense as stated in Theorem~\ref{thm:dich}.
\end{itemize} 

We now turn to the detail, addressing first the issue of breaking the scaling and translation symmetries.
The idea here is that having a nontrivial $L^3 L^{3}$ norm for the curvature requires some nontrivial
concentration of the $L^2$ norm on parabolic cubes $Q_\epsilon$ of the form
\begin{equation} \label{eq:Q-eps}
	Q_{\eps} = B_{\eps}(x_{1}) \times (s_{1}-\eps^{2}, s_{1}).
\end{equation}
We will measure the concentration via the scaling-invariant norm 
\[
\epsilon^{-1} \| F\|_{L^2 L^{2}(Q_\epsilon)} .
\]
To obtain this, we begin with a simple propagation bound:
\begin{lemma} \label{lem:en-prop}
Let $A$ be a Yang--Mills heat flow on $I$.
For any subinterval $J \subseteq I$, we have
\begin{equation*}
	\nrm{F}_{L^{\infty} L^{2}(B_{\eps}(x_{1}) \times J)} \aleq \eps^{-1} \nrm{F}_{L^{2} L^{2}(B_{2 \eps}(x_{1}) \times J)} + \nrm{\covD^{\ell} F_{\ell j}}_{L^{2}(B_{2 \eps}(x_{1}) \times J)}.
\end{equation*}
\end{lemma}
\begin{proof}
Let $\chi_\epsilon$ be a spatial cutoff adapted to $B_\epsilon(x_{1})$, which vanishes outside $2 B_{\eps}(x_{1})$. We compute 
\[
\frac{d}{ds} \int \chi_\epsilon(x) |F|^2 dx = - \int \chi_\epsilon(x) |\covD^{\ell} F_{\ell j}|^2 dx + \int \Delta \chi_\epsilon |F|^2 dx.
\]
Here $\Delta \chi_\epsilon$ has size $\epsilon^{-2}$ so integrating
this relation we obtain the desired conclusion. \qedhere

\end{proof}

Using Lemma~\ref{lem:en-prop}, we may prove the desired concentration lemma.
\begin{lemma} \label{lem:conc}
Let $A$ be a Yang--Mills heat flow on $I$ with energy $\leq \calE$, which satisfies
\begin{equation} \label{eq:selected-int} 
	\nrm{F}_{L^{3}(I; L^{3})} = 1, \qquad
	\nrm{\covD^{\ell} F_{\ell i}}_{L^{2}(I; L^{2})} < \dlt, \qquad
	\nrm{\covD F}_{L^{2}(I; L^{2})} \leq \calF. 
\end{equation}
If $\dlt \ll_{\calE, \calF} 1$, then there exists a parabolic cube $Q_{\eps} \subseteq \bbR^{4} \times I$  with the property that
\begin{equation} \label{eq:conc}
	\eps^{-1} \nrm{F}_{L^{2} L^{2}(Q_{\eps})} \ageq_{\calE, \calF} 1.
\end{equation}
\end{lemma}
\begin{proof}
In this proof, we use the shorthand $L^{p} L^{q}(I)$ for $L^{p}(I; L^{q})$. Furthermore, we suppress the dependence of the implicit constants on $\calE, \calF$.

We begin with the fixed time bound 
\[
\|F(s)\|_{L^3}^{3} \aleq \|F(s)\|_{L^2} \| \covD F(s)\|_{L^2}^2,
\]
which after time integration yields the interpolation inequality
\[
\|F\|_{L^3 L^{3}(I)}^3 \lesssim \| F\|_{L^\infty L^{2}(I)} \| \covD F\|_{L^2 L^{2} (I)}^2.
\]
Since $\| F\|_{L^3 L^{3}(I)}=1$, whereas $\nrm{F}_{L^{\infty} L^{2}(I)}, \nrm{\covD F}_{L^{2} L^{2}(I)} \aleq 1$, there must be some $s_{\ast} \in I$ so that 
\[
\|F(s_{\ast})\|^3_{L^3} \gtrsim \|F(s_{\ast})\|_{L^2} \| \covD F(s_{\ast})\|_{L^2}^2, \qquad
\|F(s_{\ast})\|^3_{L^3} \approx \abs{I}^{-1}.
\]
Moreover, going back to the first inequality it follows that 
\[
\| \covD F(s_{\ast})\|_{L^2}^2 \gtrsim |I|^{-1}.
\]
We now revisit the proof of this bound, using the improved Gagliardo-Nirenberg inequality:
\begin{equation*}
	\nrm{u}_{L^{2}} \aleq \nrm{u}_{L^{\frac{4}{3}}}^{\frac{2}{3}} \nrm{u}_{\dot{B}^{-2, \infty}_{\infty}}^{\frac{1}{3}},
\end{equation*}
where $\nrm{u}_{\dot{B}^{-2, \infty}_{\infty}} = \sup_{k} 2^{-2k} \nrm{P_{k} u}_{L^{\infty}}$. Since the kernel of $P_{k}$ rapidly decays on the scale $2^{-k}$, we have
\begin{equation*}
	\nrm{u}_{\dot{B}^{-2, \infty}_{\infty}}
	\aleq \sup_{B_{\eps}} \eps^{-2} \nrm{u}_{L^{1}(B_{\eps})}.
\end{equation*}
Putting these estimates together, at $s = s_{\ast}$, we have
\[
\begin{split}
\|F \|_{L^3}^3 = & \ \| |F|^2\|_{L^\frac32}^\frac32 \\
\aleq & \||F|^2\|_{L^1}^\frac12 \||F|^2\|_{L^2} \\
\aleq &\||F|^2\|_{L^1}^\frac12 \|\nabla |F|^2\|_{L^\frac43}^\frac23 \||F|^2\|_{B^{-2, \infty}_{\infty}}^\frac13 \\
\aleq & \|F\|_{L^2} \| \covD F\|_{L^2}^\frac23  \|F\|_{L^4}^\frac23 \sup_{B_\epsilon}( \epsilon^{-1} 
\|F\|_{L^2(B_\epsilon)})^\frac23 \\
\aleq & \|F\|_{L^2} \| \covD F\|_{L^2}^\frac43 
 \|F\|_{L^4}^\frac23 .
\end{split}
\]
By the (covariant) Sobolev embedding, the last line is bounded from above by $\nrm{F}_{L^{2}} \nrm{\covD F}_{L^{2}}^{2}$, and then by $\nrm{F}_{L^{3}}^{3}$ by our choice of $s_{\ast}$. Thus near equality must hold at the last step, i.e.,
\begin{equation} \label{eq:select-ball}
\sup_{B_\epsilon} \epsilon^{-1}  \|F(s_{\ast})\|_{L^2(B_\epsilon)} \approx \| \covD F(s_{\ast})\|_{L^2}.
\end{equation}
Further, recall that $ \| \covD F(s_{\ast})\|_{L^2} \gtrsim |I|^{-\frac12} $, so for a
near optimal $\epsilon$ we must also have $\epsilon \lesssim
|I|^\frac12$. Consider now a (nonstandard) parabolic cube $\tilde{Q}$ of the form
\begin{equation*}
	\tilde{Q} = B_{\eps_{\ast}}(x_{\ast}) \times I_{C^{-2} \eps_{\ast}^{2}}, \qquad \abs{I_{C^{-2} \eps_{\ast}^{2}}} = C^{-2} \eps_{\ast}^{2},
\end{equation*}
where $B_{\eps_{\ast}}(x_{\ast})$ is the near optimal ball in \eqref{eq:select-ball} and $s_{\ast} \in I_{C^{-2} \eps_{\ast}^{2}} \subseteq I$. By the previous lemma, it follows that
\begin{equation*}
	\eps^{-1} \nrm{F}_{L^{2} L^{2}(B_{2 \eps_{\ast}}(x_{\ast}) \times  I_{C^{-2} \eps_{\ast}^{2}})} \ageq 1.
\end{equation*}
By the pigeonhole principle, we may find a parabolic cube $Q_{\eps} \subseteq B_{2\eps_{\ast}}(x_{\ast}) \times I_{C^{-2} \eps_{\ast}^{2}}$ with $\eps = C^{-1} \eps_{\ast}$ satisfying the lower bound \eqref{eq:conc}, as desired. \qedhere
\end{proof}

Next, we handle the issue of obtaining a uniform bound for $A^{n}$. The idea is to exploit covariant parabolic regularity and use the radial (or exponential) gauge on a fixed heat time, then propagate the good bound using to other heat times in the local caloric gauge. 

\begin{lemma} \label{lem:rad-gauge}
Suppose that 
\begin{equation*}
\nrm{F(s=0)}_{L^{2}} \leq \calE, \qquad \nrm{F}_{L^{3} L^{3}(\bbR^{4} \times [0, 2])} \leq 1.
\end{equation*}
Then there exists a gauge transformation $O \in H^{2}_{loc}(\bbR^{4} \times [1, 2])$ (in the sense that $O_{;x} \in H^{1}_{loc}(\bbR^{4} \times [1, 2])$) such that $\tA = \calG(O) A = Ad(O) A - O_{;x}$ obeys
\begin{equation} \label{eq:rad-gauge}
	\nrm{\rd \tA}_{L^{\infty} L^{2}(B_{R} \times [1, 2])} + \nrm{\tA}_{L^{\infty} L^{4}(B_{R} \times [1, 2])} \aleq_{\calE, R} 1,
\end{equation}
where $B_{R}$ is the ball of radius $R$ with the same (spatial) center as $Q$, and
\begin{equation*}
	\rd_{s} \tilde{A} = 0 \qquad \hbox{ in } \bbR^{4} \times [1, 2].
\end{equation*}
\end{lemma}
\begin{proof}
Without loss of generality, we may assume that the spatial center of $Q$ is $0$. We consider first the special case when $a = A(s=0)$ is smooth, and then the general case.
\pfstep{Case~1: $a$ is smooth}
Thanks to the $L^{3} L^{3}$ norm bound for $F$, the local caloric gauge solution $A$ is smooth on $[0, 2]$.
Moreover, by Proposition~\ref{p:cov-smth}, we have full covariant parabolic regularity of $F$. In particular, on the interval $[1, 2]$ we have
\begin{equation} \label{eq:cov-smth-s=1}
\nrm{\covD^{(N)} F}_{L^{\infty} L^{\infty}(\bbR^{4} \times [1,2])}
\aleq \sum_{0 \leq k \leq N+3} \nrm{\covD^{(k)} F}_{L^{\infty} L^{2}(\bbR^{4} \times [1,2])} \aleq_{N} 1.
\end{equation}
Solving the ODE
\begin{equation*}
	O^{-1} \rd_{r} O = A_{r}(1), \qquad O(0) = Id,
\end{equation*}
which is straightforward since $A$ is smooth, we obtain a smooth $s$-independent gauge transformation $O$ such that $\tA = \calG(O) A = Ad(O) A - O_{;x}$ satisfies
\begin{equation*}
	\tA_{r}(s=1) = 0, \qquad \tA_{s} = 0 \quad \hbox{ on } [0, 2].
\end{equation*}
In particular, in the polar coordinates $(r, \Tht)$, we have $\rd_{r} \tA_{\Tht} = \tilde{F}_{r \Tht} = Ad(O) F_{r \Tht}$ at $s = 1$. Integrating \eqref{eq:cov-smth-s=1} in the radial direction, we easily obtain
\begin{equation} \label{eq:rad-gauge-s=1}
	\nrm{\rd \tA(s=1)}_{L^{2}(B_{R})} + \nrm{\tA(s=1)}_{L^{4}(B_{R})} \aleq_{R} 1.
\end{equation}
By Lemma~\ref{l:h1} (more precisely, see Remark~\ref{rem:h1-loc}), we may propagate this bound to other times. This proves the desired bound \eqref{eq:rad-gauge}.

\pfstep{Case~2: $a \in \dot{H}^{1}$}
To avoid solving the ODE for $O$, we approximate the rough solution $A$ by smooth solutions. More precisely, for each $k \in \bbR$ consider the smooth approximation $a_{<k} = P_{<k} a$ of $a$, and let $A_{<k}$ be the corresponding local caloric gauge Yang--Mills heat flow with $A_{<k}(s=0) = a_{<k}$. For $k$ sufficiently large, we know that $A_{<k}$ exists on $[0, 2]$, $A_{<k} \to A$ in $L^{\infty} \dot{H}^{1}$, and that its curvature $F_{<k}$ obeys
\begin{equation*}
	\nrm{F_{<k}(s=0)}_{L^{2}} \leq 2 \calE, \qquad
	\nrm{F_{<k}}_{L^{3} L^{3}(\bbR^{4} \times [0, 2])} \leq 2. 
\end{equation*}
Then by the previous case, we may find a smooth gauge transformation $O_{<k}$ such that $\tA_{<k} (s=1)= Ad(O_{<k}) A_{<k}(s=1) - O_{<k; x}$ obeys
\eqref{eq:rad-gauge-s=1} with a uniform constant. In particular, for each fixed $R$,
\begin{equation*}
	\nrm{O_{<k; x}}_{H^{1}(B_{R})} \aleq_{\calE, R} 1 + \nrm{a}_{\dot{H}^{1}}.
\end{equation*}
Let us fix a matrix group representation $\G \hookrightarrow O(N) \subseteq \bbR^{N \times N}$ (which exists since $\G$ is compact), and view $O_{<k}$ as a sequence in $H^{2}_{loc}(\bbR^{4}; \bbR^{N \times N})$. By the preceding bound and \eqref{eq:rad-gauge-s=1}, after passing to a subsequence, we may find a weak limit $O_{<k} \weakto O$ in $H^{2}_{loc}(\bbR^{4}; \bbR^{N \times N})$, which also converges a.e, and $\tA_{<k}(s=1) \weakto \tA(s=1)$ in $H^{1}_{loc}(\bbR^{4})$. This weak convergence is sufficient to justify that $O \in \G$ a.e., $O_{;x} = \rd_{x} O O^{-1} \in H^{1}_{loc}(\bbR^{4}; \g)$ and
\begin{equation*}
	\tA(s=1) = Ad(O) A(s=1) - O_{;x},
\end{equation*}
as well as the bound \eqref{eq:rad-gauge-s=1}. Extending $\tA$ to $s \in [0, 2]$ by defining $\tA(s) = Ad(O) A(s) - O_{;x}$, and using Lemma~\ref{l:h1} (more precisely Remark~\ref{rem:h1-loc}), the desired bound \eqref{eq:rad-gauge} follows. \qedhere
\end{proof}

We are now ready to complete the proof of Theorem~\ref{thm:dich}.
\begin{proof}[Proof of Theorem~\ref{thm:dich}]
When $\hM(a) < \infty$, $A$ extends globally thanks to Corollary~\ref{c:max}. It remains to consider the case $\hM(a) = \infty$.

\pfstep{Step~1: Selection of intervals}
Let $\calE = \spE[a]$, and let $J = [0, s_{0})$ be the maximal interval of existence of the local caloric gauge solution $A$ with data $a$, where $s_{0}$ may be finite or infinite. In either case, we fix a sequence $\tilde{s}^{n} \nearrow s_{0}$ and apply Lemma~\ref{lem:select-int} to $a^{n} = A(\tilde{s}^{n})$ (which is possible since $\hM(a^{n}) = \infty$ for each $n$). The resulting sequence of solutions and intervals, which we denote by $A^{n}$ and $I_{n}$, satisfy \eqref{eq:select-int}.

\pfstep{Step~2: Breaking the scaling and translation symmetries}
Next, we apply Lemma~\ref{lem:conc} to each $A^{n}$ on $I_{n}$, whose hypothesis is insured by \eqref{eq:select-int} for large enough $n$. 
Thus we find a parabolic cube $Q^{n} = B_{r_{n}}(x_{n}) \times (s_{n}-r_{n}^{2}, s_{n}) \subseteq \bbR^{4} \times I_{n}$ with the property that
\begin{equation*}
	r_{n}^{-1} \nrm{F^{n}}_{L^{2} L^{2}(Q^{n}} \ageq_{\calE} 1.
\end{equation*}
By construction, observe that $s_{n} \to s_{0}$ and $r_{n} \aleq s_{0} - \tilde{s}_{n} \to 0$ if $s_{0} < \infty$.
We rescale and translate each $A^{n}$ so that $(x_{n}, s_{n}, r_{n}) = (0, 0, 1)$; for simplicity of notation, we still call the resulting heat flows $A^{n}$.

\pfstep{Step~3: Gauge transformation and compactness argument}
For the sequence of solutions constructed at the previous step, we have the uniform covariant bounds 
\[
\|F^n\|_{L^\infty L^2} \leq \calE, \qquad
\|F^{n}\|_{L^3 L^{3}} \leq 1, \qquad
\|\covD[A^{n}] F^{n}\|_{L^2} \lesssim_{\calE} 1,
\]
the decay 
\[
\| \covD^j F_{jk}\|_{L^2} \to 0,
\]
and the bound from below
\[
\| F^n\|_{L^2(Q)} \gtrsim_E 1.
\]
Applying Lemma~\ref{lem:rad-gauge}, we find a sequence of $s$-independent gauge transformations $O^{n} \in H^{2}_{loc}$ such that $\tA^{n} = \calG(O^{n}) A^{n} =  Ad(O^{n}) A^{n} - O^{n}_{x}$ obeys the uniform local-in-space bounds
\[
\| \tA^n\|_{L^\infty \dot H^1 (B_{R} \times [1, 2])} \lesssim_{\calE, R} 1.
\]
By covariant Sobolev embeddings this further yields the bounds from above
\[
 \| \tF^{n}\|_{L^2 \dot H^1(B_{R} \times [1,2])}+ \| \partial_s \tF^{n}\|_{L^2 \dot H^\frac12 (B_{R} \times [1,2])} \lesssim_{\calE, R} 1.
\]

We consider a weakly convergent subsequence in the above topologies,
and denote by $(A, F)$ the corresponding limits. These must satisfy the same bounds 
from above as $\tF^n$ and $\tA^n$.

In what follows, we drop the tilde and simply write $(A^{n}, F^{n}) = (\tA^{n}, \tF^{n})$ for simplicity of notation.
  By compact Sobolev embeddings
the sequences $F^n$ and $A^n$ can be taken to converge strongly in say $L^2(B_{R} \times [-1,0])$.
This shows that $F$ is the curvature of $A$, and it also allows us to pass to the limit in the last 
two relations two obtain
\[
\covD^j F_{jk} = 0, \qquad \| F\|_{L^2(Q)} > 0
\]
We can also pass to the limit in the local caloric gauge Yang--Mills heat flow in $L^2 L^{2}$ 
to obtain 
\[
\partial_s A = 0
\]
Thus $A$ is a nontrivial, stationary $H^1_{loc}$ connection, which satisfies the harmonic Yang--Mills 
equation. \qedhere
\end{proof}

\subsection{The Threshold Theorem} \label{subsec:thr}
In this section, we prove the Threshold Theorem for the energy critical Yang--Mils heat flow in $\dot{H}^{1}$, whose precise statement is as follows:
\begin{theorem}[Threshold theorem] \label{thm:thr} Let $a$ be a connection
  1-form in $\dot{H}^{1}(\bbR^{4})$ with energy $\calE$, which is below twice the ground
  state energy $2 \Egs$.  Let $A$ be the solution to \eqref{eq:cYMHF}
  with initial data $A_{i}(s=0) = a_{i}$. Then $A$ extends globally in
  heat-time. Moreover, there exists a positive function $\hM(\calE)$
  such that
\begin{equation} \label{eq:thr-unif}
	\calQ(a) = \nrm{F}_{L^{3}([0, \infty); L^{3})} \leq \hM(\calE).
\end{equation}
\end{theorem}

As a consequence of this result and Theorem~\ref{t:ymhf-l3}, we have global-in-time bounds for both the 
Yang--Mills heat flow in the local caloric gauge with subthreshold data, as well as its linearization.

Some preliminary discussion is in order. From the Dichotomy Theorem (Theorem~\ref{thm:dich}), if $a$ fails to exist globally or $\hM(a) = \infty$, then a nontrivial harmonic Yang--Mills connection $Q$ bubbles off. Since $\calE[Q] \geq \Egs$, this scenario is ruled out when the energy of $a$ is below $\Egs$. Theorem~\ref{thm:thr} differs from this naive result in two ways:
\begin{itemize}
\item The threshold energy is $2\Egs$, instead of $\Egs$. This refinement is achieved by taking into account the ``topological triviality'' of $\dot{H}^{1}$ connections, as well as ``topological nontriviality'' of any harmonic Yang--Mills connection $Q$ with $\spE[Q] < 2 \Egs$. 
\item Instead of just the qualitative statement $\calQ(a) < \infty$, a uniform a-priori bound $\hM(a) \leq \hM(\calE)$ for any data $a$ with energy $\leq \calE < 2 \Egs$ is claimed. For this purpose, we apply the argument in Section~\ref{subsec:dich} to a sequence of solutions; in contrast, it was applied to a single solution in the Dichotomy Theorem.
\end{itemize}

Associated to a connection $a$ on $\bbR^{4}$ with curvature $f \in L^{2}$, we introduce the \emph{characteristic number}
\begin{equation} \label{eq:ch}
\boldsymbol{\chi}(a) = \int_{\bbR^{4}} - \brk{f \wedge f} = \frac{1}{4} \int_{\bbR^{4}} - \brk{f_{ij}, f_{k \ell}} \, d x^{i} \wedge d x^{j} \wedge d x^{k} \wedge d x^{\ell}.
\end{equation}
Remarkably, the characteristic number $\boldsymbol{\chi}(a)$ vanishes for $a \in \dot{H}^{1}$. This is a manifestation of ``topological triviality'' of an $\dot{H}^{1}$ connection. 
\begin{lemma} \label{lem:top-triv}
If $a \in \dot{H}^{1}$ with curvature $2$-form $f$, then 
\begin{equation*}
	\boldsymbol{\chi}(a) = \int_{\bbR^{4}} \brk{f \wedge f} = 0.
\end{equation*}
\end{lemma}

\begin{proof}
We give a direct computation. The $4$-form $-\brk{f \wedge f}$ is closed, and thus exact (since $\bbR^{4}$ is contractible). Indeed, if we introduce the $3$-form
\begin{equation*}
	Y = - \left( \brk{a_{j}, \rd_{k} a_{\ell}} + \frac{1}{3} \brk{a_{j}, [a_{k}, a_{\ell}]} \right) \, d x^{j} \wedge d x^{k} \wedge d x^{\ell}
\end{equation*}
then 
\begin{equation*}
	d Y = - \brk{f \wedge f}.
\end{equation*}
By Stokes's theorem, 
\begin{equation*}
	\boldsymbol{\chi}(a) = \liminf_{R \to \infty} \int_{\rd B_{R}(0)} Y = - \liminf_{R \to \infty} \int_{\rd B_{R}(0)} \left( \brk{a_{j}, \rd_{k} a_{\ell}} + \frac{1}{3} \brk{a_{j}, [a_{k}, a_{\ell}]} \right) \, d x^{j} \wedge d x^{k} \wedge d x^{\ell}.
\end{equation*}
But since $a \in \dot{H}^{1}$, the RHS vanishes. \qedhere
\end{proof}

\begin{remark}
In \cite{OTYM2.5}, where we define precisely the notion of topological classes of (possibly rough) finite energy connections, it is shown that (1) $a \in \dot{H}^{1}$ is equivalent to $a$ being in the same topological class as the trivial connection (hence ``topologically trivial''), and (2) $\boldsymbol{\chi}$ is a topological invariant. These facts lead to an alternative proof of the lemma.
\end{remark}

On the other hand, the integrand of \eqref{eq:ch} provides a pointwise lower bound on the energy density. This statement is often referred to as the \emph{Bogomoln'yi bound}.
\begin{lemma} \label{lem:bog}
We have the pointwise bound
\begin{equation} \label{eq:bog}
	\abs{\brk{f \wedge f}} \leq \frac{1}{2} \brk{f, f},
\end{equation}
where we use the standard inner product $(\cdot, \cdot)$ for $2$-forms, which makes $\set{d x^{j} \wedge d x^{k} : j < k}$ an orthonormal basis.
\end{lemma}
\begin{proof}
We use the Hodge star operator $\star$, which has the property
\begin{equation*}
	(\omg, \eta) d x^{1} \wedge \cdots \wedge d x^{4} = \omg \wedge \star \eta \qquad \hbox{ for $2$-forms } \omg, \eta,
\end{equation*}
as well as $\star 1 = d x^{1} \wedge \cdots \wedge d x^{4}$ and $\star d x^{1} \wedge \cdots \wedge d x^{4} = 1$.
Then
\begin{align*}
	\brk{f, f} 
	=& \star (2 \brk{f \wedge \star f}) \\
	=& \star \brk{(f \pm \star f) \wedge \star (f \pm \star f)} \pm 2 \star \brk{f \wedge f} \\
	=& \brk{(f \pm \star f) \wedge (f \pm \star f)} \pm 2 \star \brk{f \wedge f}.
\end{align*}
Since the first term on the last line is nonnegative, \eqref{eq:bog} follows. \qedhere
\end{proof}

Next, we need the fact that any nontrivial harmonic Yang--Mills connection with energy below $2\Egs$ is ``topologically nontrivial'', in the sense that $\abs{\boldsymbol{\chi}} = \Egs$. Indeed, we have:
\begin{theorem} \label{thm:gks}
Let $\G$ be a noncommutative compact Lie group. Let
\begin{equation*}
	\Egs = \inf \set{\spE[Q] : \hbox{$Q$ is a nontrivial harmonic Yang--Mills connection on $\bbR^{4}$}}.
\end{equation*}
Then the following statements hold.
\begin{enumerate}
\item There exists a nontrivial harmonic Yang--Mills connection $Q$ such that $\spE[Q] = \Egs < \infty$.
\item Let $Q$ be any nontrivial harmonic Yang--Mills connection. Then either $\spE[Q] \geq 2 \Egs$, or 
\begin{equation*}
	\abs{\boldsymbol{\chi}(Q)} = \spE[Q] \geq \Egs.
\end{equation*}
\end{enumerate}
\end{theorem}

This theorem is a combination of well-known facts concerning
instantons (i.e., energy minimizers in a topological class) and a
recent lower bound on non-instanton harmonic Yang--Mills connections
by Gursky--Kelleher--Streets \cite{GKS}. For instance, consider the
case $\G = SU(2)$ and $\brk{A, B} = - tr(AB)$, where
$\boldsymbol{\chi}(a) = 8 \pi^{2} c_{2}$ (here, $c_{2}$ is the second
Chern number computed from the connection $a$). Then part (1) is the
classical existence of BPST instantons, and part (2) follows from
\cite[Corollary~1.3]{GKS}. For the proof of Theorem~\ref{thm:gks} in
the general case of a noncommutative compact Lie group $\G$, we refer
the reader to \cite[Section~6]{OTYM2.5}.

We are now ready to prove Theorem~\ref{thm:thr}.
\begin{proof}[Proof of Theorem~\ref{thm:thr}]
We divide the proof into two steps.
\pfstep{Step~1: Contradiction argument and extraction of a bubble}
Fix a positive number $\calE < 2 \Egs$. Suppose, for the purpose of contradiction, that there does \emph{not} exist $\hM(\calE) > 0$ such that \eqref{eq:thr-unif} holds for every $a \in \dot{H}^{1}$ with $\spE[a] \leq \calE$. Then there exists a sequence $a^{n} \in \dot{H}^{1}$ of initial data, such that the corresponding solutions on the maximal time interval of existence $J_{n}$ obey
\begin{equation*}
	\nrm{F^{n}}_{L^{3}(J_{n}; L^{3})} \to \infty.
\end{equation*}
Applying Lemma~\ref{lem:select-int}, we may find a sequence of solutions $A^{n}$ on $\bbR^{4} \times I_{n}$ satisfying \eqref{eq:select-int}. Proceeding as in Steps~2 and 3 in the proof of the Dichotomy Theorem, we find a sequence $(x_{n}, s_{n}, r_{n}, O^{n})$ such that (after passing to a subsequence)
\begin{align}
	\lim_{n \to \infty} \tilde{A}^{n}(x, s) = Q(x) \quad \hbox{ in } L^{2}_{loc}(\bbR^{4} \times [-1, 0]), \\
	\lim_{n \to \infty} \tilde{F}^{n}(x, s) = F[Q](x) \quad \hbox{ in } L^{2}_{loc}(\bbR^{4} \times [-1, 0]), \label{eq:F-L2-loc-conv}
\end{align}
where 
\begin{align*}
	\tilde{A}^{n}(x, s) =& \calG(O^{n}) (r_{n} A^{n})(x_{n} + r_{n} x, s_{n} + r_{n}^{2} s), \\
	\tilde{F}^{n}(x, s) =& F[\tilde{A}^{n}](x, s) = Ad(O^{n}) (r_{n}^{2} F^{n})(x_{n} + r_{n} x, s_{n} + r_{n}^{2} s),
\end{align*}
and $Q$ is a nontrivial $H^{1}_{loc}$ harmonic Yang--Mills connection.

\pfstep{Step~2: Refinement using $\boldsymbol{\chi}$}
 By the local $L^{2}$ convergence \eqref{eq:F-L2-loc-conv}, it follows  that
\begin{equation*}
	\spE[Q] \leq \calE < 2 \Egs.
\end{equation*}
Thus, by Theorem~\ref{thm:gks}, $\abs{\boldsymbol{\chi}(Q)} = \spE[Q]$. Without loss of generality, assume that $\boldsymbol{\chi}(Q) > 0$.

Fix any $R > 0$. By Lemma~\ref{lem:top-triv}, we have
\begin{align*}
	0 = \boldsymbol{\chi}(\tilde{A}^{n}(s)) =& \int_{B_{R}} - \brk{\tilde{F}^{n} \wedge \tilde{F}^{n}}(s) + \int_{\bbR^{4} \setminus B_{R}} - \brk{\tilde{F}^{n} \wedge \tilde{F}^{n}}(s).
\end{align*}
Thus, again by the local $L^{2}$ convergence \eqref{eq:F-L2-loc-conv}, we have
\begin{equation*}
	\lim_{n \to \infty} \int_{\bbR^{4} \setminus B_{R}} \brk{\tilde{F}^{n} \wedge \tilde{F}^{n}}(s) = \int_{B_{R}} - \brk{F[Q] \wedge F{Q]}}
\end{equation*}
after passing to a subsequence, for almost every $s \in (-1, 0)$. Given any $\eps > 0$, by choosing $R$ sufficiently large, the RHS can be made equal to $\boldsymbol{\chi}(Q)$ up to an error of size as most $\eps$. Hence, by Lemma~\ref{lem:bog},
\begin{align*}
	\calE \geq & \limsup_{n \to \infty} \frac{1}{2} \int_{\bbR^{4}} \brk{\tilde{F}^{n}, \tilde{F}^{n}}(s) \\
	= & \limsup_{n \to \infty} \frac{1}{2} \int_{\bbR^{4} \setminus B_{R}} \brk{\tilde{F}^{n}, \tilde{F}^{n}}(s) + \frac{1}{2} \int_{B_{R}} \brk{\tilde{F}^{n}, \tilde{F}^{n}}(s) \\
	\geq & \limsup_{n \to \infty} \abs{\int_{\bbR^{4} \setminus B_{R}} \brk{\tilde{F}^{n} \wedge \tilde{F}^{n}}(s)} + \abs{\int_{B_{R}} \brk{\tilde{F}^{n} \wedge \tilde{F}^{n}}(s)} \\
	 \geq & 2 \boldsymbol{\chi}(Q) - 2 \eps = 2 \Egs - 2 \eps.
\end{align*}
Since $\eps > 0$ is arbitrary, we arrive at $\calE \geq 2 \Egs$, which is a contradiction.  \qedhere

\end{proof}
\section{ The caloric gauge}\label{s:caloric}

\subsection{Caloric connections and the projection map}

The results in Section~\ref{s:caloric-loc} show that for 
connection $a \in \dot{H}^{1}$ with $\hM(a) < \infty$, its Yang--Mills heat flow $A$ converges at infinity to a flat connection $A_\infty$
which has a $C^1$ dependence on $a$ in the topologies $\dot{H}^{1}$, $\bfH$ and $\dot{H}^{1} \cap \dot{H}^{\sgm}$ $(\sgm > 1)$. Moreover, in Section~\ref{s:thr}, we showed that for connections $a \in \dot{H}^{1}$ with $\spE[a] < 2 \Egs$, its Yang--Mills heat flow in the local caloric gauge is globally well-posed in the above topologies, and $\hM(a) \leq \hM(\calE)$ for some positive function $\hM(\calE)$. 

We are now ready to take advantage of these properties in order to
formulate precisely the (global) \emph{caloric} gauge. We first define caloric gauge connections.

\begin{definition}
A  $\dot H^1$ connection $a$ in $\R^4$ with $\hM(a) < \infty$ is a \emph{caloric gauge connection} if 
the corresponding Yang--Mills heat flow $A$ in the local caloric gauge satisfies $A(\infty) = 0$.
\end{definition}

 We immediately have the following:

\begin{proposition} \label{p:cal-a} 
 For each $\dot H^1$ connection $\ta$ in $\R^4$ with $\hM(a) < \infty$, there
  exists an unique (up to constant gauge transformations) gauge-equivalent connection
  $a$, which is a caloric gauge connection.  Further, the map
$\ta \to a$ is continuous in the quotient topology defined by the distance
\[
d(a_1,a_2) = \inf_{O \in \G} \| Oa_1O^{-1} - a_2\|_{\dot H^1}
\]
and $C^1$ in $\bfH$
(for a suitable choice of the associated constant gauge transformation).
\end{proposition}

In the sequel, we will denote the gauge-equivalent caloric gauge connection $a$ given by this 
proposition by 
\[
a = \Cal (\ta).
\]

\begin{proof}
  If $\ta$ is a $\dot{H}^{1}$ connection with $\hM(\ta) < \infty$, 
  then the corresponding solution to the Yang--Mills heat flow in the local caloric gauge is global (by definition) and has a limit $a_\infty$ as $s$ approaches
  infinity (Corollary~\ref{c:global}). Furthermore, $a_\infty$ has zero curvature. Then we need to  find $O$ so that
\[
O^{-1} \rd_{j} O = a_{\infty, j}.
\]
Then the connection $a$ is given by 
\[
a_j = O \ta_j O^{-1} -O_{;j}.
\]

Here we interpret this as a system of ODE's.  Formally, the zero curvature condition is viewed
as a complete integrability condition for this system. This is rigorous 
if $a_\infty$ is more regular, e.g. if $a_\infty \in \dot H^1 \cap \dot H^3$,
in which case we can initialize $O$ by\footnote{In case (i) below another natural normalization is to set $O(\infty) = Id$.} $O(0) = Id$ .

For less regular $a_\infty$, we consider a sequence of regular
approximations $a^n_\infty$, which are obtained simply by localizing
the initial data $\ta$ below frequency $2^n$.  We note that $n$ must be
sufficiently large, in order to insure that the truncated connections also obey $\hM(\ta^{n}) \leq 2 \hM(\ta) < \infty$.  
This leads to a corresponding sequence
$O^n$ of regular gauge transformations.  For the sequence $O^n$, we have
\begin{equation*}
	[O^{n} (O^{m})^{-1}]_{;x} = Ad(O^{n}) (a^{n}_{\infty} - a^{m}_{\infty}).
\end{equation*}
Hence, an easy computation shows that
\[
\|  [O^n (O^m)^{-1}]_{;x}\|_{\dot H^1} \lesssim \|a^n_\infty - a^m_\infty\|_{\dot H^1} ,
\]
but for the pointwise bound we only have
\[
\|  [O^n (O^m)^{-1}]\|_{L^\infty} \lesssim \|a^n_\infty - a^m_\infty\|_{\ell^1 \dot H^1} .
\]
This is proved in a standard manner as in the proof of the Gagliardo--Nirenberg--Sobolev,
by showing first that we have the averaged bound
\begin{align*}
	R^{-4} \int_{B_{R}(x)} d(O^{n}(O^{m})^{-1}(x), O^{n} (O^{m})^{-1}(y)) \, d y
	\aleq & \int_{B_{R}(x)} \frac{1}{\abs{x-y}^{3}} \abs{[O^{n} (O^{m})^{-1}]_{;x}(y)} \, d y \\
	\aleq & \int_{\bbR^{4}} \frac{1}{|x-y|^3} | a^n_\infty(y) - a^m_\infty(y)| dy,
\end{align*}
then observing that the RHS is bounded by $\nrm{a^{n}_{\infty} - a^{m}_{\infty}}_{\ell^{1} \dot{H}^{1}}$. The pointwise bound follows by taking $R \to 0$ and using the Lebesgue differentiation theorem.

Thus it is natural to distinguish two scenarios:

\pfstep{Scenario (i)}
If $\ta \in \bfH$ then $a_\infty \in \ell^1 \dot H^1$. Indeed, $\rd^{\ell} a_{\infty, \ell} \in \ell^{1} L^{2}$ and $\rd_{j} a_{\infty, k} - \rd_{k} a_{\infty, j} = - [a_{\infty, j}, a_{\infty, k}]$, so that
\begin{equation*}
	\lap_{a_{\infty}} a_{\infty, j} = \rd_{k} \rd^{\ell} a_{\infty, \ell} + [a_{\infty}^{\ell}, \rd_{k} a_{\infty, \ell}].
\end{equation*}
Since the RHS belongs to $\ell^{1} \dot{H}^{-1}$, it follows that $a_{\infty} \in \ell^{1} \dot{H}^{1}$ by Theorem~\ref{t:deltaA} (plus a simple interpolation).

It follows that the sequence $O^n$ converges in $L^\infty$, and thus the limit
\[
O = \lim_{n \to \infty} O^n
\]
exists. To establish $C^1$ dependence on the data, we consider a smooth 
one parameter family $\ta^h$ of initial data in $\bfH$, which in turn generates a smooth  
one parameter family $a^h_\infty$ in $\ell^1 \dot H^1$. To see how $O$ depends on $h$ 
we compute
\[
\rd_{x} (O^{-1} \rd_{h} O) = \rd_{h} a_{\infty}^{h} - [a_{\infty}^{h}, O^{-1} \rd_{h} O],
\]
or equivalently, using the $ a^h_\infty$ connection $\covD = \covD[a^{h}_{\infty}]$,
\[
\covD_{x}  (O^{-1} \rd_{h} O)= \partial_h a^h_\infty  \in \ell^1 \dot H^1.
\]
Taking the divergence this yields the elliptic equation
\begin{equation}\label{delta-Oh}
\Delta_{a^h_\infty} (O^{-1} \rd_{h} O)=  \mathrm{div}_{a_{\infty}^{h}} \partial_h a^h_\infty \in \ell^1 L^2,
\end{equation}
and thus, by Theorem~\ref{t:deltaA} and Corollary~\ref{c:global}, 
\[
\|  O^{-1} \rd_{h} O \|_{\ell^1 \dot H^2} \lesssim \| \partial_h a_{\infty}^h\|_{\ell^{1} \dot{H}^{1}} \aleq \nrm{\rd_{h} \ta^{h}}_{\bfH}.
\]
This is the only step in the argument where the extra $\ell^1$ summability is used.

Now we are ready to establish the $C^1$ dependence of $a = \Cal(\ta)$ on $\ta$.
Indeed, we have 
\begin{align*}
\rd_{h} a
= \rd_{h} (Ad(O) (\ta - O^{-1} \rd_{x} O))
= Ad(O)(\rd_{h} (\ta - a_{\infty}) + [O^{-1} \rd_{h} O, \ta - a_{\infty}]).
\end{align*}
The $L^4$ bound
\[
\| \partial_h a\|_{L^4} \lesssim \| \partial_h \ta\|_{\dot H^1}
\]
follows by the unitarity of $Ad(O)$, and the $\dot H^1$ bound
\[
\| \partial_h a\|_{\dot H^1} \lesssim \| \partial_h \ta\|_{\dot H^1}
\]
is easily obtained after one additional differentiation.

It remains to establish the $\ell^1 L^2$ bound for $\text{div} \partial_h a$. After 
peeling off some good terms, this reduces to the mapping property
\[
\ell^1 L^2 \ni b \to Ob O^{-1} \in \ell^1 L^2.
\]
This follows by interpolation from the easier $\dot H^1$ similar mapping property
and its dual $\dot H^{-1}$ bound.

\pfstep{Scenario (ii)} 
If $\ta \in \dot H^1$ then $a_\infty \in \dot H^1$ and the sequence 
$O^n$ is no longer guaranteed to converge pointwise.
However, we still have
\begin{equation*}
	\nrm{[O^{n} (O^{m})^{-1}]_{;x}}_{\dot{H}^{1}}^{2} 
	\aleq \nrm{a_{\infty}^{n} - a_{\infty}^{m}}_{\dot{H}^{1}} \aleq \sum_{k = n}^{m} c_{k}^{2},
\end{equation*}
where $c_{k}$ is a $(-1, 1)$ frequency envelope for $\ta$ in $\dot{H}^{1}$. We claim that there exists
a sequence $P^{n} \in \G$ such that $\tilde{O}^{n} = P^{n} O^{n}$ satisfy
\[
\int_Q d(\tilde O^n, \tilde O^m)^2 dx   \lesssim_{Q} \sum_{k=\min\set{n, m}}^{\infty} c_k^2 \to 0 \quad \hbox{ as } n, m \to \infty
\]
for any cube $Q$. Assuming this claim, we obtain a limit 
\[
O = \lim_{n \to \infty} \tilde{O}^{n} \qquad \text{ in } L^2_{loc}
\]
so that $O_{;x} \in \dot H^1$ and $O^{-1} \rd_{x} O = a_\infty$.  By Lebesgue's
dominated convergence theorem this suffices in order to guarantee that
the limit 
\[
 a = \lim_{n \to \infty} a^n
\]
exists in $\dot H^1$, as desired.

It remains to verify the claim. For this purpose, we take the extrinsic viewpoint by fixing a matrix group representation $\G \hookrightarrow O(N) \subseteq \bbR^{N \times N}$ (which exists since $\G$ is compact) and viewing each $O^{n}$ as a $\bbR^{N \times N}$-valued map. Fix a cube $Q$, and consider the average
\begin{equation*}
	\tilde{O}^{nm} = \frac{1}{\abs{Q}} \int_{Q} O^{n} (O^{m})^{-1} \, d x .
\end{equation*}
By Poincar\'e's inequality, it follows that
\begin{align*}
	d(\G, \tilde{O}^{nm})^{2}
	\aleq & \frac{1}{\abs{Q}} \int \abs{O^{n} (O^{m})^{-1} - \tilde{O}^{nm}}^{2} \, d x \\
	\aleq & \left( \int_{Q} \abs{[O^{n} (O^{m})^{-1}]_{;x}}^{4} \, d x \right)^{\frac{1}{2}} 
	\aleq  \sum_{k=n}^{m} c_{k}^{2}.
\end{align*}
This implies that, for sufficiently large $m, n$, $\tilde{O}^{nm}$ is close to $\G$, and its nearest-point projection $O^{nm} \in \G$ satisfies
\begin{equation*}
	\frac{1}{\abs{Q}} \int_{Q} d(O^{mn}, O^{n} (O^{m})^{-1})^{2} \, d x \aleq \sum_{k=n}^{m} c_{k}^{2}.
\end{equation*}
A similar argument using the nearest-point projection and Poincar\'e's inequality also shows that $O^{n m}$ approximately satisfy the cocycle condition, i.e.,
\begin{equation*}
	d(O^{n\ell}, O^{n m} O^{m \ell})^{2} \aleq d(\tilde{O}^{n\ell}, \tilde{O}^{n m} \tilde{O}^{m \ell}) + \sum_{k=\min\set{n, m, \ell}}^{\max\set{n, m, \ell}} c_{k}^{2}
	\aleq \sum_{k=\min\set{n, m, \ell}}^{\max\set{n, m, \ell}} c_{k}^{2}.
\end{equation*}
Now we define $P^{n}$ by the following inductive procedure: Define $P^{1}  = \lim_{m \to \infty} O^{1 n^{(1)}_{m}}$ for some convergent subsequence of $O^{1 m}$, $P^{2} = \lim_{m \to \infty} O^{2 n^{(2)}_{m}}$ for some further subsequence of $O^{2 n^{(1)}_{m}}$, etc. Then 
\begin{equation*}
	d((P^{n})^{-1} P^{m}, O^{nm})^{2} \aleq \sum_{k = \min\set{n, m}}^{\infty} c_{k}^{2}
\end{equation*}
so that $\tilde O^{n} = P^{n} O^{n}$ satisfy the claimed bound.

Finally, a similar argument yields the continuous dependence of $a$
on $\ta$ with respect to the metric in the proposition. We omit the details. \qedhere
\end{proof}

To understand the higher regularity of the map $\ta \to a = \Cal(\ta)$ we also 
establish frequency envelope bounds. 
We begin with a technical bound for $Ad(O)$. We introduce the notation
\begin{equation*}
c^{p}_{jk} = \nrm{P_{j} Ad(O) P_{k}}_{L^{p} \to L^{p}}.
\end{equation*}
where $2 \leq p < \infty$. When $p = 2$, we will often omit $p$ and simply write
\begin{equation*}
	c_{jk} = c^{2}_{jk}.
\end{equation*}

\begin{lemma}\label{l:O-conj}
\begin{enumerate}
\item Assume that $O^{-1} \rd_{x} O \in \dot H^1$. Then
\begin{equation*}
	c_{jk}^{p} \aleq_{\nrm{O^{-1} \rd_{x} O}_{\dot{H}^{1}}} 2^{-\frac{4}{p} \abs{j - k}}.
\end{equation*}

\item Let $c_k$ be a $(- 1,S)$ frequency envelope for $O^{-1} \rd_{x} O$ in $\dot H^1$.
Then
\begin{equation*}
	c_{jk}^{p} \aleq_{\nrm{O^{-1} \rd_{x} O}_{\dot{H}^{1}}} 2^{\frac{4}{p} (k-j)} c_{j}, \qquad k < j - C.
\end{equation*}
\end{enumerate}
\end{lemma}

\begin{proof}
In this proof, we suppress the dependence of implicit constants on $\nrm{O^{-1} \rd_{x} O}_{\dot{H}^{1}}$.

For part (1), by interpolation, it suffices to only consider the cases $p = \infty$ and $p = 2$. The first case is trivial. In the case $p = 2$, by duality we consider $j \geq  k$ and by scaling we set $k = 0$.
Differentiating we have
\[
\partial Ad(O) a = Ad(O)( \partial a + [O^{-1} \rd_{x} O, a]) 
\]
which yields an $\dot H^1$ bound and thus $2^{-|j-k|}$. A second differentiation gives the bound in part (1).

For part (2) we set $j = 0$.  For $a$ localized at frequency $2^k$ with $\nrm{a}_{L^{p}} \aleq 1$, using Littlewood--Paley trichotomy we can estimate
\begin{equation} \label{eq:Ox-Lp}
\nrm{P_{\ell} [O^{-1} \rd_{x} O, a]}_{L^{p}}
\aleq \left\{
\begin{array}{cl}
2^{\frac{4}{p} k} 2^{(1-\frac{4}{p}) \ell} c_{\ell} & \ell \geq k+5, \\
2^{k}  & k-5 \leq \ell < k+5, \\
2^{2 \ell - k} c_{k}  & \ell < k-5.
\end{array} 
\right.
\end{equation}
For $\ell > k+5$, we also note that
\begin{equation} \label{eq:Ox-L2}
\nrm{P_{\ell} [O^{-1} \rd_{x} O, a]}_{L^{2}} \aleq 2^{\frac{4}{p} k} 2^{-\ell} c_{\ell} .
\end{equation}
Hence for $k < j - C$, we can estimate
\begin{equation*}
	2^{j} c^{p}_{j k} \aleq 2^{k} c^{p}_{jk} + c_{k} \sum_{\ell < k} 2^{2 \ell - k} c^{p}_{j \ell} + \sum_{\ell : k < \ell < j - C} 2^{\frac{4}{p}(k-\ell)} 2^{\ell} c^{p}_{j \ell} + \sum_{\ell > j - C} 2^{(2-\frac{4}{p})j} c^{2}_{j \ell} 2^{\frac{4}{p} k} 2^{-\ell} c_{\ell}
\end{equation*}
For $C$ sufficiently large, the first term on the RHS can be absorbed into the LHS. Moreover, since $c_{\ell}$ grows slowly to the right, we can easily estimate the last term on the RHS to obtain
\begin{equation*}
	c^{p}_{jk} \aleq c_{k} \sum_{\ell < k} 2^{2 \ell - k -j} c^{p}_{j \ell} + \sum_{\ell: k < j - C} 2^{\frac{4}{p}(k - \ell)} 2^{\ell - j} c^{p}_{j \ell} + 2^{C} 2^{\frac{4}{p}(k -j)} c_{j}.
\end{equation*}
We claim that $c^{p}_{jk} \aleq_{C} 2^{\frac{4}{p}(k-j)} c_{j}$.
Indeed, choosing $C$ large enough, reiterating this bound yields strictly smaller contributions unless $c_j \approx 1$, in which case we use the bound in part (1). \qedhere
\end{proof}

We may now prove a frequency envelope bound for $Cal(\ta)$.
\begin{proposition}\label{p:Cal-fe-ah}
Let $\ta$  be a connection in $\dot H^1$ satisfying $\hM(a) \leq \hM < \infty$ and $\nrm{\ta}_{\dot{H}^{1}} \leq M_{1}$, with $(-1,S)$ frequency envelope $c_k$.
Then $a= \Cal(\ta)$ also has frequency envelope $c_k$ in $\dot H^1$, with the bound
\begin{equation*}
	\nrm{P_{j} a}_{\dot{H}^{1}} \aleq_{\hM, M_{1}} c_{j}.
\end{equation*}

\end{proposition}
\begin{proof}
Here, we suppress the dependence of implicit constants on $\hM$ and $M_{1}$.
Let $O$ be the generator of the corresponding gauge transformation, $O^{-1} \rd_{x} O = a_\infty$. Then 
\[
a = \Cal(\ta) = Ad(O) \ta - O_{;x} = Ad(O) (\ta- a_\infty).
\]
By Theorem~\ref{t:ymhf-l3-fe}, $c_{k}$ is a frequency envelope for $O^{-1} \rd_{x} O$ in $\dot{H}^{1}$.
Using Lemma~\ref{l:O-conj}, we compute
\[
\| P_j a\|_{\dot H^1} \lesssim \sum_{k} 2^{j-k} c_k c_{kj} \lesssim c_j + \sum_{k < j-C} c_k c_j 2^{k-j} \lesssim c_j 
\]
as needed. \qedhere
\end{proof}

We now consider bounds for the linearization of $\Cal$.

\begin{proposition}\label{p:fe-ah}
  Let $\ta^{(h)}$ be a $C^1$ family in $\bfH$ satisfying the uniform bounds $\hM(\ta^{(h)}) \leq \hM < \infty$ and $\nrm{\ta^{(h)}}_{\bfH} \leq M_{1}$. Let $O^{(h)}$ the
  corresponding gauge transformations into caloric gauge, normalized so that $O^{(h)}(\infty) = Id$.  Assume that $c_k$ is a
  $(-1, S)$ frequency envelope for $\ta^{(0)}$ in $\dot H^1$ ,  $d_k$
  is a $1$-compatible $(-1, S)$ frequency envelope for $\partial_h
  \ta^{(0)}$ in $\dot H^1$ and that
\[
d_k' = c_k d_k^{[1]} + d_k c_k^{[1]}
\]
is a  $(-1, S)$ frequency envelope for $\rd^{\ell} \partial_h \ta_{\ell}^{(0)}$ in $L^2$.
Then we have
\begin{equation}
\|P_k \partial_h a^{(0)} \|_{\dot H^1} \lesssim d_k+ c_k (c \cdot d)_{\leq k}.
\end{equation}
\end{proposition}

\begin{proof}
The expression $\partial_h a^{(h)}$ is given by 
\[
\partial_h a^{(h)} =  Ad(O^{(h)}) \partial_h \ta^{(h)} - \covD[a^{(h)}] a_h 
\]
where
\[
a_h = O^{(h)}_{;h}.
\]
The $O^{(h)}$ conjugation is again harmless by Lemma~\ref{l:O-conj}; precisely, this gives
\[
\|P_k (Ad(O^{(0)}) \partial_h \ta^{(0)})\|_{\dot H^1} \lesssim  \sum_{j} 2^{k-j} d_j c_{jk} \lesssim d_j + c_k \sum_{j < k-C} d_j  2^{j-k} 
\lesssim d_k.
\]
where at the last step we have used the compatibility condition.

Hence in order to estimate $\partial_h a^{(h)} $ we need to understand $a_h$. 
First, we have
\begin{equation}
\|P_k ((O^{(0)})^{-1} \rd_{h} O^{(0)}) \|_{\dot H^2} \lesssim d_k'.
\end{equation}
This is obtained directly from the elliptic equation \eqref{delta-Oh},
using the bounds for $\div \partial_h a_\infty$ provided by Theorem~\ref{t:ymhf-l3-fe}. Then since 
\begin{equation*}
a_{h} = O_{;h}^{(0)} = Ad(O^{(0)}) (O^{(0)})^{-1} \rd_{h} O^{(0)}, 
\end{equation*}
by a similar argument as before using Lemma~\ref{l:O-conj}, we obtain
\begin{equation}
\|P_k a_{h} \|_{\dot H^2} \lesssim d_k' + c_{k} \sum_{j < k} 2^{j - k} d_{j}' \aleq d_{k}'.
\end{equation}
where the last inequality follows from compatibility, i.e.,
\begin{equation*}
c_{k} \sum_{j < k} 2^{j - k} d_{j}'
\aleq c_{k} \sum_{j < k} 2^{j - k} d_{j} \aleq d'_{k}.
\end{equation*}
To complete the proof of the proposition we estimate
\[
\| \covD a_h\|_{\dot H^1} \lesssim  d_k' + \sum_{j < k} d_j' c_k \lesssim   d_k' + c_k (c \cdot d)_{\leq k}.
\]
Finally, we note that by the compatibility condition we have 
\[
 d_k'  \lesssim d_k. \qedhere
\]
\end{proof}

As a corollary of Proposition~\ref{p:fe-ah}, we obtain
\begin{corollary}
The map $\ta \to a = \Cal(\ta)$ for $\dot{H}^{1}$ connections with $\hM(a) \leq \hM < \infty$  is $C^1$ 
 in  $\bfH \cap H^\sigma$ ($\sigma > 1$), with a bound depending on $\hM$.
\end{corollary}
We omit the straightforward proof, which is similar to Scenario (i) in the proof of Proposition~\ref{p:cal-a}, but now taking into account the frequency envelope bound.

\subsection{The caloric manifold and its tangent space}
We denote by $\calC$ the set of all caloric gauge connections, and define
\begin{equation*}
	\calC_{\hM} = \set{a \in \dot{H}^{1} : \hbox{$a$ is a caloric gauge connection with } \hM(a) \leq \hM}.
\end{equation*}
Note that $\calC = \cup_{\hM > 0} \calC_{\hM}$.

We seek to describe $\calC$ as a $C^1$ infinite
dimensional manifold.  Given the results above, it is natural to seek
to do this in the $\bfH$ topology.  As a first step, we show that
$\calC$ connections are indeed in $\bfH$, and satisfy some nonlinear
form of the Coulomb gauge condition.

\begin{proposition} \label{p:cal-in-bfH}
 For $a \in \calC_{\hM}$ with energy $\leq \calE$, its caloric gauge Yang--Mills heat flow satisfies the bound
\begin{equation} \label{eq:cal-in-bfH}
\|A\|_{\dot H^1} + \| \partial^{\ell} A_{\ell} \|_{\ell^1 L^2 \cap W^{1,\frac43}} \lesssim_{\hM, \calE} 1.
\end{equation}
\end{proposition}

\begin{proof}
The $\dot H^1$ bound  is a direct consequence of Lemma~\ref{l:h1} applied on an infinite interval.
Then we use Lemma~\ref{l:da} for the divergence of $A$.
\end{proof}

Now we can prove the following:

\begin{theorem}
The set $\calC$ is an infinite dimensional $C^1$ submanifold of the Banach space $\bfH$.
\end{theorem}
The fact that we only get $C^1$ may well be an artifact of the construction; 
the difficulty is that the map $a \to a_\infty$ is $C^1$ but possibly no better.

\begin{proof}
  Consider a caloric gauge connection $a \in \calC_{\hM}$ with energy
  $\calE$.  We will show that there exists $\epsilon= \epsilon(\calE, \hM) > 0$ so
  that $\calC \cap B^{\bfH}_{\eps} (a)$ can be parametrized with a
  $C^1$ local chart. Here, $B^{\bfH}_{\eps}(a)$ is the ball of radius $\eps$ around $a$ in the $\bfH$ topology. 
  For the purpose of this proof, all covariant differentiations will be with respect to the connection $a$.
  
To achieve our goal we begin with the closed affine subspace of $\dot H^1$
\[
\mathcal B = \{ a+ b \in \dot{H}^{1} : (\covD[a])^\ell b_\ell = 0  \},
\]
which is in some sense a local Coulomb gauge adapted to the connection $a$. 
Then we  consider the caloric gauge representations of elements of $\calB$ near $a_0$,
\begin{equation}\label{chart}
\calB \ni a + b \to \Cal(a+b) \in \calC.
\end{equation}
We will prove that this map represents a $C^1$ parametrization of $\calC$ near $a$.

\pfstep{Step~1: Proof of regularity}
We first claim that this map is $C^1$ in a neighbourhood of $a$. By the previous results 
we know that
\[
\| \partial^\ell a_{\ell} \|_{\ell^1 L^2} \lesssim_{\hM, \calE} 1.
\]
 Then the map
\[
\calB \ni a+ b \to \partial^{\ell} (a +b)_{\ell} \in \ell^1 L^2
\]
is $C^1$. Hence the $C^1$ regularity of $\Cal$ restricted to $\mathcal B$ follows.

\pfstep{Step~2: Proof of local invertibility} 
In order to view this map as a local
chart on $\calC$ we need to show that it is invertible in the $\dot H^1$
topology. It suffices to show that its differential at $a$ is
bounded from below. To compute its differential we denote by $
O(b)$ the associated gauge transformation, normalized by $O(\infty) =
Id$. Then $\Cal$ has the form
\[
 \Cal(a+b)_j = O(b) (a +b)_j O(b)^{-1} - O(b)_{;j} ,
\]
and its differential at $a$ has the form
\[
( d \Cal(a) b)_j = b_j + [dO(0) b, a_{j}] - \partial_j  (dO(0) b)  
\]
Hence
\[
  (d \Cal(a) b)_j = b_j - (\covD[a])_{j}c, \qquad  c = dO(0) b.
\]
Hence we need to prove the bound
\begin{equation} \label{transverse1}
\| b - \covD[a] c\|_{\bfH} \gtrsim \|b\|_{\bfH} .
\end{equation}
Since $(\covD[a])^{\ell} b_{\ell} = 0$, it suffices to show that
\begin{equation}
\label{delta-1}
\| \covD[a] c\|_{\ell^1 \dot H^1} \lesssim \| \Delta_{a} c\|_{\ell^1 L^2}.
\end{equation}
But this is a consequence of Theorem~\ref{t:deltaA}.

\pfstep{Step~3: Proof of local surjectivity}  
Here we show that  our map \eqref{chart} is locally surjective near $a$.
Thus consider another caloric connection $a_1$ which is close to $a$. Then we have
\[
\| a - a_1\|_{\dot H^1} \ll 1, \qquad \| \partial^k (a - a_1)_k\|_{\ell^1 L^2} \ll 1 .
\]
At this point we use only these bounds, forgetting that $a_1$ is
caloric. We consider the straight line joining $a$ and $a_1$, denoted
by $a(h)$ with $h \in [0,1]$.  Along this line we construct a family
of gauge transformations $O(h)$, with $O(0) = Id$, which move this
segment into $\calB$. We need to verify the relation
\[
\covD^k ( Ad(O(h)) a(h)_k  - O(h)_{;k}) = \covD^k a_k(0).
\]
Here and below, all covariant differentiations are taken with respect
to the fixed connection $a$.  Equivalently, we can differentiate with
respect to $h$ to rewrite this condition as
\begin{align*}
0 = & \rd_{h} \covD^{k} (Ad(O) a(h)_{k} - O_{;k}) \\
= & \covD^{k} ([O_{;h}, Ad(O) a(h)_{k}] + Ad(O)(a_{1} - a)_{k}) - \rd_{k} O_{;h} + [O_{;k}, O_{;h}])
\end{align*}
We view this as an equation for $O_{;h}$:
\begin{align*}
	\covD^{k} \covD_{k} O_{;h} 
	= &\covD^{k} (Ad(O)(a_{1} - a)_{k})\\
	& + \covD^{k} (h [O_{;h}, Ad(O) (a_{1} -a)_{k}] - [O_{;h}, O_{;x}] - [O_{;h}, (Ad(O) - 1)a_{k}]) 
\end{align*}
Here we view the right hand side terms as perturbative, and integrate $O_{;h}$ in $h$ to find $O$
in the space $\ell^1 \dot H^2$. For this, we use the smallness of $a_1-a$ in $\bfH$, of $O_{;k}$ in $\ell^1 \dot H^1$, as well as of $Ad(O) - 1$ in $\bfH \to \bfH$. We skip the straightforward details. \qedhere

\end{proof}

We now take a closer look at the tangent space to $\calC$.
Since the caloric manifold $\calC$ is a $C^1$ submanifold of $\bfH$,
its tangent space $T_a\calC$ is naturally defined as a closed subspace
of $\bfH$. Precisely, given $b \in \bfH$, we denote by $B$ the
solution to the linearized Yang--Mills heat flow (i.e., \eqref{eq:cYMHF-lin}) in the local caloric gauge $a_{s} = 0$ (i.e., \eqref{lin-heat}), which we recall here
\begin{equation}\label{lin-heat-re}
\partial_s B_i = \covD^j ( \covD_j B_i - \covD_i B_j)  +  [B^j,F_{ji}] , \qquad B_i(0) = b_i.
\end{equation}
This is a well-posed flow in $\bfH$, with bounds similar to the bounds
for the Yang--Mills heat flow.  Further, the limit $B(\infty)$ exists in $\bfH$ and is curl-free.  
Then the tangent space $T \calC$ can be defined as
\begin{equation}\label{TC}
T_a \calC = \{ b \in \bfH:  B(\infty) = 0\} .
\end{equation}

For our purposes here we need to look at a larger tangent space,
namely with respect to the $L^2$ topology. We denote it by $T_a^{L^2}
\calC$, and it is defined as the closure of $T_a\calC$ in the $L^2$
topology, or equivalently as
\[
T_a^{L^2}  \calC = \{ b \in L^2 : \ B(\infty) = 0\}.
\]
Due to the linearized gauge invariance \eqref{eq:cYMHF-lin-gt} with $\rd_{s} a_{0} = 0$, it is clear that 
\[
T_a^{L^2}  \calC \cap \covD[a] \dot H^1 = 0.
\]
Our next result shows that these two closed subspaces of $L^2$ are in effect transversal:

\begin{proposition}\label{p:transverse}
Let $a \in \calC_{\hM} \subseteq \calC$ with energy $\calE$. Then the following statements hold:
\begin{enumerate}
\item Any function $w \in L^2$ admits a unique representation
\[
w = b - \covD[a] c, \qquad b \in T_a^{L^2}  \calC, \quad c \in \dot H^1
\]
and the following estimate holds:
\begin{equation}
\| b \|_{H^{\sgm}} + \|c \|_{\dot H^{\sgm+1}} \lesssim_{\hM, \calE} \| w\|_{\dot{H}^{\sgm}}, \qquad -1 <  \sgm <  1.
\end{equation}
Correspondingly, if $\sgm = 1$ then we have
\begin{equation}
\| b \|_{\bfH} + \|c \|_{\ell^1 \dot H^{2}} \lesssim_{\hM, \calE} \| w\|_{\bfH}.
\end{equation}

\item Furthermore, we can represent $c$ as 
\[
c = -\Delta^{-1} \nabla \cdot w +  L(a,w) ,
\]
where $L(a, w)$ is at least quadratic, and satisfies better bounds 
\begin{equation}
 \|L(a,w)\|_{\ell^1 \dot{H}^{\sgm+1}} \lesssim_{\hM, \calE} \|w\|_{\dot{H}^\sgm}, \qquad -1 < \sgm < 1.
\end{equation}
\end{enumerate}
\end{proposition}
The above decomposition is in effect a nonlinear div-curl decomposition.
The map from $w$ to $b$ can be viewed as a canonical projection onto $T_a^{L^2}  \calC$.
In the sequel we will denote this projection by 
\begin{equation}\label{def-pi}
b = \Pi_a w.
\end{equation}
\begin{proof}
\pfstep{Proof of (1)}
The case of $\bfH$ is essentially\footnote{Here, the $L^{2}$ frequency envelope $d'_{k}$ for $\rd^{\ell} w_{\ell}$ need not be related to $c_{k}$ and $d_{k}$, except that it must be $1$-compatible with $c_{k}$. Then the proof of Proposition~\ref{p:fe-ah} goes through, with the only exception of the last inequality $d'_{k} \aleq d_{k}$.} Proposition~\ref{p:fe-ah} with $O^{(0)} = Id$, $w = \rd_{h} \ta^{(0)}$, $\Pi_{a} w = \rd_{h} a^{(0)}$ and $c = - O_{;h}^{(0)}$.
In the case of $\dot{H}^{\sgm}$ with $-1 < \sgm < 1$, we solve the linearized equation \eqref{lin-heat-re} with $w$ as
  initial data; the solution is denoted by $W$.  By Corollary~\ref{c:global} (see also Lemma~\ref{l:global}), the map $w \to
  W(\infty)$ is bounded in $\dot{H}^{\sgm}$, and $W(\infty)$ has the form
  $W(\infty) = - \nabla c$, with $c \in \dot{H}^{\sgm+1}$. Then we simply set
\[
b = w + \covD[a] c \in T_a^{L^2}  \calC.
\]

\pfstep{Proof of (2)} 
We peel off the leading part of $c$, namely 
\[
c_0 = -\Delta^{-1} \rd^{\ell} w_{\ell} .
\]
Then $W + \covD[A] c_0$ still solves the linearized heat flow equation,
and further, its data satisfies the better (schematic) equation for the divergence
\[
\covD[a] \cdot (w + \covD[a] c_0)  
= a \cdot w + a \cdot a \cdot \Delta^{-1} \rd w + \rd a \cdot  \Delta^{-1} \rd w,
\]
which implies that $\covD[a](w + \covD_{a} c_{0}) \in \ell^{1} \dot{H}^{\sgm-1}$.
Then we propagate this regularity through the linearized flow, as in Lemma~\ref{l:db}.
\qedhere
\end{proof}

\begin{remark} \label{r:proj-a0}
Now we discuss an alternative way to derive bounds on $c$ (and therefore $b$) relying on the dynamic Yang--Mills heat flow instead of appealing to bounds for \eqref{lin-heat-re} (cf. Sections~\ref{subsec:dymhf} and \ref{subsec:dymhf-re}). This is a variant of the proof of Lemma~\ref{l:global}.

Let $w \in L^{2}$. We begin from the existence of a decomposition $w_{j} = b_{j} - \covD[a] c$ with $b \in T_{a}^{L^{2}} \calC$ and $c \in \dot{H}^{1}$; our aim is then to derive a formula for $c$, which can be analyzed without reference to \eqref{lin-heat-re}. Assume that $b_{j} = \rd_{t} a_{j}(t)$, where $(-\eps_{0}, \eps_{0}) \ni t \mapsto a_{j}(t) \in \calC$ is a $C^{1}$ curve in $\calC$ for some $\eps_{0} > 0$. As in Section~\ref{subsec:dymhf-re}, we introduce $A_{j}(t=0, x, s)$ and $A_{0}(t=0, x, s)$ by solving $F_{s j} = \covD^{\ell} F_{\ell j}$ and $F_{s 0} = \covD^{\ell} F_{\ell 0}$ on $t =0$ with $A_{j}(t=0, x, s=0) = a_{j}(x)$ and $A_{0}(t=0, x, s=0) = c(x)$. Observe that $F_{0j} (t=0, x, s) = \left(B_{j} - \covD A_{0}\right)(t=0, x, s)$, where $B_{j}$ is the solution to \eqref{lin-heat-re} with $B_{j}(s=0) = b_{j}$. Since $F_{0j}$ solves the linear covariant (nondegenerate) parabolic equation \eqref{eq:YMHF-Fij+} with $F_{0j}(t=0, \cdot, s=0) = b_{j} - \covD[a] c \in L^{2}$, by the $L^{2}$ theory in Theorem~\ref{t:heatA}, it follows that $\lim_{s \to \infty} F_{0j} = 0$ in $L^{2}$ on $\set{t = 0}$. Moreover, by the definition of $T_{a}^{L^{2}} \calC$, it follows that $\lim_{s \to \infty} B_{j} = 0$ in $L^{2}$. Hence, we see that $\lim_{s \to \infty} \covD A_{0} = 0$ in $L^{2}$ on $\set{t = 0}$. By the diamagnetic inequality as in the proof of Proposition~\ref{t:deltaA}, as well as the softer facts that $A_{0}(s=0) \in \dot{H}^{1}$ and that $F_{s0} = \covD^{\ell} F_{\ell 0}$ can be viewed as a parabolic equation for $A_{0}$, it then follows that 
\begin{equation*}
	\lim_{s \to \infty} A_{0} = 0
\end{equation*}
in $\dot{H}^{1}$ on $\set{t = 0}$. Using $\rd_{s} A_{0} = F_{s0} = \covD^{\ell} F_{\ell 0}$ (thank to $A_{s} = 0$), we arrive at
\begin{equation} \label{eq:proj-a0}
	c = A_{0}(s=0) = - \int_{0}^{\infty} \covD^{\ell} F_{\ell 0} \, \ud s,
\end{equation}
in $\dot{H}^{1}$ and on $\set{t = 0}$, which is the desired representation formula.

Observe that, while we assumed the existence of $c$ to derive \eqref{eq:proj-a0}, the right-hand side of \eqref{eq:proj-a0} can be \emph{computed in a manner that depends only on $a_{j}$ and $w_{j}$}. Indeed, $a_{j}$ already determines $A_{j}(t=0, x, s)$ in the global caloric gauge, and $F_{0j}$ satisfies the same covariant linear parabolic equation \eqref{eq:YMHF-Fij+} on $\set{t = 0}$ with $F_{0j}(s=0) = w_{j}$ \emph{regardless of $c$}. In particular, by \eqref{eq:proj-a0} and the results in Section~\ref{s:caloric-loc}, we may obtain estimates on $c$ such as $\nrm{c}_{\dot{H}^{\sgm+1}} \aleq_{\hM, \calE} \nrm{w}_{\dot{H}^{\sgm}}$ for $-1 < \sgm < 1$ etc. The advantage of this approach is that we have an explicit formula \eqref{eq:proj-a0}, which will be useful later.

We conclude by noting the following consequences of the preceding argument, which give dynamic-Yang--Mills-heat-flow characterizations of a curve in $\calC$ and $T_{a}^{L^{2}} \calC$. 
\begin{enumerate}
\item Let $A_{\mu} \ud x^{\mu}$ be a dynamic Yang--Mills heat flow on $I \times \bbR^{4} \times [0, \infty)$ with $A_{\mu}(t, \cdot, s = 0) \in \dot{H}^{1}$. Then $A_{j} \, \ud x^{j} \in \calC$ for each $t \in I$ if and only if $A_{s} = 0$ everywhere and $A_{j}(t, \cdot s=\infty) = 0$, $A_{0}(t, \cdot, s=\infty) = 0$ (in $\dot{H}^{1}$) on $t \in I$. 
\item Let $b_{j} \in L^{2}$ and $a_{j} \in \calC$. Then $b_{j} \in T_{a}^{L^{2}} \calC$ if and only if there exists a dynamic Yang--Mills heat flow $A_{\mu}$ on $\set{t = 0} \times \bbR^{4} \times [0, \infty)$ with $A_{j}(s=0) = a_{j}$, $F_{0j}(s=0) = b_{j}$ satisfying the double boundary condition (on $\set{t = 0}$):
\begin{equation}\label{TC-A0}
A_0(s=0) = A_0(s=\infty) = 0.
\end{equation}
\end{enumerate}
\end{remark}

One consequence of Proposition~\ref{p:transverse} is that if $b \in T_a^{L^2}
\calC$ then $\rd^{\ell} b_{\ell}$ has better regularity; this is a linearized analogue of Proposition~\ref{p:cal-in-bfH}:
\begin{corollary}
Let $a \in \calC_{\hM} \subseteq \calC$ with energy $\calE$, and $b \in \dot H^\sgm$ a corresponding 
linearized caloric data set. Then we have
\begin{equation}
\| \partial^{\ell} b_{\ell} \|_{\ell^1 \dot H^{\sgm-1}} \lesssim_{\hM, \calE} \| b\|_{\dot H^{\sgm}} \qquad -1 < \sgm < 1.
\end{equation}
\end{corollary}

We note that the endpoint case $\sgm = 1$ is somewhat different, in that we no longer have a bounded 
projection on the caloric tangent space. However, the bound for caloric tangent state survives, as 
a consequence of Lemma~\ref{l:db} (see also Proposition~{p:fe-ah}).

We also have a frequency envelope bound for $\Pi_{a}$:
\begin{lemma} \label{l:pi-l2}
Let $a \in \calC_{\hM} \subseteq \calC$ with energy $\calE$, and with $\dot{H}^{1}$ $(-1, S)$-frequency envelope $c_{k}$.
\begin{enumerate}
\item Let $w \in L^{2}$ with a $(-1, S)$ frequency envelope $d_{k}$, which is $1$-compatible with $c_{k}$. Then we have
\begin{equation*}
	\nrm{P_{k} \Pi_{a} w}_{L^{2}} \aleq_{\hM, \calE} d_{k}.
\end{equation*} 
\item Alternatively, let $w \in \bfH$, and let $d_{k}$, respectively $d'_{k}$, be $1$-compatible $(-1, S)$ frequency envelopes for $w$ in $\dot{H}^{1}$, respectively $\rd^{\ell} w_{\ell}$ in $L^{2}$. Then we have
\begin{equation*}
	\nrm{P_{k} \Pi_{a} w}_{\dot{H}^{1}} \aleq_{\hM, \calE} d_{k} + d'_{k} + c_{k} d'_{\leq k}.
\end{equation*}
\end{enumerate}
\end{lemma}
Note that the $1$-compatibility condition rules out application of (1) to $w \in \dot{H}^{1}$, in which case (2) must be used.
\begin{proof}
\pfstep{Proof of (1)}
In anticipation of ensuing arguments, we give a proof using the idea in Remark~\ref{r:proj-a0}. It is very similar to the proof of Lemma~\ref{l:lin-fe}, Step~1, except we integrate $\covD^{\ell} F_{\ell 0}$ from infinity to obtain $a_{0}$. 

More precisely, note that
\begin{equation*}
	\Pi_{a} w = w - \covD[a] a_{0}
\end{equation*}
where $a_{0}$ is given by the formula \eqref{eq:proj-a0}, $A$ is the Yang--Mills heat flow of $a$, and $F_{0j}$ solves
\begin{equation*}
	(\rd_{s} - \lap_{A} - 2 ad(F)) F_{0 x} =  0, \qquad F_{0j}(0) = f_{0j} = w_{j}.
\end{equation*}
As in the proof of Lemma~\ref{l:lin-fe}, we have
\begin{equation*}
	\nrm{P_{k} \covD^{\ell} F_{\ell 0}}_{L^{1} \dot{H}^{1}} \aleq d_{k}.
\end{equation*}
and thus after integration,
\begin{equation*}
	\nrm{P_{k} a_{0}}_{\dot{H}^{1}} + \nrm{P_{k} \covD a_{0}}_{L^{2}} \aleq d_{k}.
\end{equation*}
Note that $1$-compatibility is crucial to get the control of $\covD a_{0}$. The proof (1) is complete. 

\pfstep{Proof of (2)} This statement can be read off from the proof of Proposition~\ref{p:fe-ah}, where $O^{(0)} = Id$.
\qedhere
\end{proof}

\subsection{The heat flow of caloric connections: 
\texorpdfstring{$L^2$}{L2} analysis.}
Our goal here is to better describe the Yang--Mills heat flow of
caloric connections as a perturbation of the linear heat
flow. Toward that goal we begin with a caloric connection
$a \in \calC_{\hM}$ with energy $\calE$ and with a corresponding linearized caloric data $b \in T_a^{L^2} \calC$ (which we write $(a, b) \in T^{L^{2}} \calC_{\hM}$ for short), and we seek to obtain bounds for their caloric heat flows $A(s)$ and
$B(s)$. It is useful and natural to express these bounds in terms of
frequency envelopes.  The next result shows that for caloric
connections both $A$ and $F$ have full parabolic regularity:

\begin{proposition}\label{p:ab-l2}
\begin{enumerate}
\item Let $a$ be a caloric connection in $\calC_{\hM}$ with energy at most $\calE$, and let $c_{k}$ be its $(-1, S)$-frequency envelope in $\dot{H}^{1}$. Then for any $N \geq 0$, its Yang--Mills heat flow $A(s)$ satisfies
  \begin{align} 
\|P_k A(s)\|_{\dot H^1} + \| P_k F(s)\|_{L^2}  \lesssim_{\hM, \calE, N} & c_{k}(1+ 2^{2k} s)^{-N}, 	\label{AF-fe} \\
\|P_k \partial^\ell A_\ell(s)\|_{L^2}  \lesssim_{\hM, \calE, N} & c_k  c_k^{[1]} (1+ 2^{2k} s)^{-N}.	\label{DA-fe}
\end{align}

\item Let $b$ be a corresponding linearized caloric data set with a $2$-compatible $L^2$
 $(-1,S)$-frequency envelope $d_k$. Then for any $N \geq 0$, its linearized caloric flow $B$ satisfies the
  bounds
\begin{align} 
\|P_k B(s)\|_{L^2}   \lesssim_{\hM, \calE, N} & d_k (1+ 2^{2k} s)^{-N}  ,	\label{B-fe} \\
\|P_k \partial^j B_j(s)\|_{\dot H^{-1}}  \lesssim_{\hM, \calE, N}  & \Big( d_k c_k^{[1]} +  c_k d_k^{[2]} + \sum_{j > k} 2^{k-j} c_{j} d_{j}  \Big) (1+ 2^{2k} s)^{-N}. \label{DB-fe}
\end{align}
\end{enumerate}
\end{proposition}

Our starting point for the proof is the covariant curvature bounds in Proposition~\ref{p:cov-smth},
and their slightly less covariant local caloric gauge versions in Lemma~\ref{l:nocov-smth}.
Our first goal is to expand them to fully noncovariant versions, taking advantage of the 
additional information $A(\infty) = 0$, $B(\infty) = 0$. This is done in the following lemma:

\begin{lemma}
\begin{enumerate}
\item Let $a$ be a caloric connection in $\calC_{\hM}$ with energy $\calE$, and let $A$ be its Yang--Mills heat flow. Then for any $m \geq 0$, we have the bounds
\begin{align} 
 \nrm{s^{m/2} \partial_{x}^{(m)} F}_{L^{\infty}_{\frac{\ud s}{s}} ([0, \infty); L^{2} _{x})}^{2}
	+ \nrm{s^{(m+1)/2} \partial_{x}^{(m+1)} F}_{L^{2}_{\frac{\ud s}{s}} ([0, \infty); L^{2} _{x})}^{2} 
	\lesssim_{\hM, \calE} & 1, \label{nocov-est-f} \\
	 \nrm{s^{m/2} \partial_{x}^{(m+1)} A}_{L^{\infty}_{\frac{\ud s}{s}} ([0, \infty); L^{2} _{x})}^{2}
	+ \nrm{s^{(m+1)/2} \partial_{x}^{(m+2)} A}_{L^{2}_{\frac{\ud s}{s}} ([0, \infty); L^{2} _{x})}^{2} 
	\lesssim_{\hM, \calE} & 1.  \label{nocov-est-a}
\end{align}

\item In addition, if $b \in \dot{H}^{\sgm}$ is a corresponding linearized caloric data set with $-1  < \sgm < 2$,
then the corresponding solution $B$ satisfies the estimates
\begin{equation} \label{nocov-est-b}
 \nrm{s^{m/2} \partial_{x}^{(m)} B}_{L^{\infty}_{\frac{\ud s}{s}} ([0, \infty); \dot{H}^\sgm_{x})}^{2}
	+ \nrm{s^{(m+1)/2} \partial_{x}^{(m+1)} B}_{L^{2}_{\frac{\ud s}{s}} ([0, \infty); \dot{H}^\sgm_{x})}^{2} 
	\lesssim_{\hM, \calE} \|b\|_{\dot H^\sgm}.
\end{equation}
\end{enumerate}
\end{lemma}

\begin{proof}
For counting the $s$-weights, it is convenient to use the measure $\frac{\ud s}{s}$ for $s \in [0, \infty)$. In this proof, for simplicity of notation, we simply write $L^{q} L^{r} = L^{q}_{\frac{d s}{s}}([0, \infty); L^{r})$. Moreover, we suppress the dependence of implicit constants on $\hM$ and $\calE$.

\pfstep{Proof of (1)} We  start from the bounds in Proposition~\ref{p:cov-smth}, which now hold globally in time. Recall also that, by Proposition~\ref{p:cal-in-bfH}, we already have
\begin{equation*}
	\nrm{\rd_{x} A}_{L^{\infty} L^{2}} \aleq_{\hM, \calE} 1.
\end{equation*}

\pfsubstep{Step~(1).1:~Covariant bounds for $A$}
These follow from  \eqref{est-cov} after time integration from $\infty$.
Starting from the (schematic) relation
\[
\partial_s A = \covD F,
\]
we differentiate to obtain 
\[
\partial_s \covD A = \covD^{(2)} F + [\covD F,A] .
\]
Then reiterating yields
\[
\partial_s \covD^{(m)} A = \covD^{(m+1)} F + \sum_{k  = 1}^{m}    [\covD^{(k)} F, \covD^{(m-k)} A].
\]
Thus, for $m > 0$ we can estimate
\[
\begin{split}
\| s^{m/2} \covD^{(m+1)} A\|_{L^{\infty} L^2} \lesssim & \ \| s^{m/2+1} \partial_s \covD^{(m+1)} A\|_{L^{\infty} L^2} 
\\ 
\lesssim & \  \| s^{m/2+1}  \covD^{(m+2)} F\|_{L^\infty L^2}  + \sum_{k=1}^{m+1} \|  s^{m/2+1}[\covD^{(k)} F, \covD^{(m+1-k)} A]\|_{L^\infty L^2}
\end{split}
\]
and the last term is bounded inductively if $k< m+1$, as
\[
 \| s^{m/2+1} [\covD^{(k)} F, \covD^{(m+1-k)} A] \|_{L^\infty L^2}  \lesssim 
 \| s^{k/2+1} \covD^{(k)} F\|_{L^\infty L^{\infty}} \| s^{(m-k)/2} \covD^{(m-k+1)} A\|_{L^\infty L^2}
\]
and directly if $k = m+1$,
\[
\|  s^{m/2+1}[\covD^{(m+1)} F,  A]\|_{L^\infty L^2}  \lesssim \| s^{m/2+1} \covD^{(m+1)} F\|_{L^\infty L^4} \| A\|_{L^\infty L^4}.
\]
The argument for the $L^2 L^{2}$ bounds is similar, and also applies if $m = 0$.

\pfsubstep{Step~(1).2:~The bound \eqref{nocov-est-f} for $m = 1$}
The $L^\infty L^2$ part follows from \eqref{D=d}. For the $L^2$ part we estimate
\begin{align*}
\|s \partial^{(2)} F\|_{L^2 L^{2}} \lesssim & \| s \covD  \partial F\|_{L^2 L^{2}} \lesssim \| s \partial \covD F\|_{L^2 L^{2}} 
+ \| s[ \partial A,F]\|_{L^2 L^{2}} \\
\lesssim &  \| s  \covD^{(2)} F\|_{L^2 L^{2}} 
+ \| \partial A\|_{L^\infty L^2} \| s F\|_{L^2 L^\infty}.
\end{align*}

\pfsubstep{Step~(1).3:~The bound \eqref{nocov-est-a} for $m = 0$} 
We already know the $L^\infty L^2$ part. For the $L^2 L^{2}$ bound we estimate
\[
\begin{split}
\| s^{1/2} \partial_x^{(2)} A\|_{L^2 L^{2}} \lesssim & \  \| s^{1/2} \covD \partial A\|_{L^2 L^{2}} 
\lesssim \| s^{1/2} \partial \covD A\|_{L^2 L^{2}} + \| s ^{1/2} [\partial A, A]\|_{L^2 L^{2}}
\\ \lesssim  & \ \| s^{1/2} \covD^{(2)} A\|_{L^2 L^{2}} + \|\partial A\|_{L^\infty L^2}   \| s ^{1/2} A \|_{L^2 L^\infty}.
\end{split}
\]
Finally for the last factor we have
\[
  \| s ^{1/2} A \|_{L^2 L^\infty} \lesssim \| s^{3/2} \covD F\|_{L^2 L^\infty}
\]

\pfsubstep{Step~(1).4:~The bound \eqref{nocov-est-f}}
For $m = 0$ we only need the following simple bound:
\begin{equation}\label{D=d}
\| F\|_{\dot H^1} \lesssim \|\covD F\|_{L^2}.
\end{equation}

We now consider $m \geq 1$.
For the $L^\infty L^2$ bound we have 
\[
\|s^{m/2} \partial^{(m)} F\|_{L^\infty L^2} \lesssim \| s^{m/2} \covD^{(m)} F\|_{L^\infty L^2}
+ \!\!\!\! \sum_{m_0+ m_1+\cdots+ m_k = m-k} \!\!\!\!  \|  s^{m/2} \Big( \prod_{j=1}^k ad(\covD^{(m_j)} A)  \Big) \covD^{(m_0)} F\|_{L^\infty L^2}
\]
and it remains to estimate
\[
\|  s^{m/2} \Big( \prod_{j=1}^k ad(\covD^{(m_j)} A)  \Big) \covD^{(m_0)} F \|_{L^\infty L^2} \lesssim 
\Big( \prod_{j=1}^k \| s^{m_j/2 +1} \covD^{(m_j)} A\|_{L^{\infty} L^\infty} \Big) \| s^{m_0/2}  \covD^{(m_0)} F\|_{L^\infty L^2}.
\]

Similarly, for the $L^2 L^{2}$ bound we get
\begin{align*}
\|s^{(m+1)/2} \partial^{(m+1)} F\|_{L^2 L^{2}} \lesssim & \| s^{(m+1)/2} \covD^{(m+1)} F\|_{L^2 L^{2}} \\
&+ \sum_{m_0+ m_1+\cdots+ m_k = m+1-k} \|  s^{(m+1)/2} \Big( \prod_{j=1}^k ad( \covD^{(m_j)} A) \Big) \covD^{(m_0)} F\|_{L^2 L^{2}}.
\end{align*}
Here we distinguish two cases. If $m_0 > 0$ then we estimate the $F$ factor in $L^2 L^{2}$ and 
all $A$ factors in $L^\infty L^{\infty}$ as above. Else we estimate the $F$ factor in $L^2 L^4$, and one of the $A$ 
factors in $L^\infty L^4$.
\medskip

\pfsubstep{Step~(1).5: The bound \eqref{nocov-est-a}} 
For the $L^\infty L^2$ bound we have, with $m \geq 0$, 
\begin{align*}
\|s^{m/2} \partial^{(m+1)} A\|_{L^\infty L^2} \lesssim & \| s^{m/2} \covD^{(m+1)} A\|_{L^\infty L^2} \\
& + \sum_{m_0+ m_1+\cdots+ m_k = m-k+1} \|  s^{m/2} \Big( \prod_{j=1}^k ad(\covD^{(m_j)} A)  \Big) \covD^{(m_0)} A\|_{L^\infty L^2}
\end{align*}
and conclude as above using the induction hypothesis. The $L^2 L^{2}$ bound
is also similar, and that argument also applies if $m = 0$.

\pfstep{Proof of~(2)} We will approach $B$ via the $F_{0j}$ flow with same data, which also has $\dot{H}^\sgm$
regularity.  We already know from Theorem~\ref{t:heatA} that this is
well-posed in $\dot H^\sgm$, which gives the bound in \eqref{nocov-est-b}
for $N = 0$.  Then we write the equation for $s \Delta_A F_{0j}$, apply the 
same estimate and then argue as above and repeat. This yields the desired
estimate, but for $F_{0j}$ rather than $B$:
\begin{equation} \label{nocov-est-f0j}
 \nrm{s^{m/2} \partial_{x}^{(m)} F_{0j}}_{L^{\infty} \dot{H}^{\sgm}}^{2}
	+ \nrm{s^{(m+1)/2} \partial_{x}^{(m+1)} F_{0j}}_{L^{2} \dot{H}^{\sgm} }^{2} 
	\lesssim_{\hM, \calE} 1 .
\end{equation}

Next we turn our attention to $A_0$. As $b$ is a linearized caloric data set, it follows that $A_0$ vanishes at infinity and at $s = 0$.
Then $A_0$ is represented in two different ways as
\[
A_0(s_0) = -  \int_{s_0}^\infty \covD^{\ell} F_{\ell 0}(s)ds =  \int_0^{s_0} \covD^\ell F_{\ell 0}(s)ds .
\]

We first claim that 
\begin{equation} \label{nocov-est-a0}
 \nrm{s^{m/2} \partial_{x}^{(m)} A_0}_{L^{\infty} \dot{H}^{\sgm+1}}^{2}
	+ \nrm{s^{(m+1)/2} \partial_{x}^{(m+1)} A_0}_{L^{2} \dot{H}^{\sgm+1}}^{2} 
	\lesssim_{\hM, \calE} 1 .
\end{equation}
Indeed, the integrand satisfies the same bounds as for $F_{0j}$ but
with $\sgm$ replaced by $\sgm-1$.  Then by direct integration from infinity
we obtain all the desired bounds for $A_0$ (i.e. the same as for
$F_{0j}$ but with $\sgm-1$ replaced by $\sgm+1$), with the notable exception
of the $L^\infty \dot H^{\sgm+1}$ bound.  For this we combine the
$L^2\dot H^\sgm$ bound for $\covD F = \partial_s A_0$ with the $s^{-1} L^2
\dot H^{\sgm+2}$ bound for $A_0$.

Unfortunately the estimate \eqref{nocov-est-a0} does not directly yield the similar 
bounds for $\covD A_0$, precisely in the range $\sgm \geq 1$. There \eqref{nocov-est-a0}
does not provide any good control over the low frequencies of $A_0$, which is needed for 
the bilinear term $[A,A_0]$.  To remedy this we also integrate from zero
to obtain the bound
\begin{equation} \label{nocov-est-a0-0}
 \nrm{s^{-1} A_{0}}_{L^{\infty} \dot{H}^{\sgm-1}}^{2} \lesssim_{\hM, \calE} 1. 
\end{equation}
Now we can bound the term $[A,A_0]$ using a Littlewood-Paley trichotomy as follows:
if the frequency of $A_0$ is higher, then combine  \eqref{nocov-est-a} with  \eqref{nocov-est-a0},
and if the frequency of $A_0$ is lower, then combine  \eqref{nocov-est-a} with  \eqref{nocov-est-a0-0}. \qedhere
\end{proof}

We are now ready to complete the proof of Proposition~\ref{p:ab-l2}.

\begin{proof}[Completion of the proof of Proposition~\ref{p:ab-l2}]
As before, we suppress the dependence of implicit constants on $\hM, \calE, N$.

We first consider $(-1,2)$ envelopes for both $B$ and $A$.   
Then the bounds for the linearized caloric flow $B$ in part (2) of the
  proposition follow directly from the lemma.  Hence the bounds for $A$
  and $F$ in part (1) are obtained by applying the result in part (2)
  to $\partial_x A$, which solves the linearized equation. 

In order to relax the admissibility constraint on the frequency envelope from
$(-1,2)$ to $(-1,S)$, we reiterate the equations based on the linear
heat flow. Indeed, denoting by $c_k^{S}$ the minimal $(-1,S)$
frequency envelopes for the data $a$ in $\dot H^1$, we use induction
on $S$ to show that
\[
\| P_k A\|_{\dot H^1} \lesssim c_k^{S}(1+ 2^{2k} s)^{-N}.
\]
The above analysis proves this for $S < 2$. To increase $S$ to
$S+\sigma$ with $\sigma < 1$ we reiterate based on the linear (schematic) Duhamel formula
\begin{equation*}
	F(s) = e^{s \lap} F(s=0) + \int_{0}^{s} e^{(s-\tilde{s}) \lap} ([A, \rd F] + [\rd A, F] + [A, [A, F]]) \, d \tilde{s},
\end{equation*}
followed by integration from infinity for $A$.  A direct computation (whose tedious details we omit) using Littlewood--Paley trichotomy yields
\[
\| P_k A\|_{\dot H^1} \lesssim (c_k^{S+\sigma} +   
c_k^{S} \sum_{j< k} 2^{j-k} c_{j}^{S})    (1+ 2^{2k} s)^{-N} .
\]
As we have 
\[
c_k^{S} \approx \sup_{j <  k}  2^{S(1-\eps)(j-k)} c_{j}^{S+\sigma} ,
\]
we compute
\[
c_k^{S} \sum_{j< k} 2^{j-k} c_{j}^{S} \leq \sum_{j, \ell < k}
 2^{S(1-\eps)(j-k)}  c_j^{S+\sigma}  2^{\ell-k}  c_{\ell}^{S+\sigma} \lesssim \sum_{j < k}
2^{(S+\sigma+)(j-k)}  c_j^{S+\sigma} \approx  c_k^{S+\sigma}
\]
where at the second-to-last step we separate the cases $j < \ell$ and $\ell < j$ and use 
$c_k \lesssim 1$. The induction is concluded. 

Since the estimate for $A$ has been closed, a more direct argument applies for $B$. We first estimate $F_{0j}$ 
perturbatively using Theorem~\ref{t:heatB}, which yields
\begin{equation} \label{eq:F0-fe}
	\nrm{P_{k} F_{0j}(s)}_{L^{2}} \aleq d_{k} (1+2^{2k} s)^{-N}. 
\end{equation}
Then for $A_0$ we can integrate either from infinity (for the high frequencies)
or from zero (for the low frequencies). We obtain:
\begin{lemma} \label{l:ah-l2}
Under the assumptions above, we have
\begin{align}
	\nrm{P_{k} A_{0}(s)}_{\dot{H}^{1}} \aleq & d_{k} (1+2^{2k} s)^{-N}, \label{eq:ah-l2-high} \\
	\nrm{s^{-1} P_{k} A_{0}(s)}_{\dot{H}^{-1}} \aleq & d_{k}. \label{eq:ah-l2-low}
\end{align}
\end{lemma}
Combining this lemma with the \eqref{eq:F0-fe}, we obtain \eqref{B-fe} for $(-1, S)$ frequency envelopes.

It remains
  to prove the estimates for $\partial^\ell A_\ell$ and $\partial^\ell
  B_\ell$. As in Lemmas~\ref{l:da} and \ref{l:db},
   this is a direct computation based on the representations
\[
\rd^\ell A_\ell(s) = \int_{s}^\infty  [A^\ell, \covD^k F_{k \ell}] d \tilde{s},
\]
respectively,
\[
\rd^{\ell} B_{\ell}(s) + [A^{\ell}, B_{\ell}] (s) = \covD^\ell B_\ell(s) = 2\int_{s}^\infty  [B^{\ell}, \covD^{k} F_{k \ell}] d \tilde{s}.
\]
We omit the details.
\end{proof}

Our next goal is to establish difference bounds for the heat flows of
caloric connections. For this we consider two linearized caloric data sets
$(a^{(0)},b^{(0)}), (a^{(1)},b^{(1)}) \in T^{L^2} \calC_{\hM}$ with energy at most $\calE$, which are
assumed to be sufficiently close
\begin{equation}\label{dA-small}
\|a^{(0)} - a^{(1)}\|_{\dot H^1} \ll_{\hM, \calE} 1.
\end{equation}
Then we seek to compare their corresponding caloric extensions
$(A^{(0)}(s),B^{(0)}(s))$, respectively $ (A^{(1)}(s),B^{(1)}(s))$,
and provide frequency envelope bounds for the differences
\[
\delta A(s) = A^{(1)}(s)-A^{(0)}(s), \qquad \delta B(s) = B^{(1)}(s)-B^{(0)}(s),.
\]
Our main result is as follows:

\begin{proposition}\label{p:diff-l2}
Let $(a^{(0)},b^{(0)}), (a^{(1)},b^{(1)})\in T^{L^2} \calC_{\hM}$ be two linearized caloric
data sets with energy at most $\calE$, such that \eqref{dA-small} holds. Assume 
that  $c_k$ is a $(-1,S)$-frequency envelope for $(a^{(0)},b^{(0)})$ and $(a^{(1)},b^{(1)})$ in $\dot H^1 \times L^2$,
and that $d_k$  is a $1$-compatible $(-1,S)$-frequency envelope for $(a^{(0)} - a^{(1)},b^{(0)} - b^{(1)})$ in $\dot H^1 \times L^2$, such that
\begin{equation} \label{eq:e-l2}
	e_{k} = d_{k} + c_{k} (c \cdot d)_{\leq k}
\end{equation}
and $c_{k} e_{k}$ are also $(-1, S)$-admissible.
Then we have the difference bounds 
\begin{equation}\label{dA-fe}
 \|P_k \delta A(s)\|_{\dot H^1} + \|P_k \delta F(s)\|_{L^2} + 
\|P_k \delta B(s)\|_{L^2} \lesssim_{\hM, \calE, N}  e_{k} (1+ 2^{2k} s)^{-N},
\end{equation}
respectively,
\begin{equation}\label{dDA-fe}
 \|P_k \partial^\ell \delta A_\ell(s)\|_{L^2} 
+ \|P_k \partial^\ell \delta B_\ell(s)\|_{\dot H^{-1}}  \\
 \lesssim_{\hM, \calE, N} \Big( c_{k} e_{k}^{[1]} + e_{k} c_{k}^{[1]} + \sum_{j > k} 2^{k-j} c_{j} e_{j} \Big) (1+ 2^{2k} s)^{-N}.
\end{equation}
\end{proposition}
The compatibility assumption may be sharpened if we consider separate frequency envelopes for $b^{(0)}$, $b^{(1)}$ and $b^{(0)} - b^{(1)}$; however, we avoid this for the sake of simplicity.

We remark that similar bounds for the linearized equation follow
  from the previous proposition, at least for infinitesimal
  deformations of $A$.  However, if we try to transfer this directly
  to differences, then we need to address the problem of constructing
  a smooth path between $a^{(0)}$ and $a^{(1)}$ which stays within the
  caloric manifold. This will be of independent interest later, so
 we state the result separately. 

\begin{proposition}\label{p:path-l2}
Under the same assumptions as in the previous proposition, there exists a 
$C^1$ path 
\[
[0,1] \ni h \to (a^{(h)},b^{(h)}) \in T^{L^{2}} \calC 
\]
so that for $h \in [0, 1]$, $\hM(a^{(h)}) \leq 2 \hM$, $\spE[a^{(h)}] \leq 2 \calE$, and the following estimates hold uniformly:
\begin{equation}\label{h-ab}
\|P_k  a^{(h)}\|_{\dot H^1} + \|P_k  b^{(h)}\|_{L^2} \lesssim_{\hM, \calE}   c_k ,
\end{equation}
respectively,
\begin{equation}\label{dh-ab}
\|P_k \partial_h a^{(h)}\|_{\dot H^1} + \|P_k \partial_h b^{(h)}\|_{L^2} \lesssim_{\hM, \calE}  e_{k}.
\end{equation}
\end{proposition}
We now prove the two propositions above.

\begin{proof}[Proof of Proposition~\ref{p:path-l2}]
We suppress the dependence of constants on $\hM, \calE, N$.

  We first construct the path joining $a^{(0)}$ and $a^{(1)}$.  We
  begin with the straight line path $\ta^{(h)} = h a^{(1)} + (1-h) a^{(0)}$ between $a^{(0)}$ and
  $a^{(1)}$, which does not stay within the caloric manifold. In view
  of \eqref{dA-small}, this path remains in $\set{(a, b) \in \dot{H}^{1} \times L^{2} : \hM(a) \leq 2 \hM, \ \spE[a] \leq 2 \calE}$.
For $\ta^{(h)}$ and $\partial_h \ta^{(h)}$ we have 
\[
\| P_k \ta^{(h)}\|_{\dot H^1} \lesssim c_k, \qquad  \| P_k \partial_h \ta^{(h)}\|_{\dot H^1} \lesssim d_k.
\]

Hence, using Theorem~\ref{t:ymhf-l3-fe} and Lemma~\ref{l:db} (in particular, \eqref{eq:db-fe-h1}), we conclude that we have the uniform bounds
\[
\| P_k \tA^{(h)} \|_{L^\infty \dot H^1} + \| P_k\tF^{(h)} \|_{L^2 \dot H^1} \lesssim c_k,
\]
\[
\| P_k \partial_h \tA^{(h)} \|_{L^\infty \dot H^1} + \| P_k(\covD_i \partial_h \tA^{(h)}_j - \covD_j \partial_h \tA^{(h)}_i) \|_{L^2 \dot H^1} \lesssim d_k ,
\]
as well as the improved divergence bound 
\[
 \| P_k (\partial^\ell \partial_h \tA^{(h)}_\ell (s_1) - \partial^\ell  \partial_h \tA^{(h)}_\ell (s_2)) \|_{L^\infty L^2 } \lesssim c_k d_k^{[1]} 
+ d_k c_k^{[1]}.
\]
Integrating this over the interval $h \in [0,1]$ we obtain the following intermediate 
result:
\begin{lemma}\label{l:diff-cal}
Given the frequency envelopes $c_k$ and $d_k$ as above, we have the uniform difference bounds
\begin{equation}\label{diff-A}
 \| P_k \delta A \|_{ L^\infty \dot H^1} + \| P_k \delta F \|_{L^2 \dot H^1} \lesssim d_k.
\end{equation}
respectively
\begin{equation}\label{diff-dA}
\|P_k \partial^\ell \delta A_\ell \|_{L^\infty L^2} \lesssim c_k d_k^{[1]} + d_k c_k^{[1]}.
\end{equation}
\end{lemma}
We remark that for the first bounds we do not use the caloric gauge. However for the second it is critical 
that $A^{(0)}(\infty) = A^{(1)}(\infty) = 0$. In particular, its conclusion is nontrivial even at $s = 0$, where 
we get

\begin{corollary}\label{c:diff-cal}
Given the frequency envelopes $c_k$ and $d_k$ as above, we have
\begin{equation}\label{diff-b}
\| P_k \partial^{\ell} \delta a_{\ell}\|_{L^2} \lesssim  c_k d_k^{[1]} + d_k c_k^{[1]}.
\end{equation}
\end{corollary}
As a consequence of this, by construction we also get 
\begin{equation}\label{diff-tb}
\| P_k \partial^{\ell} \partial_h \ta_{\ell}^{(h)} \|_{L^2} \lesssim   c_k d_k^{[1]} + d_k c_k^{[1]}.
\end{equation}

Once  we have a good understanding of $\ta^{(h)}$ and $\partial_h\ta^{(h)}$, we project onto the caloric manifold,
setting 
\[
a^{(h)} = \Cal(\ta^{(h)}).
\]
Now the bounds for $a^{(h)}$ in the proposition follow from Proposition~\ref{p:Cal-fe-ah}, while those for 
$\partial_h a^{(h)}$ follow from Proposition~\ref{p:fe-ah}.

\bigskip

We now consider the question of choosing $b^{(h)}$. We begin with $\tb^{(h)}$ which interpolates linearly between 
$b^{(0)}$ and $b^{(1)}$, and define
\[
b^{(h)} = \Pi_{a^{(h)}} \tb^{(h)} = \tb^{(h)} - \covD[a^{(h)}] a^{(h)}_0
\]
where $a^{(h)}_0$ is the initial data for the corresponding connection component which is initialized to zero at infinity (see Remark~\ref{r:proj-a0}), i.e.,
\begin{equation*}
	a^{(h)}_{0} = \int_{0}^{\infty} (\covD[a^{(h)}])^{\ell} F_{0 \ell}^{(h)} (\tilde{s}) \, \ud \tilde{s}.
\end{equation*}	
Now the $b^{(h)} $ bound in \eqref{h-ab} is a consequence of Lemma~\ref{l:pi-l2}.
To estimate $\partial_h b^{(h)}$ we need $\partial_h \covD[a^{(h)}] a^{(h)}_0$, which we write as
\[
\partial_h \covD[a^{(h)}] a^{(h)}_0 = \covD[a^{(h)}]  \partial_h a^{(h)}_0 + [\partial_h a^{(h)}, a^{(h)}_0] .
\]
The second term is easy to estimate using the previous bound for $\partial_h a^{(h)}$ 
and for $a^{(h)}_0$ as in the proof of Proposition~\ref{p:ab-l2}.

It remains to bound $\partial_h a^{(h)}_0$ in $\dot H^1$. We claim that
\begin{equation}\label{a0-l2}
\| P_k \partial_h a^{(h)}_0 (s) \|_{\dot H^1} \lesssim   e_{k}.
\end{equation}
 
 To achieve this, we use another round of the ``infinitesimal de Turck trick'' (see Remark~\ref{r:proj-a0}). 
In what follows, we suppress the superscript $(h)$. Let us introduce a dynamical gauge component $A_h$, which satisfies
\begin{equation} \label{eq:ah-dymhf}
\partial_s A_h = \covD^j F_{jh}, \qquad A_h(\infty) = 0,
\end{equation}
where $F_{jh}$ solves the covariant heat equation
\begin{equation} \label{eq:fjh-heat}
(\rd_{s} - \lap_{A} - 2 ad (F)) F_{hj} = 0, \qquad F_{hj}(0) = \partial_h a_{j}.
\end{equation}
Then we also have
\begin{equation} \label{eq:ah0=0}
	A_h(0) = 0.
\end{equation}
Differentiating with respect to $s$, we see that
\begin{equation} \label{eq:dsdhA0}
	\rd_{s} \covD_{h} A_{0}
	= [\covD^{\ell} F_{\ell h}, A_{0}] + \covD_{h} \covD^{\ell} F_{\ell 0}
	= [\covD^{\ell} F_{\ell h}, A_{0}] + [\tensor{F}{_{h}^{\ell}}, F_{\ell 0}] + \covD_{j} \covD_{h} F_{j0}.
\end{equation}
On the other hand, $\covD_{h} F_{j0}$ obeys the inhomogeneous covariant heat equation
\begin{equation} \label{eq:heat-dhF0}
	(\rd_{s} - \lap_{A} - 2 ad (F)) \covD_{h} F_{0j} = G_{j}, \qquad \covD_{h} F_{0 j}(0) = \rd_{h} \tilde{b}_{j},
\end{equation}
where $G_{i}$ is (schematically) of the form\footnote{For this computation, we also need the Bianchi identity $\covD_{h} F_{ij} = \covD_{i} F_{h j} - \covD_{j} F_{h i}$.}
\begin{equation} \label{eq:heat-dhF0-RHS}
	G_{i} = [\covD F_{h}, F_{0}] + [F_{h}, \covD F_{0}].
\end{equation}
Now  we have all the equations we need in order to prove the estimates.

By Proposition~\ref{p:ab-l2}.(2), we already have
\begin{equation}\label{fhj}
\| P_k F_{hj} \|_{\dot H^1} \lesssim  e_{k}   (1+2^{2k} s)^{-N}
, \qquad \| P_k F_{0j} \|_{L^2}  \lesssim c_k(1+2^{2k} s)^{-N}.
\end{equation}
Hence for $G_i$ we obtain
\[
\| P_k G_i \|_{\dot H^{-1}} \lesssim  s^{-\frac12} e_{k} (1+2^{2k} s)^{-N}.
\]
Thus solving the parabolic equation for $\covD_h F_{0i}$ using Theorems~\ref{t:heatB} and \ref{t:heatB-inhom}, we obtain 
\begin{equation}\label{dh-f0j}
\| P_k \covD_h F_{0i} \|_{L^2} \lesssim    e_{k}     (1+2^{2k} s)^{-N}.
\end{equation}

This implies that 
\[
\| P_k  \partial_s \covD_h A_0 \|_{\dot H^{-1}} \lesssim   e_{k}   (1+2^{2k} s)^{-N},
\]
which in turn yields the desired bound for $\covD_h A_0$,
\begin{equation}\label{A0-l2}
\| P_k \covD_h A_0 \|_{\dot H^1} \lesssim   e_{k} (1+2^{2k} s)^{-N},
\end{equation}
which at $s = 0$ gives \eqref{a0-l2}. \qedhere
\end{proof}

\begin{proof}[Proof of Proposition~\ref{p:diff-l2}]
We use the path $(a^{(h)}, b^{(h)})$ constructed in Proposition~\ref{p:path-l2}.
The difference bounds in the proposition are obtained by integrating with respect to 
$h \in [0,1]$ the corresponding bounds for $\partial_h A^{(h)}(s)$ and $\partial_h B^{(h)}(s)$. 

The frequency envelope bounds for the data $\partial_h a^{(h)}$ translate to similar bounds 
for $\partial_h A^{(h)}(s)$ by Proposition~\ref{p:ab-l2}, and thus to bounds for $\partial_h F^{(h)}_{ij} = \covD_{i} \rd_{h} A^{(h)}_{j} - \covD_{j} \rd_{h} A^{(h)}_{i}$. We remark that, by $1$-compatibility of $d_{k}$ and $e_{k}$ with $c_{k}$, we end up with simply the frequency envelope $e_{k}$ on the RHS.

For $\partial_h B^{(h)}$, we again introduce the auxiliary dynamic component $A_{h}$ as in the preceding proof. Suppressing the superscript $(h)$, we have
\[
\partial_h B = \covD_h  B - [A_h, B].
\]
For the second term we combine the $B$ bound given by Proposition~\ref{p:ab-l2}.(2)
with the $A_h$ bound in Lemma~\ref{l:ah-l2} (note that $A_{0}$ in the lemma corresponds to $A_{h}$ here). For the first term we write
\[
\covD_h  B = \covD_h F_{0j} + \covD_h \covD A_0 = \covD_h F_{0j} + \covD \covD_h  A_0 + [F_{hj}, A_0].
\]
To estimate the RHS, we combine the bounds \eqref{dh-f0j} for $\covD_h F_{0j}$, \eqref{A0-l2} for $\covD_h A_0$,
\eqref{fhj} for $F_{hj}$ and the following bound for $A_{0}$:
\begin{equation*}
	\nrm{P_{k} A_{0}(s)}_{\dot{H}^{1}} \aleq c_{k} (1+2^{2k} s)^{-N},
\end{equation*}
which is obtained by integrating $\rd_{s} A_{0} = \covD^{\ell} F_{\ell 0}$ from $s = \infty$, and using \eqref{fhj} for $F_{\ell 0}$.
Again, by compatibility, the frequency envelope bound on the RHS simplifies to $e_{k}$.

Our final goal is to prove the bound for $\rd_{h} \rd^{\ell} A_{\ell}$ and $\partial_h \partial^\ell B_\ell$; we only give a sketch of the proof.
For $\rd_{h} \rd^{\ell} A_{\ell}$, its $s$ derivative is:
\begin{equation*}
	\rd_{s} \rd_{h} \rd^{\ell} A_{\ell}
	= \rd_{h} [A_{\ell}, \covD^{i} \tensor{F}{_{i}^{\ell}}]
	= [\rd_{h} A_{\ell}, \covD^{i} \tensor{F}{_{i}^{\ell}}]
	+ [A_{\ell}, \covD^i ( \covD_i \partial_h A^{\ell} - \covD^{\ell} \partial_h A_i)]
\end{equation*}
We then integrate from infinity, estimating the RHS using the bounds for $A$ in Proposition~\ref{p:ab-l2} and $\rd_{h} A$ from above.

The case of $\rd_{h} \rd^{\ell} B_{\ell}$ is dealt with similarly. Given the estimates for $\partial_h A$, this is easily seen to be equivalent to the bound for  $\partial_h \covD^\ell B_\ell$. For this  we compute its $s$ derivative,
\[
\begin{split}
\partial_s \partial_h \covD^\ell B_\ell = & -2 \partial_h   [B_\ell,\covD^i \tensor{F}{_{i}^{\ell}}] 
\\ = & -2 [\partial_h B_{\ell},\covD^i \tensor{F}{_{i}^{\ell}}]
- 2    [B_\ell, [\partial_h A^i, \tensor{F}{_{i}^{\ell}}]] 
-2 [B_\ell, \covD^i ( \covD_i \partial_h A^{\ell} - \covD^{\ell} \partial_h A_i)] .
\end{split}
\]
Now it suffices to integrate from infinity, estimating all terms on the RHS, using the bounds for $A$ in Proposition~\ref{p:ab-l2} and $\rd_{h} B$ from above.
\end{proof}

Next we compare the Yang--Mills heat flow of a caloric connection $a \in \calC$ with its linear heat flow.
Precisely, for its heat flow and the associated curvature tensor, we consider the following representations:
\begin{equation}
A_j(s) = e^{s \Delta} a_j + \bfA_j(s), \qquad F_{ij}(s) = e^{s \Delta} f_{ij} + \bfF_{ij}(s),
\end{equation}
as well as
\begin{equation}
\partial^{\ell} A_{\ell}(s) =  \DA(s).
\end{equation}
Here $\bfA_j$, $\bfF_{ij}$ and $\DA$ are viewed as maps on $\calC \times  [0,\infty)$.

Similarly, if $b$ is a corresponding linearized caloric data, then for its (local caloric gauge) linearized 
Yang--Mills heat flow  $B$, we consider the representation
\begin{equation}
B_j(s) = e^{s \Delta} b_j + \bfB_j(s), \qquad 
\partial^{\ell} B_{\ell}(s) =  \DB(s),
\end{equation}
where $\bfB_j$ and $\DB$ are  maps on $T^{L^2}\calC  \times  [0,\infty)$.

 Our goal is now to 
show that these maps satisfy favorable quadratic bounds
with Lipschitz dependence on $a$.  
For each heat-time $s$, recall that $k(s)$ refers to the associated frequency with $2^{2k(s)} s = 1$. 
As part of our analysis, we will show that $\bfA_j(s)$, $\bfF_{ij}(s)$ and $\bfB(s)$ are primarily concentrated
at frequency $k(s)$.

For the following proposition, let $0 < \dlt \ll 1$.
\begin{proposition}\label{p:AB2-l2}
Let $(a,b) \in T^{L^2}\calC_{\hM}$ with energy at most $\calE$ be equipped with $\dot{H}^{1} \times L^{2}$ $(-\dlt, S)$ frequency envelope $c_{k}$. Then we have the bounds
\begin{equation}\label{F-first}
\|(1-s\Delta)^N \bfA_{i}(s)\|_{\dot W^{1,\frac43}} +\|(1-s\Delta)^N \bfF_{ij}(s)\|_{L^{\frac43}}+
\|(1-s\Delta)^N \bfB_j(s)\|_{ L^{\frac43}}  \lesssim_{\hM, \calE, N}  2^{-k(s)} c_{k(s)} ,
\end{equation}
respectively,
\begin{equation}\label{DA-first}
\|P_k  \DA\|_{ \dot{W}^{1,\frac43}}  + \|P_k  \DB\|_{  L^\frac43} \lesssim_{\hM, \calE, N}   c_k^{[1]} (1+ 2^{2k}s)^{-N}.
\end{equation}

Similarly, if $(a^{(0)},b^{(0)})$ and $(a^{(1)},b^{(1)})$ are two close linearized caloric data sets, with a joint $\dot{H}^{1} \times L^{2}$ $(-\dlt, S)$ frequency envelope $c_{k}$ and $1$-compatible $(-\dlt, S)$ frequency envelope $d_{k}$ for the difference in $\dot{H}^{1} \times L^{2}$. Let
\begin{equation*}
	e_{k} = d_{k} + c_{k} (c \cdot d)_{\leq k}.
\end{equation*}
Then we have the difference bounds
 \begin{equation}\label{dF-first}
\|(1-s\Delta)^N \delta\bfA_{i}(s)\|_{\dot W^{1,\frac43}} +\|(1-s\Delta)^N \delta \bfF_{ij}(s)\|_{L^{\frac43}}+
\|(1-s\Delta)^N \delta \bfB_j(s)\|_{ L^{\frac43}}  \lesssim_{\hM, \calE, N}  2^{-k(s)} e^{[1]}_{k(s)},
\end{equation}
respectively,
\begin{equation}\label{dDA-first}
\|P_k  \delta \DA\|_{ \dot{W}^{1,\frac43}}  + \|P_k   \delta\DB\|_{ L^\frac43} \lesssim_{\hM, \calE, N}   e^{[1]}_{k} (1+ 2^{2k}s)^{-N}.
\end{equation}

\end{proposition}
Here we note that  all quantities estimated here are of quadratic type.  In particular we can easily get $c_k$ replaced  by $c_k^{2}$ 
if we give up a bit in terms of Sobolev embeddings and taper off rapid decay of high frequencies of $c_{k}$ (i.e., consider $(-\dlt, \dlt)$ frequency envelopes).
See Proposition~\ref{p:AB2-sp-lin} for such a statement in the $\dot{W}^{\sgm, p}$ setting.

\begin{proof}
The proof is tedious but straightforward. We focus on the structure of the equations, and only sketch the details.

First we establish the curvature bound,
using the equations \eqref{eq:YMHF-Fij} and Duhamel's principle to write
\begin{equation} \label{rep-F}
\begin{aligned}
\bfF_{ij}(s_{0}) =& F_{ij} (s_0) - e^{s_0 \Delta }f_{ij} \\
=& \int_{0}^{s_0} e^{(s_0-s)\Delta}\left(  -2 [\tensor{F}{_i^\ell},F_{j \ell}] + 2 [A_\ell, \covD^\ell F_{ij}]
+ [\rd^\ell A_\ell, F_{ij}]\right) \,  ds.
\end{aligned}
\end{equation}
Then we use the bounds \eqref{AF-fe} for $A$ and $F$ to estimate the integrand, using Littlewood--Paley trichotomy; the worst term is $[A, \covD^{\ell} F_{ij}]$ in the $low \times high$ scenario. Combined with heat flow bounds, this yields the $\bfF$ bound in \eqref{F-first}.

Next we establish the corresponding $\DA$ bound, this time integrating from infinity 
using the representation \eqref{DA-rep}:
\begin{equation} \label{rep-DA}
\DA(s) = \int_{s}^\infty [A^k,\covD^j F_{kj}] ds .
\end{equation}
The desired bound follows again from \eqref{AF-fe} by Littlewood--Paley trichotomy. Again, the worst case is the $low \times high$ scenario.

Finally, we establish the $A$ bound, solving again from zero in the equation \eqref{caloric-re} to obtain 
\begin{equation} \label{rep-A}
\begin{aligned}
\bfA_i(s_0) =& A_{i}(s_{0}) - e^{s_0 \Delta }a_i \\
=& \int_{0}^{s_0} e^{(s_0-s)\Delta}\left(  - \partial_i \partial^\ell A_{\ell}-
 [A^\ell, \partial_i A_\ell] + 2 [A_\ell, \covD^\ell A_i]
+ [\rd^\ell A_\ell, A_i]\right) \,  ds.
\end{aligned}\end{equation}
Here we treat all $\partial^\ell A_\ell$ terms perturbatively, using the previously 
obtained bound. 

We now consider the $B$ bounds. Here $B$ is  obtained via the $F_{0j}$ and $A_0$
route (see Section~\ref{subsec:dymhf}). The first step is to consider the quadratic part of the curvature $\bfF_{0j}$,
for which the analysis and estimates are identical to that for $\bfF_{ij}$. The next step
is to obtain bounds for $A_0$, for which we have the double boundary condition
$A_0(0) = A_0(\infty) = 0$ thanks to the fact that $b$ is caloric. Thus, we obtain the double representation
\begin{equation}
A_0(s) = \int_{s}^\infty  \covD^j F_{j0} ds  = - \int_{0}^s  \covD^j F_{j0} ds
\end{equation}
Peeling off the linear heat flow of $F_{j0}$
we obtain 
\[
A_0(s) = \int_{s}^\infty  \covD^j \bfF_{j0}  + [A_j,e^{\tilde{s} \Delta } b]  d\tilde{s} + \Delta^{-1} e^{s \Delta} \partial^j b_j .
\]
respectively
\[
A_0(s) = -\int_{0}^s  \covD^j \bfF_{j0}  + [A_j,e^{\tilde{s} \Delta } b]  d\tilde{s} -  \Delta^{-1} (1- e^{s \Delta}) \partial^j b_j . 
\]
Combining the two we arrive at
\begin{equation}\label{A0-bi}
A_0(s) = (1- e^{s \Delta}) \int_{s}^\infty  \covD^j \bfF_{j0}  + [A_j,e^{\tilde{s} \Delta } b]  d\tilde{s}  - e^{s\Delta}
\int_{0}^s  \covD^j \bfF_{j0}  + [A_j,e^{\tilde{s} \Delta } b]  d \tilde{s}.
\end{equation}
Now we can estimate this in $\dot W^{1,\frac43}$
as follows,
\begin{equation}
\|(1-s\Delta)^N \bfA_{0}(s)\|_{\dot W^{1,\frac43}} \lesssim  2^{-k(s)} c_{k(s)} .
\end{equation}
This immediately leads to the $\bfB$ bound in \eqref{F-first}, as $B_j = F_{0j} + \covD_j A_0$. 

We now consider the bounds for $\DB$, which is given by 
\begin{equation}\label{rep-DB}
\DB(s) = 2 \int_{s}^\infty [B^j, \covD^i F_{ij}] d \tilde{s} .
 \end{equation}
The corresponding bound in \eqref{DA-first} immediately follows.

Lastly, the difference bounds \eqref{dF-first}, \eqref{dDA-first} follow in the same fashion by using the difference 
estimates in the previous proposition. \qedhere

\end{proof}

\subsection{The heat flow of caloric connections: \texorpdfstring{$L^p$}{Lp} analysis} \label{subsec:lp}

We are also interested in the $L^p$ regularity properties for caloric connections.
For this we will primarily be interested in considering the caloric flow for subthreshold 
linearized caloric data $(a,b) \in \dot W^{\sgm,p} \times\dot W^{\sgm-1,p} $, for
 a range of indices $(\sgm,p)$  related to the Strichartz estimates for the wave equation with $\dot H^1$ data,
which corresponds to the line between the spaces 
\[
\dot H^1, \qquad \dot W^{\frac16,6}
\]
However, in order to have a good range of admissible envelopes we want
to be able to vary somewhat the number of derivatives. We also want to
be able to work with weaker spaces, obtained from the above ones by
Sobolev embeddings. This will be useful in order to take full
advantage of the energy dispersion later in the paper.  Because of
these, we will use a range $(\sgm,p)$ as well as associated frequency
envelopes as follows:
\begin{equation}\label{sp-range}
2 \leq p \leq \infty,  \qquad -1 < \sgm <  \frac{4}{p} .
\end{equation}
The above range insures global well-posedness of the covariant heat
flow in both $\dot W^{\sgm,p}$ and $\dot W^{\sgm,p-1}$ for caloric connections $a \in
\dot H^1$, with good parabolic decay, see Theorem~\ref{t:heatB}.
The last condition asserts that $\dot W^{\sgm,p}$ scales below $L^\infty$. 

For an appropriately constant $\dlt > 0$, which is small depending on $p, \sgm$, we will denote by $c_k^{\sgm, p}$, respectively $c_{k}$, $(-\delta,S)$ frequency envelopes for $(a,b)$  in $ \dot W^{\sgm, p} \times\dot W^{\sgm-1,p} $, respectively in $\dot H^{1} \times L^{2}$.
 We will also compare flows corresponding to two pairs of data $(a^{(0)}, b^{(0)})$ and
$(a^{(1)},b^{(1)})$. In that case we will use the notation $c_k^{\sgm, p}$, respectively $c_{k}$,
for joint $ \dot W^{\sgm, p} \times\dot W^{\sgm - 1, p}$, respectively $\dot H^{1} \times L^{2}$, $(-\dlt, S)$-frequency envelopes.
Also, we will denote by $d_k^{\sgm, p}$, respectively $d_{k}$, for a $(-\dlt, S)$
frequency envelope for their difference in $ \dot W^{\sgm, p} \times\dot W^{\sgm-1,p}$, respectively $\dot{H}^{1} \times L^{2}$. 
We will assume that $c_{k}^{\sgm}$, $d^{\sgm, p}_{k}$ and $d_{k}$ are $\dlt$-compatible with $c_{k}$.

Our estimates for differences will primarily involve the following modification of $ d_k^{\sgm, p}$:
\begin{equation} \label{eq:e-sp}
e_k^{\sgm, p} = d_k^{\sgm, p}+    c_k^{\sgm,p} d_k^{[\dlt]}   +     c_k^{\sgm, p} (c \cdot d)_{\leq k}.
\end{equation}
By the above compatibility properties, it can be verified that $e^{\sgm, p}_{k}$ is $2\dlt$-compatible with $c_{k}$.

With the setup as above, our first main goal will be to prove:

\begin{proposition}\label{p:main-sp}
\begin{enumerate}
\item Let  $(a,b)$ be a linearized caloric data set in $T^{L^{2}} \calC_{\hM}$ with energy $\leq \calE$, equipped with frequency envelopes $c_{k}$ and $c^{\sgm, p}_{k}$ as above. Then the corresponding Yang--Mills heat flow $A$ and the linearized Yang--Mills heat flow $B$, respectively, satisfy the bounds
\begin{equation}
\| P_k A(s) \|_{\dot W^{\sgm, p}}+ \| P_k F(s) \|_{\dot W^{\sgm-1,p}}
+  \| P_k B(s) \|_{\dot W^{\sgm-1,p}}  \lesssim_{\hM, \calE, N} c^{\sgm,p}_{k} (1+ 2^{2k} s)^{-N}.
\end{equation}

\item Let $(a^{(0)}, b^{(0)})$ and $(a^{(1)},b^{(1)})$ be two linearized caloric data sets in $T^{L^{2}} \calC_{\hM}$ with energy $\leq \calE$,  which are sufficiently close as in  \eqref{dA-small}, and equipped with frequency envelopes $c_{k}$, $c^{\sgm, p}_{k}$, $d_{k}$, $d^{\sgm, p}_{k}$ as above. Then the difference of their Yang--Mills heat flow pairs satisfies
\begin{equation}\label{cal-diff-sp}
\| P_k \delta A(s) \|_{\dot W^{\sgm, p}}+ \| P_k \delta F(s) \|_{\dot W^{\sgm-1,p}}
+  \| P_k \delta B(s) \|_{\dot W^{\sgm-1,p}}  \lesssim_{\hM, \calE, N} e_k^{\sgm,p} (1+ 2^{2k} s)^{-N}.
\end{equation}
\end{enumerate}
\end{proposition}

\begin{proof}
  We proceed in several steps.  Our first result is concerned with the
  $\dot W^{\sgm, p}$ bounds for the caloric projection map (cf. Proposition~\ref{p:Cal-fe-ah}).

\begin{lemma}\label{l:Cal-sp}
  Let $(\sgm,p)$ be as in \eqref{sp-range}. Let $\ta$ be a $\dot{H}^{1} \cap \dot{W}^{\sgm, p}$
  connection with $\hM(\ta) \leq \hM < \infty$ and $\nrm{\ta}_{\dot{H}^{1}} \leq M_{1}$. Let $c_{k}$ be a $(-\dlt, S)$ frequency envelope for $\ta$ in $\dot{H}^{1}$, and let $c_{k}^{\sgm, p}$ be a $(-\dlt, S)$ frequency envelope for $\ta$ in $\dot W^{\sgm,p}$ which is $\dlt$-compatible with $c_{k}$.  Then
  the caloric projection $a = \Cal(\ta)$ satisfies the bounds
\begin{equation}
\| P_k a \|_{\dot W^{\sgm,p}} \lesssim_{\hM, M_{1}} c_k^{\sgm, p}.
\end{equation}
\end{lemma}

\begin{proof}
We suppress the dependence of implicit constants on $\hM$ and $M_{1}$. The idea is to first estimate $F$ in the caloric gauge, and then pass to $a$ using the caloric gauge representation
\begin{equation*}
	a = - \int_{0}^{\infty} \covD^{\ell} F_{\ell i} \, d s.
\end{equation*}

To begin, let $O$ be the corresponding gauge transformation, i.e.,
\begin{equation*}
	O^{-1} \rd_{x} O = a_{\infty}.
\end{equation*}
By Theorem~\ref{t:ymhf-l3-fe}, $c_{k}$ is a frequency envelope for $a_{\infty}$ in $\dot{H}^{1}$. Then by Lemma~\ref{l:O-conj} and  compatibility, it follows that $c^{\sgm, p}_{k}$ is a $\dot{W}^{\sgm - 1, p}$ frequency envelope for
\begin{equation*}
	f_{jk} = Ad(O) \tf_{jk}.
\end{equation*}
By Propositions~\ref{p:Cal-fe-ah} and \ref{p:ab-l2}, we have
\begin{equation*}
	\nrm{P_{k} A(s)}_{\dot{H}^{1}} \aleq c_{k} (1+2^{2k} s)^{-N}.
\end{equation*}
Therefore, solving the parabolic equation for $f$ using Theorem~\ref{t:heatB}, we obtain
\begin{equation*}
	\nrm{P_{k} F(s)}_{\dot{W}^{\sgm-1, p}} \aleq c^{\sgm, p}_{k} (1+2^{2k} s)^{-N}.
\end{equation*}
By Littlewood--Paley trichotomy for $[A^{\ell}, F_{\ell i}]$, we have
\begin{equation} \label{eq:cal-DF-sp}
	\nrm{P_{k} \covD^{\ell} F_{\ell i}(s)}_{\dot{W}^{\sgm, p}} \aleq 2^{2k} c^{\sgm, p}_{k} (1+2^{2k} s)^{-N} +  \sum_{j > k} 2^{(2 + \sgm) (k-j)} c_{j} c^{\sgm, p}_{j} 2^{2j} (1+2^{2j} s)^{-N-10}.
\end{equation}
Then integrating in $s$, the desired bound follows. \qedhere
\end{proof}

Next we consider the heat flow of $\dot W^{\sgm, p}$ caloric data, and prove the $A$ bound in part (1) of the 
proposition: 

\begin{lemma}\label{l:cal-sp}
Let $a \in \calC_{\hM}$ with energy $\leq \calE$, and with a $\dot{H}^{1}$ $(-\dlt, S)$ frequency envelope $c_{k}$.
Let $(\sgm, p)$ be as in \eqref{sp-range}, and let $c^{\sgm, p}_{k}$ be a $(-\dlt, S)$ frequency envelope for $a$ in $\dot{W}^{\sgm, p}$, which is $\dlt$-compatible with $c_{k}$.  Then we have the bounds
\begin{equation}
\| P_k A(s) \|_{\dot W^{\sgm,p}}+ \| P_k F(s) \|_{\dot W^{\sgm-1,p}}  \lesssim_{\hM, \calE, N} c_k^{\sgm, p} (1+ 2^{2k} s)^{-N}.
\end{equation}
\end{lemma}

\begin{proof}
This is the same as the previous argument but with $O = I$. Note that $M_{1} \aleq_{\hM, \calE} 1$ by Proposition~\ref{p:cal-in-bfH}. Moreover, we have the representation
\begin{equation*}
	A_{j}(s) = - \int_{s}^{\infty} \covD^{\ell} F_{\ell j}(\tilde{s}) \, d \tilde{s}.
\end{equation*}
Hence, the desired bound follows by integrating \eqref{eq:cal-DF-sp} from $\infty$ to $s$. \qedhere
\end{proof}

Next we turn our attention to the linearized caloric flow and and the corresponding projection map.
To set the notations, let $a$ be a caloric connection, and $\tb$ a linearized data set.
Its projection   $b = \Pi_a \tb$ is a linearized caloric data set. Our goals will be to 

\begin{itemize}
\item Provide $\dot W^{\sgm, p}$ frequency envelope bounds for the projection map $\tilde b \to b$.
\item Provide $\dot W^{\sgm, p}$ frequency envelope bounds for the linearized Yang--Mills heat flow of $b$.
\end{itemize}

We begin with a frequency envelope bound for the projection map, which is analogous to Lemma~\ref{l:pi-l2}:
\begin{lemma}  \label{l:pi-sp}
Let $a \in \calC_{\hM}$ with energy $\leq \calE$, and with a $\dot{H}^{1}$ $(-\dlt, S)$ frequency envelope. Let $\tb \in L^{2}$, and let $b = \Pi_{a} \tilde{b}$ be its caloric projection.

\begin{enumerate}
\item Let $(\sgm, p)$ be in the range $-2 < \sgm < \frac{4}{p} - 1$, and let $d^{\sgm, p}_{k}$ be a $(-\dlt, S)$ frequency envelope for $\tb$ in $\dot{W}^{\sgm, p}$, which is $\dlt$-compatible with $c_{k}$. Then $b$ satisfies the bound
\begin{equation} \label{b-wsp}
	\nrm{P_{k} b}_{\dot{W}^{\sgm, p}} \aleq_{\hM, \calE} d_{k}^{\sgm, p}.
\end{equation}

\item Suppose that $\frac{4}{p} - 1 \leq \sgm < \frac{4}{p}$. Let $c^{\sgm, p}_{k}$ be a $(-\dlt, S)$ frequency envelope for $a$ in $\dot{W}^{\sgm, p}$, and let $d_{k}$, respectively $d^{\sgm, p}_{k}$ be $(-\dlt, S)$ frequency envelopes for $\tb$ in $\dot{H}^{1}$, respectively $\dot{W}^{\sgm, p}$, which are $\dlt$-compatible with $c_{k}$. Assume also that $d'_{k} = c_{k} d_{k}^{[1]} + d_{k} c_{k}^{[1]}$ is a $(-\dlt, S)$ frequency envelope for $\rd^{\ell} \tb_{\ell}$ in $L^{2}$. Then
\begin{equation} \label{b-wsp+}
	\nrm{P_{k} b_{j}}_{\dot{W}^{\sgm, p}} + \nrm{P_{k} F_{0j}(s)}_{\dot{W}^{\sgm, p}} \aleq_{\hM, \calE} d_{k}^{\sgm, p} + c^{\sgm, p}_{k} (c \cdot d)_{\leq k}.
\end{equation}
\end{enumerate} 
\end{lemma}
Note that the frequency envelope in \eqref{b-wsp+} is bounded by $e^{\sgm, p}_{k}$ in \eqref{eq:e-sp}.
\begin{proof}
\pfstep{Proof of (1)}
As in the proof of Lemma~\ref{l:pi-l2}, we begin by solving the covariant parabolic flow for $F_{0j}$, with data 
\[
f_{0j} = \tb.
\]
By Theorem~\ref{t:heatB}, this yields the parabolic bounds
\begin{equation}\label{pk-F-0j}
\| P_k F_{0j} \|_{\dot W^{\sgm,p}} \lesssim d_k^{\sgm, p}(1+2^{2k}s)^{-N}.
\end{equation}
Now 
\[
b = \tb - \covD a_0
\]
where 
\[
a_0 = - \int_{0}^\infty \covD^\ell F_{\ell 0} ds.
\] 
By the previous estimate and compatibility, we obtain the bound
\begin{equation}\label{pk-a0}
\| P_k a_0 \|_{\dot W^{\sgm+1,p}} \lesssim d_k^{\sgm,p}
\end{equation}
which in turn leads to the desired bound for $\covD a_0$,
\begin{equation}\label{pk-Da0}
\| P_k \covD a_0 \|_{\dot W^{\sgm,p}} \lesssim d_k^{\sgm,p},
\end{equation}
where we again used compatibility.

\pfstep{Proof of (2)}
Here the bound \eqref{pk-a0} still
holds, but it no longer implies \eqref{pk-Da0}. Precisely, the only difficulty occurs in the 
 expression $[a,a_0]$ for the high-low interactions. In that case we use
 \begin{equation*}
	\nrm{P_{\leq k} a_{0}}_{L^{\infty}} \aleq \sum_{j \leq k} \nrm{P_{j} a_{0}}_{\dot{H}^{2}} \aleq (c \cdot d)_{\leq k},
\end{equation*}
which is derived from the bound on $\rd^{\ell} \tb_{\ell}$ (see the proof of Proposition~\ref{p:fe-ah}). This is combined with the bound $\nrm{P_{k} a}_{\dot{W}^{\sgm, p}} \leq c^{\sgm, p}_{k}$. \qedhere
\end{proof}

Next we consider the regularity of linearized caloric flows, and in particular prove the $B$ bounds in part (1) of the proposition:

\begin{lemma} \label{l:lincal-sp}
Let $a \in \calC_{\hM}$ with energy $\leq \calE$ and with a $\dot{H}^{1}$ $(-\dlt, S)$ frequency envelope $c_{k}$.
Let $(\sgm, p)$ be as in \eqref{sp-range}, and let $d^{\sgm, p}_{k}$ be a $(-\dlt, S)$ frequency envelope for $b$ in $\dot{W}^{\sgm-1, p}$, which is $\dlt$-compatible with $c_{k}$. Then $b$ satisfies the bound
\begin{equation}
\| P_k B(s) \|_{\dot W^{\sgm-1,p}} \lesssim_{\hM, \calE, N} d_k^{\sgm,p} (1+ 2^{2k} s)^{-N}.
\end{equation}
\end{lemma}

\begin{proof}
We proceed exactly as in part (1) of the previous proof (with a caution, however, that the scaling of $d^{\sgm, p}_{k}$ is different!). Observe that
\begin{equation*}
	B_{j}(s) = \rd_{0} A_{j}(s) = F_{0j}(s) - \covD_{j} A_{0}(s).
\end{equation*}
In view of \eqref{pk-F-0j}, it only remains to estimate $\covD A_{0}(s)$. For $A_{0}(s)$, we have the representation
\begin{equation*}
	A_{0}(s) = -\int_{s}^{\infty} \covD^{\ell} F_{\ell}(\tilde{s}) \, d \tilde{s}.
\end{equation*}
Thus, the desired estimate for $\covD A_{0}(s)$ follows by integration from infinity to $s$, as in the proof of Lemma~\ref{l:pi-sp}. \qedhere
\end{proof}

Finally,  we consider difference bound (which is the only use we have for the previous projection bounds): 
\begin{lemma} \label{l:diff-sp}
\begin{enumerate}
\item Consider two caloric connections $a^{(0)}, a^{(1)} \in \calC_{\hM}$ and energy $\leq \calE$ so that \eqref{dA-small} holds. Let $c_{k}$ and $c^{\sgm, p}_{k}$, respectively $d^{\sgm, p}_{k}$, be frequency envelopes for $a^{(0)}, a^{(1)}$, respectively $a^{(0)} - a^{((1)}$, as in the assumptions of Proposition~\ref{p:main-sp}. Then there exists a one parameter family $a^{(h)}$ of caloric data in $\calC_{2 \hM}$ with energy $\leq 2 \calE$, so that   
\begin{equation}\label{a-wsp}
\|P_k a^{(h)} \|_{\dot W^{\sgm,p}} \lesssim_{\hM, \calE} c_k^{\sgm,p},
\end{equation}
as well as
\begin{equation}\label{dha-wsp}
\|P_k \partial_h  a^{(h)} \|_{\dot W^{\sgm,p}} \lesssim_{\hM, \calE}   e_k^{\sgm,p} .
\end{equation}

\item In addition, let $b^{(0)}, b^{(1)}$ be corresponding linearized caloric data sets with $\dot W^{\sgm-1,p}$
$(-\dlt, S)$ frequency envelopes $c_k^{\sgm,p}$, respectively  $d_k^{\sgm,p}$ for the difference, as in the assumptions of Proposition~\ref{p:main-sp}. Then there exists a corresponding family of linearized caloric data sets $b^{(h)}$ with similar bounds,
\begin{equation}\label{bh-wsp}
\|P_k b^{(h)} \|_{\dot W^{\sgm-1,p}} \lesssim_{\hM, \calE}  c_k^{\sgm,p},
\end{equation}
as well as
\begin{equation}\label{dhb-wsp}
\|P_k \partial_h  b^{(h)} \|_{\dot W^{\sgm-1,p}} \lesssim_{\hM, \calE}   e_k^{\sgm,p}.
\end{equation}
\end{enumerate}
\end{lemma}

\begin{proof}
\pfstep{Proof of (1)}
We will prove that the family $a^{(h)}$ constructed in Proposition~\ref{p:path-l2} has all the desired properties.
We use the same notation as there. The first bound \eqref{a-wsp} is a direct consequence of the mapping properties of the caloric projection operator $\Cal$ in Lemma~\ref{l:Cal-sp}.

Next we consider the bound for $\partial_h a^{(h)}$, which is obtained by 
\[
\partial_h a^{(h)} = Ad(O^{(h)}) \rd_{h} \ta^{(h)}_{j} - (\covD[a])_{j} O^{(h)}_{;h} = \Pi_{a^{(h)}} Ad(O^{(h)}) \partial_h \ta^{(h)}  ,
\] 
where the last equality follows from the uniqueness statement in Proposition~\ref{p:transverse}. By Lemma~\ref{l:O-conj} and compatibility, it follows that $d^{\sgm, p}_{k}$, $d_{k}$ and $d'_{k} = c_{k} d^{[1]}_{k} + d_{k} c^{[1]}_{k}$, respectively, are frequency envelopes for $Ad(O^{(h)}) \rd_{h} \ta^{(h)}$ in $\dot{W}^{\sgm, p}$, for $Ad(O^{(h)}) \rd_{h} \ta^{(h)}$ in $\dot{H}^{1}$ and for $\rd^{\ell} (Ad(O^{(h)}) \rd_{h} \ta_{\ell}^{(h)})$ in $L^{2}$, respectively. Then the desired bound \eqref{dha-wsp} follows from Lemma~\ref{l:pi-sp}.(2). 
\pfstep{Proof of (2)}
We proceed again as in Proposition~\ref{p:path-l2}, from which we borrow the notation and equations.
 Given $a^{(h)}$ constructed in part (1), we first define $\tb^{(h)}$ by linearly interpolating between $b^{(0)}$ and $b^{(1)}$.
These have the desired regularity but are not yet on the tangent space of the caloric manifold, so we project,
setting
\[
b^{(h)} = \Pi_{a^{(h)}} \tb^{(h)}.
\]
Now the $b^{(h)} $ bound \eqref{a-wsp} is a consequence of Lemma~\ref{l:pi-sp}.

To estimate $\partial_h b^{(h)}$ we need $\partial_h \covD a_0$.  
\[
\partial_h \covD a_0 = D  \partial_h a_0 + [\partial_h a,a_0] .
\]
The second term is easy to estimate using the previous $\dot W^{\sgm,p}$ bound for $\partial_h a$ and 
the $\dot H^1$ bound
for $a_0$ in Lemma~\ref{l:ah-l2}, as 
\[
\dot W^{\sgm, p} \cdot \dot H^1 \to \dot W^{\sgm-1,p}.
\]

It remains to bound $\partial_h a_0$ in $\dot W^{\sgm,p}$. 
Exactly as in Proposition~\ref{p:path-l2}, we use another round of the ``infinitesimal de Turck trick'',
where we introduce a dynamical component $A_{h}$ satisfying \eqref{eq:ah-dymhf}, \eqref{eq:fjh-heat} and \eqref{eq:ah0=0}. We have
\begin{equation*}
	\rd_{s} \covD_{h} A_{0} = [\covD^{\ell} F_{\ell h}, A_{0}] + [\tensor{F}{_{0}^{\ell}}, F_{\ell 0}] + \covD_{j} \covD_{h} F_{j 0},
\end{equation*}
where $\covD_{h} F_{j0}$ solves the covariant heat equation \eqref{eq:heat-dhF0} with inhomogeneity $G_{j}$ as in \eqref{eq:heat-dhF0-RHS}, and with initial data $\rd_{h} \tb$.

By Theorem~\ref{t:heatB} and the prior bounds at $s =0$, we already have
\[
\| P_k F_{hj} \|_{\dot H^1} \lesssim (d_k + c_{k} (c \cdot d)_{\leq k}) (1+2^{2k} s)^{-N}
, \qquad \| P_k F_{0j} \|_{L^2}  \lesssim c_k(1+2^{2k} s)^{-N},
\]
\[
\| P_k F_{hj} \|_{\dot W^{\sgm,p}} \lesssim (d^{\sgm, p}_k + c^{\sgm, p}_{k} (c \cdot d)_{\leq k})  (1+2^{2k} s)^{-N}
, \qquad \| P_k F_{0j} \|_{\dot W^{\sgm-1,p}}  \lesssim c^{\sgm, p}_k(1+2^{2k} s)^{-N}.
\]
Hence for $G_i$ (which is, schematically, $[\covD F_{h}, F_{0}] + [F_{h}, \covD F_{0}]$) we obtain, by Littlewood--Paley trichotomy,
\[
\| P_k G_i \|_{s^{\frac{1}{2}} \dot W^{\sgm-2,p} + s^{1-} \dot{W}^{\sgm-1 - \dlt, p}} \lesssim     e^{\sgm, p}_k  (1+2^{2k} s)^{-N}.
\]
where, of course, both $s^{\frac{1}{2}} \dot W^{\sgm-2,p}$ and $s^{1-} \dot{W}^{\sgm-1-, p}$ have the same scaling as $\dot{W}^{\sgm-3, p}$.
The worst terms arise from (i) $high \times low$ interaction in $[\covD F_{h}, F_{0}]$, (ii) $low \times high$ interaction in $[F_{h}, \covD F_{0}]$ , and (iii) $high \times high \to low$ for both terms. These give (i) $d^{\sgm, p}_{k} + c^{\sgm, p}_{k} (c \cdot d)_{\leq k}$ in $s^{\frac{1}{2}} \dot W^{\sgm-2,p}$, (ii) $c^{\sgm, p} d^{[1]}_{k}$ in $s^{\frac{1}{2}} \dot W^{\sgm-2,p}$ and (iii) $d^{\sgm, p}_{k} + c^{\sgm, p}_{k} (c \cdot d)_{\leq k}$ in $s^{1-} \dot{W}^{\sgm-1 - \dlt, p}$, respectively. Recall from \eqref{eq:e-sp} that the resulting frequency envelopes add up to $e^{\sgm, p}_{k}$.

Then solving the parabolic equation for $\covD_{h} F_{0i}$ we obtain 
\[
\| P_k \covD_h F_{0i} \|_{\dot W^{\sgm-1,p}} \lesssim  e^{\sgm, p}_k   (1+2^{2k} s)^{-N}.
\]
By \eqref{eq:dsdhA0}, this implies that
\[
\|  \partial_s \covD_h A_0 \|_{\dot W^{\sgm-2,p}} \lesssim  e^{\sgm,p}_k  (1+2^{2k} s)^{-N},
\]
which in turn, after integration from infinity, yields the desired bound for $\covD_h A_0$,
\[
\| \covD_h A_0 \|_{\dot W^{\sgm, p}} \lesssim  e^{\sgm, p}_k   (1+2^{2k} s)^{-N}. \qedhere
\]
\end{proof}

As a consequence of the last result, we are able to provide the difference bounds in part (2) of the proposition.
These are obtained by combining the last lemma with the linearized bounds in Lemma~\ref{l:lincal-sp}.
\end{proof}

Our next goal is to compare our caloric heat flow for $(A,B)$ with the corresponding linear heat flow.
We do this first at the linear level, where we can prove a $\dot W^{\sgm, p}$ counterpart of Proposition~\ref{p:AB2-l2}.
From here on we assume that 
\[
0 < \sgm < \frac{4}{p}, \qquad 2 < p < \infty. 
\]
To the pair $(\sgm,p)$  we associate another pair $(\sgm_1,p_1)$ 
so that the following relations hold:
\[
0 \leq \sgm_1 < \sgm,
\]
respectively,
\[
\frac{4}{p_1}- \sgm_1 = 2(\frac{4}{p} - \sgm) .
\]
These are so that we have the scaling equivalence (and bilinear multiplicative property)
\[
 \dot W^{\sgm,p} \cdot \dot W^{\sgm,p} \sim \dot W^{\sgm_1,p_1}.
\]

\begin{proposition}\label{p:AB2-sp-lin}
 Let $(a,b) \in T^{L^2} \calC$ be a linearized caloric initial data set, with 
$\dot H^1 \times L^{2}$ $(-\dlt, S)$-frequency envelope $c_k$ and $\dot W^{\sgm,p} \times \dot W^{\sgm-1,p}$ $(-\dlt, S)$-frequency envelope $c_k^{\sgm,p}$, which is $\dlt$-compatible with $c_{k}$. Then its Yang--Mills heat flow satisfies the bounds
\begin{equation} \label{cal-linf-af}
 \| (1-s\Delta)^N \bfA_j(s)\|_{\dot W^{\sgm_1,p_1}} + \|  (1-s\Delta)^N\bfF_{ij}(s)\|_{\dot W^{\sgm_1-1,p_1}} 
 \lesssim_{\hM, \calE, N}  2^{-k(s)}  c_{k(s)}^{\sgm,p} c_{k(s)}^{\sgm,p [\dlt]} ,
\end{equation}
\begin{equation}\label{cal-linf-bf}
 \| (1-s\Delta)^N \bfB_{j}(s)\|_{\dot W^{\sgm_1-1,p_1}} + \|  (1-s\Delta)^N\bfF_{i0}(\sgm)\|_{\dot W^{\sgm_1-1,p_1}} \lesssim_{\hM, \calE, N}  2^{-k(s)}  c_{k(s)}^{\sgm,p} c_{k(s)}^{\sgm,p [\dlt]}
\end{equation}
respectively,
\begin{equation}\label{cal-linf-da}
 \|P_k \DA\|_{\dot W^{\sgm_1,p_1}} +   \|P_k \DB\|_{\dot W^{\sgm_1-1,p_1}}\lesssim_{\hM, \calE}  c_{k}^{\sgm,p} c_{k}^{\sgm,p [\dlt]}  .
\end{equation}

Similarly, if $(a^{(0)},b^{(0)})$ and $(a^{(1)},b^{(1)})$ are two close linearized caloric initial data sets with frequency envelopes $c_{k}, c_{k}^{\sgm, p}, d_{k}, d_{k}^{\sgm, p}$ satisfying the assumptions of Proposition~\ref{p:main-sp},
then we have the difference bounds
\begin{equation} \label{cal-linf-af-diff}
\begin{aligned}
& \| (1-s\Delta)^N \delta \bfA_j(s)\|_{\dot W^{\sgm_1,p_1}} + \|  (1-s\Delta)^N\delta \bfF_{ij}(s)\|_{\dot W^{\sgm_1-1,p_1}} \\
&  \lesssim_{\hM, \calE, N}  2^{-k(s)} (c_{k(s)}^{\sgm, p} e_{k(s)}^{\sgm, p [\dlt]} + e_{k(s)}^{\sgm, p} c_{k(s)}^{\sgm, p [\dlt]}),
\end{aligned}
\end{equation}
\begin{equation} \label{cal-linf-bf-diff}
\begin{aligned}
& \| (1-s\Delta)^N \delta \bfB_{j}(s)\|_{\dot W^{\sgm_1-1,p_1}}  + \|  (1-s\Delta)^N\delta \bfF_{i0}(s)\|_{\dot W^{\sgm_1-1,p_1}} \\
 & \lesssim_{\hM, \calE, N}  2^{-k(s)} (c_{k(s)}^{\sgm, p} e_{k(s)}^{\sgm, p [\dlt]} + e_{k(s)}^{\sgm, p} c_{k(s)}^{\sgm, p [\dlt]}),
\end{aligned}
\end{equation}
respectively,
\begin{equation}\label{cal-linf-da-diff}
 \|P_k \delta \DA(s)\|_{\dot W^{\sgm_1,p_1}} + \|P_k \delta \DB(s) \|_{\dot W^{\sgm_1-1,p_1}} \lesssim_{\hM, \calE}  c_{k}^{\sgm, p} e_{k}^{\sgm, p [\dlt]} + e_{k}^{\sgm, p} c_{k}^{\sgm, p [\dlt]}.
\end{equation}
Here for all $\DB$ bounds we make the additional assumption $\sgm > \frac12$.

\end{proposition}

\begin{proof}
  We proceed as in the proof of Proposition~\ref{p:AB2-l2}. Again, the proof is tedious but straightforward. 
  We focus on the key structural aspects, and omit the details.
  
  First we
  estimate $\bfF_{ij}$, using the representation \eqref{rep-F}.  The
  same argument applies to $F_{0j}$.  For the bilinear terms in the
  integral we estimate both factors using the $\dot W^{\sgm, p}$ envelopes.
   To fix the notations, consider the worst
  term $[A,\partial F]$.  Then using the Littlewood-Paley trichotomy
  we have two main contributions,
\[
\bfF_{ij}(s) \approx 2^{-2k} P_{\leq k(s)} [A_{<k(s)},\partial F_{<k(s)}] + \sum_{k > k(s)} 2^{-2k} P_{\leq k(s)} [A_k,\partial F_k] 
\]
In the second term we have additional off-diagonal decay as we only need to apply Bernstein for the product.
In the first, however, if $A$ has lower frequency then we need to apply  Bernstein separately for $A$, and so we can only use
an  $L^q$ bound for $A$.

Next we estimate $\DA$, for which we use the representation \eqref{rep-DA}. 
The integrand is quite similar to the one above, but now we integrate to infinity. Thus the leading term 
is 
\[
\DA(s) \approx   \sum_{k_1,k_2 \leq k(s)}  2^{-2k_{max}}  [A_{k_1},\partial F_{k_2}]
\] 
and when we consider $P_k \DA(s)$ we arrive at the same two cases as above, with the same final result.

To estimate $A_i$ we use the same analysis one derivative higher, via \eqref{rep-A}. The same applies to $A_0$
via \eqref{A0-bi}, which implies the $B$ bound.

Finally we consider the $\DB$ bound, where we use 
\[
\covD^k B_k = - 2\int_{0}^\infty [B^j, \covD^i F_{ij}] ds 
\]
The worst contribution is in the $high \times high \to low$ case,
\[
\covD^k B_k \approx  \sum_{j > k}  2^{-2j} P_{k} (P_{j} B \cdot \partial P_{j} F).
\]
The derivative in front of $P_{j} F$ gives one more $2^{j}$, so we only have $2^{-j}$ left. This is where we need to assume that $\sgm > \frac{1}{2}$.

The estimates for the differences are similar, using part (2) of Proposition~\ref{p:main-sp} and Proposition~\ref{p:diff-l2}.
\end{proof}

We end this subsection by fixing some parameters for the ensuing analysis. For our goal below, we will need to work with five different sets of exponents $(\smin,\pmin)$, $(\smid,\pmid)$, \ldots, $(\smaxp,\pmaxp)$. Their choice is somewhat flexible, within a range.
We describe it in the following table\footnote{$L^{q}_{t}$ refers to the $t$-integrability in the matching $\dot{H}^{1}$-Strichartz norm.}:
\renewcommand{\arraystretch}{3}

\begin{table}[ht]
\begin{tabular}{| c | c | c | c | c | c |}
\hline
$(\sgm, p)$ & Scaling &$L^q_t$ match & $\sgm$ constraint     &$p$  constraint     & $(\sgm_1,p_1)$ 
\\ 
\hline
$(\smin,\pmin)$  & $\dot H^1$ & $L^\infty$ &  $\dfrac12 < \smin < 1$  &  $2 < p < 4$   & $(0,2)$
\\
\hline
$(\smid,\pmid)$ & $\dot H^{\frac54}$ & $L^4$&   $\dfrac12 < \smid < \dfrac{7}{12} $ & $3 < p < \dfrac{16}5$   &  $(\dfrac12,2)$
\\
\hline
$(\smidp,\pmidp)$ & $\dot H^{\frac54+\smexp}$ & $L^{4-}$&  $\dfrac12+2\smexp < \smidp < \dfrac{7}{12}-\dfrac{5\smexp}{3} $ & $p < 3$   &  $(\dfrac12+2\smexp,2)$
\\ 
\hline
  $(\smax,\pmax)$ &  $\dot H^{\frac43}$ &$ L^3$&  $\dfrac13 < \smax < \dfrac{4}{9} $  &  $ \dfrac{18}5  < p < 4$     &  $(\dfrac13,\dfrac{12}5)$
\\ 
\hline
  $(\smaxp,\pmaxp)$ &  $\dot H^{\frac32-2\smexp}$ & $ L^{2+}$ &  $\dfrac16-4\smexp  < \smaxp < \dfrac{1}{6}-\dfrac{10\smexp}{3} $  &  $  p < 6$     &  $(\dfrac16-4\smexp,\dfrac{24}{5})$
\\ 
\hline
\end{tabular}
\end{table}

Here $\smexp$ is a small parameter,
\[
0 < \smexp < \frac{1}{22}.
\]
The above proposition applies the these sets of indices as follows: 

\begin{corollary}
\begin{enumerate}
\item All bounds in Proposition~\ref{p:AB2-sp-lin} apply  for the sets of indices 
$(\smin,\pmin)$, $(\smid,\pmid)$ and $(\smidp,\pmidp)$.
\item  The bounds in Proposition~\ref{p:AB2-sp-lin}, except for the $\DB$ bounds, 
apply   for the sets of indices $(\smax,\pmax)$  and  $(\smaxp,\pmaxp)$.
\end{enumerate}
\end{corollary}

To keep the notation simpler, in what follows we denote the corresponding homogeneous Sobolev spaces,
and the associated frequency envelopes by $\dot W^{(j)}$, respectively $c_k^{(j)}$ for $j = 0,1,2,3,4$.
For the most part, the following embeddings will suffice for our estimates:
\begin{equation}
\begin{split}
\dot W^{(0)} \subset \dot W^{\frac12, \frac83}, \qquad \dot W^{(1)} \subset \dot W^{\frac12, \frac{16}5}, \qquad 
 \dot W^{(2)} \subset \dot W^{\frac12-\smexp, 3}, \\
\dot W^{(3)} \subset \dot W^{\frac13, 4}, \qquad
 \dot W^{(4)} \subset \dot W^{\frac16-2\smexp, 6}.
\end{split}
\end{equation}
The reason we go past the range of these embeddings is to be able to gain off-diagonal decay in several quadratic
and cubic estimates.  

We fix a universal constant $\dlt_{0} > 0$, which is sufficiently small relative to the five pairs $(\smin, \pmin)$, \ldots, $(\smaxp, \pmaxp)$, as well as $\smexp$. This will be our lower admissibility range, as well as compatibility parameter, for all the frequency envelopes we use. 

Given a linearized caloric data set $(a, b) \in T^{L^{2}} \calC$, we let $\hM, \calE$ be upper bounds $\hM(a) \leq \hM$ and $\spE[a] \leq \calE$. We let $c_{k}$, respectively $c_{k}^{(j)}$, be $(-\dlt_{0}, S)$ frequency envelopes for $(a, b)$ in $\dot{H}^{1} \times L^{2}$, respectively in $\dot{W}^{\sgm^{(j)}, p^{(j)}} \times \dot{W}^{\sgm^{(j)}-1, p^{(j)}}$. 

Next, given a pair of linearized caloric data sets $(a^{(0)}, b^{(0)}), (a^{(1)}, b^{(1)}) \in T^{L^{2}} \calC$, with again let $\hM, \calE$ be upper bounds $\hM(a^{(0)}), \hM(a^{(1)}) \leq \hM$ and $\spE[a^{(0)}], \spE[a^{(1)}] \leq \calE$, which are close in the sense of \eqref{dA-small}. We let $c_{k}$, respectively $c_{k}^{(j)}$, be joint $(-\dlt_{0}, S)$ frequency envelopes in $\dot{H}^{1} \times L^{2}$, respectively in $\dot{W}^{\sgm^{(j)}, p^{(j)}} \times \dot{W}^{\sgm^{(j)}-1, p^{(j)}}$. For the difference $(a^{(0)} - a^{(1)}, b^{(0)} - b^{(1)})$, we let $d_{k}$, respectively $d_{k}^{(j)}$, be $(-\dlt_{0}, S)$ frequency envelopes in $\dot{H}^{1} \times L^{2}$, respectively in $\dot{W}^{\sgm^{(j)}, p^{(j)}} \times \dot{W}^{\sgm^{(j)}-1, p^{(j)}}$. 

In all cases above, we assume that $c_{k}^{(j)}$, $d_{k}$, and $d_{k}^{(j)}$ are $\dlt_{0}$-compatible with $c_{k}$. 

\subsection{Generalized Coulomb condition and \texorpdfstring{$\bfQ$}{bfQ}}
By now, we have repeatedly seen (and took advantage of the fact) that caloric connections $a$ and their linearizations $b$ satisfy a generalized type of Coulomb gauge condition
\[
\partial^{\ell} a_{\ell} = \DA(a), \qquad \partial^\ell b_\ell = \DB(a,b),
\]
where the smooth maps $\DA$ and $\DB$ contain only quadratic and
higher terms, and have better regularity. 
As a remarkable corollary of the results proved in the preceding subsection, we are now able to provide a better description of these maps. In particular, the main quadratic part is described in terms of the explicit symmetric bilinear form $\bfQ$ with symbol
\begin{equation} \label{eq:Q-sym}
	\bfQ(\xi, \eta) = \frac{\abs{\xi}^{2} - \abs{\eta}^{2}}{2 (\abs{\xi}^{2} + \abs{\eta}^{2})}.
\end{equation}
Later, in the analysis of the hyperbolic Yang--Mills equation, we will use
the explicit form of the quadratic part, while the cubic and higher
terms will only play a perturbative role.

\begin{proposition}\label{p:DAB}
 Let $(a,b) \in T^{L^{2}} \calC_{\hM}$ be a linearized caloric data set with energy $\leq \calE$. 
Then $\rd^{\ell} a_{\ell} = \DA(a)$ and $\rd^{\ell} b_{\ell} = \DB(a, b)$ decompose into the quadratic and the higher order parts
\begin{align*}
\DA(a) =& \bfQ(a, a) + \DA^3(a), \\
\DB(a, b) =& \frac{1}{2} \left( [a, b] + 2 \bfQ(a, b) \right) + \DB^{3}(a,b),
\end{align*}
where $\bfQ$ is the symmetric bilinear form with symbol \eqref{eq:Q-sym}, and the remainders $\DA^3$, $\DB^3$ are  maps containing cubic and higher order terms. Under the assumptions at the end of Section~\ref{subsec:lp}, they obey the following bounds:
\begin{align}
 \|P_k \DA^3\|_{L^{2}}+  \|P_k \DB^3\|_{\dot{H}^{-1}} \lesssim_{\hM, \calE} & c_{k}^{(0)}(c_{k}^{(0), [\dlt_{0}]})^2, \label{DAB3-0}  \\
  \|P_k \DA^3\|_{\dot H^{\frac{1}{2}}}+  \|P_k \DB^3\|_{\dot{H}^{-\frac{1}{2}}} \lesssim_{\hM, \calE} & c_{k}^{(1)}(c_{k}^{(1), [\dlt_{0}]})^2, \label{DAB3-1}  \\
 \|P_k \DA^3\|_{\dot H^1}+  \|P_k \DB^3\|_{L^{2}} \lesssim_{\hM, \calE} & c_{k}^{(3)}(c_{k}^{(3), [\dlt_{0}]})^2. \label{DAB3}
\end{align}
as well as the corresponding difference bounds
\begin{align}
 \|P_k \delta \DA^3\|_{L^{2}}+  \|P_k \delta \DB^3\|_{\dot{H}^{-1}} \lesssim_{\hM, \calE} &   (c_{k}^{(0), [\dlt_{0}]})^2 e_{k}^{(0), [\dlt_{0}]} , \label{dDAB3-0} \\
 \|P_k \delta \DA^3\|_{\dot{H}^{\frac{1}{2}}}+  \|P_k \delta \DB^3\|_{\dot{H}^{-\frac{1}{2}}} \lesssim_{\hM, \calE} &   (c_{k}^{(0), [\dlt_{0}]})^2 e_{k}^{(0), [\dlt_{0}]} , \label{dDAB3-1} \\
  \|P_k \delta \DA^3\|_{\dot H^1}+  \|P_k \delta \DB^3\|_{L^{2}} \lesssim_{\hM, \calE} &   (c_{k}^{(3), [\dlt_{0}]})^2 e_{k}^{(3), [\dlt_{0}]} . \label{dDAB3} 
\end{align}
\end{proposition}
We will also often write
\begin{equation*}
	\DA^{2}(a) = \bfQ(a, a), \qquad
	\DB^{2}(a, b) = \frac{1}{2} \left( [a, b] + 2 \bfQ(a, b) \right).
\end{equation*}
To understand their mutual relation, note that $\DB$ is the \emph{not} the linearization of $\DA$, but rather $\covD^{\ell} b_{\ell} = \DB + [a^{\ell}, b_{\ell}]$ is.
We also remark that \eqref{DAB3} and \eqref{dDAB3}, respectively \eqref{DAB3} and \eqref{dDAB3}, will be dynamically accompanied with $L^{2}_{t}$, respectively $L^{1}_{t}$, in Section~\ref{s:wave} below.

\begin{proof}
We use the representation \eqref{rep-DA}. The leading quadratic part is obtained 
using the linear heat flow for $A$, and has the form 
\[
\DA^2(a) = \int_0^{\infty} [e^{s \Delta} a^j,    \partial^k (e^{s \Delta} \partial_k a_j -  e^{s \Delta} \partial_j a_k)] ds.
\]
Integrating this and symmetrizing in $j$ and $k$ yields the symbol
\[
 \DA^2 (\xi,\eta) = - \frac{\eta^2}{\xi^2+\eta^2} .
\]
The desired expression \eqref{eq:Q-sym} follows after antisymmetrization.

To estimate $\DA^{3}$ we write
\[
\DA^3(a) = \int_0^{\infty} [\bfA^2,    \covD^k F_{kj}]  +  [e^{s \Delta} a^j,[A^k, F_{kj}]] + [e^{s \Delta} a^j, \partial^k \bfF_{kj}] \, ds.
\]
All of \eqref{DAB3-0}, \eqref{DAB3-1} and \eqref{DAB3} are proved by estimating the integral on the right by Littlewood--Paley trichotomy.
\begin{itemize}
\item For \eqref{DAB3-0}, we use the $\dot W^{\smin, \pmin} \subset L^{4}$ bound for $A$ and $e^{s \Delta} a$, the $\dot{W}^{-1, 4}$ bound for $F$, the $L^{2}$ bound for $\bfA^{2}$ and the $\dot{H}^{-1}$ bound for $\bfF$.
\item For \eqref{DAB3-1}, we use the $\dot W^{\smid, \pmid} \subset L^{\frac{16}{3}}$ bound for $A$ and $e^{s \Delta} a$, the $\dot{W}^{-1, \frac{16}{3}}$ bound for $F$, the $\dot{H}^{\frac{1}{2}}$ bound for $\bfA^{2}$ and the $\dot{H}^{-\frac{1}{2}}$ bound for $\bfF$.
\item For \eqref{DAB3}, we use the $\dot W^{\smax,\pmax} \subset L^6$ bound for $A$ and $e^{s \Delta} a$, the 
$\dot W^{\smax-1,\pmax}\subset \dot W^{-1,6}$ bound for $F$, the $\dot W^{\frac13,\frac{12}5}$ bound for $\bfA^2$ and the $\dot W^{-\frac23,\frac{12}5}$ bound for  $\bfF$. 
\end{itemize}

The argument for $\DB^3$ is similar, there we use the representation 
\[
\covD^k B_k = - 2\int_{0}^\infty [B^j, \covD^i F_{ij}] ds 
\]
where we expand all terms as the linear heat flow plus a quadratic
error, so that
\[
\DB^{3}  = - 2\int_{0}^\infty [\bfB^j, \covD^i F_{ij}] + [e^{t\Delta} b^j, [A^i,F_{ij}]] +     [e^{t\Delta} b^j,\partial^j \bfF_{kj}]  + [B^j, \covD_j \partial^k A_k] ds .
\]
Here $\partial^k A_k$ yields only cubic contributions. \qedhere
\end{proof}

\section{The dynamic Yang--Mills heat flow and the 
caloric Yang--Mills waves}
Consider a sufficiently regular space-time connection $A_{t,x}$ on $J \times \bbR^{4}$, which solves the inhomogeneous hyperbolic Yang--Mills equation
\begin{equation} \label{YMWC-s}
	\covD^{\alp} F_{\alp \bt} = w_{\bt}.
\end{equation}
Here, $w$ is called the \emph{Yang--Mills tension field}, and satisfies the constraint equation
\begin{equation}\label{w-constraint}
\covD^\beta w_\beta = 0.
\end{equation}
Assume in addition that for each $t$, $A(t) = A_{x}(t)$ is a caloric connection and $B(t) = \rd_{t} A_{x}(t) \in T_{A(t)}^{L^{2}} \calC$; in short, we call $A_{t,x}$ an \emph{inhomogeneous caloric Yang--Mills wave}.

To take advantage of the caloric gauge condition, we extend $A_{t,x} = A_{t,x,s}$ as a dynamic Yang--Mills heat flow on $J \times \bbR^{4} \times [0, \infty)$. Precisely, we adjoint the heat-time $s \in [0, \infty)$ and consider the dynamic Yang--Mills heat flow $A_{t, x, s}$
\begin{equation*}
	F_{s \alp} = \covD^{\ell} F_{\ell \alp}, \qquad A_{t,x}(t, x, s=0) = A_{t,x}(t, x),
\end{equation*}
under the local caloric gauge condition $A_{s} = 0$. By the (global) caloric gauge assumption, the Yang--Mills heat flow $A(t, x, s)$ exists globally in heat-time $s$, and tends to $0$ as $s \to \infty$. Afterwards, it follows that $A_{0} (t,x,s) = A_{t}(t, x, s)$ also exists globally in heat-time $s$ and tends to $0$ as $s \to \infty$.

In order to study the problem \eqref{YMWC-s} in the caloric gauge, we
first need to clarify what is a proper initial data set.  Our starting
point is the notion, introduced earlier, of a gauge invariant data set
$(a,e) \in \dot H^1 \times L^2$, where $e_j= F_{0j}$ is subject to the
constraint $\covD^j e_j=w_0$.  On the other hand, once the gauge is
fixed we expect to have control of a full initial data set
$(A, \partial_t A)$. However, these are not all independent due to the
gauge condition, and at least conceptually, we expect to see the same
pattern as in the Coulomb gauge, namely that $A = A_x \in \calC$ and
$B = \partial_t A_x \in T_A^{L^2} \calC$ are the independent
variables. Thus, we will call the pair $(A(t), B(t))$ the initial data
for the Yang--Mills connection $A_{t,x}$ at time $t$ in the caloric
gauge (see Definiton~\ref{d:cal-data}).

Our goals are now three-fold:
\begin{itemize}

\item  To establish a one to one correspondence between the two initial data sets
$(a,e)$ and $(A, B)$.

\item To show that the remaining initial data components $A_0$ and $\partial_0 A_0$ 
can be recovered in an elliptic fashion from $A_{x}$ and $B_{x}$.

\item To understand the evolution of $w_{\nu}$ with respect to the heat-time.
\end{itemize}

Of course, our main interest lies in the homogeneous case $w_{\nu} =
0$; for this purpose, it is not immediately apparent why the third
goal is important. However, it will shortly become clear that there
are multiple reasons.  On the one hand, this turns out to be closely
related to the second goal above, even for the homogeneous case
$w_{\nu} = 0$. On the other hand, knowing that the dynamic Yang--Mills
heat flow $A(s)$ at a heat-time $s$ is a good approximate hyperbolic
Yang--Mills connection plays a key role in our induction on energy
argument in \cite{OTYM2}.

Unrelated to the above objectives, in the last part of this section 
we turn the tables and prove that we can transfer some $L^\infty$
type bounds in the opposite direction, namely from the curvature 
$(f,e)$ to the caloric data $(a,b)$. This part has no further continuation
in the present paper, but will be very useful in the next article \cite{OTYM2}
in the context of the energy dispersion.

\bigskip

We begin with the equivalence of the two notions of initial data sets:
\begin{theorem} \label{t:data}
\begin{enumerate}
\item Given any Yang--Mills initial data pair $(a_k,e_k) \in \dot H^1 \times L^2$
 such that $\hM(a) < \infty$, there exists a unique caloric gauge
  Yang--Mills data set $(\tilde a_k,b_k) \in \dot H^1 \times L^2$ and $a_0 \in \dot
  H^1$,  so that the initial data pair
  $(\tilde a_k, \tilde e_k)$ is gauge equivalent to $(a_k,e_k)$, where
\[
 \tilde e_k = b_k - \covD_k a_0 .
\]
    In addition, $(\ta, b)$ and $a_0$ are unique up to constant
    gauge transformations, and depend continuously on $(a,e)$ in the
    corresponding quotient topology.  Further, the map $(a,e) \mapsto
    (\ta,b)$ is locally $C^1$ in the stronger topology\footnote{Here
      we impose again the condition $\lim_{\abs{x} \to \infty} O(a) =
      I$ in order to fix the choice of $O(a)$.}  $\bfH \times L^2 \to
    \bfH \times L^2$, as well as in more regular spaces $H^{N} \times H^{N-1} \to H^{N}
    \times H^{N-1}$ $(N \geq 2)$.

\item Given any  caloric gauge data $(a_k,b_k) \in T^{L^{2}} \calC$, there exists an unique $a_0 \in \dot H^1$,
depending smoothly on $(a_k,b_k)$ so that 
\[
e_k = b_k - \covD_k a_0 
\]
satisfies the constraint equation \eqref{eq:YMconstraint}. Further, the map $(a, b) \to a_{0}$ is also Lipschitz from $H^{N} \times H^{N-1} \to H^{N}$ for $N \geq 3$.
\end{enumerate}
\end{theorem}
This proves Theorem~\ref{t:data-simple}. 
\begin{proof}
\pfstep{Proof of (1)} For the first part we note that we can first place $a_j$ in the caloric gauge, and 
thus reduce the problem to the case when $\tilde a_j = a_j$. Then the fact that $e_k$ 
and $\tilde e_k$ are gauge equivalent simply means that $e = \tilde e$. 

Both the existence and the uniqueness part for the decomposition
\[
 e_k = b_k - \covD_k a_0 
\]
comes from Proposition~\ref{p:transverse}.

\pfstep{Proof of (2)}
For the second part, we note that the divergence equation for $e_k$ gives
\[
\covD^k \covD_k a_0 = \covD^k b_k + w_0
\]
so that $a_0$ is obtained by solving this elliptic equation, see Theorem~\ref{t:deltaA}. \qedhere

\end{proof}

Next we turn our attention to the expressions for $A_0$ and $\partial_0 A_0$.
For $A_0$ we will directly use the above elliptic equation,
\begin{equation}
\covD^k \covD_k A_0(s) = \covD^k B_k(s) + w_0(s).
\end{equation}
In particular this will uniquely identify $A_0(0)$ as a smooth 
function 
\begin{equation}
A_0 = \bfA_0 (A,B)= \bfA_0^2(A,B) + \bfA_0^3(A,B)
\end{equation}
where we will further separate the quadratic part and the higher order terms.

Alternately, we can also obtain $A_0$ by integrating \eqref{A0-cal}
to obtain the following formula (see Remark~\ref{r:proj-a0}):
\begin{equation} \label{eq:l-caloric-A0}
A_{0}(s) = \int_{s}^{\infty} \covD^{\ell} F_{0 \ell}(s') \, \ud s' = \int_{s}^{\infty} w_0(s') \, \ud s'.
\end{equation}
We will use this expression to gain control of $\covD^0 A_0$.
Indeed, differentiating with respect to $t$ we arrive at
\begin{equation} \label{eq:l-caloric-d0a0}
\rd^{0} A_{0}(s) = \int_{s}^{\infty} \partial^0 w_0(s')  \, \ud s' =  - \int_{s}^{\infty}   \covD^k w_k(s')
+[A^0,w_0(s')]  \, \ud s'
\end{equation}

To continue we need to understand the evolution of $w_0$, which is
coupled to the evolution of all $w_\nu$'s:

\begin{lemma}[Deformation of the Yang--Mills
  tension] \label{lem:dymhf-tension} 
Let $A = A_{j} \, \ud x^{j} +  A_{s} \, \ud s$ be a sufficiently regular dynamic covariant 
Yang--Mills heat flow, i.e. solution 
to  \eqref{eq:cdYMHF}. Then the  Yang--Mills tension $w_{\mu}$ obeys the
  following covariant parabolic equation.
\begin{equation} \label{eq:ymhf-w}
\covD_{s} w_{\nu} - \covD^{\ell} \covD_{\ell} w_{\nu} = 2 \LieBr{\tensor{F}{_{\nu}^{\ell}}}{w_{\ell}}
+ 2 \LieBr{F^{\mu \ell}}{\covD_{\mu} F_{\nu \ell} + \covD_{\ell} F_{\nu \mu}}
\end{equation}
\end{lemma}
For a proof, see \cite[Appendix~A]{Oh1}. We remark that in the last expression
by symmetry all terms cancel unless $\mu = 0$, so we can rewrite it as
\begin{equation} \label{eq:ymhf-w1}
\covD_{s} w_{\nu} - \covD^{\ell} \covD_{\ell} w_{\nu} = 2 \LieBr{\tensor{F}{_{\nu}^{\ell}}}{w_{\ell}}
+ 2 \LieBr{F^{0 \ell}}{\covD_{0} F_{\nu \ell} + \covD_{\ell} F_{\nu 0}}
\end{equation}
For this system to be self-contained at fixed $t$, we need to avoid the $D_0$ derivatives 
on the right. This is achieved differently depending on whether $\nu $ is zero or not.
For $\nu \neq 0$ we simply apply the Bianchi identities to get
\begin{equation} \label{eq:ymhf-w2}
\covD_{s} w_{\nu} - \covD^{\ell} \covD_{\ell} w_{\nu} = 2 \LieBr{\tensor{F}{_{\nu}^{\ell}}}{w_{\ell}}
+ 2 \LieBr{F^{0 \ell}}{\covD_{\nu} F_{0 \ell} + 2\covD_{\ell} F_{\nu 0}}, \qquad \nu \neq 0
\end{equation}
which does not involve $\nu = 0$ at all. On the other hand if $\nu = 0$ then we 
have
\begin{equation} \label{eq:ymhf-w3}
\covD_{s} w_{0} - \covD^{\ell} \covD_{\ell} w_{0} = 2 \LieBr{\tensor{F}{_{0}^{\ell}}}{w_{\ell}}
- 2 \LieBr{{F_{0}}^\ell}{w_\ell + \covD^k F_{k \ell}} =- 2 \LieBr{{F_{0}}^\ell}{ \covD^k F_{k \ell}} .
\end{equation}

Thus the above computation shows that we can express $\covD^0 A_0(0)$ as a function
\[
\covD^0 A_0 = \DA_0(A, B):=  \DA_0^2(B, B) +  \DA_0^3(A, B)
\]
which is again decomposed into a quadratic term and a higher order term. 
The aim of the remaining subsections is to make all these decompositions
quantitative rather than qualitative.

In what follows, $A$ or $B$ without any subscripts refer to the spatial components $A_{x}$ or $B_{x}$. Moreover, $A$, $B$, $A_{0}$, $B_{0}$ etc. without $(s)$ refers to the corresponding components at $s = 0$. We use the convention set up at the end of Section~\ref{subsec:lp}, with $(a, b)$ is replaced by $(A, B)$.

\subsection{The analysis of \texorpdfstring{$w_\nu$}{w-nu}}

We begin with the case when the initial data for $w$ is $w(s=0) = 0$, i.e., our 
map is a homogeneous Yang--Mills wave. Then we have the following:
\begin{proposition}\label{p:w}
Let $A_{t, x}$ be a caloric Yang--Mills wave on $I \times \bbR^{4}$ satisfying $(A_{0}, A) \in C_{t}(I; \dot{H}^{1} \times \calC_{\hM})$ with $\calE[a] \leq \calE$. 
Then at fixed heat-time $s > 0$ we have
\begin{equation}
w(s) = \bfw(A,B,s) = \bfw^2(A,B,s) + \bfw^3(A,B,s) 
\end{equation}
where the  quadratic part $\bfw^2$  has the form
\begin{equation}
\bfw^2_\nu(B,B) = -2 \bfW( B^l, \partial_\nu B_l - 2\partial_l B_\nu), \qquad \nu \neq 0
\end{equation}
\begin{equation}
\bfw^2_0(A,B) = 2 \bfW( B^l, \partial_k^2 A_l), \qquad \nu \neq 0
\end{equation}
where $\bfW$ is a symmetric bilinear form with symbol 
  \begin{equation} \label{W-sym}
    \begin{aligned}
      \bfW (\xi, \eta, s)
    = & \int_{0}^{s} e^{- (s - s') \abs{\xi + \eta}^{2}} e^{- s' (\abs{\xi}^{2} + \abs{\eta}^{2})} \, \ud s'  \\
      = & - \frac{1}{2 \xi \cdot \eta} e^{ - s \abs{\xi + \eta}^{2}}
      \left( 1 - e^{2 s (\xi \cdot \eta)} \right).
    \end{aligned}
  \end{equation}
Further,  $\bfw$ satisfies the following bounds:
\begin{equation}\label{w-linf}
\| (1-s\Delta)^N  \bfw(s) \|_{\dot H^{-\frac32}} \lesssim_{\hM, \calE, N}  2^{-\frac{k(s)}2} c_{k(s)}^{(0)}  c_{k(s)}^{(0)[\dlt_{0}]}
\end{equation}
\begin{equation}\label{w-l2}
\| (1-s\Delta)^N \bfw(s) \|_{\dot H^{-\frac12-\smexp}} \lesssim_{\hM, \calE, N} 2^{-\smexp k(s)} c_{k(s)}^{(1)} c_{k(s)}^{(1)[\dlt_{0}]}
\end{equation}
respectively 
\begin{equation}\label{w-l1}
\|(1-s\Delta)^N  \bfw^3(s) \|_{\dot H^{-\frac1{12}}} \lesssim_{\hM, \calE, N} 2^{-\frac{k(s)}{12}}  c_{k(s)}^{(2)} c_{k(s)}^{(2)[\dlt_{0}]} c_{k(s)}^{(4)[\dlt_{0}]} + (c_{k(s)}^{(2)[\dlt_{0}]})^2 c_{k(s)}^{(4)} 
\end{equation}
as well as corresponding difference bounds:
\begin{equation}\label{dw-linf}
\| (1-s\Delta)^N  \dlt \bfw(s) \|_{\dot H^{-\frac32}} \lesssim_{\hM, \calE, N}  2^{-\frac{k(s)}2} c_{k(s)}^{(0)[\dlt_{0}]}  e_{k(s)}^{(0)[\dlt_{0}]}
\end{equation}
\begin{equation}\label{dw-l2}
\| (1-s\Delta)^N \dlt \bfw(s) \|_{\dot H^{-\frac12-\smexp}} \lesssim_{\hM, \calE, N} 2^{-\smexp k(s)} c_{k(s)}^{(1)[\dlt_{0}]} e_{k(s)}^{(1)[\dlt_{0}]}
\end{equation}
respectively 
\begin{equation}\label{dw-l1}
\|(1-s\Delta)^N  \dlt \bfw^3(s) \|_{\dot H^{-\frac1{12}}} \lesssim_{\hM, \calE, N} 2^{-\frac{k(s)}{12}}  e_{k(s)}^{(2)[\dlt_{0}]} c_{k(s)}^{(2)[\dlt_{0}]} c_{k(s)}^{(4)[\dlt_{0}]} + (c_{k(s)}^{(2)[\dlt_{0}]})^2 e_{k(s)}^{(4)[\dlt_{0}]} 
\end{equation}
where
\begin{equation*}
	\dlt \bfw^{2,3}(s) = \bfw^{2,3}(A^{(0)}, B^{(0)}, s) -  \bfw^{2,3}(A^{(1)}, B^{(1)}, s).
\end{equation*}
\end{proposition}

\begin{proof}
  Here we use the equations \eqref{eq:ymhf-w2}, respectively
  \eqref{eq:ymhf-w3}, recalling that at the initial time $F_{0\ell} -
  B_{\ell} = \covD_{\ell} A_0$ is a quadratic term which is better
  behaved.

  To compute the leading quadratic component of $w_\nu$ we proceed as
  follows, first for $\nu \neq 0$:
\[
\begin{split}
\bfw^2_\nu \approx & \ 2 \int_{0}^{s_0} e^{(s-s_0)\Delta} [ e^{s\Delta} F^{0l},   
e^{s\Delta}(\covD_\nu F_{0 \ell} + 2 \covD_\ell F_{\nu 0})] ds
\\
\approx & \  - 2 \int_{0}^{s_0} e^{(s-s_0)\Delta} [ e^{s\Delta} B^\ell,   e^{s\Delta}
(\partial_\nu B_\ell  - 2  \partial_\ell B_{\nu})] ds
\\
= & \  -2 \bfW( B^\ell, \partial_\nu B_\ell  - 2  \partial_\ell   B_{\nu})
\end{split}
\]
where $\bfW$ has the symbol
\[
\bfW(\xi,\eta) = \int_{0}^{s_0} e^{-(s-s_0)(\xi+\eta)^2}  e^{-s(\xi^2+\eta^2)} ds
\]
We remark that $B$ has size $(\xi^2+\eta^2)$ localized in the region
$|\xi+\eta| \lesssim s_0^{-\frac12}$.

Next we consider $\nu=0$, where we use  \eqref{eq:ymhf-w3} instead.
Then a computation which is similar to the one above yields
\[
\bfw^2_0 = -2  \bfW( B^\ell,\partial^k \partial_k A_\ell)
\]

To prove the  bounds in the proposition we use \blue{Theorem~\ref{t:heatB-inhom}}.
Thus we need to estimate the right hand side in the equations~\eqref{eq:ymhf-w2}
respectively \eqref{eq:ymhf-w3}. 

For the bound \eqref{w-linf} we use Proposition~\ref{p:main-sp} to estimate
\[
\| (1 - s \lap)^{N} [F(s),\covD F(s)] \|_{\dot H^{-\frac32}} \lesssim 2^{\frac{3k(s)}2} (c_{k(s)}^{(0)})^2
\]
and then apply heat flow bounds.

For the bound \eqref{w-l2} we use the same bounds  to similarly estimate 
\[
\| (1 - s \lap)^{N} [F(s),\covD F(s)] \|_{\dot H^{-\frac12 - \smexp}} \lesssim_E 2^{-(2+\smexp)k(s)} (c_{k(s)}^{(1)})^2
\]
The same applies for \eqref{w-l1}.   Here we use again Proposition~\ref{p:main-sp}
 for $F$,  while the contribution of the
nonlinear terms $\bfF$ in $F$ is easy to account for based on Proposition~\ref{p:AB2-sp-lin}. 

Finally, the difference bounds are proved similarly. \qedhere

\end{proof}

\subsection{The analysis of \texorpdfstring{$A_0$}{A0}}
Our main result is as follows:

\begin{proposition}\label{p:A0}
Let $A_{t, x}$ be a caloric Yang--Mills wave on $I \times \bbR^{4}$ satisfying $(A_{0}, A) \in C_{t}(I; \dot{H}^{1} \times \calC_{\hM})$ with $\calE[a] \leq \calE$.  Then for $A_0$ we have the representation:
\begin{equation}\label{a0-rep}
A_0 = \bfA_0(A, B) =  \bfA_0^2(A, B) +   \bfA_0^3(A, B)
\end{equation}
where $\bfA_0^2(A,B) $ is a bilinear form of the form
\begin{equation} \label{A0-sym}
\bfA_0^2(A, B) = (-\lap)^{-1} ([A, B] + 2 \bfQ(A, B)).
\end{equation}
and  $\bfA_0^3(A, B)$ is a higher order term, linear in $B$, so that the following bounds hold:
\begin{equation}\label{A0-est1}
\| P_k  \bfA_0^{2,3}(A,B) \|_{\dot H^1} \lesssim_{\hM, \calE} c_k^{(0)} c_k^{(0)[\dlt_{0}]}
\end{equation}
and (corresponding to $L^{2}_t$)
\begin{equation}\label{A0-est2}
\| P_k \bfA_0^{2,3}(A,B) \|_{\dot H^{\frac32}} \lesssim_{\hM, \calE} c_k^{(1)} c_k^{(1)[\dlt_{0}]}
\end{equation}
as well as (corresponding to $L^1_t$)
\begin{equation}\label{A03-est}
    \|P_k \bfA_0^3(A,B)  \|_{\dot H^2} \lesssim_{\hM, \calE} c_k^{(2)[\dlt_{0}]} c_k^{(2)} c_k^{(4)} + (c_k^{(2)})^2  c_k^{(4)[\dlt_{0}]}.
\end{equation}
We also have the corresponding difference bounds:
\begin{equation}\label{dA0-est1}
\| P_k  \dlt \bfA_0^{2,3} \|_{\dot H^1} \lesssim_{\hM, \calE} c_k^{(0)[\dlt_{0}]} e_k^{(0)[\dlt_{0}]},
\end{equation}
\begin{equation}\label{dA0-est2}
\| P_k \dlt \bfA_0^{2,3} \|_{\dot H^{\frac32}} \lesssim_{\hM, \calE} c_k^{(1)[\dlt_{0}]} e_k^{(1)[\dlt_{0}]},
\end{equation}
\begin{equation}\label{dA03-est}
    \|P_k \dlt \bfA_0^3  \|_{\dot H^2} \lesssim_{\hM, \calE} e_k^{(2)[\dlt_{0}]} c_k^{(2)[\dlt_{0}]} c_k^{(4)[\dlt_{0}]}  + (c_k^{(2)[\dlt_{0}]})^2  e_k^{(4)[\dlt_{0}]},
\end{equation}
where
\begin{equation*}
\dlt \bfA_0^{2,3} = \bfA_0^{2,3}(A^{(0)},B^{(0)}) - \bfA_0^{2,3}(A^{(1)},B^{(1)}).
\end{equation*}

\end{proposition}
\begin{proof}
We only sketch the proof, emphasizing the structural points. 

For $A_0$ we already have the 
elliptic equation
\[
\Delta_A A_0 = \covD^k B_k 
\]
On the other hand for $\covD^k b_k$ we have  the representation in 
Proposition~\ref{p:DAB}.  Thus we have 
\[
\Delta_A A_0 = \DB^2(a,b) + \DB^3(a,b)
\]
In particular the quadratic part of $A_0$ is given by 
\[
\bfA_0^2(a,b) = \Delta^{-1} \DB^2(a,b)
\]
and its symbol is directly obtained from the symbol of  $\DB^2$,
\[
\bfA_0^2(\xi,\eta) = \frac{1}{(\xi+\eta)^2} \DB^2(\xi,\eta) 
\]
The bounds \eqref{A0-est1} and \eqref{A0-est2} are immediate
consequences of the estimates in Propositions~\ref{p:AB2-sp-lin},~\ref{p:DAB},
combined with Theorem~\ref{t:deltaA}.

It remains to prove the bound \eqref{A03-est}.
The cubic part of $\bfA_0(a,b)$  is given by 
\[
 \Delta_A \bfA_0^3(a,b) =  \DB^3(a,b) - 2[A^k, \partial_k \A_0^2]  - [\partial^k A_k + A_k^2, \A_0^2]   
\]
We will separately bound the three terms in the right hand side above in $ \ell^1 L^{2}$. 
We have already proved this for the first term in Proposition~\ref{p:DAB}, and the remaining two terms are 
similar using only the $ \dot H^{\frac32+2\smexp}$ bound for $\A_0$ (combined with the $\dot W^{\frac16-2\smexp,6}$ bound for $A$,
in the worst case), which in turn is proved similarly as \eqref{A0-est2}. \qedhere

\end{proof}

The above description of $A_0$ suffices for our description of the caloric Yang--Mills wave at heat-time $s = 0$.
However, we will also need to show that at $s > 0$,  $A(s)$ is a good approximate caloric Yang--Mills wave.
One difference between the two is that $A_0(s) \neq \bfA_0(A(s),B(s))$; this is because solving the 
$w$ equation with zero Cauchy data at time $0$, respectively zero Cauchy data at time $s$, yields 
different results. Nevertheless, we need to compare the two:
\begin{proposition}\label{p:A0-diff}
Let $A_{t, x}$ be a caloric Yang--Mills wave on $I \times \bbR^{4}$ satisfying $(A_{0}, A) \in C_{t}(I; \dot{H}^{1} \times \calC_{\hM})$ with $\calE[a] \leq \calE$. Let $A_{t,x}(s)$ be the corresponding dynamic Yang--Mills heat flow. Then for $A_0(s)$ we have the representation
\begin{equation}\label{A0-diff}
A_0(s) =  \bfA_0(A(s),B(s)) +    \bfA_{0;s}^2(A,B) + \bfA_{0;s}^3(A,B)
\end{equation}
where $\bfA_{0;s}^2(A,B) $ is a bilinear form 
\[
\bfA_{0;s}^2(A,B) = \Delta^{-1} \bfw_0^2(A,B,s).
\]
Moreover, under the additional assumption that all frequency envelope bounds are $(-\dlt, \dlt)$-admissible, the following bounds hold:
\begin{equation}\label{a0s}
\|  (1-s\Delta)^N \bfA_{0;s}^{2,3}(A,B) \|_{\dot H^{\frac12}} \lesssim_{\hM, \calE, N} 2^{-\frac{k(s)}2} (c_{k(s)}^{(0)})^2,
\end{equation}
\begin{equation}\label{a0s1}
\| (1-s\Delta)^N  \bfA_{0;s}^{2,3}(A,B) \|_{\dot H^{\frac32 - \smexp}} \lesssim_{\hM, \calE, N} 2^{-\smexp k(s)} (c_{k(s)}^{(1)})^2,
\end{equation}
respectively
\begin{equation}\label{a0s2}
\|(1-s\Delta)^N  \bfA_{0;s}^3(A,B) \|_{\dot H^{2-\frac{1}{12}}} \lesssim_{\hM, \calE, N} 2^{-\frac{k(s)}{12}} (c_{k(s)}^{(2)})^2  c_{k(s)}^{(4)}.
\end{equation}
\end{proposition}
In the next section, the three bounds above would be dynamically accompanied by $L^\infty_t,L^{2}_t$ respectively $L^1_t$.

\begin{proof}
Denote by $\tw$ the solutions to the $w$ equations \eqref{eq:ymhf-w2}-\eqref{eq:ymhf-w} 
but with initial data $\tw(s) = 0$. Then we have 
\[
A_0(s) -   \bfA_0(A(s),B(s)) = \int_s^\infty (w_0-\tw_0)(s') ds'.
\]
The function $z_0 = w_0-\tw_0$ solves the homogeneous heat equation
\[
(\partial_s - \Delta_A)z_0 = 0, \qquad z_0(s) = \bfw_0(s).
\]
For $w_0(s)$ we can use the $\dot H^{-\frac32}$, respectively $\dot{H}^{-\frac12-\smexp}$ bounds
in \eqref{w-linf} and \eqref{w-l2} to estimate $z_0(s')$ in the same spaces in a parabolic fashion.
Then \eqref{a0s} and \eqref{a0s1} directly follow. The same applies to the contribution 
of  $\bfw_0^3(s)$ in \eqref{a0s2}. It remains to consider the contribution of $\bfw_0^2(s)$ in \eqref{a0s2}.
This corresponds to replacing $z_0$ above by the solution $z_0^2$ to
\[
(\partial_s - \Delta_A)z_0^2 = 0, \qquad z_0(s) = \bfw_0^2(s).
\]
We consider to the expansion
\[
z_0^2(s') = e^{s' \Delta} \bfw^2_0(s) + \int_{s}^{s'} e^{(s''-s)\Delta} ([A,\covD z^2_0] + [\covD A,z^2_0])(s') ds''.
\]
For $z_0^2$, we use the  $\dot H^{-\frac12-\smexp}$ derived from \eqref{w-l2}, and for $A$ we use the $\dot W^{\smaxp, \pmaxp}$ norm.  \qedhere
\end{proof}

\subsection{The analysis of \texorpdfstring{$\covD^0 A_0$}{DA0}}

Here we have a representation as follows:
\begin{proposition}\label{p:DA0}
Let $A_{t, x}$ be a caloric Yang--Mills wave on $I \times \bbR^{4}$ satisfying $(A_{0}, A) \in C_{t}(I; \dot{H}^{1} \times \calC_{\hM})$ with $\calE[a] \leq \calE$. Then for $\partial^0 A_0$ we have the representation
\begin{equation}
\partial^0 A_0 =  \DA_0^2(B,B) + \DA_0^3(A,B)
\end{equation}
where the two terms are quadratic, respectively cubic and higher in $A,B$, 
and  $\DA_0^2(B,B)$ takes the form
\begin{equation*}
\DA_0^2(B,B) = -2 \lap^{-1} \bfQ(B, B).
\end{equation*}
Further, they satisfy the 
bounds
\begin{equation}\label{bfb0}
\| P_k \DA_0^{2,3} \|_{L^2} \lesssim_{\hM, \calE} c_k^{(0)} c_k^{(0)[\dlt_{0}]}
\end{equation}
\begin{equation}\label{bfb01}
\| P_k \DA_0^{2,3} \|_{\dot H^{\frac12}} \lesssim_{\hM, \calE}   c_k^{(1)} c_k^{(1)[\dlt_{0}]}
\end{equation}
\begin{equation}\label{bfb02}
\|P_k \DA_0^3 \|_{\dot H^1} \lesssim_{\hM, \calE}  c_k^{(2)} c_k^{(2)[\dlt_{0}]} c_k^{(4)[\dlt_{0}]} + (c_k^{(2)[\dlt_{0}]})^2 c_k^{(4)}
\end{equation}
as well as the corresponding difference bounds (cf. Proposition~\ref{p:A0}).
\end{proposition}
\begin{proof}
This is obtained by integrating the previous representation and bounds for $w_k$ via the formula 
\eqref{eq:l-caloric-d0a0}.
Precisely, we have
\[
\partial^0 A_0 = \int_{0}^\infty [A^\nu,w_\nu] + \partial_j w_j \, ds.
\]
The first term above is cubic, and it suffices to  combine $A$ and $w$ bounds 
($\dot W^{\frac16-2\delta}$ and $\dot H^{-\frac12+\delta}$ in the worst case for \eqref{bfb02}).

So it remains to consider the $ \partial_j w_j $ term. For the quadratic part 
we integrate the symbol of $\bfW(s)$
\[
\int_0^\infty \int_0^{s_0}  e^{-s(\xi^2+\eta^2)} e^{(s_0-s)(\xi+\eta)^2} ds ds_0 = \frac{1}{(\xi^2+\eta^2)(\xi+\eta)^2}
\]
which combines with the argument of $W$, namely $\partial^j [ B^l,\partial_j B_l] =  [ B^l,\Delta B_l]$.
It remains to account for the cubic term in $w$, for which we use the bounds in Proposition~\ref{p:w}.
This is exactly the same argument as for $A_0$.
\end{proof}

\begin{remark} 
 In the proof of the preceding proposition, we can also obtain the quadratic symbol in a more direct
  fashion, by returning to the $A_0$ equation. Retaining only
  quadratic terms, we have
\[
\partial_0 A_0 \approx \partial_0 \bfA_0(a,b) = \bfA_0(b,b) +  \bfA_0(a,\partial_t b) 
\approx \bfA_0(b,b) +  \bfA_0(a,\Delta a) 
\]
At the symbol level we get for the first expression after antisymmetrization
\[
\frac{\xi^2-\eta^2}{(\xi+\eta)^2 (\xi^2+\eta^2)}
\]
whereas the second expression vanishes after antisymmetrization.
\end{remark}

As in the case of $A_0$, we also need to compare $\covD^0 A_0(s)$ with 
$\DA_0(A(s),B(s))$.

\begin{proposition}\label{p:DA0-diff}
Let $A_{t, x}$ be a caloric Yang--Mills wave on $I \times \bbR^{4}$ satisfying $(A_{0}, A) \in C_{t}(I; \dot{H}^{1} \times \calC_{\hM})$ with $\calE[a] \leq \calE$. Let $A_{t,x}(s)$ be the corresponding dynamic Yang--Mills heat flow. 
Then for $\covD^0A_0(s)$ we have the representation
\begin{equation}\label{DA0-diff}
\covD^0A_0(s) =  \DA_0(A(s),B(s)) +    \DA_{0;s}^2(A, B) + \DA_{0;s}^3(A, B)
\end{equation}
where $\DA_{0;s}^2(A, B) $ is a bilinear form 
\[
\DA_{0;s}^2(A, B) = \Delta^{-1} \partial^k \bfw_k^2(A, B,s).
\]
Moreover, under the additional assumption that all frequency envelope bounds are $(-\dlt, \dlt)$-admissible, the following bounds hold:
\begin{equation}\label{bfb0s}
\| (1+s \lap)^{N} \DA_{0;s}^{2,3} \|_{\dot H^{-\frac12}} \lesssim_{\hM, \calE, N} 2^{-\frac{k(s)}2} (c_{k(s)}^{(0)})^2
\end{equation}
\begin{equation}\label{bfb0s1}
\| (1+s \lap)^{N} \DA_{0;s}^{2,3} \|_{\dot H^{\frac12-\smexp}} \lesssim_{\hM, \calE, N}  2^{-\smexp k(s)} (c_{k(s)}^{(1)})^2
\end{equation}
respectively
\begin{equation}\label{bfb0s2}
\| (1+s \lap)^{N} \DA_{0;s}^3 \|_{\dot H^{1-\frac1{12}}} \lesssim_{\hM, \calE, N}  2^{-\frac{k(s)}{12}}(c_{k(s)}^{(2)})^2 c_{k(s)}^{(4)}.
\end{equation}
\end{proposition}

\begin{proof}
The proof is similar to that of Proposition~\ref{p:A0-diff} for $\bfA_{0;s}^{2,3}$; we omit the details.
\end{proof}

\subsection{Turnabout: from curvature to caloric data}
Throughout this section so far, we have adopted the viewpoint that
$(a,b)$ should be considered as the canonical initial data
set. However, we also briefly need to turn the tables, and prove an
estimate for caloric data $(a,b)$ and its caloric flow which is
derived from information about initial curvature $(f,e)$. This is one
of the end results of this paper, which will be used in \cite{OTYM2}
to transfer small ``inhomogeneous energy dispersion'' information from
$(f, e)$ to $(A, B)$.

\begin{proposition} \label{prop:ed}
Let $c_{k}$ be a $(-\dlt_{0}, \dlt_{0})$ frequency envelope for $(a, b)$ in $\dot{H}^{1} \times L^{2}$, and let $d_{k}$ be a $(-\dlt_{0}, \dlt_{0})$ frequency envelope for $(f, e)$ in $\dot{W}^{-2, \infty}$. Then the following bounds hold:
\begin{align} 
2^{-k} \nrm{P_{k} A(s)}_{L^{\infty}} + 2^{-2k} \nrm{P_{k} B(s)}_{L^{\infty}} \aleq_{\hM, \calE, N} & (d_{k})^{\frac{1}{2}} (1+2^{2k} s)^{-N},
\label{ed-AB}
 \\
\nrm{P_{k} \rd^{j} A_{j}(s)}_{L^{2}} + \nrm{P_{k} \rd^{j} B_{j}(s)}_{\dot{H}^{-1}} \aleq_{\hM, \calE, N} & (d_{k})^{\frac{1}{2}} c_{k} (1+2^{2k} s)^{-N},
\label{ed-DAB}
 \\
\nrm{(1+s \lap)^{N} \bfA(s)}_{L^{2}} + \nrm{(1+s \lap)^{N} \bfB(s)}_{\dot{H}^{-1}} \aleq_{\hM, \calE, N}  & 2^{-k(s)} (d_{k(s)})^{\frac{1}{2}} c_{k(s)}.\label{ed-bfAB}
\end{align}
\end{proposition}
One can view this as a non-symmetric variant of Propositions~\ref{p:main-sp}, \ref{p:AB2-sp-lin} 
and \ref{p:DAB}. Here, there is no need to consider more general $(-\dlt_{0}, S)$ frequency envelopes.
\begin{proof}
We proceed in several steps, omitting the dependence of implicit constants on $\hM, \calE, N$:

\pfstep{ Step 1: $F_{ij}$ and $F_{0j}$ bounds}  For the curvature components we have the covariant heat equations 
\eqref{eq:YMHF-Fij+} therefore we are in a position to apply the bounds in Theorem~\ref{t:heatB}. Unfortunately the 
$\dot W^{-2,\infty}$ norm is borderline inadmissible there. To rectify this we work instead in the intermediate $\dot W^{-1,4}$
norm, for which by interpolation we have the initial data bounds
\[
\| P_k (f,e)\|_{\dot W^{-1,4}} \lesssim (c_k d_k)^{\frac12}
\]
By Theorem~\ref{t:heatB} this yields the corresponding parabolic bounds for their caloric flows,
\begin{equation}\label{ed:f}
\| P_k F_{\alpha \beta}\|_{\dot W^{-1,4}} \lesssim (c_k d_k)^{\frac12} (1+2^{2k} s)^{-N}.
\end{equation}

\pfstep{ Step 2: $A_{j}$ and $A_{0}$ bounds} These are obtained by integrating from infinity
\[
A_\alpha(s) = - \int_{s}^\infty D^j F_{j\alpha}(s_1) d s_1.
\]
We estimate in $L^4$ using \eqref{ed:f}, Bernstein's inequality and the Littlewood-Paley trichotomy 
\[
\begin{split}
\| P_k D^j F_{j\alpha}(s_1) \|_{L^4} & \  \lesssim \| P_k \partial^j F_{j\alpha}(s_1) \|_{L^4}+ \| P_k [A^j, F_{j\alpha}(s_1)] \|_{L^4}
\\
\lesssim & \ 2^{2k} \left[(c_k d_k)^{\frac12} (1+2^{2k} s_1)^{-N} +   \sum_{j > k} c_j^\frac32  d_j        (1+2^{2j} s_1)^{-N}\right] .
\end{split}
\]
After integration in $s_1$ this yields 
\begin{equation}\label{ed:a}
\| P_k A_{\alpha}(s)\|_{L^4} \lesssim (c_k d_k)^{\frac12} (1+2^{2k} s)^{-N}.
\end{equation}

\pfstep{ Step 3: $B$ bounds} Recalling that 
\[
B_j = F_{0j} + D_j A_0 = F_{0j} + \partial_j A_0 +[A_j,A_0] 
\]
we use \eqref{ed:f} and \eqref{ed:a} for the first two terms and combine $\dot H^1$ and $L^4$ bounds 
for the last term to obtain
\begin{equation}\label{ed:b}
\| P_k B_{j}(s)\|_{\dot W^{-1,4}} \lesssim (c_k d_k)^{\frac12} (1+2^{2k} s)^{-N}.
\end{equation}
By Bernstein's inequality, this bound together with \eqref{ed:a} complete the proof of \eqref{ed-AB}.

\pfstep{ Step 4: the remaining bounds \eqref{ed-DAB} and \eqref{ed-bfAB}}
These follow from the estimates \eqref{cal-linf-af} and \eqref{cal-linf-da} by choosing $(\sigma_1,p_1) = (0,2)$,
and appropriate $(\sigma,p)$ interpolating between $(1,2)$ and $(0,4)$.
\end{proof}

\section{The wave equation for \texorpdfstring{$A_x$}{Ax}} \label{s:wave}

Our main goal here is to interpret the hyperbolic Yang--Mills equation in the caloric gauge as 
a system of nonlinear wave equations for $A_x$. To be more precise, we seek to formulate the equations in a 
form where all the quadratic terms are explicit, while the cubic terms satisfy favorable frequency envelope
bounds which only involve the non-endpoint Strichartz type norms for $A$. 

In this section, by \emph{time} we always refer to the hyperbolic-time $t$. Accordingly, in this section the shorthand $L^{q} L^{p}$ means the space-time norm $L^{q}_{t} L^{p}_{x}$, \emph{not} the space-heat-time norm $L^{q}_{s} L^{p}_{x}$ as it were in the prior sections.
Otherwise, the conventions fixed at the end of Section~\ref{subsec:lp} are still in effect.

For economy of notation, we introduce the following definition:
\begin{definition} \label{def:env-pres} Let $X, Y$ be dyadic norms.
  \begin{itemize}
  \item A map $\bfF : X \to Y$ is said to be \emph{envelope-preserving
      of order $\geq n$} ($n \in \bbN$ with $n \geq 2$) if the
    following property holds: Let $c$ be a $(-\dlt_{0}, S)$
    frequency envelope for $a$ in $X$. Then
    \begin{equation*}
      \nrm{P_{k} \bfF(a)}_{Y} \aleq_{\nrm{a}_{X}} (c_{k}^{[\dlt_{0}]})^{n-1} c_{k}.
    \end{equation*}
  \item A map $\bfF: X \to Y$ is said to be \emph{Lipschitz
      envelope-preserving of order $\geq n$} if, in addition to being
    envelope preserving of order $\geq n$, the following additional
    property holds: Let $c$ be a common $(-\dlt_{0}, \dlt_{0})$ frequency
    envelopes for $a_{1}$ and $a_{2}$ in $X$, and let $d$ be a
    $(-\dlt_{0}, \dlt_{0})$ frequency envelope for $a_{1} - a_{2}$ in $X$. Then
    \begin{equation*}
      \nrm{P_{k}(\bfF(a_{1}) - \bfF(a_{2}))}_{Y_{k}} 
      \aleq_{\nrm{a_{1}}_{X}, \nrm{a_{2}}_{X}}  c_{k}^{n-1} e_{k}.
    \end{equation*}
    where $e_{k} = d_{k} + c_{k} (c \cdot d)_{\leq k}$,
  \end{itemize}
\end{definition}

Our main result is as follows:

\begin{theorem} \label{thm:main-eq} Let $A_{t,x} = (A_{0}, A) \in
  C_{t}(I; \dot{H}^{1} \times \calC_{\hM})$ with $(\rd_{t} A_{0},
  \rd_{t} A) \in C_{t}(I; L^{2} \times T^{L^{2}}_{A(t)} \calC_{\hM})$
  be a solution to \eqref{ym} with energy $\nE$. Then its spatial components $A =
  A_{x}$ satisfy an equation of the form
  \begin{equation}\label{eq:main-wave}
    \Box_{A} A_{j} = \P_{j} [A, \rd_{x} A] + 2\Delta^{-1} \partial_{j} \bfQ (\rd^{\alp} A, \rd_{\alp} A) + R_{j}(A),
  \end{equation}
  together with a compatibility condition
  \begin{equation}\label{eq:main-compat}
    \partial^{\ell} A_{\ell} = \DA(A) :=  \bfQ(A, A) + \DA^{3}(A).
  \end{equation}
  Moreover, the temporal component $A_{0}$ and its time derivative
  $\rd_{t} A_{0}$ admit the expressions
  \begin{align}
    A_{0} =& \bfA_{0}(A) := \lap^{-1}[A, \rd_{t} A] + 2 \lap^{-1} \bfQ(A, \rd_{t} A) + \bfA_{0}^{3}(A), \label{eq:main-A0} \\
    \rd_{t} A_{0} =& \DA_{0}(A) := - 2 \lap^{-1} \bfQ (\rd_{t} A,
    \rd_{t} A) + \DA_{0}^{3}(A). \label{eq:main-DA0}
  \end{align}
  Here $\P$ is the Leray projector, and $\bfQ$ is the symmetric
 bilinear form with symbol as in \eqref{eq:Q-sym}.
  Moreover, $R_{j}(t)$, $\DA^{3}(t)$, $\bfA_{0}^{3}(t)$ and
  $\DA_{0}^{3}(t)$ are uniquely determined by $(A, \rd_{t} A)(t) \in
  T^{L^{2}} \calC$, and are Lipschitz envelope preserving maps of
  order $\geq 3$ on the following spaces: 
  \begin{align}
    R_{j}(t): & \ \dot{H}^{1} \to \dot{H}^{-1}, \label{eq:Rj-t} \\
    \DA^3(t): & \ \dot{H}^{1} \to L^2, \label{eq:DA-tri-t} \\
    \bfA_{0}^{3}(t): & \  \dot{H}^{1} \to \dot{H}^{1}, \label{eq:A0-tri-t} \\
    \DA_{0}^{3}(t) : & \  \dot{H}^{1} \to L^{2}. \label{eq:DA0-tri-t}
  \end{align}
  Finally, on any interval $I \subseteq \bbR$, $R_{j}$, $\DA^{3}$,
  $\bfA_{0}^{3}$ and $\DA_{0}^{3}$ are Lipschitz envelope preserving
  maps of order $\geq 3$ (with bounds independent of $I$) on the
  following spaces:
  \begin{align}
    R_{j}: & \ \Str^{1}[I] \to L^{1} L^{2} \cap L^{2} \dot{H}^{-\frac{1}{2}}[I], \label{eq:Rj} \\
    \DA^3: & \ \Str^1[I] \to \blue{L^1 \dot H^1 \cap L^{2} \dot{H}^{\frac{1}{2}}}[I], \label{eq:DA-tri} \\
    \bfA_{0}^{3}: & \ \Str^1[I] \to  L^1 \dot H^2 \cap L^{2} \dot{H}^{\frac{3}{2}}[I], \label{eq:A0-tri} \\
    \DA_{0}^{3} : & \ \Str^{1}[I] \to  \blue{L^1 \dot H^1 \cap L^{2} \dot{H}^{\frac{1}{2}}}[I].
     \label{eq:DA0-tri}
  \end{align}
  All implicit constants depend on $\hM$ and $\nE$.
\end{theorem}

\begin{proof}
We expand the equations \eqref{ym} in terms of the connection $A$,
\[
\Box_A A_j = D^\alpha \partial_j A_\alpha = \partial_j \partial^\alpha A_\alpha+
[A^\alpha, \partial_j A_\alpha] 
\]
In the first term on the right we use the expressions for $\partial_0 A_0$ and $\partial^j A_j$,
\[
\Box_A A_j = \partial_j (\DA - \DA_0) + [A^\alpha, \partial_j A_\alpha] 
\]
We separate the quadratic and cubic terms to obtain 
\begin{equation}\label{wave-A}
\Box_A A_j = \partial_j (\DA^2 - \DA_0^2) + [A^\alpha, \partial_j A_\alpha] +  \partial_j (\DA^3 - \DA_0^3)
\end{equation}
Then we denote 
\begin{equation}
R_j (A_x,\partial_t A_x) = [A^0,\partial_j A_0] +  \partial_j (\DA^3 - \DA_0^3)
\end{equation}
To complete the proof of \eqref{eq:main-wave} we need to compare the above quadratic expressions with those
in \eqref{eq:main-wave}. We begin with the $A_x$ bilinear forms. From Proposition~\ref{p:DAB} we have the 
relation $\DA^2 = \bfQ(A_x,A_x)$ therefore the $A_x$ bilinear forms are 
\[
\begin{split}
[A^k, \partial_j A_k] + \partial_j Q(A^k,A_k) = & \  P_j [A^k, \partial_x A_k]  + 
2 \Delta^{-1} \partial_j \bfQ(\partial^\ell A^k, \partial_\ell A_k) 
\\ & \ + \Delta^{-1} \partial_j  [A^k, \Delta A_k]  +   \Delta^{-1} \partial_j [ \bfQ(\Delta A_k, A_k)+ \bfQ( A_k, \Delta A_k)]
\end{split}
\]
and the terms on the last line cancel in view of the expression \eqref{eq:Q-sym}. On the other hand for the bilinear 
term in $\partial_t A_x$ we have from Proposition~\ref{p:DA0}
\[
\DA_0^2(\partial_0 A, \partial_0 A) = 2 \Delta^{-1} \bfQ( \partial^0 A,\partial_0 A)
\]

Next we prove the  estimates for $R$, $\DA^3$, $\bfA_0^3$, $\A_0^3$ and $\DA_0$. 
In terms of subcritical Strichartz norms we will use the components 
\[
L^{q^{(j)}} \dot W^{\sgm^{(j)},p^{(j)}} \subset \Str^{1}
\]
If $\cstr  \in \ell^2$ is a $(-\delta,S)$ admissible $\Str^1$ frequency
envelope for $(A,B=\partial_t A)$ in a time interval $[0,T]$, then we
denote by $c_k^{(j)}(t)$ a (minimal) $(-\delta,S)$ admissible frequency envelope for
$(A(t),B(t))$ in $\dot W^{\sgm^{(j)},p^{(j)}}$. Then we observe that we
must have the relation
\begin{equation}\label{ckt}
\| c_k^{(j)}(t)\|_{L^{q^{(j)}}} \lesssim \cstr_k
\end{equation}
We will always use this relation in order to transition from the fixed time bounds in 
the previous section to the space-time bounds  here.

\pfstep{1.~The bounds for $\DA^3$}  The fixed time bound
\eqref{eq:DA-tri-t} is a direct consequence of \eqref{DAB3-0}, while
the Lipschitz property is due to the difference bound
\eqref{dDAB3-0}.  For the space-time bound \eqref{eq:DA-tri} we
first estimate separately the term $\DA^3$ using the bound
\eqref{DAB3} at fixed $t$,
\[
\|P_k \DA^3 (t)\|_{\dot H^1} \lesssim c_k^{(3)}(t)(c_k^{(3),[\delta_{0}]})^{2}(t).
\]
Since $q^{(3)} = 3$, by \eqref{ckt} this yields the space-time bound 
\[
\|P_k \DA^3 \|_{L^1 \dot H^1} \lesssim  \cstr_k (\cstrd_k)^2.
\]
Finally the $L^2 \dot H^\frac12$ bound is obtained similarly using \eqref{DAB3-1}.

\pfstep{2.~The bounds for $\A_0^3$} These follow as above but starting from
the bounds \eqref{A0-est1}-\eqref{A03-est}. For later use, we also note the 
quadratic bounds
\begin{equation}\label{eq:A02-re-t}
\A_0^2(t) : \dot H^1 \to \dot H^1
\end{equation}
\begin{equation}\label{eq:A02-re}
\A_0^2 : \Str^1[I] \to L^2 \dot H^\frac32[I]
\end{equation}
which are a consequence of  \eqref{A0-est1} and \eqref{A0-est2}.

\pfstep{3.~The bounds for $\DA_0^3$} Again the same argument applies, but now starting from 
\eqref{bfb0}-\eqref{bfb02}.

\pfstep{4.~The bounds for $R$} Given our definition of $R$ above, the bounds \eqref{eq:Rj-t} and \eqref{eq:Rj}
are a consequence of the similar bounds for $\DA^3$, $\DA_0^3$, $\A_0^3$ together with the estimates
\eqref{eq:A02-re-t} and \eqref{eq:A02-re} for $\A_0^2$. \qedhere

\end{proof}

For our study in subsequent work \cite{OTYM2}, \cite{OTYM2.5}, \cite{OTYM3}
of the large data hyperbolic Yang--Mills flow  will also need a hyperbolic evolution for the connection $A$ 
at a nonzero parabolic time $s > 0$. The added difficulty is that $A(s)$ no longer solves exactly the 
 hyperbolic Yang--Mills equation \eqref{ym}. Instead we have $\covD^{\bt}
F_{\alp \bt}(s) = w_{\alp} \neq 0$ in general. We expect the
``heat-wave commutator'' $w_{\alp}$ (called the Yang--Mills tension
field) to be concentrated primarily at frequency comparable to
$s^{-\frac{1}{2}}$.   Other errors are  also expected to have a similar concentration.  Precisely, we have

\begin{theorem} \label{thm:main-eq-s} Let $A_{t,x} = (A_{0}, A) \in
  C_{t}(I; \dot{H}^{1} \times \calC_{\hM})$ with $(\rd_{t} A_{0},
  \rd_{t} A) \in C_{t}(I; L^{2} \times T^{L^{2}}_{A(t)} \calC_{\hM})$
  be a solution to \eqref{ym} with energy $\nE$. Let $A_{t,x}(s) = A_{t,x}(t, x, s)$
  be the dynamic Yang--Mills heat flow development of $A_{t,x}$ in the
  caloric gauge. Then the spatial components $A(s) = A_{x}(s)$ of
  $A_{t,x}(s)$ satisfy an equation of the form
  \begin{equation}\label{eq:main-wave-s}
    \begin{aligned}
      \Box_{A(s)} A_{j}(s) =& \P_{j} [A(s), \rd_{x} A(s)] + 2\Delta^{-1} \partial_{j} \bfQ (\rd^{\alp} A(s), \rd_{\alp} A(s)) + R_{j}(A(s)) \\
      & + \P_{j} \bfw_{x}^{2}(\rd_{t} A, \rd_{t} A, s) + R_{j; s}(A)
    \end{aligned}
  \end{equation}
  together with the compatibility condition
  \begin{equation}\label{eq:main-compat-s}
    \partial^{\ell} A_{\ell}(s) = \DA(A(s)).
  \end{equation}
  Moreover, the temporal component $A_{0}(s)$ and its time derivative
  $\rd_{t} A_{0}(s)$ admit the expansions
  \begin{align}
    & \begin{aligned}
      A_{0}(s) =& \bfA_{0}(A(s)) + \bfA_{0; s}(A) \\
      :=&\bfA_{0}(A(s)) + \lap^{-1} \bfw_{0}^{2}(A, A, s) + \bfA_{0;
        s}^{3}(A),
    \end{aligned} \label{eq:A0-struct} \\
    & \begin{aligned} \rd_{t} A_{0}(s) = \DA_{0}(A(s)) + \DA_{0; s}(A)
    \end{aligned} \label{eq:DA0-struct}
  \end{align}
  Here $\P$, $\bfQ$, $R_{j}$, $\DA$, $\bfA_{0}$ and $\DA_{0}$ are as
  before, and $\bfw_{\alp}^{2}$ are defined as
  \begin{align}
    \bfw_{0}^{2}(A, B, s) =& - 2 \bfW(\rd_{t} A, \lap B, s), \label{eq:w2-0-def}\\
    \bfw_{j}^{2}(A, B, s) =& - 2 \bfW(\rd_{t} A, \rd_{j} \rd_{t} B - 2
    \rd_{x} \rd_{t} B_{j}, s), \label{eq:w2-x-def}
  \end{align}
  where $\bfW(\cdot, \cdot, s)$ is a bilinear form with symbol as in \eqref{W-sym}.
  
  Moreover, $R_{j; s}(t)$, $\bfA_{0; s}^{3}(t)$ and $\DA_{0;s}(t)$ are
  uniquely determined by $(A, \rd_{t} A)(t) \in T^{L^{2}} \calC$ for
  each $s > 0$, and satisfy the following properties
  \begin{itemize}
  \item $R_{j; s}(t) : \dot{H}^{1} \to \dot{H}^{-1}$ is a Lipschitz
    map with output concentrated at frequency $s^{-\frac{1}{2}}$. More
    precisely,
    \begin{equation} \label{eq:Rj-s-t} (1- s \lap)^N R_{j; s}(t) :
      \dot{H}^{1} \to \blue{2^{-\smexp k(s)} \dot{H}^{-1-\smexp}}.
    \end{equation}

  \item $\bfA_{0; s}^{3}(t): \dot{H}^{1} \to \dot{H}^{1}$ is a
    Lipschitz map with output concentrated at frequency
    $s^{-\frac{1}{2}}$, i.e.,
    \begin{equation} \label{eq:A0-tri-s-t} (1- s \lap)^N \bfA_{0;
        s}^3(t) : \dot{H}^{1} \to \blue{2^{-\smexp k(s)}
      \dot{H}^{1-\smexp}}
    \end{equation}

  \item $\DA_{0; s}(t): \dot{H}^{1} \to L^{2}$ is a Lipschitz map with
    output concentrated at frequency $s^{-\frac{1}{2}}$, i.e.,
    \begin{equation} \label{eq:DA0-tri-s-t} (1- s\lap)^N \DA_{0; s}(t)
      : \dot{H}^{1} \to \blue{2^{-\smexp k(s)} \dot{H}^{-\smexp}}.
    \end{equation}
  \end{itemize}
  Finally, on any time interval $I \subseteq \bbR$ (with bounds
  independent of $I$), $R_{j; s}$, $\bfA_{0; s}^{3}$ and $\DA_{0; s}$
  satisfy the following properties:
  \begin{itemize}
  \item $R_{j; s} : \Str^{1}[I] \to L^{1} L^{2} \cap L^{2}
    \dot{H}^{-\frac{1}{2}} [I]$ is a Lipschitz map with output
    concentrated at frequency $s^{-\frac{1}{2}}$, i.e.,
    \begin{equation} \label{eq:Rj-s} (1- s \lap)^N R_{j; s} : \Str^1
      [I] \to \blue{2^{-\smexp k(s)} (L^1 \dot H^{-\smexp} \cap
      L^{2} \dot{H}^{-\frac{1}{2}-\frac{1}{12}})[I]}
    \end{equation}

  \item $\bfA_{0; s}^{3}: \Str^1 [I] \to L^1 \dot H^{2} \cap L^{2}
    \dot{H}^{\frac{3}{2}} [I]$ is a Lipschitz map with output
    concentrated at frequency $s^{-\frac{1}{2}}$, i.e.,
    \begin{equation} \label{eq:A0-tri-s} (1- s \lap)^N \bfA_{0; s}^3 :
      \Str^1[I] \to \blue{2^{-\smexp k(s)} (L^{1} \dot
      H^{2-\smexp} \cap L^{2} \dot{H}^{\frac{3}{2} -
        \smexp})[I]}
    \end{equation}

  \item $\DA_{0; s}: \Str^1[I] \to L^{2} \dot H^{\frac{1}{2}} [I]$ is
    a Lipschitz map with output concentrated at frequency
    $s^{-\frac{1}{2}}$, i.e.,
    \begin{equation} \label{eq:DA0-tri-s} (1- s\lap)^N \DA_{0; s} :
      \Str^1[I] \to \blue{2^{-\smexp k(s)} L^{2}
      \dot{H}^{\frac{1}{2}-\smexp}[I]}
    \end{equation}
  \end{itemize}
  All implicit constants depend on $\hM$ and $\nE$.
\end{theorem}

\begin{remark}
 Compared with the prior theorem, here we have additional contributions
$R_0$, $\bfA_{0s}$ and $\DA_{0s}$ as well as the $\bfw$ terms. These have the downside
that they depend on $A(0)$ and $\partial_t A(0)$ rather than 
 $A(s)$ and $\partial_t A(s)$. The redeeming feature is that  these terms will not only 
be small due to the energy dispersion, but also, critically,  concentrated at frequency $s^{-\frac12}$.
\end{remark}

\begin{remark}
 The other change here is due to the inhomogeneous terms $\bfw^2$; these are matched 
in the $A_k(s)$ and the $A_0(s)$ equations, and will interact in the trilinear analysis for the hyperbolic Yang--Mills flow.
\end{remark}

\begin{proof}
  Using \eqref{A0-diff} and \eqref{DA0-diff} we obtain that instead of
  the equation \eqref{wave-A} we now have the equation
\begin{equation}\label{wave-As}
\begin{split}
\Box_{A(s)} A_j(s)  = & \  \partial_j (\DA^2 - \DA_0^2) + [A^\alpha, \partial_j A_\alpha] +  \partial_j (\DA^3 - \DA_0^3) 
\\    & \  + (\bfw^2_j(s) - \partial_j \DA_{0;s}^2) +(\bfw^3_j(s) - \partial_j \DA_{0;s}^3)
\end{split}
\end{equation}
where $\DA^3$ and $\DA_0^3$ now depend on $A_k(s)$, $\partial_0 A_k(s)$ and
$w(s)$.  On the second line we have  separated  the effect of $w$, which is nonzero at $s > 0$.  

The terms on the first line are as in the previous theorem. For the second line,
we define 
\[
R_{j;s} = \bfw^3_j(s) - \partial_j \DA_{0s}^3.
\]
For the quadratic part, on the other hand, using Proposition~\ref{p:DA0-diff} we have
\[
\bfw^2_j(s) - \partial_j \DA_{0;s}^2 =  \bfw^2_j(s) - \Delta^{-1} \partial_j \partial_k  \bfw^2_j(s) = 
P_j \bfw^2_x
\]
The remaining algebraic relations \eqref{eq:A0-struct} and \eqref{eq:DA0-struct} are obtained from
Propositions~\ref{p:A0-diff} and \ref{p:DA0-diff}. We now consider the estimates in the theorem:

\pfstep{1. The $\bfw^3_j(s)$ component of $R_{j;s}$} The corresponding parts of the bounds \eqref{eq:Rj-s-t} 
and \eqref{eq:Rj-s}  follow from the estimates \eqref{w-linf}-\eqref{w-l1} in Proposition~\ref{p:w}. 

\pfstep{2. The $\DA_{0;s}^3$ component of $R_{j;s}$} Here we use instead the bounds \eqref{bfb0s}-\eqref{bfb0s2}.

\pfstep{3. The $\A_{0;s}$ bound} The estimates \eqref{eq:A0-tri-s} and \eqref{eq:A0-tri} are consequences of the 
bounds \eqref{a0s}-\eqref{a0s2}.

\pfstep{4. The $\DA_{0;s}$ bound} The estimates \eqref{eq:DA0-tri-s} and \eqref{eq:DA0-tri} are consequences of the 
bounds \eqref{bfb0s}-\eqref{bfb0s2}.

\end{proof}

\bibliographystyle{ym}
\bibliography{ym}

\end{document}